\documentclass[10pt, a4paper, oneside]{article}

\usepackage[scaled]{helvet} % ss
\usepackage{courier} % tt
\usepackage{eulervm} % a better implementation of the euler package (not in gwTeX)
\usepackage{lscape}

\normalfont
\usepackage[english]{babel}
\usepackage[utf8]{inputenc}
\usepackage{indentfirst}				% rientra il primo capoverso di ogni sezione
\usepackage{booktabs}				% tabelle
\usepackage{multirow}					
\usepackage{tabularx}					% tabelle di larghezza prefissata
\usepackage{graphicx}					% immagini
\usepackage{caption}					% didascalie
\usepackage{amsmath,amssymb,amsthm,mathrsfs,mathtools}	% matematica
\usepackage{bm,faktor,braket,stmaryrd,esint}
\usepackage[english]{varioref}			% riferimenti completi della pagina
\usepackage{mparhack,relsize}			% finezze tipografiche
\usepackage[dvipsnames]{xcolor}			% colori
\usepackage{hyperref}					% collegamenti ipertestuali
\usepackage{bookmark}					% segnalibri
\usepackage{tikz}					%immagini PGF
\usepackage{fullpage}
\usepackage{tcolorbox}
\usepackage[export]{adjustbox}
\usepackage{rotating}
\usepackage{pbox}
\newsavebox{\mytable}

\usepackage[textsize=footnotesize]{todonotes}
\usepackage{tikz-cd}
\usetikzlibrary{arrows,decorations.pathreplacing,decorations.markings,calc}

\usepackage{enumitem}
\numberwithin{equation}{section}
\addto\captionsenglish{\renewcommand{\figurename}{Table}}
\usepackage[labelsep=endash]{caption}

\newcommand{\dd}{\mathrm{d}}
\newcommand{\NN}{\mathbb{N}}

\newcommand{\Z}{\mathbb{Z}}

\newcommand{\R}{\mathbb{R}}
\newcommand{\CC}{\mathbb{C}}

\let\Im\relax
\DeclareMathOperator{\Im}{Im}
\let\Re\relax
\DeclareMathOperator{\Re}{Re}

\DeclareMathOperator*{\Res}{Res}

\newcommand{\beq}{\begin{equation}}
\newcommand{\eeq}{\end{equation}}
\newcommand{\bea}{\begin{eqnarray}}
\newcommand{\eea}{\end{eqnarray}}

\newcommand{\bbraket}[1]{\llbracket #1 \rrbracket}

\theoremstyle{plain}
\newtheorem{thm}{Theorem}[section]
\newtheorem{prop}[thm]{Proposition}
\newtheorem{defn}[thm]{Definition}

\newtheorem{lem}[thm]{Lemma}
\newtheorem{cor}[thm]{Corollary}

\newtheorem{conj}[thm]{Conjecture}
\newtheorem*{que*}{\textcolor{BrickRed}{Question}}

\theoremstyle{definition}
\newtheorem{rem}[thm]{Remark}

\definecolor{webbrown}{rgb}{0.65, 0.16, 0.16}
\newcommand{\ben}{\begin{eqnarray*}}
\newcommand{\een}{\end{eqnarray*}}
\newcommand{\be}{\begin{equation}}
\newcommand{\ee}{\end{equation}}
\newcommand{\half}{\frac{1}{2}}

\makeatletter
\renewenvironment{abstract}{%
    \if@twocolumn
      \section*{\abstractname}%
    \else \normalsize %% <- here I've removed \small
      \begin{center}%
        {\bfseries \normalsize\abstractname\vspace{\z@}}%  %% <- here I've added \Large
      \end{center}%
      \quotation
    \fi}
    {\if@twocolumn\else\endquotation\fi}
\makeatother

\hypersetup{
colorlinks=true, linktocpage=true, pdfstartpage=1, pdfstartview=FitV,breaklinks=true, pdfpagemode=UseNone, pageanchor=true, pdfpagemode=UseOutlines,plainpages=false, bookmarksnumbered, bookmarksopen=true, bookmarksopenlevel=1,hypertexnames=true, pdfhighlight=/O,urlcolor=webbrown, linkcolor=RoyalBlue, citecolor=ForestGreen}

\title{\textsc{Topological recursion for Masur--Veech volumes}\vspace{1cm}}
\date{\vspace{-5ex}}

\author{
	J\o{}rgen Ellegaard Andersen\footnote{Center for Quantum Mathematics, Danish Institute for Advanced Study, Campusvej 55, 5230 Odense, Denmark.}\,\,,
	Ga\"etan Borot\footnote{Max Planck Institut f\"ur Mathematik, Vivatsgasse 7, 53111 Bonn, Germany.}\,\,,
	S\'everin Charbonnier\footnotemark[2]\,\,,\\
	Vincent Delecroix\footnotemark[2]\,\,\footnote{LaBRI, UMR 5800, Batiment A30, 351 cours de la Lib\'eration, 33405 Talence Cedex, France.}\,\,,
	Alessandro Giacchetto\footnotemark[2]\,\,,\\
	Danilo Lewa\'nski\footnotemark[2]\,\,,
	Campbell Wheeler\footnotemark[2]}

\begin{document}

\maketitle 

\vspace{2cm}

\begin{abstract}
	We study the Masur--Veech volumes $MV_{g,n}$ of the principal stratum of the moduli space of quadratic differentials of unit area on curves of genus $g$ with $n$ punctures. We show that the volumes $MV_{g,n}$ are the constant terms of a family of polynomials in $n$ variables governed by the topological recursion/Virasoro constraints. This is equivalent to a formula giving these polynomials as a sum over stable graphs, and retrieves a result of~\cite{Delecroix} proved by combinatorial arguments. Our method is different: it relies on the geometric recursion and its application to statistics of hyperbolic lengths of multicurves developed in~\cite{GRpaper}. We also obtain an expression of the area Siegel--Veech constants in terms of hyperbolic geometry. The topological recursion allows numerical computations of Masur--Veech volumes, and thus of area Siegel--Veech constants, for low $g$ and $n$, which leads us to propose conjectural formulas for low $g$ but all $n$. We also relate our polynomials to the asymptotic counting of square-tiled surfaces with large boundaries.
\end{abstract}

\thispagestyle{empty}
\bigskip

\newpage

\tableofcontents

\newpage

%***************************************
\section{Introduction}
\label{SIntro}
%***************************************

We consider two facets of the geometry of surfaces. On the one hand, hyperbolic geometry with associated Teichm\"uller space and Weil--Petersson metric, and on the other hand, flat geometry associated with quadratic differentials and the Masur--Veech measure. We will show that invariants of flat geometry of surfaces, namely the Masur--Veech volumes and the area Siegel--Veech constants, can be expressed as asymptotics of certain statistics of multicurves on hyperbolic surfaces. Using the geometric recursion developed in~\cite{GRpaper} for these statistics, we prove that the Masur--Veech volumes satisfy some form of topological recursion \`a la Eynard--Orantin~\cite{EORev}.

%***************************************
\subsection{The Masur--Veech volumes}
%***************************************

We will let $\Sigma$ denote a smooth, compact, oriented, not necessarily connected surface, which can be closed, punctured or bordered. We consider those cases to be mutually exclusive and we shall indicate which situation is considered when necessary. When $\Sigma$ is not closed, the punctures or boundary components are labelled $\partial_1\Sigma,\ldots,\partial_{n}\Sigma$. We assume that $\Sigma$ is stable, \textit{i.e.} the Euler characteristic of each connected component is negative. We say that $\Sigma$ has type $(g,n)$ if it is connected of genus $g$ with $n$ boundary components. We use $P$ (respectively $T$) to refer to surfaces with the topology of a pair of pants (resp. of a torus with one boundary component).

\medskip

The Teichm\"uller space $\mathcal{T}_{\Sigma}$ of a bordered $\Sigma$ is the set of hyperbolic metrics on $\Sigma$ such that the boundary components are geodesic, modulo diffeomorphisms of $\Sigma$ that restrict to the identity on $\partial \Sigma$ and which are isotopic to ${\rm Id}_{\Sigma}$ among such. The Teichm\"uller space $\mathcal{T}_{\Sigma}$ fibers over $\mathbb{R}_{+}^n$ and we denote the fiber over $L = (L_1, \ldots, L_n) \in \mathbb{R}_{+}^n$ by $\mathcal{T}_{\Sigma}(L)$. For a surface of type $(g,n)$, $\mathcal{T}_{\Sigma}(L)$ is a smooth manifold of dimension $6g - 6 + 2n$. Here $\mathbb{R}_{+}$ is the positive real axis, excluding $0$.
In several places, we will also consider $L_i=0$, which means that the $i$-th boundary corresponds to a cusp for the hyperbolic metric. The slice $\mathcal{T}_\Sigma(0,\ldots,0) = \mathfrak{T}_{\Sigma}$ is the Teichm\"uller space of complete hyperbolic metrics of finite area on $\Sigma - \partial\Sigma$, which is then considered as a punctured surface.
$\mathfrak{T}_{\Sigma}$ can also be seen as the space of Riemann surface structures on the punctured surface.
The cotangent bundle to $\mathfrak{T}_{\Sigma}$ is isomorphic to the bundle $Q\mathfrak{T}_\Sigma$ of holomorphic integrable quadratic differentials on the punctured surface. For any $(\sigma,q) \in Q\mathfrak{T}_\Sigma$, the quadratic differential $q$ has either a removable singularity or a simple pole at each puncture of $\Sigma$. These spaces also exist for closed surfaces.

\medskip

The mapping class group ${\rm Mod}_{\Sigma}$ is the group of isotopy classes of orientation-preserving diffeomorphisms of $\Sigma$. It admits as subgroup the pure mapping class group ${\rm Mod}_{\Sigma}^{\partial}$, consisting of the isotopy classes of diffeomorphisms that restrict to the identity on $\partial \Sigma$. The pure mapping class group acts on the Teichm\"uller spaces $\mathcal{T}_\Sigma(L)$ and on $\mathfrak{T}_{\Sigma}$ and on the space of quadratic differentials $Q\mathfrak{T}_\Sigma$. This action is properly discontinuous and the quotient spaces $\mathcal{M}_{\Sigma}(L)$, $\mathfrak{M}_{\Sigma}$ and $Q\mathfrak{M}_\Sigma$ are smooth orbifolds, called respectively the moduli space of bordered surfaces, the moduli space of punctured surfaces, and the moduli space of quadratic differentials. The moduli spaces for all surfaces of given type $(g,n)$ are all canonically isomorphic and simply denoted by $\mathcal{M}_{g,n}(L)$, $\mathfrak{M}_{g,n}$ and $Q\mathfrak{M}_{g,n}$.

\medskip

The spaces $\mathcal{T}_\Sigma(L)$ for $L \in \mathbb{R}_{+}^n$ and $\mathfrak{T}_{\Sigma}$ are endowed with the Weil--Petersson measures $\mu_{{\rm WP}}$. These measures are invariant under the action of the pure mapping class group and descend to the quotients $\mathcal{M}_{g,n}(L)$ and $\mathfrak{M}_{g,n}$. If $Y_{\Sigma}$ is a ${\rm Mod}_{\Sigma}^{\partial}$-invariant function on $\mathcal{T}_{\Sigma}$, we denote by $Y_{g,n}$ the function it induces on $\mathcal{M}_{g,n}$ and we introduce
\begin{equation}\label{VYnot}
  VY_{g,n}(L) = \int_{\mathcal{M}_{g,n}(L)} Y_{g,n}(\sigma)\,\dd\mu_{{\rm WP}}(\sigma)
\end{equation}
if this integral makes sense.

\medskip

Likewise, if $\Sigma$ is a closed or punctured surface, $Q\mathfrak{M}_\Sigma$ is endowed with the Masur--Veech measure $\mu_{{\rm MV}}$ coming from its piecewise linear integral structure. The function which associates to a quadratic differential $q$ on $\Sigma$ its area $\int_{\Sigma} |q|$ provides a natural way to define an induced measure on the space $Q^1\mathfrak{M}_{g,n}$ of quadratic differentials of unit area (see Section~\ref{S3Prem}). By a theorem of Masur and Veech~\cite{Masur82,Veech82} the total mass of this measure is finite. Its value is, by definition, the Masur--Veech volume and it is denoted by $MV_{g,n}$. Its computation is relevant in the study of the geometry of moduli spaces and the dynamics of measured foliations and has been the object of numerous investigations~\cite{AEZ,Delecroix,EskinOkounkov,Goujard-vol,Mirzagrowth}.

\medskip

%***************************************
\subsection{Topological recursion for Masur--Veech volumes: overview}
%***************************************

In Section~\ref{SPre}, we review the definition and main properties of the geometric and topological recursion, mainly taken from \cite{GRpaper}.

\medskip

In Section~\ref{S3}, for each connected bordered surface $\Sigma$ of genus $g$ with $n > 0$ boundaries, we construct a ${\rm Mod}_{\Sigma}^{\partial}$-invariant continuous function $\Omega^{{\rm MV}}_\Sigma\colon\mathcal{T}_\Sigma \to \R$. It is such that the integral $V\Omega^{{\rm MV}}_{g,n}(L_1, \ldots, L_n)$ is a polynomial function in the variables $L_1, \ldots, L_n$ and the Masur--Veech volume $MV_{g,n}$ of $Q^1\mathfrak{M}_{g,n}$ consisting of unit area quadratic differentials satisfies
\begin{equation}\label{MVcoinc}
  MV_{g,n} = \frac{2^{4g-2+n} (4g-4+n)!}{(6g-7+2n)!}\,V \Omega^{{\rm MV}}_{g,n}(0,\ldots,0).
\end{equation}

The family of functions $\Omega^{{\rm MV}}$ can be defined via the geometric recursion of \cite{GRpaper}, with initial data found in Proposition~\ref{MVtr1}. The polynomials $V\Omega^{{\rm MV}}$, which we call Masur--Veech polynomials, have five different descriptions:
\begin{itemize}
   \item[$(\mathbf{1})$]
      they are sums over stable graphs (Section~\ref{MVpo}), which we reproduce in \eqref{sumstabgraph} below;
   \item[$(\mathbf{2})$]
      they encode the asymptotic growth of the integral (against $\mu_{{\rm WP}}$)  of statistics of the hyperbolic lengths of multicurves on a surface of type $(g,n)$ with large boundaries, see Section~\ref{SMVvolumes} for the precise statement;
   \item[$(\mathbf{3})$]
      they are obtained by integration of $\Omega^{{\rm MV}}$, in coherence with the notation \eqref{VYnot};
   \item[$(\mathbf{4})$]
      they satisfy the topological recursion -- which is equivalent to the Virasoro constraints stated in Theorem~\ref{thm:intro:2} below -- for the spectral curve
      \[
      x(z) = \frac{z^2}{2},\qquad y(z) = -z,\qquad \omega_{0,2}^{{\rm MV}}(z_1,z_2) = \frac{1}{2}\,\frac{\dd z_1 \otimes \dd z_2}{(z_1 - z_2)^2} + \frac{1}{2}\,\frac{\pi^2\,\dd z_1 \otimes \dd z_2}{\sin^2(\pi(z_1 - z_2))}.
      \]
      \item[$(\mathbf{5}$)] they govern the asymptotic counting of square-tiled surfaces with large boundaries, see Section~\ref{Sectinin} for the precise statement.
\end{itemize}
The identity between $(\textbf{1})$ and $(\mathbf{2})$ is proved in Theorem~\ref{stablg}, which is the crux of our argument. The identity between $(\textbf{1})$, $(\mathbf{3})$ and $(\mathbf{4})$ is proved in Proposition~\ref{MVtr1} and follows from general properties of the geometric and the topological recursion. In Corollary~\ref{MVcoco}, we prove the relation \eqref{MVcoinc} between the constant term of these polynomials and the Masur--Veech volume. Lemma~\ref{dilatonthm0} implies that the value of the Masur--Veech volumes for closed surfaces of genus $g \geq 2$ can be retrieved from $V\Omega_{g,1}^{{\rm MV}}$.

\medskip

In Section~\ref{S4}, we extend these arguments to show in Corollary~\ref{Co45} that the area Siegel--Veech constants can be expressed in terms of asymptotics of certain derivative statistics of hyperbolic lengths of multicurves. Our current proof of Corollary~\ref{Co45} uses Goujard's recursion \cite{Goujard} (here quoted in Theorem~\ref{thGouj1}) for the area Siegel--Veech constants of the principal stratum in $Q^1\mathfrak{M}_{g,n}$ in terms of Masur--Veech volumes. It would be more satisfactory if one could obtain an independent proof of the identity of Corollary~\ref{Co45}, as our Section~\ref{S4} would then give a new proof of Goujard's recursion for the principal stratum. Section~\ref{S5} is devoted to explicit computation of Masur--Veech volumes and conjectures that can be drawn from them.

\medskip

In Section~\ref{Sectinin}, we discuss the enumeration of square-tiled surfaces with boundaries, via generating series including a parameter $q$ coupled to the number of tiles. We show in Proposition~\ref{MVpolyasstscount} that sending $q \rightarrow 1$ while rescaling the boundary lengths by $1/\ln(1/q)$ retrieves the Masur--Veech polynomials, \textit{i.e.} give the identity between $(\mathbf{1})$ and $(\mathbf{5})$. In absence of boundaries, a related asymptotic enumeration of square-tiled surfaces with a number of tiles $\leq N \rightarrow \infty$ was the crucial ingredient in the proof of \eqref{MVcoinc} given by \cite{Delecroix}. We also prove in Proposition~\ref{propsq} that -- before taking any asymptotics -- the $q$-series counting square-tiled surfaces are governed by Eynard--Orantin topological recursion for the spectral curve
\[
x(z) = z + \frac{1}{z},\qquad y(z) = z,\qquad \omega_{0,2}(z_1,z_2) = \frac{1}{2}\,\frac{\dd z_1 \otimes \dd z_2}{(z_1 - z_2)^2} + \frac{\dd u_1 \otimes \dd u_2}{2}\bigg(\wp(u_1 - u_2;\mathsf{q}) + \frac{\pi^2 E_2(\mathsf{q})}{3}\bigg)
\]
where $z_j = e^{2{\rm i}\pi u_j}$, $\wp$ is Weierstra\ss{} elliptic function and $E_2$ the second Eisenstein series. In the $\mathsf{q} \rightarrow 1$ limit, we show that it retrieves  the topological recursion for Masur--Veech polynomials mentioned in $(\mathbf{4})$, thus giving a second proof of the identity between $(\mathbf{1})$ and $(\mathsf{4})$. 

%***************************************
\subsubsection*{Main results for the computation of Masur--Veech volumes and polynomials}
%***************************************

Concretely, our results lead to two ways of computing Masur--Veech volumes. Firstly, the Masur--Veech polynomials are expressed as a sum over the set $\mathbf{G}_{g,n}$ of stable graphs (see Definition~\ref{defgrstss}). Stable graphs encode topological types of primitive multicurves, which naturally appear via (\textbf{2}). Let us introduce the polynomials
\begin{equation} \label{eq:Kpoly}
  V\Omega^{{\rm K}}_{g,n}(L_1,\ldots,L_n)
  =
  \int_{\overline{\mathfrak{M}}_{g,n}} \exp\bigg(\sum_{i = 1}^n \frac{L_i^2}{2}\,\psi_i\bigg),
\end{equation}
which from Kontsevich's work~\cite{Kontsevich} compute the volume of the combinatorial moduli spaces. The application of Theorem~\ref{stablg} of this paper to the computation of Masur--Veech volumes can be summarised as follows.

\begin{thm}\label{thm:intro:1}
  For $g,n \geq 0$ such that $2g - 2 + n > 0$, the Masur--Veech polynomials can be expressed as
  \begin{equation}\label{sumstabgraph}
    V\Omega_{g,n}^{{\rm MV}}(L_1,\ldots,L_n) = \sum_{\Gamma \in \mathbf{G}_{g,n}} \frac{1}{|{\rm Aut}\,\Gamma|} \int_{\mathbb{R}_{+}^{E_{\Gamma}}}  \prod_{v \in V_{\Gamma}} V\Omega^{{\rm K}}_{h(v),k(v)}\big((\ell_e)_{e \in E(v)},(L_{\lambda})_{\lambda \in \Lambda(v)}\big)\prod_{e \in E_{\Gamma}} \frac{\ell_e\dd\ell_e }{e^{\ell_e} - 1},
  \end{equation}
  where $V_{\Gamma}$ is the set of vertices of $\Gamma$ and $E(v)$ (respectively, $\Lambda(v)$) is the set of edges (respectively, leaves) incident to $v$. In particular the Masur--Veech volumes can be computed as
  \begin{equation}\label{MVgraphsss}
  \begin{aligned}
  	MV_{g,n} & = \frac{2^{4g - 2 + n}(4g - 4 + n)!}{(6g - 7 + 2n)!}  \\
  	&\quad \times \sum_{\Gamma \in \mathbf{G}_{g,n}} \frac{1}{|{\rm Aut}\,\Gamma|} \int_{\mathbb{R}_{+}^{E_{\Gamma}}}  \prod_{v \in V_{\Gamma}} V\Omega^{{\rm K}}_{h(v),k(v)}\big((\ell_e)_{e \in E(v)},(0)_{\lambda \in \Lambda(v)}\big)\prod_{e \in E_{\Gamma}} \frac{\ell_e\dd\ell_e }{e^{\ell_e} - 1}.
  \end{aligned}
  \end{equation}
\end{thm}

Formula \eqref{MVgraphsss} was obtained prior to our work in \cite{Delecroix} by combinatorial methods. It was presented by V.D. in a reading group organised by A.G. and D.L. The discussions which followed led to the present work where, in particular, we give a new proof of formula \eqref{MVgraphsss}.

\medskip

Secondly, the coefficients of the Masur--Veech polynomials satisfy Virasoro constraints, expressed in terms of values of the Riemann zeta function at even integers. This is summarised by the following theorem, which combines the results of Lemma~\ref{dilatonthm0}, Corollary~\ref{MVcoco}, Theorem~\ref{MVtr1} and Section~\ref{Virrec} of this paper.

\begin{thm}\label{thm:intro:2}
  For any $g \geq 0$ and $n > 0$ such that $2g - 2 + n > 0$, we have a decomposition
  \[
    V\Omega_{g,n}^{{\rm MV}}(L_1,\ldots,L_n)
    =
    \sum_{\substack{d_1,\ldots,d_n \geq 0 \\ d_1 + \cdots + d_n \leq 3g - 3 + n}}
      F_{g,n}[d_1,\ldots,d_n] \, \prod_{i = 1}^n \frac{L_i^{2d_i}}{(2d_i + 1)!}.
  \]
  Let us set $F_{0,1}[d_1] = F_{0,2}[d_1,d_2] = 0$ for all $d_1,d_2 \geq 0$. The base cases
  \[
    F_{0,3}[d_1,d_2,d_3] = \delta_{d_1,d_2,d_3,0},
    \qquad
    F_{1,1}[d] = \delta_{d,0} \, \frac{\zeta(2)}{2} + \delta_{d,1} \, \frac{1}{8}
  \]
  determine uniquely all other coefficients via the following recursion on $2g - 2 + n \geq 2$, for $d_1,\ldots,d_n \geq 0$
  \[
  \begin{split}
  	F_{g,n}[d_1,\ldots,d_n] & = \sum_{m = 2}^n \sum_{a \geq 0} B^{d_1}_{d_m,a}\,F_{g,n - 1}[a,d_2,\ldots,\widehat{d_m},\ldots,d_n] + \\
  	&\quad + \frac{1}{2} \sum_{a,b \geq 0} C^{d_1}_{a,b}\bigg(F_{g - 1,n + 1}[a,b,d_2,\ldots,d_n] + \!\!\!\! \sum_{\substack{h + h' = g \\ J \sqcup J' = \{d_2,\ldots,d_n\}}} \!\!\!\! F_{h,1+|J|}[a,J]\,F_{h',1 + |J'|}[b,J']\bigg), 
  \end{split}
  \]
  where
  \[
  \begin{split}
  	B^{i}_{j,k} & = (2j + 1)\,\delta_{i + j,k + 1} + \delta_{i,j,0}\,\zeta(2k + 2), \\
  	C^i_{j,k} & = \delta_{i,j + k + 2} + \tfrac{(2j + 2a + 1)!\zeta(2j + 2a + 2)}{(2j + 1)!(2a)!}\,\delta_{i + a,k + 1} + \tfrac{(2k + 2a + 1)!\zeta(2k + 2a + 2)}{(2k + 1)!(2a)!}\,\delta_{i + a,j + 1} + \zeta(2j + 2)\zeta(2k + 2)\delta_{i,0}.
  \end{split}
  \]
  For surfaces of genus $g$ with $n > 0$ boundaries, the Masur--Veech volumes are identified as
  \[
    MV_{g,n} = \frac{2^{4g - 2 + n}(4g - 4 + n)!}{(6g - 7 + 2n)!}\,F_{g,n}[0,\ldots,0],
  \]
  while for closed surfaces of genus $g \geq 2$ they are obtained through
  \[
    MV_{g,0} = \frac{2^{4g - 2}(4g - 4)!}{(6g - 6)!}\,F_{g,1}[1].
  \] 
\end{thm}

\medskip

We use Theorem~\ref{thm:intro:2} to compute many Masur--Veech volumes and Masur--Veech polynomials for low $g$ and $n$ (Section~\ref{S5}). Based on numerical evidence, we propose polynomiality conjectures for $MV_{g,n}$ for all $n$ and fixed $g$ (Conjecture~\ref{conjMV}), with explicit coefficients up to $g \leq 6$. Conditionally on this conjecture, we discuss the consequences for area Siegel--Veech constants in Corollary~\ref{conjSV} and for the $n \rightarrow \infty$ asymptotics in Section~\ref{Sasymss}.

\medskip

The paper is supplemented with three appendices. In Appendix~\ref{AppA}, we establish a closed formula for all $\psi$ classes intersections in genus one, which we have not found in the literature and which we use for computations of $V\Omega_{1,n}^{{\rm MV}}$ via stable graphs. In Appendix~\ref{AppB}, we illustrate the computation of Masur--Veech polynomials and generating series of square-tiled surfaces from Propositions~\ref{EOMV} and \ref{propsq} using the original formulation of the topological recursion \`a la Eynard--Orantin, via residues on the associated spectral curves. Appendix~\ref{AppC} contains tables of coefficients for the Masur--Veech polynomials and area Siegel--Veech constants.

\begin{rem} \textrm{Since the first arXiv release of our work, our polynomiality Conjectures~\ref{conjMV} and thus \ref{conjSV} have been proved via intersection-theoretic methods by Chen, M\"oller and Sauvaget \cite{Moeller}.}
\end{rem}
\medskip

%Throughout the paper we make use of the symbol $\blacksquare$ at the end of those statements whose proof is not part of the paper, whereas the symbol $\star$ is used if the proof is included, but not immediately after the statement. No symbol is used if the proof follows the statement.

%***************************************
\section*{Acknowledgements}
%***************************************

We thank Don~Zagier for suggesting a more compact formula in Conjecture~\ref{conjMV}, and Martin Möller for discussions related to the intersection theory aspects of the paper. J.E.A. was supported in part by the Danish National Sciences Foundation Centre of Excellence grant ``Quantum Geometry of Moduli Spaces" (DNRF95) at which part of this work was carried out. JEA is currently funded in part by the ERC Synergy grant ``ReNewQuantum". G.B., S.C., V.D., A.G., D.L. and C.W. are supported by the Max-Planck-Gesellschaft. 

%***************************************
\section{Review of geometric and topological recursion}
\label{SPre}
%***************************************

We review some aspects of the formalism of geometric recursion developed in \cite{GRpaper} and its relation to topological recursion which are directly relevant for the analysis carried out in Section~\ref{S3} and onwards in the present paper.

%***************************************
\subsection{Preliminaries}
%***************************************

%\todo{Where $T_\Sigma^{(\epsilon)}$ is used?}
%For a given point in $\mathcal{T}_{\Sigma}$, the systole is the length of the shortest closed geodesic on $\Sigma$ (it could possibly be a boundary component). The $\epsilon$-thick part of the Teichm\"uller space is denoted by $\mathcal{T}_{\Sigma}^{(\epsilon)}$: it consists of those classes of hyperbolic metrics for which the systole is bounded below by $\epsilon$.

\medskip

Let $S_{\Sigma}^{\circ}$ be the set of isotopy classes of simple closed curves which are neither puncture nor boundary parallel on $\Sigma$, $M_{\Sigma}$ the set of multicurves (\textit{i.e.} isotopy classes of finite disjoint unions of simple closed curves which are not homotopic to boundary components of $\Sigma$) and $M'_{\Sigma}$ the subset of primitive multicurves (the components of the multicurve must be pairwise non-homotopic). By convention $M_{\Sigma}$ and $M'_{\Sigma}$ contain the empty multicurve, but $S_{\Sigma}^{\circ}$ does not contain the empty closed curve. In particular
\[
   M_{\Sigma} \cong \Set{(\gamma,m) | \gamma \in M'_{\Sigma},\,\,\,m \in \mathbb{Z}_{+}^{\pi_0(\gamma)}},
\]
where $\mathbb{Z}_{+}$ is the set of positive integers. 

%***************************************
\subsection{Geometric recursion}
%***************************************

In the present context, the geometric recursion (in brief, GR) is a recipe to construct ${\rm Mod}_{\Sigma}^{\partial}$-invariant functions $\Omega_{\Sigma}$ on $\mathcal{T}_{\Sigma}$ for bordered surfaces $\Sigma$ of all topologies, by induction on the Euler characteristic of $\Sigma$. The initial data for GR is a quadruple $(A,B,C,D)$ where $A,B,C$ are functions on the Teichm\"uller space of a pair of pants, and $D$ is a function on the Teichm\"uller space of a torus with one boundary component. Since $\mathcal{T}_{P} \cong \mathbb{R}_{+}^3$, the functions $A$, $B$ and $C$ are just functions of three positive variables. We further require that $A$ and $C$ are invariant under exchange of their two last variables. In the construction we need that the initial data satisfy some decay conditions. Let $[x]_{+} = \max(x,0)$.

\begin{defn}\label{defadm}
   We say that an initial data $(A,B,C,D)$ is \emph{admissible} if
   \begin{itemize}
      \item[$\bullet$]
         $A$ is bounded on $\mathcal{T}_{P}$ and $D$ is bounded on $\mathcal{T}_{T}$,
      \item[$\bullet$]
         For any $s > 0$ and some $\eta \in [0,2)$, 
      \begin{align*}
      	\sup_{L_1,L_2,\ell \geq 0} \big(1 + [\ell - L_1 - L_2]_{+}\big)^{s}\,|B(L_1,L_2,\ell)|\, \ell^{\eta} & < +\infty, \\
      	\sup_{L_1,\ell,\ell' \geq 0} \big(1 + [\ell + \ell' - L_1]_{+}\big)^{s}\,|C(L_1,\ell,\ell')|\, (\ell\ell')^{\eta} & < +\infty.
      \end{align*}
   \end{itemize}
\end{defn}

Let us now briefly recall the recursion introduced in \cite{GRpaper}, which uses successive excisions of pairs of pants. Assume that $\Sigma$ has genus $g$ and $n$ boundary components such that $2g -2 + n \geq 2$. We consider the set of homotopy classes of embedded pairs of pants $\phi \colon P \hookrightarrow \Sigma$ such that
\begin{itemize}
   \item[$\bullet$]
      $\partial_1P$ is mapped to $\partial_1\Sigma$,
   \item[$\bullet$]
      $\partial_2P$ is either mapped to a boundary component of $\Sigma$, or mapped to a curve that is not null-homotopic neither homotopic to a boundary component of $\Sigma$.
\end{itemize}
Let $\mathcal{P}_{\Sigma}$ the set of homotopy classes of such embeddings. It is partitioned into the subsets $\mathcal{P}_{\Sigma}^{\varnothing}$ and $\mathcal{P}_{\Sigma}^{m}$ for $m \in \{2,\ldots,n\}$, consisting respectively of those classes of embeddings such that $\partial_2P$ is mapped to the interior of $\Sigma$, resp. mapped to $\partial_m\Sigma$. Given a hyperbolic metric $\sigma$ with geodesic boundaries on $\Sigma$, each element of $\mathcal{P}_{\Sigma}$ has a representative $P$ such that $\phi(P)$ has geodesic boundaries. We denote by $\vec{\ell}_{\sigma}(\partial P)$ the ordered triple of lengths of $\phi(P)$ for the metric $\sigma$. Removing this embedded pair of pants from $\Sigma$ gives a bordered surface $\Sigma - P$. Our assumptions imply that $\Sigma - P$ is stable. It is also equipped with a hyperbolic metric $\sigma|_{\Sigma - P}$ with geodesic boundaries. We decide to label the boundary components of $\Sigma - P$ by putting first the boundary components that came from those of $P$ (respecting the order in which they appeared in $\partial P$) and then the boundary components that came from those of $\Sigma$ (with the order in which they appeared in $\Sigma$). 

\medskip

The GR amplitudes $\Omega_{\Sigma}$ are now defined as follows. For surfaces with Euler characteristic $-1$, we declare
\[
   \Omega_{P} = A,\qquad \Omega_{T} = D.
\]
For disconnected surfaces, we use the identification $\mathcal{T}_{\Sigma_1 \cup \Sigma_2} \cong \mathcal{T}_{\Sigma_1} \times \mathcal{T}_{\Sigma_2}$ to set
\[
   \Omega_{\Sigma_1 \cup \Sigma_2}(\sigma_1,\sigma_2) = \Omega_{\Sigma_1}(\sigma_1)\Omega_{\Sigma}(\sigma_2),
\]
and for connected surfaces with Euler characteristic $\leq -2$, we set
\begin{equation}\label{serieGR}
   \Omega_{\Sigma}(\sigma)
   =
   \sum_{m = 2}^n \sum_{[P] \in \mathcal{P}_{\Sigma}^{m}}
      B(\vec{\ell}_{\sigma}(\partial P))\,\Omega_{\Sigma - P}(\sigma|_{\Sigma - P})
   +
   \frac{1}{2} \sum_{[P] \in \mathcal{P}_{\Sigma}^{\varnothing}}
      C(\vec{\ell}_{\sigma}(\partial P))\,\Omega_{\Sigma - P}(\sigma|_{\Sigma - P}).
\end{equation}
The latter is a countable sum and its absolute convergence was addressed\footnote{The notion of admissibility adopted in the present paper is more restrictive than the one appearing in \cite{GRpaper}, but is sufficient for our purposes.} in \cite{GRpaper}. We recall the main construction theorem of that paper here. Let ${\mathcal F}(\mathcal{T}_{\Sigma},\mathbb{C})$ be the set of complex valued functions on $\mathcal{T}_{\Sigma}$.

\begin{thm}\label{THGR} 
  If $(A,B,C,D)$ is an admissible initial data, then $\Sigma \mapsto \Omega_{\Sigma} \in {\mathcal F}(\mathcal{T}_{\Sigma},\mathbb{C})$ is a well-defined assignment. More precisely:
  \begin{itemize}
    \item[$\bullet$] the series \eqref{serieGR} is absolutely convergent for the supremum norm over any compact subset of $\mathcal{T}_{\Sigma}$;
  \item[$\bullet$] $\Omega_{\Sigma}$ is invariant under all mapping classes in ${\rm Mod}_{\Sigma}$ which preserve $\partial_{1}\Sigma$;
    \item[$\bullet$] if the initial data is continuous (or measurable), $\Omega_{\Sigma}$ is also continuous (or measurable).
  \end{itemize}
  \hfill $\blacksquare$
\end{thm}

%***************************************
\subsection{Two examples}
\label{Sexamples}
%***************************************

We describe two examples of initial data which play a special role for us. The first one appears in Mirzakhani's generalisation  \cite{Mirza1} of McShane identity \cite{McShane}, which is a prototype of GR and which we can reformulate in GR terms as follows.

\begin{thm}[Mirzakhani \& McShane]
   \label{thrmi}The initial data
   \begin{equation}\label{Mirzaini}
   \begin{aligned}
   	A^{{\rm M}}(L_1,L_2,L_3)
   	& =  1, \\
   	B^{{\rm M}}(L_1,L_2,\ell)
   	& =  1 - \frac{1}{L_1}\ln\bigg(\frac{\cosh\big(\frac{L_2}{2}\big) + \cosh\big(\frac{L_1 + \ell}{2}\big)}{\cosh\big(\frac{L_2}{2}\big) + \cosh\big(\frac{L_1 - \ell}{2}\big)}\bigg), \\
   	C^{{\rm M}}(L_1,\ell,\ell')
   	& = \frac{2}{L_1}\,\ln\bigg(\frac{e^{\frac{L_1}{2}} + e^{\frac{\ell + \ell'}{2}}}{e^{-\frac{L_1}{2}} + e^{\frac{\ell + \ell'}{2}}}\bigg), \\
   	D^{{\rm M}}(\sigma)
   	& = \sum_{\gamma \in S_{T}^{\circ}} C^{{\rm M}}\big(\ell_{\sigma}(\partial T),\ell_{\sigma}(\gamma),\ell_{\sigma}(\gamma)\big),
   \end{aligned}
   \end{equation}
   are admissible, and for any bordered $\Sigma$ the corresponding GR amplitude $\Omega_{\Sigma}^{{\rm M}}$ is the constant function $1$ on $\mathcal{T}_{\Sigma}$. \hfill $\blacksquare$
\end{thm}

The second example is obtained by rescaling all length variables in Mirzakhani initial data as follows 
\begin{equation}
\label{KKlim} X^{{\rm K}}(L_1,L_2,L_3) = \lim_{\beta \rightarrow \infty} X^{{\rm M}}(\beta L_1,\beta L_2,\beta L_3),\qquad X \in \{A,B,C\}.
\end{equation}
More explicitly
\begin{equation}\label{WKini}
\begin{aligned}
	A^{{K}}(L_1,L_2,L_3)
	& = 1,  \\
	B^{{K}}(L_1,L_2,\ell)
	& = \frac{1}{2L_1}\big([L_1 - L_2 - \ell]_{+} - [- L_1 + L_2 - \ell]_{+} + [L_1 + L_2 - \ell]_{+}\big), \\
	C^{{\rm K}} (L_1,\ell,\ell')
	& = \frac{1}{L_1}\,[L_1 - \ell - \ell']_{+}, \\
	D^{{\rm K}}(\sigma)
	& = \sum_{\gamma \in S_{T}^{\circ}} C^{{\rm K}}(\ell_{\sigma}(\partial T),\ell_{\sigma}(\gamma),\ell_{\sigma}(\gamma)).
\end{aligned}
\end{equation}
It is easy to check that these initial data are admissible, and we call them the Kontsevich initial data. Unlike the previous situation, the resulting GR amplitudes $\Omega_{\Sigma}^{\rm K}$ are non-trivial functions on $\mathcal{T}_{\Sigma}$. Their geometric interpretation and basic properties are studied in \cite{CombGRpaper}.

%***************************************
\subsection{Hyperbolic length statistics and twisting of initial data}
%***************************************
\label{Twistsec} 
Let $\mathbb{D} \subset \mathbb{C}$ be the open unit disk. Let $f\colon\mathbb{R}_{+} \rightarrow \mathbb{C}$ and $\widetilde{f}\colon\mathbb{R}_{+} \rightarrow \mathbb{D}$ be two functions related by
\begin{equation}\label{f:ftilde:relation}
   f(\ell) = \sum_{k \geq 1} (\widetilde{f}(\ell))^{k} = \frac{\widetilde{f}(\ell)}{1 - \widetilde{f}(\ell)}.
\end{equation}

\begin{defn}
   We call $f\colon\mathbb{R}_{+} \to \mathbb{C}$ an \emph{admissible test function} if $f$ is Riemann-integrable on $\mathbb{R}_{+}$ and for any $s > 0$
   \begin{equation}\label{condh}
      \sup_{\ell > 0}\,\,(1 + \ell)^s\, |f(\ell)|\ < +\infty.
   \end{equation}
\end{defn}

This condition is stronger than what is needed in \cite{GRpaper}, but is sufficient here.

\medskip

Following  \cite{GRpaper}, we consider multiplicative statistics of hyperbolic lengths of multicurves
\begin{equation}
\label{nh} N_{\Sigma}(f;\sigma) = \sum_{c \in M_{\Sigma}'} \prod_{\gamma \in \pi_0(c)} f(\ell_{\sigma}(\gamma)) = \sum_{c \in M_{\Sigma}} \prod_{\gamma \in \pi_0(c)} \widetilde{f}(\ell_{\sigma}(\gamma)).
\end{equation}
It can be written either as a sum over all multicurves or as a sum over primitive multicurves only, the two expressions being related via the geometric series \eqref{f:ftilde:relation}. According to our conventions, the empty multicurve gives a term equal to $1$ in this sum.

\medskip

In fact, these statistics satisfy the geometric recursion. If $(A,B,C,D)$ are some initial data, we define its twisting as follows
\begin{equation}\label{initwist}
\begin{aligned}
	A[f](L_1,L_2,L_3) & = A(L_1,L_2,L_3), \\
	B[f](L_1,L_2,\ell) & = B(L_1,L_2,\ell) + A(L_1,L_2,\ell)\,f(\ell), \\
	C[f](L_1,\ell,\ell') & = C(L_1,\ell,\ell') + B(L_1,\ell,\ell') f(\ell) + B(L_1,\ell',\ell)f(\ell') + A(L_1,\ell,\ell')f(\ell)f(\ell'), \\
	D[f](\sigma) & = D(\sigma) + \sum_{\gamma \in S_{T}^{\circ}} A\big(\ell_{\sigma}(\partial T),\ell_{\sigma}(\gamma),\ell_{\sigma}(\gamma)\big)\,f(\ell_{\sigma}(\gamma)).
\end{aligned}	
\end{equation}

\begin{thm} \cite{GRpaper}
  If we choose $(A,B,C,D)$ to be the Mirzakhani initial data \eqref{Mirzaini} and $f$ is an admissible test function, the twisted initial data \eqref{initwist} are admissible and the resulting GR amplitudes coincide with the assignment $\Sigma \mapsto N_{\Sigma}(f;\,\cdot\,)$. \hfill $\blacksquare$
\end{thm}

The idea of the proof is, for each $c \in M_{\Sigma}'$, to multiply the product in \eqref{nh} by $1$, seen as a function on the Teichm\"uller space of $\Sigma - c$. Then, one decomposes $1$ using Mirzakhani's identity on $\mathcal{T}_{\Sigma - c}$, and interchanges the summation over primitive multicurves with the summation over embedded pairs of pants. As the curves do not intersect the pair of pants, the structure of the geometric recursion \eqref{serieGR} appears again, but the initial data are modified as in \eqref{initwist}. It is important to consider only \emph{simple} closed curves, as otherwise $\Sigma - c$ would not be anymore a bordered surface and the recursive procedure could not be carried out in this way.

%***************************************
\subsection{Relation to the topological recursion}
%***************************************
\label{S25}
Being invariant under the pure mapping class group, the GR amplitudes $\Omega_{\Sigma}$ descend to functions on the moduli space $\mathcal{M}_{g,n}$, and we denote them by $\Omega_{g,n}$. The structure of the geometric recursion is compatible with factorisations of the Weil--Petersson volume form $\mu_{{\rm WP}}$ when excising pairs of pants. This means that, if we integrate GR amplitude against $\mu_{{\rm WP}}$, the outcome will again be governed by a recursion with respect to the Euler characteristic, which is called the topological recursion (TR for short). The (countable) sum over homotopy classes of pairs of pants is replaced with a sum over the (finitely many) diffeomorphism classes of embeddings of pair of pants. 

\medskip

Recall the notation
\[
   V\Omega_{g,n}(L_1,\ldots,L_n) = \int_{\mathcal{M}_{g,n}(L_1,\ldots,L_n)} \Omega_{g,n}(\sigma)\,\dd\mu_{{\rm WP}}(\sigma),
\]
whenever the integral on the right-hand side makes sense; by convention, we set $V\Omega_{g,n} = 0$ when $2g - 2 + n \leq 0$.

\begin{thm}[From GR to TR, \cite{GRpaper}]
  If $(A,B,C,D)$ are admissible, $V\Omega_{g,n}$ is well-defined as the integrand is Riemann-integrable, and it satisfies the topological recursion, that is for any $g \geq 0$ and $n \geq 1$ such that $2g - 2 + n \geq 2$
  \begin{equation}\label{TReqn}
  \begin{split}
  	& V\Omega_{g,n}(L_1,L_2,\ldots,L_n) \\
   	& = \sum_{m = 2}^n \int_{\mathbb{R}_{+}} B(L_1,L_m,\ell) V\Omega_{g,n - 1}(\ell,L_2,\ldots,\widehat{L_m},\ldots,L_n) \ell\, \dd \ell \\
   	& + \frac{1}{2} \int_{\mathbb{R}_{+}^2} \!\! C(L_1,\ell,\ell') \bigg(V\Omega_{g-1,n + 1}(\ell,\ell',L_2,\ldots,L_n) + \!\!\!\!\! \!\!\!\!\! \sum_{\substack{h + h' = g \\ J \sqcup J' = \{L_2,\ldots,L_n\}}} \!\!\!\!\!\!\!\!\! V\Omega_{h,1+|J|}(\ell,J)V\Omega_{h',1 + |J'|}(\ell',J')\bigg) \ell \ell'\,\dd \ell\,\dd \ell' \\
  \end{split}
  \end{equation}
  The base cases are
  \[
     V\Omega_{0,3}(L_1,L_2,L_3) = A(L_1,L_2,L_3),\qquad V\Omega_{1,1}(L_1) = VD(L_1) = \int_{\mathcal{M}_{1,1}(L_1)} D(\sigma)\,\dd\mu_{{\rm WP}}(\sigma).
  \]
  \hfill $\blacksquare$
\end{thm}

We call any sequence of functions $V\Omega_{g,n}$ satisfying a recursion of the form \eqref{TReqn} TR amplitudes. Let us come back to the two examples of Section~\ref{Sexamples}.

\medskip

According to Theorem~\ref{thrmi}, $V\Omega^{{\rm M}}_{g,n}(L_1,\ldots,L_n)$ is the Weil--Petersson volume of $\mathcal{M}_{g,n}(L_1,\ldots,L_n)$, and the topological recursion \eqref{TReqn} in this case is Mirzakhani's recursion for these volumes \cite{Mirza1}. To be complete, we should record the Weil--Petersson volume for $\mathcal{M}_{1,1}(L_1)$
\[
   VD^{{\rm M}}(L_1) = \frac{\pi^2}{6} + \frac{L_1^2}{48},
\]
which is also mentioned in \cite{Mirza1}. Mirzakhani also expressed the Weil--Petersson volumes via intersection theory on the Deligne--Mumford compactified moduli space of punctured surfaces $\overline{\mathfrak{M}}_{g,n}$.

\begin{thm}\cite{Mirzaint}\label{mirzath2}
  The Weil--Petersson volumes satisfy
  \[
  V\Omega^{{\rm M}}_{g,n}(L_1,\ldots,L_n) = \int_{\overline{\mathfrak{M}}_{g,n}} \exp\bigg(2\pi^2\kappa_1 + \sum_{i = 1}^n \frac{L_i^2}{2}\psi_i\bigg).
  \]
  \hfill $\blacksquare$
\end{thm}

Similar considerations applies to $V\Omega^{{\rm K}}$. Actually, the topological recursion for $V\Omega^{{\rm K}}$ is equivalent to the set of Virasoro constraints for the intersection of $\psi$-classes on $\overline{\mathfrak{M}}_{g,n}$.

\begin{thm}[Conjecture \cite{Witten}, theorem of \cite{Kontsevich} and \cite{DVV}]\label{konth}
  The amplitudes $V\Omega^{{\rm K}}$ satisfy
  \[
  V\Omega^{{\rm K}}_{g,n}(L_1,\ldots,L_n) = \int_{\overline{\mathfrak{M}}_{g,n}} \exp\bigg(\sum_{i = 1}^n \frac{L_i^2}{2}\psi_i \bigg).
  \]
  In particular, $VD^{{\rm K}}(L_1) = \frac{L_1^2}{48}$.
  \hfill $\blacksquare$
\end{thm}

This is also a corollary of Theorem~\ref{mirzath2}, as can be seen if we multiply all length variables by $\beta$ in the Mirzakhani initial data, let $\beta \rightarrow \infty$ and recall definition \eqref{KKlim}. The main analysis carried out in this paper consists in rescaling length variables by $\beta \rightarrow \infty$ in the twisted GR amplitudes to understand properties of the asymptotic number of multicurves.

\medskip

There are several other ways to see that Theorem~\ref{mirzath2} implies or is implied by Theorem~\ref{konth}, see \cite{Safnukwp,Do,Mulasevolume}. They will be discussed in the broader context of the geometric recursion in \cite{CombGRpaper}.

%***************************************
\subsubsection*{Symmetry issues}
%***************************************

The GR amplitudes $\Omega_{\Sigma}$ are a priori invariant under mapping classes that preserve $\partial_1\Sigma$ (see Theorem~\ref{THGR}). Therefore, after integration, the TR amplitudes $V\Omega_{g,n}(L_1,\ldots,L_n)$ are symmetric functions of $L_2,\ldots,L_n$. The topological recursion also gives a special role to the length $L_1$ of the first boundary. 

\medskip

The framework of quantum Airy structures \cite{KSTR} provides sufficient conditions for the invariance of TR amplitudes under all permutations of $(L_1,\ldots,L_n)$. These conditions are quadratic constraints on $(A,B,C,VD)$ which are explicitly written down in \cite[Section 2.2]{ABCD}. They are satisfied by the Mirzakhani and Kontsevich initial data obtained from spectral curves in the Eynard--Orantin description (Section~\ref{SEO}), and they are stable under the twisting operation \cite{ABCD}. Therefore, all TR amplitudes that are considered in this article have the full $\mathfrak{S}_{n}$-symmetry.

\medskip

The situation is different at the level of GR amplitudes. For instance, one can prove that $\Omega^{{\rm K}}_{\Sigma}$ is \emph{not} always invariant under mapping classes that do not respect $\partial_{1}\Sigma$ \cite{CombGRpaper}.
%***************************************
\subsection{Twisting and stable graphs}
%***************************************
\label{Stwists}
If $(A,B,C,D)$ are admissible initial data, the upper bound on the number of multicurves of bounded length directly implies that the twisted initial data $(A[f],B[f],C[f],D[f])$ remain admissible when $f$ is an admissible test function \eqref{condh}. Therefore, the integrals
\[
   V\Omega_{g,n}(f;L_1,\ldots,L_n) = \int_{\mathcal{M}_{g,n}(L_1,\ldots,L_n)} \Omega_{g,n}(f;\sigma)\,\dd\mu_{{\rm WP}}(\sigma)
\]
of the GR amplitudes $\Omega_{g,n}(f;\,\cdot\,)$ satisfy TR \eqref{TReqn} for the initial data $(A[f],B[f],C[f])$, completed by
\begin{equation}
\label{VDtwist}    VD(f;L_1) = VD(L_1) + \frac{1}{2} \int_{\mathbb{R}_{+}}f(\ell)\,A(L_1,\ell,\ell)\,\ell \dd \ell.
\end{equation}
The function $V\Omega_{g,n}(f;\,\cdot\,)$ can also be evaluated by direct integration, exploiting the factorisation of the Weil--Petersson volume form when cutting along simple closed curves -- which is clear from its expression in Fenchel--Nielsen coordinates. The result is that, while $\Omega_{g,n}(f;\,\cdot\,)$ is a (countable) sum over primitive multicurves, its integral $V\Omega_{g,n}(f;\,\cdot\,)$ is a sum over the (finitely many) topological types of such multicurves. The latter are described by stable graphs, which we now define. We use the notation $\NN$ for the non-negative integers.

\begin{defn}\label{defgrstss}
	A stable graph $\Gamma$ of type $(g,n)$ consists of the data
	\[
		\bigl(
			V_{\Gamma}, H_{\Gamma},\Lambda_{\Gamma}, h, \mathsf{v}, \mathsf{i}
		\bigr)
	\]
	satisfying the following properties.
	\begin{enumerate}
		\item
         $V_{\Gamma}$ is the set of vertices, equipped with a function $h \colon V_{\Gamma} \to \NN$, called the genus.
		\item
         $H_{\Gamma}$ is the set of half-edges, $\mathsf{v} \colon H_{\Gamma} \to V_{\Gamma}$ associate to each half-edge the vertex it is incident to, and $\mathsf{i} \colon H_{\Gamma} \to H_{\Gamma}$ is the involution.
		 \item 
         $E_{\Gamma}$ is the set of edges, consisting of the $2$-cycles of $\mathsf{i}$ in $H_{\Gamma}$ (loops at vertices are permitted).
		\item
         $\Lambda_{\Gamma}$ is the set of leaves, consisting of the fixed points of $\mathsf{i}$, which are equipped with a labelling from $1$ to $n$.
		\item
         The pair $(V_{\Gamma}, E_{\Gamma})$ defines a connected graph.
		\item
         If $v$ is a vertex, $\overline{E}(v)$ (resp. $E(v)$) is the set of half edges incident to $v$ including (resp. excluding) the leaves  and $k(v) = |\overline{E}(v)|$ is the valency of $v$. We require that for each vertex $v$, the stability condition $2h(v) - 2 + k(v) > 0$ holds.
		\item
         The genus condition
   		\[
   			g = \sum_{v \in V_{\Gamma}} h(v) + b_1(\Gamma)
   		\]
   		holds. Here $b_1(\Gamma) = |E_\Gamma| - |V_\Gamma| + 1$ is the first Betti number of the graph $\Gamma$.
	\end{enumerate}
	An automorphism of $\Gamma$ consists of bijections of the sets $V_{\Gamma}$ and $H_{\Gamma}$ which leave invariant the structures $h$, $\mathsf{v}$, and $\mathsf{i}$ (and hence respect $E_{\Gamma}$ and $\Lambda_{\Gamma}$). We denote by ${\rm Aut}\,\Gamma$ the automorphism group of $\Gamma$.
\end{defn}

We denote by $\mathbf{G}_{g,n}$ the set of stable graphs of type $(g,n)$. It parametrises the topological types of primitive multicurves on a bordered surface $\Sigma$ of genus $g$ with $n$ labelled boundaries:
\[
   \mathbf{G}_{g,n} = M'_{\Sigma}\ /\ {\rm Mod}_{\Sigma}^{\partial}.
\]
The stable graph with a single vertex of genus $g$ corresponds to the empty multicurve. The other stable graphs are in bijective correspondence with the boundary strata of $\overline{\mathfrak{M}}_{g,n}$; more precisely $\Gamma \in \mathbf{G}_{g,n}$ refers to a boundary stratum of complex codimension $|E_{\Gamma}|$ that contains the union over $v \in V_{\Gamma}$ of smooth complex curves of genus $h(v)$ with $k(v)$ punctures, glued in a nodal way along punctures that correspond to the two ends of the same edge.

\medskip

By direct integration, we have that

\begin{thm}\cite{GRpaper}\label{stablth}
  Assume $V\Omega_{g,n}$ is $\mathfrak{S}_{n}$-invariant for any $g \geq 0$ and $n \geq 1$ such that $2g - 2 + n > 0$. Then, for any admissible test function $f\colon\mathbb{R}_+ \to \mathbb{C}$ we have
  \[
    V\Omega_{g,n}(f;L_1,\ldots,L_n)
    =
    \sum_{\Gamma \in \mathbf{G}_{g,n}} \frac{1}{|{\rm Aut}\,\Gamma|}
      \int_{\mathbb{R}_{+}^{E_{\Gamma}}}
      \prod_{v \in V_{\Gamma}}
        V\Omega_{h(v),k(v)}\big((\ell_e)_{e \in E(v)},(L_{\lambda})_{\lambda \in \Lambda(v)}\big)
      \prod_{e \in E_{\Gamma}}
        \ell_e\,f(\ell_e) \dd \ell_e .
  \]
  \hfill $\blacksquare$
\end{thm}

We record two useful combinatorial identities, valid for any $\Gamma \in \mathbf{G}_{g,n}$.

\begin{lem}\label{combstgraph}
   \begin{align*}
	  \chi_{\Gamma} \coloneqq \sum_{v \in V_{\Gamma}} \big(2 - 2h(v) - k(v)\big) & = 2 - 2g + n, \\
     d_{\Gamma} \coloneqq \sum_{v \in V_{\Gamma}} \big(3h(v) - 3 + k(v)\big) & = 3g - 3 + n - |E_{\Gamma}|.
   \end{align*}
\end{lem}
\begin{proof}
The claim follows by combining edge counting with the definition of the first Betti number $b_1(\Gamma)$, namely
\[
   \sum_{v \in V_{\Gamma}} k(v) = 2|E_{\Gamma}| + n,\qquad 1 - |V_{\Gamma}| + |E_{\Gamma}| + \sum_{v \in V_{\Gamma}} h(v) = g.
\]
\end{proof}

%***************************************
\subsection{Equivalent forms of the topological recursion}
\label{Sequivalent:forms}
%***************************************

In this section, we describe equivalent forms of the topological recursion \eqref{TReqn}, which can be convenient for either carrying out calculations or for exploiting properties proved in the context of Eynard--Orantin topological recursion.

%***************************************
\subsubsection{Polynomial cases}
\label{Spoly}
%***************************************

Let $\phi_1: \mathbb{R}_+ \to \mathbb{R}$ and $\phi_2: \mathbb{R}_+ \times \mathbb{R}_+ \to \mathbb{R}$ be measurable functions. The operators
\begin{equation}\label{BCop}
   \hat{B}[\phi_1](L_1,L_2) = \int_{\mathbb{R}_{+}} B(L_1,L_2,\ell)\,\phi_1(\ell)\,\ell\,\dd\ell,\qquad \hat{C}[\phi_2](L_1) = \int_{\mathbb{R}_{+}^2} C(L_1,\ell,\ell')\,\phi_2(\ell,\ell')\,\ell\,\ell'\,\dd \ell\,\dd \ell'
\end{equation}
play an essential role in the topological recursion \eqref{TReqn}. It turns out that for Mirzakhani or Kontsevich initial data, these operators preserve the space of polynomials in one (for $\hat{B}$) or two (for $\hat{C}$) variables that are even with respect of each variable (we call them even polynomials). Since in both examples the base cases $(g,n) = (0,3)$ and $(1,1)$ are even polynomials in the length variables, it implies that $V\Omega^{{\rm M}}_{g,n}$ and $V\Omega^{{\rm K}}_{g,n}$ are even polynomials.

\begin{defn}
  We say that an initial data $(A,B,C,D)$ is polynomial if $(B,C)$ are such that the operators $\hat{B}$ and $\hat{C}$ defined in~\eqref{BCop} preserve the spaces of even polynomials and $A$ and $VD$ are themselves even polynomials.
\end{defn}

For polynomial initial data, it is sometimes more efficient for computations to decompose $V\Omega_{g,n}$ on a basis of monomials and write the effect of $\hat{B}$ and $\hat{C}$ on these monomials. For instance, let us decompose
$$
V\Omega_{g,n}(L_1,\ldots,L_n) = \sum_{d_1,\ldots,d_n \geq 0} F_{g,n}[d_1,\ldots,d_n]\,\prod_{i = 1}^n {\rm e}_{d_i}(L_i),\qquad  {\rm e}_{d}(\ell) = \frac{\ell^{2d}}{(2d + 1)!},
$$
and
$$
\hat{B}[{\rm e}_{d_3}](L_1,L_2) = \sum_{d_1,d_2 \geq 0} B^{d_1}_{d_2,d_3} \,{\rm e}_{d_1}(L_1){\rm e}_{d_2}(L_2),\qquad \hat{C}[{\rm e}_{d_2} \otimes {\rm e}_{d_3}](L_1) = \sum_{d_1 \geq 0} C^{d_1}_{d_2,d_3}\,{\rm e}_{d_1}(L_1).
$$
The topological recursion \eqref{TReqn} then takes the form, if $2g + n -2\geq 2$
\begin{equation}\label{recfgb}
\begin{split}
	& F_{g,n}[d_1,\ldots,d_n] \\
	&\quad = \sum_{m = 2}^n \sum_{a \geq 0} B^{d_1}_{d_m,a}\,F_{g,n - 1}[a,d_2,\ldots,\widehat{d_m},\ldots,d_n] \\
	&\qquad + \frac{1}{2} \sum_{a,b \geq 0} C^{d_1}_{a,b}\bigg(F_{g - 1,n + 1}[a,b,d_2,\ldots,d_n] + \sum_{\substack{h + h' = g \\ J \sqcup J' = \{d_2,\ldots,d_n\}}} F_{h,1 + |J|}[a,J]\,F_{h',1+|J'|}[b,J']\bigg). 
\end{split}
\end{equation}
For the sake of uniformity, we introduce a similar notation for the base cases of the recursion
$$
F_{0,3}[d_1,d_2,d_3] = A^{d_1}_{d_2,d_3},\qquad F_{1,1}[d_1] = D^{d_1}.
$$

For the Kontsevich initial data, we have in the chosen basis
$$
F_{g,n}[d_1,\ldots,d_n] = \prod_{i = 1}^n (2d_i + 1)!! \int_{\overline{\mathfrak{M}}_{g,n}} \prod_{i = 1}^n \psi_i^{d_i},
$$
which vanishes unless $d_1 + \cdots + d_n = 3g - 3 + n$. Translating the Virasoro constraints of \cite{Witten} -- or computing directly with \eqref{WKini} -- we find
\begin{equation}
\label{K1} F_{0,3}[d_1,d_2,d_3] = \delta_{d_1,d_2,d_3,0},\qquad F_{1,1}[d] = \frac{\delta_{d,1}}{8}.
\end{equation}
and
\begin{equation}
\label{K2} B^{d_1}_{d_2,d_3} = (2d_2 + 1)\delta_{d_1 + d_2,d_3},\qquad C^{d_1}_{d_2,d_3} = \delta_{d_1,d_2 + d_3 + 2}.
\end{equation}
A similar computation for Mirzakhani initial data can be found in \cite{Mirza1} and is reviewed in \cite{CoursToulouse} with notations closer to ours.

\medskip

Other bases of the space of even polynomials are sometimes useful to consider. For instance, the linear isomorphism given by the Laplace transform
$$
\mathcal{L}\colon\begin{array}{ccc} \mathbb{C}[L^2] & \longrightarrow & \mathbb{C}[p^{-2}]\dd p \\
\phi & \longrightarrow & \Big(\int_{\mathbb{R}_{+}} e^{-p\ell} \phi(\ell)\,\ell\,\dd \ell\Big)\dd p \end{array}
$$
is essential for the bridge to the Eynard--Orantin form of the topological recursion (see Section~\ref{SEO}).

%***************************************
\subsubsection{Twisting}
%***************************************

The operation of twisting \eqref{initwist} preserves the polynomiality of initial data. Indeed, the condition \eqref{condh} guarantees that all moments of the test function $f$ exist and if we set
\begin{equation}
\label{udd} u_{d_1,d_2} = \int_{\mathbb{R}_{+}} \frac{\ell^{2d_1 + 2d_2 + 1}}{(2d_1 + 1)!(2d_2 + 1)!}\,f(\ell)\, \dd \ell,
\end{equation}
we obtain
\begin{equation}\label{polytwist}
\begin{split}
	A[f]_{d_2,d_3}^{d_1} & = A^{d_1}_{d_2,d_3}, \\
	B[f]_{d_2,d_3}^{d_1} & = B^{d_1}_{d_2,d_3} + \sum_{a \geq 0} A^{d_1}_{d_2,a}\,u_{a,d_3}, \\
	C[f]_{d_2,d_3}^{d_1} & = C^{d_1}_{d_2,d_3} + \sum_{a \geq 0} \big(B^{d_1}_{a,d_3}\,u_{a,d_2} + B^{d_1}_{a,d_2}\,u_{a,d_3} \big) + \sum_{a,b \geq 0} A^{d_1}_{a,b}u_{a,d_2}u_{b,d_3}, \\
	D[f]^{d_1} & = D^{d_1} + \frac{1}{2} \sum_{a,b \geq 0} A^{d_1}_{a,b} u_{a,b}.
\end{split}
\end{equation}

Let us denote by $F_{g,n}[f;\,\cdot\,]$ the coefficients of decomposition of the twisted TR amplitudes. According to Theorem~\ref{stablth}, it can be expressed as a sum over decorated stable graphs
\begin{equation}
\label{Fgntwist} F_{g,n}[f;d_1,\ldots,d_n] = \sum_{\substack{\Gamma \in \mathbf{G}_{g,n} \\ d\colon H_{\Gamma} \rightarrow \mathbb{N}}} \frac{1}{|{\rm Aut}\,\Gamma|} \prod_{\substack{e \in E_{\Gamma} \\ e = (h,h')}} u_{h,h'} \prod_{v \in V_{\Gamma}} F_{h(v),k(v)}\big[(d_{e})_{e \in \overline{E}(v)}\big].
\end{equation}
In this sum we impose that the decoration of the $i$-th leaf is $d_i$. Similar twisting operations appear in the context of Givental group action on cohomological field theories, see \cite{GRpaper} for the comparison.

%***************************************
\subsubsection{Eynard--Orantin form}
\label{SEO}
%***************************************

Originally, the topological recursion was formulated by Eynard and Orantin as a residue computation on spectral curves \cite{EORev}. We present it in a restricted setting adapted to our needs. A local spectral curve is a triple $(x,y,\omega_{0,2})$ where
\begin{itemize}
\item[$\bullet$] $x$ and $y$ are holomorphic functions on a smooth complex curve $\mathcal{C}$;
\item[$\bullet$] the set $\mathfrak{a}$ of zeros of $\dd x$ is finite; each zero $\alpha \in \mathfrak{a}$ is simple and such that $\dd y(\alpha) \neq 0$;
\item[$\bullet$] $\omega_{0,2}$ is a meromorphic symmetric bidifferential on $\mathcal{C}^2$ with a double pole on the diagonal with biresidue $1$. The latter means that for any choice of local coordinate $p$ on $\mathcal{C}$, the bidifferential $\omega_{0,2}(z_1,z_2) - \frac{\dd p(z_1) \otimes \dd p(z_2)}{(p(z_1) - p(z_2))^2}$ is holomorphic near the diagonal in $\mathcal{C}^2$. 
\end{itemize}
We consider $x\colon\mathcal{C} \rightarrow \mathbb{C}$ as a double branched cover in a neighbourhood of $\alpha \in \mathfrak{a}$; it admits (locally) a non-trivial holomorphic automorphism $\tau_{\alpha}$ exchanging the two sheets, \textit{i.e.} $\tau_{\alpha}^2 = {\rm id}$ and $x \circ \tau_{\alpha} = x$, but $\tau_{\alpha} \neq {\rm id}$ and $\tau_{\alpha}(\alpha) = \alpha$. We introduce the recursion kernel
$$
K_{\alpha}(z_1,z) = \frac{1}{2}\,\frac{\int_{\tau_{\alpha}(z)}^{z} \omega_{0,2}(\cdot,z_1)}{(y(z) - y(\tau_{\alpha}(z)))\dd x(z)},
$$
and proceed to define multidifferentials $\omega_{g,n}(z_1,\ldots,z_n)$ for $g \geq 0$ and $n \geq 1$ as follows. We set $\omega_{0,1} = y\dd x$, further $\omega_{0,2}$ is part of the data of the local spectral curve, and for $2g - 2 + n > 0$ we define inductively
\begin{equation}\label{TREOeqn}
\begin{split}
	\omega_{g,n}(z_1,z_2,\ldots,z_n) &= \sum_{\alpha \in \mathfrak{a}} \Res_{z = \alpha} K_{\alpha}(z_1,z)\biggl( \omega_{g - 1,n + 1}(z,\tau_{\alpha}(z),z_2,\ldots,z_n) \\
	&\qquad + \sum_{\substack{h + h' = g \\ J \sqcup J' = \{z_2,\ldots,z_n\}}}^{{\rm no}\,\,(0,1)} \omega_{h,1 + |J|}(z,J)\otimes \omega_{h',1 + |J'|}(\tau_{\alpha}(z),J')\biggr),
\end{split}
\end{equation}
where $\sum\limits^{{\rm no }\,\,(0,1)}$ means that the sum excludes the cases where $(h,1+|J|)=(0,1)$ or $(h',1+|J'|)=(0,1)$. For $n = 0$ and $g \geq 2$, we also define the numbers
\begin{equation}
\label{omg0def} \omega_{g,0} = \frac{1}{2 - 2g} \sum_{\alpha \in \mathfrak{a}} \Res_{z = \alpha} \bigg(\int^{z}_{\alpha} y\dd x\bigg)\omega_{g,1}(z).
\end{equation}

As in the $(A,B,C,D)$ formulation, there is an operation of twisting in the Eynard--Orantin topological recursion, which consists in shifting $\omega_{0,2}$.

\begin{thm}
\label{variationomega02} Let $(\mathcal{C},x,y,\omega_{0,2})$ and $(\mathcal{C},x,y,\widetilde{\omega}_{0,2})$ be two local spectral curves, $\omega_{g,n}$ and $\widetilde{\omega}_{g,n}$ the respective outputs of the Eynard--Orantin topological recursion. We define two projectors $\mathscr{P}$ and $\widetilde{\mathscr{P}}$ acting on the space of meromorphic $1$-forms on $\mathcal{C}$, by the formulas
$$
\mathscr{P}[\phi](z) = \sum_{\alpha \in \mathfrak{a}} \Res_{z' = \alpha} \bigg(\int^{z'} \omega_{0,2}(\cdot,z)\bigg)\,\phi(z'),
$$
and likewise $\widetilde{\mathscr{P}}$ with $\widetilde{\omega}_{0,2}$. We denote $\mathscr{V} \coloneqq {\rm Im}\,\mathscr{P}$. We assume there exists a $2$-cycle $\mathscr{C} \subset \mathcal{C}^2$ and a germ $\Upsilon$ of holomorphic function at $\mathscr{C}$ such that
$$
\widetilde{\omega}_{0,2}(z_1,z_2) - \omega_{0,2}(z_1,z_2) = \int_{\mathscr{C}} \Upsilon(z_1',z_2')\,\omega_{0,2}(z_1,z_1')\omega_{0,2}(z_2,z_2').
$$
We define a linear form $\mathscr{O}$ on $\mathscr{V}^{\otimes 2}$ by the formula
$$
\mathscr{O}[\varpi] = \int_{\mathscr{C}} \Upsilon(z_1',z_2')\,\varpi(z_1',z_2').
$$
Then, we have
$$
\widetilde{\omega}_{g,n}(z_1,\ldots,z_n) = \sum_{\Gamma \in \mathbf{G}_{g,n}} \frac{1}{|{\rm Aut}\,\Gamma|} \bigg(\bigotimes_{i = 1}^n \widetilde{\mathscr{P}}_{z_i} \otimes \bigotimes_{e \in E_{\Gamma}} \mathscr{O}_{z'_{\vec{e}},z'_{-\vec{e}}}\bigg)\bigg[\bigotimes_{v \in V_{\Gamma}} \omega_{h(v),k(v)}\big((z'_{e})_{e \in \vec{E}(v)},(z_{\lambda})_{\lambda \in \Lambda(v)}\big)\bigg].
$$
In this formula: $\vec{E}(v)$ is the set of oriented edges pointing to $v$; each oriented edge $\vec{e}$ carries a variable $z'_{\vec{e}}$ ; to each $e \in E_{\Gamma}$ corresponds two oriented edges $\vec{e}$ and $-\vec{e}$; we indicate in subscripts of the operators which are the variables they act on.
\end{thm}
\begin{proof}[Sketch of proof] When $\mathcal{C}$ is a compact Riemann surface and $\tilde{\omega}_{0,2} = \omega_{0,2} + t h$ where $h$ is a fixed symmetric holomorphic bidifferential, \cite[Theorem 6.1, proved in Appendix B]{EOFg} establishes a formula for the first derivative of $\omega_{g,n}$ with respect to $t$. Integrating this relation with respect to $t$ yields the result -- in that case $\mathscr{C}$ is an element of ${\rm Sym}^2 H^1(\mathcal{C},\mathbb{C})$ and $\Upsilon = (2{\rm i}\pi)^{-2}$. The same proof in fact works under the assumptions of the theorem. Notice that the order of integration of the variables $z_{\vec{e}}$ is irrelevant, by the assumptions on $(\mathscr{C},\Upsilon)$ and the fact that $\omega_{h,k}$ only has poles on $\mathfrak{a}^k \subset \mathcal{C}^k$.
\end{proof}
%***************************************
\subsubsection{Equivalences}
%***************************************

The correspondence with Section~\ref{Spoly} appears if we decompose the $\omega_{g,n}$ on a suitable basis of $1$-forms. We explain it when $\dd x$ has a single zero, which is the only case where we are going to use this correspondence. Let us choose a coordinate $p$ near $\alpha$ such that $x = p^2/2 + x(\alpha)$. We introduce the $1$-form globally defined on $\mathcal{C}$
\begin{equation}
\label{xidform} \xi_{d}(z_0) = \Res_{z = \alpha} \frac{\dd p(z)}{p(z)^{2d + 2}}\,\bigg(\int_{\alpha}^{z} \omega_{0,2}(\cdot,z_0)\bigg).
\end{equation}
We also introduce
\begin{align*}
	\xi^*_{d}(z) & = (2d + 1)p(z)^{2d + 1}, \\
	\theta(z) & = \frac{-2}{(y(z) - y(\tau(z)))\dd x(z)} \mathop{\sim}_{z \rightarrow \alpha} \sum_{k \geq -1} \theta_{k}\,\frac{p(z)^{2k}}{\dd p(z)}, \\
	u_{0,0} & = \lim_{z_1 \rightarrow z_2} \bigg(\frac{\omega_{0,2}(z_1,z_2)}{\dd p(z_1)\dd p(z_2)} - \frac{1}{(p(z_1) - p(z_2))^2}\bigg).
\end{align*}

\begin{thm} \cite{ABCD} 
 \label{thmABCD} For $2g - 2 + n > 0$, we have
$$
\omega_{g,n}(z_1,\ldots,z_n) = \sum_{\substack{d_1,\ldots,d_n \geq 0 \\ d_1 + \cdots + d_n \leq 3g - 3 + n}} F_{g,n}[d_1,\ldots,d_n]\,\bigotimes_{i = 1}^n \xi_{d_i}(z_i),
$$
where the $F_{g,n}$'s are given by the recursion \eqref{recfgb} with
\begin{align*}
	A^{d_1}_{d_2,d_3} & = \Res_{z = \alpha} \xi_{d_1}^*(z)\dd \xi_{d_2}^*(z)\dd\xi_{d_3}^*(z)\,\theta(z), \\
	B^{d_1}_{d_2,d_3} & = \Res_{z = \alpha} \xi_{d_1}^*(z)\dd \xi_{d_2}^*(z)\xi_{d_3}(z)\,\theta(z), \\
	C^{d_1}_{d_2,d_3} & = \Res_{z = \alpha} \xi_{d_1}^*(z)\xi_{d_2}(z)\xi_{d_3}(z)\,\theta(z), \\
	VD^{d} & = \frac{\theta_{0} + u_{0,0}\theta_{-1}}{8}\,\delta_{d,0}  +  \frac{\theta_{-1}}{24}\,\delta_{d,1}.
\end{align*}
\hfill $\blacksquare$
\end{thm}

The $F_{g,n}$'s associated with Kontsevich or Mirzakhani initial data are described by Theorem~\ref{thmABCD} for the spectral curve $\mathcal{C} = \mathbb{C}$, $x(z) = z^2/2$ and $\omega_{0,2}(z_1,z_2) = \frac{\dd z_1 \otimes \dd z_2}{(z_1 - z_2)^2}$, for which $\tau(z) = -z$ and
\begin{equation}\label{ykm}
   y^{{\rm K}}(z) = -z,
   \qquad
   y^{{\rm M}}(z) = -\frac{\sin(2\pi z)}{2\pi}.
\end{equation}
In other words
\begin{equation}
   \theta^{{\rm K}}(z) = \frac{1}{z^2 \, \dd z},
   \qquad
   \theta^{{\rm M}}(z) = \frac{2\pi}{z\sin(2\pi z) \, \dd z}.
\end{equation}
More generally, if we assume that a polynomial GR initial data $(A,B,C,D)$ leads to TR amplitudes described by Theorem~\ref{thmABCD} for a certain spectral curve, then $\omega_{g,n}(z_1,\ldots,z_n)$ and $V\Omega_{g,n}(L_1,\ldots,L_n)$ are two equivalent ways of collecting the numbers $F_{g,n}$'s, which are related by the Laplace transform. Indeed, we notice that $\xi_{d} = p^{-(2d + 2)}\dd p + O(\dd p)$ and $p^{-(2d + 2)}\dd p = \mathcal{L}[{\rm e}_{d}]$. Let us introduce the projection operator
\[
   \mathcal{P}[\phi](p_0) = \Res_{z = \alpha} \frac{\phi(z)}{p(z) - p_0},
\]
which takes as input a meromorphic $1$-form on $\mathcal{C}$ and outputs the element of $\mathbb{C}[p_0^{-1}]\dd p_0$ such that $\phi(z_0) - \mathcal{P}[\phi](p_0)$ is holomorphic when $z_0 \rightarrow \alpha$. Hence $\mathcal{P}[\xi_{d}] = \mathcal{L}[{\rm e}_{d}]$ and
\begin{equation}\label{pdmi}
   \mathcal{L}^{\otimes n}[V\Omega_{g,n}](p_1,\ldots,p_n)
   =
   \mathcal{P}^{\otimes n}[\omega_{g,n}](p_1,\ldots,p_n).
\end{equation}
Furthermore, twisting the GR initial data amounts to shifting \cite{GRpaper}
\begin{equation}\label{02twist}
   \omega_{0,2}(z_1,z_2) \longrightarrow \omega_{0,2}(z_1,z_2) +  \mathcal{L}[f]\bigl( \pm p(z_1) \pm p(z_2) \bigr)\,\dd p(z_1)\dd p(z_2),
\end{equation}
where the two choice of signs $\pm$ are independent and arbitrary -- they do not affect the right-hand side of \eqref{pdmi}.

%***************************************
\section{Asymptotic growth of multicurves}
\label{S3}
%***************************************
%***************************************
\subsection{Preliminaries}
\label{S3Prem}
%***************************************

We review some aspects of the space of measured foliations which play a key role in this article. For a more complete description we refer to \cite{FLP12}.

\medskip

Let $\Sigma$ be a closed or punctured surface. A measured foliation is an ordered pair $\lambda = (\mathcal{F},\nu)$ where $\mathcal{F}$ is a foliation of $\Sigma$ whose leaves are $1$-dimensional submanifolds, except for the possible existence of isolated singular points of valency $p \geq 3$ away from the punctures and univalent at the punctures and $\nu$ is a transverse measure invariant along $\mathcal{F}$. Two measured foliations are Whitehead equivalent if they are related by a sequence of isotopies (relatively to the punctures), and contraction or expansion of edges between two singularities (that should not be both punctures). We denote by ${\rm MF}_{\Sigma}$ the set of Whitehead equivalence classes of measured foliations. For each $\sigma \in \mathfrak{T}_{\Sigma}$, ${\rm MF}_{\Sigma}$ is equipped with a hyperbolic length function which we denote by $\ell_{\sigma} \colon {\rm MF}_{\Sigma} \rightarrow \mathbb{R}_{+}$.

\medskip

The space ${\rm MF}_\Sigma$ is endowed with an integral piecewise linear structure, and the set of multicurves $M_\Sigma$ is in (length-preserving) bijection with the set of integral points of ${\rm MF}_\Sigma$. One can then define a measure $\mu_{\rm Th}$ by lattice point counting, which is called the Thurston measure in this context; we normalise $\mu_{\rm Th}$ such that $M_\Sigma$ has covolume one in ${\rm MF}_\Sigma$. Let us emphasise that our normalisation differs from the Thurston symplectic volume form by a constant factor, see~\cite{Arana,MoninTelpukhovskiy}.
 
\medskip

The space of quadratic differentials $Q\mathfrak{T}_\Sigma$ is intimately linked to ${\rm MF}_\Sigma$ by considering the horizontal and vertical foliations associated to a quadratic differential. More precisely we have a homeomorphism
\begin{equation}\label{eq:QvsMF}
\begin{array}{rcl}
   Q\mathfrak{T}_\Sigma & \longrightarrow & {\rm MF}_\Sigma \times {\rm MF}_\Sigma \setminus \Delta_{\Sigma} \\
   q & \longmapsto & \big(\big[\sqrt{|\Im(q)|}\,\big], \big[\sqrt{|\Re(q)|}\,\big]\big)
\end{array}
\end{equation}
where
\[
   \Delta_{\Sigma} = \Set{ (\lambda_1, \lambda_2) \in {\rm MF}_{\Sigma}^{2} \;|\; \exists \eta \in {\rm MF}_{\Sigma}, \; \iota(\eta, \lambda_1) + \iota(\eta, \lambda_2) = 0 },
\]
and  $\iota \colon {\rm MF}_{\Sigma} \times {\rm MF}_{\Sigma} \rightarrow {\mathbb R}_{\geq 0}$ is the geometric intersection pairing, which extends continuously the topological intersection of (formal $\mathbb{Q}_{+}$-linear combinations of) simple closed curves, see e.g. \cite{Bonahon88}.

\medskip

The subset of $Q\mathfrak{T}_{\Sigma}$ made of quadratic differentials with only simple zeros, the so called principal stratum, has an integral piecewise linear structure defined in terms of holonomy coordinates. The Masur--Veech measure $\mu_{{\rm MV}}$ is defined from this structure by lattice point counting \cite{Masur82,Veech82}. We define the Masur--Veech measure on the bundle $Q^1\mathcal{T}_{\Sigma}$ of quadratic differentials of unit area as follows. If $Y \subseteq Q^1\mathfrak{T}_{\Sigma}$, we put
$$
\mu_{\rm MV}^1(Y) = (12g - 12 + 4n)\mu_{{\rm MV}}(\widetilde{Y}),\qquad \widetilde{Y} = \Set{ t q  | t \in (0,\tfrac{1}{2})\,\,\text{ and }\,\,q \in Y}
$$
when $\tilde{Y}$ is measurable. This normalisation follows the one chosen in \cite{AEZ,Delecroix,Goujard}. Then the Masur--Veech volume is by definition the total mass $MV_{g,n} = \mu_{\rm MV}^1(Q^1\mathfrak{M}_{g,n}) < \infty$.

\medskip

Finally, we need to discuss Teichm\"uller spaces with zero boundary lengths. We introduce the space
$$
\widehat{\mathcal{T}}_{\Sigma} = \bigcup_{L_1,\ldots,L_n \geq 0} \mathcal{T}_{\Sigma}(L_1,\ldots,L_n),
$$
which is a stratified manifold. Its top-dimensional stratum is $\mathcal{T}_{\Sigma}$ and lower-dimensional strata correspond to some of the boundary length $L_i$ equal to zero. The lowest-dimensional stratum $\mathfrak{T}_{\Sigma} = \mathcal{T}_\Sigma(0, \ldots, 0)$ is identified with the Teichm\"uller space of punctured Riemann surfaces on $\Sigma$. The quotient of the action of ${\rm Mod}_{\Sigma}^{\partial}$ on $\widehat{\mathcal{T}}_{\Sigma}$ obviously respects the stratification, and is denoted by $\widehat{\mathcal{M}}_{g,n}$. Inside this moduli space, the lowest-dimensional stratum $\mathfrak{M}_{g,n} = \mathcal{M}_{g,n}(0,\ldots,0)$ is identified with the usual moduli space of complex curves with punctures.

\medskip

Following Thurston~\cite{Thurston}, we consider an asymmetric pseudo-distance on $\widehat{\mathcal{T}}_\Sigma$ defined for $\sigma, \sigma' \in \widehat{\mathcal{T}}_\Sigma$ as
\[
   \operatorname{d}_{\rm Th}(\sigma, \sigma')
   =
   \sup_{\gamma \in S_\Sigma^{\circ}}  \ln\bigg(\frac{\ell_{\sigma'}(\gamma)}{\ell_{\sigma}(\gamma)}\bigg).
\]
The fact that this quantity is finite follows from the compactness of the space of projective measured foliations. We emphasise that $S_{\Sigma}^{\circ}$ does not include boundary curves and hence $\operatorname{d}_{\rm Th}$ is constant equal to zero on the Teichm\"uller space $\widehat{\mathcal{T}}_P$ of the pair of pants $P$. It is expected that, on any other stable surface, $\operatorname{d}_{\rm Th}$ is actually an asymmetric distance, but this is irrelevant for our purposes. We will simply use the facts that $\operatorname{d}_{\rm Th}$ is non-negative, continuous and vanishes on the diagonal, \textit{i.e.} $\operatorname{d}_{\rm Th}(\sigma, \sigma) = 0$.

%***************************************
\subsection{Masur--Veech volumes}
\label{SMVvolumes}
%***************************************

In this paragraph, $g$ and $n$ are non-negative integers and $2g - 2 + n > 0$. Let $\phi\colon \mathbb{R}_{+} \rightarrow \mathbb{C}$ be an admissible test function and $\Sigma$ a surface of type $(g,n)$. We introduce the additive statistics for $\sigma \in \widehat{\mathcal{T}}_\Sigma$
$$
	N^+_{\Sigma}(\phi;\sigma) = \sum_{c \in M_{\Sigma}} \phi (\ell_{\sigma}(c)).
$$
We are interested in some scaling limit of the additive statistics $N^+_\Sigma(\phi; \sigma)$. Namely, we define for $\beta > 0$ the scaling operator
\begin{equation}
\label{rhobetastar}   \rho_\beta^* \phi (x) = \phi(x / \beta).
\end{equation}
and we want to understand the behaviour of $N^{+}_{\Sigma}(\rho^*_{\beta}\phi;\sigma)$ and its integrals over the moduli spaces with fixed boundary lengths.

\medskip

The result (Lemma~\ref{Xfn:cv} below) will be governed by two ingredients. First, the dependence on the test function will involve the following linear forms, for $k \geq 0$
\begin{equation}\label{cdeff}
   c_k[\phi] = \int_{\mathbb{R}_{+}} \frac{ \ell^{k-1}}{(k-1)!}\,\phi(\ell)\,\dd \ell.
\end{equation}
Note that $c_{0}[\phi]$ is not always well-defined; we will assume it is only when necessary. Second, the dependence on the metric will be governed by the function
$$
   X_\Sigma\colon
   \begin{array}{lll}
      \widehat{\mathcal{T}}_\Sigma & \longrightarrow & \mathbb{R}_{+} \\
      \sigma  & \longmapsto & (6g-6+2n)!\ \mu_{\rm Th}\big(\Set{ \lambda \in {\rm MF}_\Sigma | \ell_\sigma(\lambda) \leq 1}\big)
   \end{array}.
$$
The function $X_{\Sigma}$ is an important ingredient in~\cite{Mirza1}, where most of its properties
are proven. In particular, its integral over moduli space is proportional to the Masur--Veech volumes.

\begin{lem}\label{Xfn:prop}
   The function $X_{\Sigma}$ descends to a function $X_{g,n}$ on the moduli space $\widehat{\mathcal{M}}_{g,n}$. Further, the following properties hold.
   \begin{itemize}
      \item[$\bullet$]
         The logarithm $\ln(X_\Sigma)$ is Lipschitz with respect to $d_{\rm Th}$, namely
         $$
            \frac{X_\Sigma(\sigma)}{X_{\Sigma}(\sigma')}  \leq e^{(6g-6+2n) d_{\rm Th}(\sigma, \sigma')}.
         $$
      \item[$\bullet$]
         The average $VX_{g,n}(L_1,\ldots,L_n)$ exists and is a continuous function of $(L_1, \ldots, L_n) \in (\mathbb{R}_{\geq 0})^n$.
      \item[$\bullet$]
         We have that
         $$
            MV_{g,n} = \frac{2^{4g - 2 + n}(4g - 4 + n)!}{(6g - 7 + 2n)!}\,VX_{g,n}(0,\ldots,0).
         $$
   \end{itemize}
\end{lem}

\begin{lem}\label{Xfn:cv}
   Let $\sigma \in \widehat{\mathcal{T}}_\Sigma$ and $\phi\colon \mathbb{R}_{+} \to \mathbb{C}$ be an admissible test function. Then
   $$
      \lim_{\beta \rightarrow \infty} \beta^{-(6g - 6 + 2n)}\,N^+_{\Sigma}(\rho_\beta^* \phi;\sigma)
      =
      c_{6g - 6 + 2n}[\phi]\ X_\Sigma(\sigma),
   $$
   and further, the following limit exists and it equals
   $$
      \lim_{\beta \rightarrow \infty} \beta^{-(6g - 6 + 2n)}\,VN^{+}_{g,n}(\rho_{\beta}^*\phi;L_1,\ldots,L_n) = c_{6g - 6 + 2n}[\phi]\,VX_{g,n}(L_1,\ldots,L_n).
   $$ 
   for all $(L_1,\ldots,L_n) \in \mathbb{R}_{\geq 0}^n$.
\end{lem}

\begin{proof}[Proof of Lemma~\ref{Xfn:prop}]
The first property follows from the inclusion of the unit $\ell_{\sigma}$-ball in a $\ell_{\sigma'}$-ball:
$$ 
   \Set{\lambda \in {\rm MF}_\Sigma | \ell_\sigma(\lambda) \leq 1 } \subseteq \Set{ \lambda \in {\rm MF}_\Sigma | \ell_{\sigma'}(\lambda) \leq e^{d_{\rm Th}(\sigma, \sigma')} }.
$$
The integrability of $X_\Sigma$ is proven in Theorem~3.3 of \cite[p.~106]{Mirzagrowth}. Namely, the function $X_\Sigma$ is bounded by the function
\[
   K_{\Sigma}(\sigma) = \kappa\,\prod_{\substack{\gamma \in S_{\Sigma}^{\circ} \\ \ell_\sigma(\gamma) \leq \epsilon}} \frac{1}{\ell_\sigma(\gamma)}
\]
for appropriate constants $\kappa,\epsilon > 0$ that depend only on $g$ and $n$. The function $K_{\Sigma}$ is invariant under the action of the mapping class group and we denote by $K_{g,n}$ the function it induces on the moduli space. Mirzakhani showed that $K_{g,n}$ is integrable with respect to $\mu_{{\rm WP}}$ over $\mathcal{M}_{g,n}(L)$
for any $L \in \mathbb{R}_{\geq 0}^n$ (see her proof of Theorem~3.3 in~\cite[pp.~111-112]{Mirzagrowth}).

\medskip

We now prove that the integral $V X_{g,n}(L)$ is a continuous function of $L$. Let us choose a pair of pants decomposition of $\Sigma$ and consider the corresponding Fenchel--Nielsen coordinates $(\ell_i,\tau_i)_{i = 1}^{3g - 3 + n}$ realising $\mathcal{T}_{\Sigma}(L) \simeq (\mathbb{R}_{+} \times \mathbb{R})^{3g - 3 + n}$. By continuity of $X_\Sigma(\sigma)$, for any compact set $Z \subset (\R_{+} \times \R)^{3g-3+n}$ the following function is continuous
\[
	L \longmapsto \int_{\{L\} \times Z} X_\Sigma(\sigma) \dd\mu_{\rm WP}(\sigma).
\] 
In order to show the continuity of $VX_{g,n}$, it remains to show that the contribution coming from the set $\mathcal{M}_{g,n}^{< \epsilon'}(L) \subset \mathcal{M}_{g,n}(L)$ of surfaces with a non-peripheral curve of length smaller than $\epsilon'$ is uniformly small in $\epsilon'$. We use again the function $K_{g,n}$, for which
\[
   \int_{\mathcal{M}_{g,n}^{< \epsilon'}(L)} X_{g,n}(\sigma) \dd\mu_{\rm WP}
   \leq
   \int_{\mathcal{M}_{g,n}^{< \epsilon'}(L)} K_{g,n}(\sigma) \dd\mu_{\rm WP}.
\]
The set $\mathcal{M}_{g,n}^{< \epsilon'}(L)$ is covered by the $(3g-3+n) 2^{3g-4+n}$ sets
\[
   Y_{i_0, J}^{<\epsilon'}(L) = \pi\Big(\{L\}\times
       \Set{
       	(\ell_i,\tau_i)_{i = 1}^{3g - 3 + n} |
       	\ell_{i_0} \leq \epsilon',
       	\; \ell_j \leq \epsilon \;\; \forall j \in J,
       	\; \ell_i \leq b_{g,n}(L) \;\; \forall i,
       	\; 0 \leq \tau_i \leq \ell_i
       } \Big),
\]
where $J$ is a subset of $\{1, 2, \ldots, 3g-3+n\}$, $i_0$ an integer in the complement of $J$, $\pi\colon\mathcal{T}_{\Sigma} \rightarrow \mathcal{M}_{\Sigma}$ is the projection map, and $b_{g,n}(L)$ is the Bers constant of $\mathcal{T}_{\Sigma}(L)$. It is shown in~\cite{BalacheffParlierSabourau} that $b_{g,n}(L)$ is uniformly bounded for $L$ in compact subsets of $\mathbb{R}_{\geq 0}^n$. Now, given a point in $\mathcal{M}_{g,n}^{< \epsilon'}$, one can always choose a hyperbolic structure in its $\pi$-preimage so that all curves shorter than $\epsilon$ are contained in the pants decomposition. Hence
\begin{align*}
   \int_{\mathcal{M}_{g,n}^{< \epsilon'}(L)} K_{g,n}(\sigma)\,\dd\mu_{\rm WP}(\sigma)
   &\leq
   \kappa\ \sum_{i_0, J} \int_{Y_{i_0,J}^{< \epsilon'}} \prod_{j \in J} \frac{1}{\ell_j} \prod_{i = 1}^{3g - 3 + n} \dd \ell_i \dd \tau_i  \\
   &\leq
   \kappa\ \sum_{i_0, J} \epsilon'\,\epsilon^{|J|}\ \big(b_{g,n}(L)\big)^{2(3g-4+n-|J|)} \\
   &\leq
   \kappa (3g-3+n)\ 2^{3g-4+n} \big(b_{g,n}(L)\big)^{2(3g-4+n)}\epsilon'.
\end{align*}
This concludes the proof of the continuity.
 
\medskip

The proportionality with the Masur--Veech volume is derived in~\cite{Mirzakhani-earthquake} for closed surfaces and extended to punctured surfaces in \cite{Delecroix}. We only sketch the idea. Associated to any maximal measured foliation $\lambda$, Thurston \cite{Thurston} and Bonahon \cite{Bonahon} constructed an analytic embedding
\[
   G_\lambda \colon {\mathfrak T}_{\Sigma} \longrightarrow {\rm H}_{\Sigma}(\lambda),
\]
where ${\rm H}_{\Sigma}(\lambda)$ are the transverse H\"older distributions on (the support of) $\lambda$. The transverse H\"older distributions form a vector space of dimension $6g-6+2n$ which plays the role of the tangent space at $\lambda$ in ${\rm MF}_\Sigma$, see~\cite{Bonahon}. Mirzakhani then proved that $G_\lambda$ factors through the space of measured foliations as $G_\lambda = I_\lambda \circ F_\lambda$, where
$$
   F_\lambda \colon {\mathfrak T}_\Sigma \longrightarrow {\rm MF}_\Sigma(\lambda)
   \qquad
   I_\lambda \colon {\rm MF}_\Sigma(\lambda) \longrightarrow {\rm H}_{\Sigma}(\lambda)
$$
are respectively the \emph{horocyclic foliation} and \emph{shearing coordinates}, and where
\[
   {\rm MF}_\Sigma(\lambda) = \Set{\eta \in {\rm MF}_\Sigma\,\,|\,\,\forall \gamma \in S_\Sigma^{\circ},\,\,\,\,\iota(\lambda,\gamma) + \iota(\eta,\gamma) > 0}.
\]
It is shown in \cite{Bonahon,Mirzakhani-earthquake} that these maps are symplectomorphisms
with respect to the Weil--Petersson symplectic form on $\mathfrak{T}_\Sigma$ and the Thurston symplectic forms on ${\rm H}_{\Sigma}(\lambda)$ and ${\rm MF}_\Sigma$. As a consequence, on the subset ${\rm MF}^{{\rm max}}_\Sigma$ of maximal foliations -- which has full measure in ${\rm MF}_\Sigma$ -- we obtain a map
$$
\begin{array}{ccc}
   {\mathfrak T}_\Sigma \times {\rm MF}^{{\rm max}}_\Sigma & \longrightarrow & {\rm MF}_{\Sigma} \times {\rm MF}_{\Sigma} \\
   (\sigma,\lambda) & \longmapsto & \big(\lambda,F_\lambda(\sigma) \big)
\end{array}
$$
which is again a symplectomorphism.

\medskip

On the other hand, the homeomorphism \eqref{eq:QvsMF} $Q\mathfrak{T}_{\Sigma} \rightarrow {\rm MF}_{\Sigma} \times {\rm MF}_{\Sigma} \setminus \Delta_{\Sigma}$ maps $\mu_{\rm MV}$ to $\mu_{\rm Th} \otimes \mu_{\rm Th}$, up to a constant factor. In order to match our normalisation of $\mu_{{\rm MV}}$ one has to include
the factor that corresponds to the ratio between the Thurston symplectic volume form and the measure obtained via integral points in ${\rm MF}_\Sigma$, see~\cite{Arana,MoninTelpukhovskiy}.
\end{proof} 

\begin{proof}[Proof of Lemma~\ref{Xfn:cv}]
Since $\phi$ is Riemann integrable, we have
\begin{equation}\label{eq:conv:add:stat}
   \lim_{\beta \to \infty} \beta^{-6g-6+2n} N^+_\Sigma(\rho_{\beta}^*\phi;\sigma)
   =
   \int_{{\rm MF}_\Sigma} \phi \circ \ell_\sigma(\lambda)\,\dd\mu_{\rm Th}(\lambda).
\end{equation}
Now, we can desintegrate the Thurston measure with respect to the function $\ell_\sigma$.
We denote by $\overline{\mu}$ the projectivised measure on $\mathbb{P}{\rm MF}_\Sigma$ defined by
$$
   \overline{\mu}(A) = \mu_{{\rm Th}}\big(\Set{\lambda \in {\rm MF}_\Sigma | [\lambda] \in A\,\,\,\text{and}\,\,\,\ell_\sigma(\lambda) \leq 1 } \big),
$$
where $[\lambda]$ denotes the projective class of $\lambda$. Then we have the ``polar form'' of the Thurston measure
$$
   \mu_{{\rm Th}} = (6g-6+2n)\,t^{6g-7+2n}\,\dd t\,\dd \overline{\mu}.
$$
The right hand side in~\eqref{eq:conv:add:stat} hence can be rewritten as
$$
   (6g-6+2n)\bigg(\int_{\mathbb{R}_{+}} t^{6g-7+2n}\,\phi(t)\, \dd t\bigg)\,\mu_{{\rm Th}}\big(\Set{ \lambda \in {\rm MF}_\Sigma | \ell_\sigma(\lambda) \leq 1 }\big).
$$
The above is equivalent to the first part of the Lemma.

\medskip

To complete the proof, we should justify that the limit $\beta \rightarrow \infty$ and the integral over the moduli spaces can be exchanged. We will do so by dominated convergence. Let us denote
$$
   \mathscr{N}_{\Sigma}(R;\sigma) = \Set{c \in M_\Sigma | \ell_\sigma(c) \leq R }.
$$
By~\cite[Proposition~3.6 and Theorem~3.3]{Mirzagrowth} we have
$$
   |\mathscr{N}_{\Sigma}(R;\sigma)| \leq K_{\Sigma}(\sigma)\,R^{6g-6+2n}.
$$
Now we have
   \begin{align*}
   \beta^{-(6g - 6 + 2n)}\,N^+_\Sigma(\rho^*_{\beta}\phi;\sigma) & = \beta^{-(6g - 6 + 2n)} \sum_{k \geq 0} \sum_{\substack{c \in M_\Sigma \\ \beta k\leq \ell_{\sigma}(c) < \beta(k+1)}} \,\phi\bigg(\frac{\ell_\sigma(c)}{\beta} \bigg) \\
   & \leq K_{\Sigma}(\sigma) \bigg(\sum_{k \geq 0} (k + 1)^{6g - 6 + 2n} \sup_{k \leq \ell < k + 1}\,|\phi(\ell)|\bigg) \\
   & \leq K_{\Sigma}(\sigma) \bigg( \sum_{k \geq 0} (k+1)^{-2} \bigg) \sup_{\ell \geq 0}\, (\ell+2)^{6g-4+2n} |\phi(\ell)|.
\end{align*}
The right hand side is bounded by the decay assumption~\eqref{condh}. By Lemma~\ref{Xfn:prop}, for any $L \in \mathbb{R}_{\geq 0}^n$ the right-hand side is integrable against the Weil--Petersson measure over $\mathcal{M}_{g,n}(L)$. It is independent of $\beta$, so the conclusion follows by dominated convergence.
\end{proof}

%***************************************
\subsection{Definition of the Masur--Veech polynomials}
\label{MVpo}
%***************************************

We introduce the Masur--Veech polynomials, for any $g,n \geq 0$ such that $2g - 2 + n > 0$, by setting
\begin{equation}\label{defn:MVpoly}
	V\Omega_{g,n}^{{\rm MV}}(L_1,\ldots,L_n) = \sum_{\Gamma \in \mathbf{G}_{g,n}} \frac{1}{|{\rm Aut}\,\Gamma|} \int_{\mathbb{R}_{+}^{E_{\Gamma}}}  \prod_{v \in V_{\Gamma}} V\Omega^{{\rm K}}_{h(v),k(v)}\big((\ell_e)_{e \in E(v)},(L_{\lambda})_{\lambda \in \Lambda(v)}\big)\prod_{e \in E_{\Gamma}} \frac{\ell_e\dd\ell_e }{e^{\ell_e} - 1}.	
\end{equation}
These are polynomials in the variables $(L_i^2)_{i = 1}^n$ of total degree $3g - 3 + n$. Its terms of maximal total degree come from the stable graph with a single vertex of genus $g$ with $n$ leaves; therefore coincide with the Kontsevich volumes of the combinatorial moduli space $V\Omega^{{\rm K}}_{g,n}(L_1,\ldots,L_n)$. In order to evaluate the sum over stable graphs, we need the following integral.

\begin{lem}\label{Lzeta}
   The function $f^{{\rm MV}}(\ell) = \frac{1}{e^{\ell} - 1}$ is such that, for any $k \geq 0$,
   $$
      \int_{\mathbb{R}_{+}} f^{{\rm MV}}(\ell)\ell^{2k + 1}\,\dd \ell = (2k + 1)!\zeta(2k + 2).
   $$
\end{lem}

\begin{proof}
We compute
\bea
(2k + 1)!\zeta(2k + 2) & = & \sum_{n \geq 1} \frac{1}{n^{2k + 2}}\,\int_{\mathbb{R}_{+}}\,e^{-t}\,t^{2k + 1}\,  \dd t= \sum_{n \geq 1} \int_{\mathbb{R}_{+}} \,e^{-n\ell}\,\ell^{2k + 1}\,\dd \ell \nonumber \\
& = & \int_{\mathbb{R}_{+}} \frac{e^{-\ell}}{1 - e^{-\ell}}\,\ell^{2k + 1}\,\dd\ell = \int_{\mathbb{R}_{+}} \frac{1}{e^{\ell} - 1}\,\ell^{2k + 1}\, \dd \ell.\nonumber
\eea
\end{proof}

For $n = 0$ and $g \geq 2$, $V\Omega_{g,0}^{{\rm MV}}$ is a number, which can also be extracted from $V\Omega_{g,1}^{{\rm MV}}(L_1)$, as a particular case of the following formula.

\begin{lem}\label{dilatonthm0}
   For any $g,n \geq 0$ such that $2g - 2 + n > 0$, we have the dilaton equation
   $$
      \big[\tfrac{L_{n + 1}^2}{2}\big]\,V\Omega_{g,n + 1}^{{\rm MV}}(L_1,\ldots,L_{n + 1}) = (2g - 2 + n)\,V\Omega_{g,n}^{{\rm MV}}(L_1,\ldots,L_n),
   $$
   where $\big[\tfrac{\ell^2}{2}\big]$ extracts the coefficient of $\tfrac{\ell^2}{2}$ in the polynomial to its right. In particular, for $g \geq 2$ we have that
   $$
      \big[\tfrac{L^2}{2}\big]\,V\Omega_{g,1}^{{\rm MV}}(L) = (2g - 2)V\Omega_{g,0}^{{\rm MV}}.
   $$
\end{lem}

\begin{proof}
We introduce
$$
   \mathbf{G}_{g,n}^{\bullet} = \Set{ (\Gamma,v) | \Gamma \in \mathbf{G}_{g,n}\,\,\,{\rm and}\,\,\,v \in V_{\Gamma} }
$$
and the surjective map $\pi\colon\mathbf{G}_{g,n + 1} \rightarrow \mathbf{G}_{g,n}^{\bullet}$ which erase the $(n + 1)$-th leaf from the stable graph, but records the information of the vertex $v$ to which this leaf was incident. In general $\pi$ is not injective, but one can check that for any $\Gamma \in \mathbf{G}_{g,n}$, $v \in V_{\Gamma}$ and $\tilde{\Gamma} \in \pi^{-1}(\Gamma,v)$, we have
\begin{equation}\label{autfug}
	|{\rm Aut}\,\Gamma|= |\pi^{-1}(\Gamma,v)|\,|{\rm Aut}\,\tilde{\Gamma}|.
\end{equation}
The dilaton equation for the $\psi$ classes intersections yields, for $2h - 2 + (k + 1) > 0$
$$
   \big[\tfrac{\ell_{k + 1}^2}{2}\big]\,\,V\Omega^{{\rm K}}_{h,k + 1}(\ell_1,\ldots,\ell_{k + 1}) = (2h - 2 + k)\,V\Omega^{{\rm K}}_{h,k}(\ell_1,\ldots,\ell_k),
$$
and this expression vanishes when $2h - 2 + k = 0$. Therefore
\[
\begin{split}
   & \big[\tfrac{L_{n + 1}^2}{2}\big]\,\,V\Omega^{{\rm MV}}_{g,n + 1}(L_1,\ldots,L_{n +1}) \\
   & =  \sum_{(\Gamma,v) \in \mathbf{G}_{g,n}^{\bullet}} \frac{2h(v) - 2 + k(v)}{|{\rm Aut}\,\tilde{\Gamma}|} \int_{\mathbb{R}_{+}^{E_{\Gamma}}} \prod_{w \in V_{\Gamma}} V\Omega^{{\rm K}}_{h(w),k(w)} \big((\ell_e)_{e \in E(w)},(L_{\lambda})_{\lambda \in \Lambda(w)}\big) \prod_{e \in E_{\Gamma}} \frac{\ell_e\,\dd\ell_e}{e^{\ell_e} - 1},
\end{split}
\]
where $\tilde{\Gamma}$ is any element of $\pi^{-1}(\Gamma,v)$, $E(w)$ and $\Lambda(w)$ are the sets of edges and leaves incident to $w$ in the graph $\Gamma$ (and not in $\tilde{\Gamma}$). Using \eqref{autfug} we deduce that
\[
\begin{split}
	\big[\tfrac{L_{n + 1}^2}{2}\big]\,\,V\Omega^{{\rm MV}}_{g,n + 1}(L_1,\ldots,L_{n + 1}) & = \sum_{\Gamma \in \mathbf{G}_{g,n}} \bigg(\sum_{v \in V_{\Gamma}} 2h(v) - 2 + k(v)\bigg) \\ 
	& \times \frac{1}{|{\rm Aut}\,\Gamma|} \int_{\mathbb{R}_{+}^{E_{\Gamma}}} \prod_{w \in V_{\Gamma}} V\Omega^{{\rm K}}_{h(w),k(w)} \big((\ell_e)_{e \in E(w)},(L_{\lambda})_{\lambda \in \Lambda(v)}\big) \prod_{e \in E_{\Gamma}} \frac{\ell_e\,\dd\ell_e}{e^{\ell_e} - 1}.
\end{split}
\]
By Lemma~\ref{combstgraph}, the sum of Euler characteristics at the vertices is $2g - 2 + n$, hence the claim.
\end{proof}

%***************************************
\subsection{Main result}
%***************************************

In \cite{Delecroix}, V.D., Goujard, Zograf and Zorich obtained by combinatorial methods a formula for the Masur--Veech volumes as a sum over stable graphs, exploiting the relation between Masur--Veech volumes and lattice point counting in the moduli space of quadratic differentials. Our proof is different and relies on ideas of the geometric recursion reviewed in Section~\ref{SPre}. Our method gives access to more general quantities, which we introduced under the name of Masur--Veech polynomials. We now prove that they record the asymptotic growth of the number of multicurves on surfaces with large boundaries, after integration against the Weil--Petersson measure. As a consequence of Lemma~\ref{Xfn:cv}, we then show that the Masur--Veech volumes arise as the constant term of the Masur--Veech polynomials, up to normalisation.

\medskip 

Let us denote $\hat{c}[\phi] \colon \mathbb{C}[t^{-1}] \rightarrow \mathbb{C}$ the linear operator sending $t^{-k}$ to $c_{k}[\phi]$ for $k > 0$, and $t^0$ to $\phi(0)$ when it exists. Recall the definition \eqref{rhobetastar} of the rescaling operator $\rho_{\beta}^*$.

\begin{thm}\label{stablg}
   Let $\phi$ be an admissible test function admitting a Laplace representation
   $$
      \phi(\ell) = \int_{\mathbb{R}_{+}} \Phi(t)\,e^{-t \ell}\,\dd t
   $$
   for a measurable function $\Phi$ such that $t \mapsto |\Phi(t)|$ is integrable on $\mathbb{R}_{+}$. In particular, $\phi(0) = \lim_{\ell \rightarrow 0} \phi(\ell)$ exists. Then, for any $g,n \geq 0$ such that $2g - 2 + n > 0$, we have that
   \begin{align*}
   	\lim_{\beta \rightarrow \infty} \beta^{-(6g - 6 + 2n)}\,VN^{+}_{g,n}\big(\rho^*_{\beta}\phi;L_1,\ldots,L_n\big)
   	& =  c_{6g - 6 + 2n}[\phi]\,V\Omega^{{\rm MV}}_{g,n}(0,\ldots,0), \\
   	\lim_{\beta \rightarrow \infty} \beta^{-(6g - 6 + 2n)}\,VN^{+}_{g,n}\big(\rho^*_{\beta}\phi;\beta L_1,\ldots,\beta L_n\big)
   	& = \hat{c}[\phi]\big( t^{-(6g - 6 + 2n)}\,V\Omega^{{\rm MV}}_{g,n}(tL_1,\ldots,tL_n)\big),
   \end{align*}
   and the convergence is uniform for $L_i$ in any compact of $\mathbb{R}_{\geq 0}$.
\end{thm}

Notice that the contribution of the test function factors out for finite boundary lengths. The assumption that $\phi$ has a Laplace representation is not essential. It could be waived by an approximation argument, if we had an integrable upper bound for the number of multicurves whose lengths belong to a segment $[\beta L_1,\beta L_2]$. This is not currently available in the literature and we do not address this question here.

\medskip

In particular, comparing the last formula of Lemma~\ref{Xfn:prop} with the second formula in Lemma~\ref{Xfn:cv}, we obtain

\begin{cor}\label{MVcoco}
   For any $g,n \geq 0$ such that $2g - 2 + n > 0$, $VX_{g,n}(L_1,\ldots,L_n) = V\Omega^{{\rm MV}}_{g,n}(0,\ldots,0)$ is independent of $L_1,\ldots,L_n \in \mathbb{R}_{\geq 0}$, and the Masur--Veech volumes are
   $$
      MV_{g,n} = \frac{2^{4g - 2 + n}(4g - 4 + n)!}{(6g - 7 + 2n)!}\,V\Omega^{{\rm MV}}_{g,n}(0,\ldots,0).
   $$
   \hfill $\blacksquare$
\end{cor}

It would be interesting to provide an a priori explanation of why $VX_{g,n}$ is independent of the boundary lengths $L_1,\ldots,L_n$; for us it is merely the consequence of a computation.

\begin{proof}[Proof of Theorem~\ref{stablg}]
We fix once and for all $g$ and $n$ such that $2g - 2 + n > 0$. In the admissibility assumption, we will only use a weaker form of decay
\begin{equation}\label{conhs}
   \sup_{\ell > 0} \,(1 + \ell)^{6g - 6 + 2n + \delta} |\phi(\ell)| < +\infty,
\end{equation}
with $\delta = 1$. 

\medskip

The Laplace representation of $\phi$ allows us to convert additive statistics into multiplicative statistics. We are going to apply many times the Fubini--Tonelli and the dominated convergence theorems.

\medskip

Admissibility implies convergence of the series
$$
   N^{+}_{\Sigma}(\rho^*_{\beta}\phi;\sigma) = \sum_{c \in M_{\Sigma}} \phi\bigg(\frac{\ell_{\sigma}(c)}{\beta}\bigg) = \sum_{c \in M_{\Sigma}} \int_{\mathbb{R}_{+}} \Phi(t)\,\prod_{\gamma \in \pi_0(c)} e^{-t\ell_{\sigma}(c)/\beta}\dd t.
$$
By the Fubini--Tonelli theorem applied twice, we have
\begin{align}
   \label{Nplus}
   N^{+}_{\Sigma}(\rho^*_{\beta}\phi;\sigma) & = \int_{\mathbb{R}_{+}}\Phi(t)\,N^{t/\beta}_{\Sigma}(\sigma)\, \dd t, \\
   \label{VNplus}
   VN^{+}_{g,n}(\rho^*_{\beta}\phi;\beta L_1,\ldots,L_n) & = \int_{\mathbb{R}_{+}} \Phi(t)\,VN^{t/\beta}_{g,n}(\beta L_1,\ldots,\beta L_n) \,\dd t,
\end{align}
where
\begin{align*}
   N^{t}_{\Sigma}(\sigma)
   & =
   \sum_{c \in M_{\Sigma}} \prod_{\gamma \in \pi_0(c)} e^{-t\ell_{\sigma}(\gamma)}
   =
   \sum_{c \in M'_{\Sigma}} \prod_{\gamma \in \pi_0(c)} \frac{1}{e^{t\ell_{\sigma}(\gamma)} - 1}, \\
   VN^{t}_{\Sigma}(L_1,\ldots,L_n)
   & =
   \int_{\mathcal{M}_{g,n}(L_1,\ldots,L_n)} N^{t}_{g,n}(\sigma)\,\dd\mu_{{\rm WP}}(\sigma),
\end{align*}
are now multiplicative statistics, to which we can apply the theory reviewed in Section~\ref{SPre}. 

\medskip

For $t > 0$ and $n \geq 1$, by Theorem~\ref{stablth} we have
\begin{equation}\label{VNtb} 
\begin{split}
	& VN^{t/\beta}_{g,n}(\beta L_1,\ldots,\beta L_n)  \\
	& = \int_{\mathcal{M}_{g,n}(\beta L_1,\ldots,\beta L_n)} N^{t/\beta}_{g,n}(\sigma)\,\dd\mu_{{\rm WP}}(\sigma) \\
	& = \sum_{\Gamma \in \mathbf{G}_{g,n}} \frac{1}{|{\rm Aut}\,\Gamma|} \int_{\mathbb{R}_{+}^{E_{\Gamma}}}  \prod_{v \in V_{\Gamma}} V\Omega^{{\rm M}}_{h(v),k(v)}\big((\ell_e)_{e \in E(v)},(\beta L_{\lambda})_{\lambda \in \Lambda(v)}\big)\prod_{e \in E_{\Gamma}} \frac{\ell_e\,\dd \ell_e}{e^{t \ell_e/\beta} - 1} \\
	& = \sum_{\Gamma \in \mathbf{G}_{g,n}} \frac{(\beta/t)^{2|E_{\Gamma}|}}{|{\rm Aut}\,\Gamma|} \int_{\mathbb{R}_{+}^{E_{\Gamma}}}  \prod_{v \in V_{\Gamma}} V\Omega^{{\rm M}}_{h(v),k(v)}\big((\beta \ell_e/t)_{e \in E(v)},(t\cdot \beta L_{\lambda}/t)_{\lambda \in \Lambda(v)}\big)\prod_{e \in E_{\Gamma}} \frac{\ell_e\,\dd \ell_e}{e^{\ell_e} - 1}.
\end{split}
\end{equation}
This formula is also true for $n = 0$, as can be shown by returning to the computations proving Theorem~\ref{stablth}.

\medskip

We remark that $\beta^{-(6g - 6 + 2n)}VN_{g,n}^{t/\beta}(\beta L_1,\ldots,\beta L_n)$ is a polynomial in $t^{-1}$ and $\beta^{-1}$ of bounded degree by (\ref{VNtb}). We observe that
$$
   \int_{\mathbb{R}_{+}}\Phi(t) \,\dd t\, = \phi(0),
$$
which here is assumed to exist, while for $k \geq 1$
$$
   \int_{\mathbb{R}_{+}} \frac{1}{t^{k}}\,\Phi(t)\,\dd t =  \int_{\mathbb{R}_{+}} \Phi(t) \int_{\mathbb{R}_{+}} \frac{\ell^{k - 1}}{(k - 1)!}\,e^{-t\ell}\,\dd \ell \,\dd t\, = c_{k}[\phi].
$$
The assumptions on $\phi$ guarantee that $c_{k}[\phi]$ are finite for all $k \geq 0$. Hence \eqref{VNplus} is finite for a fixed $\beta > 0$.

\medskip

We now study the $\beta \rightarrow \infty$ limit. For $\beta \geq 1$, we can bound the aforementioned polynomial by a $\beta$-independent polynomial in $t^{-1}$ and integrating the latter against $\Phi(t)\dd t$ gives a finite result. Therefore, by dominated convergence, we have
$$
   \lim_{\beta \rightarrow \infty} \beta^{-(6g - 6 + 2n)}\,VN^{+}_{g,n}(\rho_{\beta}^*\phi;\beta L_1,\ldots,\beta L_n) = \int_{\mathbb{R}_{+}} \Phi(t)\,\Big(\lim_{\beta \rightarrow \infty} \beta^{-(6g - 6 + 2n)}\,VN^{t/\beta}_{g,n}(\beta L_1,\ldots,\beta L_n)\Big)\,\dd t .
$$
Comparing Theorems~\ref{mirzath2} and \ref{konth} yields
$$
   \lim_{\beta \rightarrow \infty} \beta^{-2(3h - 3 + k)}\,V\Omega^{{\rm M}}_{h,k}(\beta \ell_1/t,\ldots,\beta \ell_k/t) =  t^{-2(3h - 3 + k)}\,V\Omega^{{\rm K}}_{h,k}(\ell_1,\ldots,\ell_k),
$$
and the limit is uniform for $(\ell_{1},\ldots,\ell_k,t^{-1})$ in any compact of $\mathbb{R}_{\geq 0}^{k + 1}$. Thus, uniformly for $(L_1,\ldots,L_n,t^{-1})$ in any compact of $\mathbb{R}_{\geq 0}^{n + 1}$, we have that
\begin{equation}\label{Ncollect}
\begin{split}
	\lim_{\beta \rightarrow \infty} & \beta^{-(6g - 6 + 2n)} VN^{t/\beta}_{g,n}(\beta L_1,\ldots,\beta L_n) \\
	& = \frac{1}{t^{6g - 6 + 2n}} \sum_{\Gamma \in \mathbf{G}_{g,n}} \int_{\mathbb{R}_{+}^{E_{\Gamma}}} \prod_{v \in V_{\Gamma}} V\Omega^{{\rm K}}_{h(v),k(v)}\big((\ell_e)_{e \in E(v)},(t L_{\lambda})_{\lambda \in \Lambda(v)}\big)\prod_{e \in E_{\Gamma}} \frac{\ell_e\,\dd \ell_e}{e^{\ell_e} - 1} ,
\end{split}
\end{equation}
where we recognise the Masur--Veech polynomials introduced in Section~\ref{MVpo}. We arrive at
\begin{equation}\label{2828}
\begin{split}
	\lim_{\beta \rightarrow \infty} \beta^{-(6g - 6 + 2n)}\,VN_{g,n}^{+}(\rho^*_{\beta};\beta L_1,\ldots,\beta L_n) & = \int_{\mathbb{R}_{+}} \Phi(t)\,t^{-(6g - 6 + 2n)}\,V\Omega^{{\rm MV}}(tL_1,\ldots,tL_n)\,\dd t \\
	& = \hat{c}[\phi]\big(t^{-(6g - 6 + 2n)}\,V\Omega^{{\rm MV}}(tL_1,\ldots,t L_n)\big).
\end{split}
\end{equation}

\medskip

Using finite boundary lengths $L_i$ instead of rescaling them by $\beta$ amounts to replacing $L_i$ by $L_i/\beta$ in \eqref{VNtb}, and by the aforementioned uniformity we then have 
\begin{equation}\label{2929}
\begin{split}
	\lim_{\beta \rightarrow \infty} \beta^{-(6g - 6 + 2n)}\,VN_{g,n}^{+}(\rho^*_{\beta}\phi;L_1,\ldots,L_n) & = \hat{c}[\phi]\big( t^{-(6g - 6 + 2n)}\,V\Omega^{{\rm MV}}_{g,n}(0,\ldots,0)\big) \\
	& = c_{6g - 6 + 2n}[\phi]\,V\Omega^{{\rm MV}}_{g,n}(0,\ldots,0).
\end{split}
\end{equation}
This concludes the proof of the theorem.
\end{proof}

\begin{proof}[Proof of Theorem \ref{thm:intro:1}]
The expression of the Masur--Veech polynomials in terms of stable graphs is actually our Definition~\ref{defn:MVpoly}. Note that this is not a circular argument: in the beginning of the paper we stated that Masur--Veech polynomials have four different equivalent formulations, we then chose the formulation in terms of stable graphs to be their definition, and we show in the rest of the paper that the same polynomials are expressed in the remaining three formulations. Therefore, the only non-trivial statement left to prove is the second part of the theorem, \textit{i.e.} formula~\eqref{MVgraphsss}, which follows immediately from Corollary \ref{MVcoco}.
\end{proof}

%***************************************
\subsection{Expression via geometric and topological recursion}
%***************************************

By comparison with Theorem~\ref{stablth}, the structure of formula  (\ref{defn:MVpoly}) implies that the Masur--Veech polynomials satisfy the topological recursion.

\begin{prop}[Geometric recursion for Masur--Veech volumes]\label{MVtr1} Let $\Omega^{{\rm MV}}$ be the GR amplitudes produced by the initial data
\begin{equation}
\label{eq314} \begin{aligned}
	A^{{\rm MV}}(L_1,L_2,L_3) & = 1, \\
	B^{{\rm MV}}(L_1,L_2,\ell) & = \frac{1}{(e^{\ell} - 1)} + \frac{1}{2L_1}\big([L_1 - L_2 - \ell]_{+} - [-L_1 + L_2 - \ell]_{+} + [L_1 + L_2 - \ell]_{+}\big), \\
	C^{{\rm MV}}(L_1,\ell,\ell') & = \frac{1}{(e^{\ell} - 1)(e^{\ell'} - 1)} + \frac{1}{L_1}\,[L_1 - \ell - \ell']_{+} \\
	&\quad + \frac{1}{2L_1}\bigg(\frac{1}{e^{\ell} - 1}\,\Big([L_1 - \ell - \ell']_{+} - [-L_1 + \ell - \ell']_{+} + [L_1 + \ell - \ell']_{+}\Big) \\
	&\quad \hphantom{+\frac{1}{2L_1}\bigg(} + \frac{1}{e^{\ell'} - 1}\,\Big([L_1 - \ell - \ell']_{+} - [-L_1 - \ell + \ell']_{+} + [L_1 - \ell + \ell']_{+}\Big) \bigg), \\
	D^{{\rm MV}}(\sigma) & = D^{{\rm K}}(\sigma) + \sum_{\gamma \in S_{T}^{\circ}} \frac{1}{e^{\ell_{\sigma}(\gamma)} - 1}.
\end{aligned}
\end{equation}
Then, for any $g \geq 0$ and $n \geq 1$ such that $2g - 2 + n > 0$, the Masur--Veech polynomials satisfy
\[
	V\Omega^{{\rm MV}}_{g,n}(L_1,\ldots,L_n) = \int_{\mathcal{M}_{g,n}(L_1,\ldots,L_n)} \Omega^{{\rm MV}}_{g,n}\,\dd\mu_{{\rm WP}}.
\]
In particular, they are computed by the topological recursion~\eqref{TReqn}.
\hfill $\blacksquare$
\end{prop}

Notice that the notation $V\Omega^{{\rm MV}}$ is consistent with its use in \eqref{VYnot}. The above initial data is obtained by twisting the Kontsevich initial data \eqref{WKini} by the function $f^{{\rm MV}}(\ell) = \frac{1}{e^{\ell} - 1}$ -- it is admissible according to Definition~\ref{defadm} with $\eta = 1$. The function $\Omega^{{\rm MV}}_{g,n}$ is a non-trivial function on $\mathcal{T}_{\Sigma}$, which is not equal to the function $X_{g,n}$ from Lemma~\ref{Xfn:prop}. For instance, we saw in Corollary~\ref{MVcoco} that $VX_{g,n}(L_1,\ldots,L_n)$ does not depend on $L_1,\ldots,L_n$, while $V\Omega_{g,n}^{{\rm MV}}(L_1,\ldots,L_n)$ are non-trivial polynomials whose constant term is $VX_{g,n}$. 

\medskip

Recall the decomposition
$$
V\Omega^{{\rm MV}}_{g,n}(L_1,\ldots,L_n) = \sum_{d_1 + \cdots + d_n \leq 3g - 3 + n} F_{g,n}^{{\rm MV}}[d_1,\ldots,d_n]\, \prod_{i = 1}^n \frac{L_i^{2d_i}}{(2d_i + 1)!}.
$$
By Section~\ref{Sequivalent:forms} we can give two equivalent forms of Proposition~\ref{MVtr1}, in terms of the $F_{g,n}$'s. The first one is the recursion of Theorem~\ref{thm:intro:2}, of which we give a proof in the following. This recursion is spelled out explicitly in Section~\ref{Virrec}.

\begin{proof}[Proof of Theorem~\ref{thm:intro:2}]
From the topological recursion in Proposition~\ref{MVtr1}, it follows that the $F_{g,n}$'s are computed by the recursion \eqref{recfgb}, by twisting the Kontsevich initial data \eqref{K1}-\eqref{K2} with $f^{{\rm MV}}(\ell) = \frac{1}{e^{\ell} - 1}$, that is, by
$$  
u_{d_1,d_2} = \int_{\mathbb{R}_{+}}\frac{\ell^{2d_1 + 2d_2 + 1}}{(2d_1 + 1)!(2d_2 + 1)!}\,\frac{ \dd \ell}{e^{\ell} - 1} = \frac{(2d_1 + 2d_2 + 1)!}{(2d_1 + 1)!(2d_2 + 1)!} \zeta(2d_1 + 2d_2 + 2)
$$  
according to Lemma~\ref{Lzeta}.
\end{proof}

The second equivalent form is the topological recursion à la Eynard--Orantin. 
Let us introduce the even part of the Hurwitz zeta function, for $k \geq 1$ 
$$
\zeta_{{\rm H}}(2k;z) = \frac{1}{z^{2k}} + \frac{1}{2} \sum_{m \in \mathbb{Z}^*} \frac{1}{(z + m)^{2k}},
$$
and define the multidifferentials
$$
\omega_{g,n}^{{\rm MV}}(z_1,\ldots,z_n) = \sum_{d_1 + \cdots + d_n \leq 3g - 3 + n} F_{g,n}[d_1,\ldots,d_n]\,\bigotimes_{i = 1}^n \zeta_{{\rm H}}(2d_i + 2;z_i)\,\dd z_i.
$$

\medskip

\begin{prop} 
\label{EOMV} For $g,n \geq 0$ such that $2g - 2 + n > 0$, the $\omega_{g,n}^{{\rm MV}}(z_1,\ldots,z_n)$ are computed by Eynard--Orantin topological recursion \eqref{TREOeqn} for the spectral curve
\[
\mathcal{C} = \mathbb{C},
\qquad
x(z) = \frac{z^2}{2},
\qquad
y(z) = -z,
\qquad
\omega_{0,2}^{{\rm MV}}(z_1,z_2) = \bigg(\frac{1}{(z_1 - z_2)^2} + \frac{\pi^2}{\sin^2\pi(z_1 - z_2)}\bigg)\frac{\dd z_1 \otimes \dd z_2}{2}.
\]
\end{prop}

\begin{proof}
Recall the spectral curve \eqref{ykm} associated with the Kontsevich initial data. The effect of twisting amounts to shifting $\omega_{0,2}^{{\rm K}}(z_1,z_2) = \frac{\dd z_1 \otimes \dd z_2}{(z_1 - z_2)^2}$ according to \eqref{02twist}. We compute, for ${\rm Re}\,z > 0$
\begin{equation}
\label{into} \int_{\mathbb{R}_{+}} \frac{1}{e^{\ell} - 1}\,e^{-\ell z}\,\ell\,\dd \ell = \sum_{m \geq 1} \int_{\mathbb{R}_{+}} e^{-\ell(z + m)}\,\ell\,\dd \ell = \sum_{m \geq 1} \frac{1}{(z + m)^2}.
\end{equation}
As the choice of signs in \eqref{02twist} is arbitrary, we can also take
\begin{align*}
\omega_{0,2}^{{\rm MV}} (z_1,z_2) & = \bigg(\frac{1}{(z_1 - z_2)^2} + \frac{1}{2}\sum_{m \geq 1} \frac{1}{(z_1 - z_2 + m)^2} + \frac{1}{(z_1 - z_2 - m)^2}\bigg) \dd z_1 \otimes \dd z_2 \\
& = \bigg(\frac{1}{(z_1 - z_2)^2} + \frac{\pi^2}{\sin^2\pi(z_1 - z_2)}\bigg)\frac{\dd z_1 \otimes \dd z_2}{2}.
\end{align*}
The sector of convergence for the integral \eqref{into} is irrelevant, as we only need the (well-defined) Taylor expansion when $z_i \rightarrow 0$ to compute the $\omega_{g,n}$. Finally, we compute the differential forms $\xi_{d}$ defined in \eqref{xidform} and which are used in Theorem~\ref{thmABCD} to decompose the $\omega_{g,n}^{{\rm MV}}$:
\begin{align*}
	\frac{\xi_{d}(z_0)}{\dd z_0} & = \Res_{z = 0} \frac{\dd z}{z^{2d + 2}}\bigg(\frac{1}{z_0 - z} + \frac{1}{2} \sum_{m \geq 1} \frac{1}{z_0 - z - m} + \frac{1}{z_0 - z + m}\bigg) \\
	& = \frac{1}{z_0^{2d + 2}} + \frac{1}{2} \sum_{m \geq 1}  \frac{1}{(z_0 + m)^{2d + 2}} + \frac{1}{(z_0 - m)^{2d + 2}} \\
	& = \zeta_{{\rm H}}(2d + 2;z_0).
\end{align*}

For $n = 0$ and $g \geq 2$, Lemma~\ref{dilatonthm0} gives
$$
V\Omega_{g,0}^{{\rm MV}} = \frac{1}{2g - 2}\,\frac{F_{g,1}[1]}{3}.
$$
This agrees with the definition \eqref{omg0def} of $\omega_{g,0}^{{\rm MV}}$ by the following computation
\begin{align*}
	\omega_{g,0}^{{\rm MV}} & = \frac{1}{2 - 2g} \Res_{z = 0} \bigg(\int_{0}^{z} y \dd x\bigg) \omega_{g,1}(z) = \frac{1}{2g - 2} \Res_{z = 0} \frac{z^3}{3}\,\omega_{g,1}^{{\rm MV}}(z) \\
	& = \frac{1}{2g - 2}\,\sum_{d \geq 0} \bigg(\Res_{z = 0} \frac{z^3}{3}\,\zeta_{{\rm H}}(2d + 2;z)\,\dd z\bigg)\,F_{g,1}[d] \\
	& = \frac{1}{2g - 2}\,\frac{F_{g,1}[1]}{3},
\end{align*}
where we used that $\zeta_{{\rm H}}(2d + 2;z) = z^{-(2d + 2)}\dd z + O(1)$ when $z \rightarrow 0$, which implies that only the $d = 1$ term contributes to the residue.
\end{proof}

%***************************************
\subsection{Equivalent expression in intersection theory}
%***************************************

We can express Masur--Veech polynomials as a single integral over moduli space of curves of a certain class, which involves boundary divisors. This is just another way of expressing the sum over stable graphs (\textit{i.e.} boundary strata of $\overline{\mathfrak{M}}_{g,n}$).

\medskip

We first introduce some notations. Consider the set $\mathbf{G}_{g,n}$ of stable graphs of type $(g,n)$. For every $\Gamma \in \mathbf{G}_{g,n}$, we have the moduli space $\overline{\mathfrak{M}}_{\Gamma}$ and the maps $\xi_{\Gamma}$ and $p_{v}$:
\[
  \overline{\mathfrak{M}}_{\Gamma} = \prod_{v \in V_{\Gamma}} \overline{\mathfrak{M}}_{h(v),k(v)},
  \qquad
  \xi_{\Gamma} \colon \overline{\mathfrak{M}}_{\Gamma} \to \overline{\mathfrak{M}}_{g,n},
  \qquad
  p_v \colon \overline{\mathfrak{M}}_{\Gamma} \to \overline{\mathfrak{M}}_{h(v),k(v)}.
\]
The image of $\xi_{\Gamma}$ is the boundary stratum associated to the graph $\Gamma$, while $p_v$ is the projection on the moduli space attached to the vertex $v$. We also define the map
\[
  \jmath = \sum_{\Gamma \in \mathbf{G}_{g,n}^{1}} \frac{\xi_{\Gamma\ast}}{|{\rm Aut}\,{\Gamma}|},
\]
where $\mathbf{G}_{g,n}^{k}$ is the set of stable graphs of type $(g,n)$ with $k$ edges. In other words, $\jmath$ is a sum over boundary divisors of $\overline{\mathfrak{M}}_{g,n}$. Further, denote by $\psi_{\bullet}$ and $\psi_{\circ}$ the cotangent classes at the nodes, so that it makes sense to consider the push-forward by $\jmath$ of any monomial in $\psi_{\bullet}$ and $\psi_{\circ}$.

\medskip

Recall that the even zeta values are related to Bernoulli numbers by 
\[
  \zeta(2m + 2) = (-1)^m \frac{B_{2m+2} (2\pi)^{2m+2}}{2 (2m+ 2)!},
\]
with $B_2 = \tfrac{1}{6}$, $B_4 = -\frac{1}{30}$, etc.

\begin{prop}\label{interMV}
  For $2g - 2 + n > 0$, the Masur--Veech polynomials $V\Omega^{{\rm MV}}_{g,n}$ satisfy
  \begin{equation}\label{eqn:interMV1}
    V\Omega^{{\rm MV}}_{g,n}(L_1,\ldots,L_n) =
    \int_{\overline{\mathfrak{M}}_{g,n}}\sum_{\Gamma \in \mathbf{G}_{g,n}}
        \frac{1}{|{\rm Aut}\,{\Gamma}|} \, \xi_{\Gamma\ast}
        \prod_{\substack{e \in E_{\Gamma} \\ e = (h,h')}}
          \widetilde{\mathcal{E}}(-\psi_{h} -\psi_{h'})
        \exp\bigg( \sum_{\lambda \in \Lambda_{\Gamma}} \frac{L_i^2}{2} \, \psi_{\lambda} \bigg)
  \end{equation}
  for
  \[
    \widetilde{\mathcal{E}}(u) = \sum_{D \ge 0} (2\pi^{2})^{D + 1}\,\frac{B_{2D + 2}}{2D + 2} u^{D},
  \]
  or equivalently
  \begin{equation}\label{eqn:interMV2}
    V\Omega^{{\rm MV}}_{g,n}(L_1,\ldots,L_n) = \int_{\overline{\mathfrak{M}}_{g,n}}
     \exp\bigg( \Xi_{g,n} + \sum_{i = 1}^n \frac{L_i^2}{2}\,\psi_i\bigg)
  \end{equation}
  for
  \[
    \Xi_{g,n} = \jmath_{\ast} \mathcal{E}(-\psi_{\bullet} - \psi_{\circ}) \in H^{\bullet}(\overline{\mathfrak{M}}_{g,n}),
    \qquad\qquad
   \mathcal{E}(u) = u^{-1}\ln\bigg(1 + \sum_{k \ge 1} (2\pi^2)^k\frac{B_{2k}}{2k}\,u^k\bigg).
  \]
\end{prop}

Once the spectral curve for a certain enumerative geometric problem satisfying topological recursion is known (here Proposition~\ref{EOMV}), one could apply Eynard's formula \cite[Theorem 3.1]{Einter} to obtain such a representation for $\omega^{{\rm MV}}_{g,n}(z_1,\ldots,z_n)$, and thus the Masur--Veech polynomials. To be self-contained, we prove the result by direct computation.

\medskip

It would be interesting to obtain this formula by algebro-geometric methods. A first hint in this direction would be to express $\Xi_{g,n}$ in a more intrinsic way, as a characteristic class of a bundle over $\overline{\mathfrak{M}}_{g,n}$, maybe obtained by push-forward from the moduli space of quadratic differentials.
 
\begin{proof}
We shall examine the contribution in Equation~\eqref{defn:MVpoly} of a given $\Gamma \in \mathbf{G}_{g,n}$ before integration over the product of moduli spaces at the vertices. Given a decoration $d\colon H_{\Gamma} \rightarrow \mathbb{N}$, an edge $e = (h,h')$ receives a weight  $(2d_{h} + 2d_{h'} + 1)!\zeta(2d_{h} + 2d_{h'} + 2)$. We remark that it only depends on the total degree $D_{e} = d_{h} + d_{h'}$ associated to this edge. On the other hand, the contribution of the $\psi$ classes at the ends of the edge is
$$
\frac{\psi^{d_{h}} (\psi')^{d_{h'}}}{2^{D_e}\,d_{h}!\,d_{h'}!}.
$$ 
Therefore, we can replace the sum over decorations of half-edges $d\colon H_{\Gamma} \rightarrow \mathbb{N}$ by the sum over decorations of edges $D\colon E_{\Gamma} \rightarrow \mathbb{N}$, and attach to each edge a contribution of 
$$
\frac{(2D_e + 1)!\,\zeta(2D_e + 2)}{2^{D_e}\,D_e!}\,\big(\psi_{h} + \psi_{h'}\big)^{D_e} =(2\pi^{2})^{D_{e} + 1}\,\frac{B_{2D_e + 2}}{2D_e + 2} \big(- \psi_{h} - \psi_{h'}\big)^{D_{e}}.
$$ 
In other words,
\[
  V\Omega^{{\rm MV}}_{g,n}(L_1,\ldots,L_n) =
  \int_{\overline{\mathfrak{M}}_{g,n}}\sum_{\Gamma \in \mathbf{G}_{g,n}}
      \frac{1}{|{\rm Aut}\,{\Gamma}|} \, \xi_{\Gamma\ast}
      \prod_{\substack{e \in E_{\Gamma} \\ e = (h,h')}}
        \widetilde{\mathcal{E}}(-\psi_{h}- \psi_{h'})
      \exp\bigg( \sum_{\lambda \in \Lambda_{\Gamma}} \frac{L_i^2}{2} \, \psi_{\lambda} \bigg)
\]
for
\[
  \widetilde{\mathcal{E}}(u) = \sum_{D \ge 0} (2\pi^{2})^{D + 1}\,\frac{B_{2D + 2}}{2D + 2} u^{D}.
\]
This proves Equation~\eqref{eqn:interMV1}. The equivalence between Equation~\eqref{eqn:interMV1} and Equation~\eqref{eqn:interMV2} is shown via the following lemma.
\end{proof}

\begin{lem}\label{exp:vs:sum:stable}
  Consider two formal power series $\mathcal{E} \in \CC\bbraket{x,y}^{\mathfrak{S}_2}$ and $T \in \CC\bbraket{u}$, $T(u) = \sum_{m \ge 0} t_m u^m$, and define the cohomology class
  \begin{equation}\label{eqn:class}
    \Theta_{g,n} =
      \exp\bigl( \jmath_{\ast} \mathcal{E}(\psi_{\bullet},\psi_{\circ}) \bigr)
      \,
      \exp\bigl( T(\kappa) \bigr)
      \,
      \prod_{i=1}^n \psi_{i}^{d_i}
  \end{equation}
  on $\overline{\mathfrak{M}}_{g,n}$, where $T(\kappa) = \sum_{m \ge 0} t_m \kappa_m$. Then
  \begin{equation}\label{eqn:sum:stable:graphs}
    \Theta_{g,n} =
      \sum_{\Gamma \in \mathbf{G}_{g,n}}
      \frac{1}{|{\rm Aut}\,{\Gamma}|}\,\xi_{\Gamma\ast}
      \prod_{\substack{e \in E_{\Gamma} \\ e = (h,h')}}
        \widetilde{\mathcal{E}}(\psi_{h},\psi_{h'})
      \prod_{v \in V_{\Gamma}}
        \exp\bigl( T(p_v^{\ast}\kappa) \bigr)
      \prod_{\lambda \in \Lambda_{\Gamma}}
        \psi_{\lambda}^{d_{\lambda}}
  \end{equation}
  for $\widetilde{\mathcal{E}} \in \CC\bbraket{x,y}^{\mathfrak{S}_2}$ defined by
  \begin{equation}
 \label{Eeqn2}   \widetilde{\mathcal{E}}(x,y) = \frac{1 - e^{-(x + y)\mathcal{E}(x,y)}}{x + y}.
  \end{equation}
  Conversely, consider a class $\Theta_{g,n}$ given by Equation~\eqref{eqn:sum:stable:graphs} for certain formal power series $\widetilde{\mathcal{E}} \in \CC\bbraket{x,y}^{\mathfrak{S}_2}$ and $T \in \CC\bbraket{u}$. Then $\Theta_{g,n}$ can be expressed by Equation~\eqref{eqn:class}, for $\mathcal{E}$ defined as
  \begin{equation}
    \label{Eeqn3} \mathcal{E}(x,y) = - \frac{1}{x + y} \ln{\bigl(1 - (x + y)\widetilde{\mathcal{E}}(x,y) \bigr)}.
  \end{equation}
\end{lem}
\begin{proof}
  Firstly, notice that \eqref{Eeqn2} and \eqref{Eeqn3} make sense because $(x+y)$ formally divide $1 - e^{-(x + y)\mathcal{E}(x,y)}$ and $\ln{\bigl(1 - (x + y)\widetilde{\mathcal{E}}(x,y) \bigr)}$, respectively. Further, the contribution at the vertices follow from the projection formula and the relation
  \[
    \xi_{\Gamma}^{\ast}\kappa_m = \sum_{v \in V_{\Gamma}} p_{v}^{\ast} \kappa_m,
  \]
  while the legs contribution follows from the correspondence between legs of $\Gamma$ and markings. Computing the edge contribution amounts to understand how to intersect push-forwards of classes via boundary maps. Let $(\Gamma,\Delta)$ be a stable graph in $\mathbf{G}_{g,n}$ together with a decoration of each edge $e$ of the form
  \[
    \Delta(e;\psi_{h},\psi_{h'}) \in \CC\bbraket{\psi_{h},\psi_{h'}}^{\mathfrak{S}_{2}},
    \qquad\qquad
    e = (h,h').
  \]
  The associated class on $\overline{\mathfrak{M}}_{g,n}$ is
  \[
    \xi_{\Gamma\ast} \prod_{\substack{e \in E_{\Gamma} \\ e = (h,h')}}
      \Delta(e;\psi_{h},\psi_{h'}).
  \]
  In general, for two decorated stable graphs $(\Gamma_A, \Delta_A)$ and $(\Gamma_B, \Delta_B)$ in $\mathbf{G}_{g,n}$, the intersection of the corresponding classes is determined as follows: enumerate all decorated stable graphs $(\Gamma,\Delta)$ whose edges are marked by $A$, $B$ or both, in such a way that contracting all edges outside $A$ yields $(\Gamma_A, \Delta_A)$ and contracting all edges outside $B$ yields $(\Gamma_B, \Delta_B)$. Notice that each edge $e = (h,h')$ that is marked by both $A$ and $B$ corresponds to a boundary divisor in the Poincar\'e dual of both $\Gamma_A$ and $\Gamma_B$. To such an edge is therefore assigned a factor that is corresponding to its self-intersection, namely, the first Chern class of the normal bundle $N_e$ (of the gluing morphism) associated to the edge $e$:
  \[
    c_1(N_e) = - \psi_{h} - \psi_{h'}.
  \]
 Summing up the push-forwards over these decorated graphs of the product over edges of the associated decorations, represents the intersection of the classes associated to $(\Gamma_A, \Delta_A)$ and $(\Gamma_B, \Delta_B)$.

  \smallskip

 Let us apply this general argument to our case, that is, to the class $\exp(\jmath_{\ast} \mathcal{E}(\psi_{\bullet},\psi_{\circ}))$. Notice that $\jmath_{\ast}\mathcal{E}(\psi_{\bullet},\psi_{\circ})$ is a sum over stable graphs with a single edge, decorated by a factor of $\mathcal{E}$. In the $k$-th term of the exponential expansion, we have to consider the sum over stable graphs whose global decoration involves exactly $k$ factors of $\mathcal{E}$, distributed in all possible ways on $k$ edges counted with multiplicity $(m_e)_e$, taking into account the self-intersection of the edges with multiplicity $m_e > 1$. This results for each edge $e = (h,h')$ into the factor
\begin{equation}
\label{Eeqn1}  \sum_{m_e \ge 1} \frac{1}{m_e!} (- \psi_{h} - \psi_{h'})^{m_e - 1} \, \mathcal{E}(\psi_{h},\psi_{h'})^{m_e} = \frac{1 - \exp( - (\psi_{h} + \psi_{h'})\mathcal{E}(\psi_{h},\psi_{h'}))}{\psi_{h} + \psi_{h'}} = \widetilde{\mathcal{E}}(\psi_{h},\psi_{h'}).
\end{equation}
 Therefore we obtain
  \[
  \begin{split}
    \exp\bigl( \jmath_{\ast} \mathcal{E}(\psi_{\bullet},\psi_{\circ}) \bigr)
    & =
    \sum_{\Gamma \in \mathbf{G}_{g,n}}
      \frac{\xi_{\Gamma\ast}}{|{\rm Aut}\,{\Gamma}|}
      \prod_{\substack{e \in E_{\Gamma} \\ e = (h,h')}}
        \widetilde{\mathcal{E}}(\psi_{h},\psi_{h'}).
  \end{split}
  \]
Notice that the relation \eqref{Eeqn1} can be inverted as \eqref{Eeqn3}. This concludes the proof of the lemma.
\end{proof}

%***************************************
\section{Statistics of hyperbolic lengths for Siegel--Veech constants}
%***************************************
\label{S4}

\subsection{Preliminaries}

The area Siegel--Veech constant $SV_{g,n}$ of $Q\mathfrak{M}_{g,n}$ is a positive real number related to the asymptotic number of
flat cylinders of a generic quadratic differential. Given a quadratic differential $q \in Q\mathfrak{M}_{g,n}$, we define
\[
  \mathscr{N}_{{\rm area}}(q, L) = \frac{1}{\operatorname{Area}(q)} \sum_{\substack{\mathbf{c} \subset q \\ w(\mathbf{c}) \leq L}} \operatorname{Area}(\mathbf{c}),
\]
where the sum is over flat cylinders $\mathbf{c}$ of $q$ whose width $w(\mathbf{c})$ (or circumference) is less or equal to $L$ and $\operatorname{Area}$ refers to the total mass of the measure induced by the flat metric of $q$. By a theorem of Veech \cite{Veecha} and Vorobets \cite{Vorobets}, the number
$$
  SV_{g,n} = \frac{1}{MV_{g,n}} \frac{1}{\pi L^2} \int_{Q^1\mathfrak{M}_{g,n}} \mathscr{N}_{{\rm area}}(q,L)\,\dd \mu_{\rm MV}^1(q) 
$$
exists and is independent of $L > 0$. It is called the (area) Siegel--Veech constant of $Q\mathfrak{M}_{g,n}$.

%***************************************
\subsection{Goujard's formula}
\label{SGouj}
%***************************************

Goujard showed in \cite[Section 4.2, Corollary 1]{Goujard} how to compute $SV_{g,n}$ in terms of the Masur--Veech volumes. Her result is in fact more general, as it deals with all strata of the moduli space of quadratic differentials, while the present article is only concerned with the principal stratum. 

\begin{thm}\cite{Goujard}
\label{thGouj1} For $g,n \geq 0$ such that $2g + n -2\geq 2$, we have
\begin{equation}\label{eqGouj1}
\begin{aligned}
	& SV_{g,n}\cdot MV_{g,n} \\
	&\quad = \tfrac{1}{4}\,\tfrac{(4g - 4 + n)(4g - 5 + n)}{(6g - 7 + 2n)(6g - 8 + 2n)}\,MV_{g - 1,n + 2} \\
	&\qquad + \tfrac{1}{8} \sum_{\substack{g_1 + g_2 = g \\ n_1 + n_2 = n}} \tfrac{n!}{n_1!n_2!}\,\tfrac{(4g - 4  + n)!}{(4g_1 - 3 + n_1)!(4g_2 - 3 + n_2)!}\,\tfrac{(6g_1 - 5 + 2n_1)!(6g_2 - 5 + 2n_2)!}{(6g - 7 + 2n)!}\,MV_{g_1,1 + n_1}MV_{g_2,1 + n_2}.
\end{aligned}
\end{equation}
\hfill $\blacksquare$
\end{thm}

In \cite{Goujard} the contribution of $MV_{0,3}\cdot MV_{g,n - 1}$ was written separately, but this term can be included in the sum if we remark that  $MV_{0,3} = 4$ (see Section~\ref{Genus01row}) and
$$
   \frac{(2n - 5)!}{(n - 3)!}\Big|_{n = 2} = \lim_{n \rightarrow 2} \frac{\Gamma(2n - 4)}{\Gamma(n - 2)} = \frac{1}{2}.
$$

The structure of this formula becomes more transparent if we rewrite it in terms of the rescaled Masur--Veech volumes that are sums over stable graphs
$$
   MV_{g,n} = \frac{2^{4g - 2 + n}(4g - 4 + n)!}{(6g - 7 + 2n)!}\,V\Omega_{g,n}^{{\rm MV}}(0,\ldots,0).
$$

\begin{cor}\label{corSv}
   For $g,n \geq 0$ such that $2g + n -2\geq 2$, we have
   \begin{equation}\label{SVVV} 
      SV_{g,n}\cdot V\Omega_{g,n}^{{\rm MV}}(0) = \frac{1}{4}\bigg( V\Omega_{g-1,n+2}^{{\rm MV}}(0) + \frac{1}{2} \sum_{\substack{g_1 + g_2 = g \\ n_1 + n_2 = n}} \frac{n!}{n_1!n_2!}\,V\Omega_{g_1,1 + n_1}^{{\rm MV}}(0)\,V\Omega_{g_2,1 + n_2}^{{\rm MV}}(0)\bigg).
   \end{equation}
   \hfill $\blacksquare$
\end{cor} 

We can give an even more compact form to this relation, in terms of generating series. If we introduce
\begin{equation}\label{ZMVpart}
   \mathscr{Z}_{\hslash}(x) = \exp\bigg(\sum_{g \geq 0} \hslash^{g - 1} \sum_{\substack{n \geq 1 \\ 2g - 2 + n > 0}} \frac{x^n}{n!}\,\frac{V\Omega^{{\rm MV}}_{g,n}(0)}{\pi^{6g - 6 + 2n}}\bigg),
\end{equation}
then Corollary \eqref{corSv} is equivalent to

\begin{cor}\label{cor2Sv}
   We have that
   \begin{equation}\label{sqZ}
      \sum_{g \geq 0} \hslash^{g} \sum_{\substack{n \geq 0 \\ 2g - 2 + n \geq 2}} \frac{x^n}{n!}\,\frac{SV_{g,n}\cdot V\Omega^{{\rm MV}}_{g,n}(0)}{\pi^{6g - 4 + 2n}} = \frac{1}{2}\,\frac{\hslash^2\partial_{x}^2 \sqrt{\mathscr{Z}_{\hslash}(x)}}{\sqrt{\mathscr{Z}_{\hslash}(x)}}.
   \end{equation}
\end{cor}

\begin{proof} Let us write $\mathscr{Z}_{\hslash}(x) = \exp\big(\sum_{g \geq 0} \hslash^{g - 1}\,\mathscr{F}_{g}(x)\big)$. For $\alpha \in \mathbb{C}$, we compute 
\begin{equation}\label{SVVVV2}
\begin{aligned}
	& \frac{\hslash^2 \partial_{x}^2 \mathscr{Z}_{\hslash}^{\alpha}(x)}{\mathscr{Z}_{\hslash}^{\alpha}(x)} \\
	& = \sum_{g \geq 0} \hslash^{g}\bigg(\alpha \partial_{x}^2 \mathscr{F}_{g - 1}(x) + \alpha^2 \sum_{g_1 + g_2 = g} \partial_{x}\mathscr{F}_{g_1}(x)\cdot \partial_{x}\mathscr{F}_{g_2}(x)\bigg) \\
	& = \sum_{g \geq 0} \hslash^{g} \sum_{\substack{n \geq 0 \\ 2g + n > 0}} \frac{x^n}{n!}\bigg(\alpha\,V\Omega_{g-1,n+2}^{{\rm MV}}(0) + \alpha^2\,\sum_{\substack{g_1 + g_2 = g \\ n_1 + n_2 = n}} \frac{n!}{n_1!n_2!}\,V\Omega_{g_1,1 + n_1}^{{\rm MV}}(0)\,V\Omega_{g_2,1 + n_2}^{{\rm MV}}(0)\bigg),
\end{aligned}
\end{equation} 
where we noticed that the restriction $2g - 2 + n > 0$ in \eqref{ZMVpart} implies that there are no terms for $2g + n \leq 0$ in \eqref{SVVVV2}. The relative factor of $\tfrac{1}{2}$ between the two types of terms in \eqref{SVVV} is reproduced by choosing $\alpha = \tfrac{1}{2}$, and we need to multiply \eqref{SVVVV2} by an overall factor of a $\tfrac{1}{2}$ to reproduce the prefactor $\frac{1}{4}$ in \eqref{SVVV}. The factors of $\pi$ also match since
\begin{equation}\label{gcollect}
   6(g - 1) - 6 + 2(2 + n) = 6g_1 - 6 + 2(1 + n_1) + 6g_2 - 6 + 2(1 + n_2) = 6g - 4 + 2n.
\end{equation}
They have been included so that the coefficients of $\hslash^{g}x^{n}$ in the generating series are rational numbers.
\end{proof}

The contributions in \eqref{SVVV} correspond to the topology of surfaces obtained by cutting along a simple closed curve in a surface of genus $g$ with $n$ boundaries. It is important to note that the (somewhat unusual) feature that separating curves receive an extra factor of a $\frac{1}{2}$, which is reflected in the squareroot in the right-hand side of \eqref{sqZ}. Such sums (without this relative factor of a $\tfrac{1}{2}$)
can be obtained by differentiating a sum over stable graphs with respect to the edge weight. Therefore, they also arise by integrating over the moduli space derivatives of the statistics of hyperbolic lengths of multicurves with respect to the test function. We make this precise in the next paragraphs.

%***************************************
\subsection{Derivatives of hyperbolic length statistics}
%***************************************

We define two natural derivative statistics for which we are going to study the scaling limit. First, if $\gamma_0 \in S_{\Sigma}^{\circ}$, we denote
$$ 
\jmath(\gamma_0) = \begin{cases}
	1 & \text{if $\gamma_0$ is separating}, \\
	0 & \text{otherwise}\,.
	\end{cases} 
$$  
Let $\psi,\phi$ be admissible test functions, and consider
\begin{align}
	\label{Nfirst}
	N_{\Sigma}^{+}(\phi;\psi;\sigma) & = \sum_{c \in M_{\Sigma}} \sum_{\gamma_0 \in \pi_0(c)} 2^{-\jmath(\gamma_0)}\,\psi(\ell_{\sigma}(\gamma_0))\cdot \phi(\ell_{\sigma}(c)), \\
	\label{Nfirsttilde}
	\widetilde{N}_{\Sigma}^{+}(\phi;\psi;\sigma) & = \sum_{c \in M_{\Sigma}} \sum_{\gamma_0 \in \pi_0(c)} 2^{-\jmath(\gamma_0)}\,\psi\big(\ell_{\sigma}(\gamma_0)\big)\cdot \phi\big(\ell_{\sigma}(c - \gamma_0)\big).
\end{align}

\begin{thm}\label{SVcomc}
Assume that $\psi$ is bounded, $\ell \mapsto \ell^{-1}\psi(\ell)$ is integrable over $\mathbb{R}_{+}$ and recall that $c_{0}[\psi] = \int_{\mathbb{R}_{+}} \frac{\dd \ell}{\ell}\,\psi(\ell)$. Assume that $\phi$ has a Laplace representation
$$
\phi(\ell) = \int_{\mathbb{R}_{+}} \Phi(t)\,e^{-t\ell}\,\dd t
$$
for some measurable function $\Phi$ such that $t \mapsto |\Phi(t)|$ is integrable over $\mathbb{R}_{+}$. For $g,n \geq 0$ such that $2g + n -2\geq 2$ and fixed $L_1,\ldots,L_n \geq 0$, we have
\begin{align*}
	\lim_{\beta \rightarrow \infty} & \beta^{-(6g - 6 + 2n)}\,VN_{\Sigma}^{+}(\rho_{\beta}^*\phi;\psi;\beta L_1,\ldots,\beta L_n)  \\
	& = \frac{1}{2}\,c_0[\psi]\,\hat{c}[\phi]\bigg[t^{-(6g - 6 + 2n)} \\
	& \quad \cdot \bigg(V\Omega_{g - 1,2 + n}^{{\rm MV}}(0,0,L_1,\ldots,L_n) + \frac{1}{2}\,\sum_{\substack{g_1 + g_2 = g \\ J_1 \sqcup J_2 = \{L_1,\ldots,L_n\}}} V\Omega_{g_1,1 + |J_1|}^{{\rm MV}}(0,J_1)\,V\Omega_{g_2,1 + |J_2|}^{{\rm MV}}(0,J_2)\bigg)\bigg], 
\end{align*}
and
\begin{align*}
	\lim_{\beta \rightarrow \infty} & \beta^{-(6g - 6 + 2n)}\,VN_{\Sigma}^{+}(\rho_{\beta}^*\phi;\psi;L_1,\ldots,L_n)  \\
	& = \frac{1}{2}\,c_0[\psi]\,c_{6g - 6 + 2n}[\phi]\bigg(V\Omega_{g - 1,2 + n}^{{\rm MV}}(0) + \frac{1}{2}\,\sum_{\substack{g_1 + g_2 = g \\ n_1 + n_2 = n}} \frac{n!}{n_1!n_2!}\,V\Omega_{g_1,1 + n_1}^{{\rm MV}}(0)\,V\Omega_{g_2,1 + n_2}^{{\rm MV}}(0)\bigg).
\end{align*}
In particular, this last expression is independent of $L_1,\ldots,L_n$. Furthermore, replacing $N$ with $\widetilde{N}$ gives the same limits.
\end{thm}
By comparison with Goujard's formula (Corollary~\ref{corSv}), we deduce that
\begin{cor}
\label{Co45} Under the assumptions of Theorem~\ref{SVcomc}, for any fixed $L_1,\ldots,L_n \in \mathbb{R}_{+}^n$, we have
$$
2c_0[\psi]\,SV_{g,n}= \lim_{\beta \rightarrow \infty} \frac{VN^{+}_{g,n}(\rho_{\beta}^*\phi;\psi;L_1,\ldots,L_n)}{VN^{+}_{g,n}(\rho_{\beta}^*\phi;L_1,\ldots,L_n)}.
$$
The same equality holds if $N$ in the numerator is replaced with $\widetilde{N}$. \hfill $\blacksquare$
\end{cor}

The corollary gives a hyperbolic geometric interpretation of the area Siegel--Veech constant. However, our proof is done by comparison of values from Goujard's formula. It would be desirable to find a direct and geometric proof of this identity, which would give a new proof of Goujard's formula. The derivative statistics of hyperbolic lengths $N^+_{g,n}(\phi;\psi;L_1, \ldots, L_n)$ is indeed reminiscent of the Siegel--Veech transform. Via the Hubbard--Masur correspondence \cite{HubbardMasur}, the multicurve $c \in M_\Sigma$ is associated to a holomorphic quadratic differential $q$ and the component $\gamma_0$ is the core curve of a cylinder of $q$. The difficulty, however, lies in comparing hyperbolic and flat lengths.

\begin{proof}[Proof of Theorem~\ref{SVcomc}]
The assumption $2g + n -2\geq 2$ is made so that $M_{\Sigma}$ does not only consist of the empty multicurve. If we encode multicurves as a pair consisting of a primitive multicurve and integers $k$ remembering the multiplicity for each of its component, we have that
\begin{equation}
\label{oungifung}
  N_{\Sigma}^{+}(\rho_{\beta}^*\phi;\psi;\sigma) = \sum_{\substack{c \in M'_{\Sigma} \\ m\colon \pi_0(c) \rightarrow \mathbb{N}^*}} \sum_{\gamma_0 \in \pi_0(c)} 2^{-\jmath(\gamma_0)} m(\gamma_0)\,\psi(\ell_{\sigma}(\gamma_0))\,\phi\bigg(\beta^{-1} \sum_{\gamma \in \pi_0(c)} m(\gamma)\,\ell_{\sigma}(\gamma)\bigg),
\end{equation} 
since $\gamma_0$ in \eqref{Nfirst} can be any of the $m(\gamma_0)$ component of the multicurve.

\medskip

As in the proof of Theorem~\ref{stablg}, we rely on the Laplace representation for $\phi$
\begin{equation}
\label{Lphis}
  \phi(\ell) = \int_{\mathbb{R}_{+}} \Phi(t)\,e^{-t\ell}\,\dd t
\end{equation}
to convert additive statistics into multiplicative statistics. As their application is similar to the proof of Theorem~\ref{stablg}, we will silently use the Fubini--Tonelli and dominated convergence theorems at many places -- the estimates necessary for their application use that $\psi$ is bounded and $c_{0}[\psi]$ exists.

\medskip

The Laplace representation allows us to convert \eqref{oungifung} into derivatives (with respect to the test function) of multiplicative statistics, namely
\begin{equation}\label{1111A}
\begin{aligned}
	N_{\Sigma}^{+}(\rho_{\beta}^*\phi;\psi;\sigma) & = \int_{\mathbb{R}_{+}} \Phi(t)\,\bigg( \sum_{c \in M'_{\Sigma}} \sum_{\gamma_0 \in \pi_0(c)} \frac{\psi(\ell_{\sigma}(\gamma_0))}{2^{\jmath(\gamma_0)}(1 - e^{-t\ell_{\sigma}(\gamma_0)/\beta})}\, \prod_{\gamma \in \pi_0(c)} \frac{e^{-t\ell_{\sigma}(c)/\beta}}{1 - e^{-t\ell_{\sigma}(\gamma)/\beta}}\bigg)\,\dd t \\ 
	& = \int_{\mathbb{R}_{+}} \Phi(t)\,\partial_{z = 0}\big(N_{\Sigma}^{t/\beta,z}(\psi;\sigma)\big)\dd t,
\end{aligned}
\end{equation}
where
$$
N_{\Sigma}^{t,z}(\psi;\sigma) = \sum_{c \in M'_{\Sigma}} \prod_{\gamma \in \pi_0(c)} \frac{1}{e^{t\ell_{\sigma}(\gamma)} - 1}\bigg(1 +  \frac{z\,\psi(\ell_{\sigma}(\gamma))}{2^{\jmath(\gamma)}(1 - e^{-t\ell_{\sigma}(\gamma)})}\bigg) \nonumber
$$ 
is a polynomial of degree $(3g - 3 + n)$ in the variable $z$. Integrating over the moduli space, we obtain a sum over the topological types of primitive multicurves, that is, over stable graphs:
\begin{equation}\label{3939}
\begin{split}
	& VN_{g,n}^{t,z}(\psi;L_1,\ldots,L_n) \\
	& = \sum_{\Gamma \in \mathbf{G}_{g,n}} \frac{1}{|{\rm Aut}\,\Gamma|} \int_{\mathbb{R}_{+}^{E_{\Gamma}}} \!\! \prod_{v \in V_{\Gamma}} V\Omega^{{\rm M}}_{h(v),k(v)}\big((\ell_e)_{e \in E(v)},(L_{\lambda})_{\lambda \in \Lambda(v)}\big)\prod_{e \in E_{\Gamma}} \frac{1}{e^{t\ell_e} - 1}\bigg(1 + \frac{z\,\psi(\ell_e)}{2^{\jmath_{e}}(1 - e^{-t\ell_e})}\bigg)\,\ell_e\,\dd \ell_e,
\end{split}
\end{equation}
where $\jmath_{e} = 1$ if $e$ is separating and $\jmath_{e} = 0$ otherwise. The coefficient of $z$ in this sum reads
\begin{equation}\label{sumglu}
\begin{split}
	& \partial_{z = 0}\big(VN_{g,n}^{t,z}(\psi;L_1,\ldots,L_n)\big) \\
	& = \!\! \sum_{\substack{\Gamma \in \mathbf{G}_{g,n} \\ e_0 \in E_{\Gamma}}} \!\!\! \frac{1}{|{\rm Aut}\,\Gamma|} \int_{\mathbb{R}_{+}^{E_{\Gamma}}} \prod_{v \in V_{\Gamma}} V\Omega^{{\rm M}}_{h(v),k(v)}\big((\ell_e)_{e\in E(v)},(L_{\lambda})_{\lambda \in \Lambda(v)}\big)\frac{\psi(\ell_{e_0})\,e^{-t\ell_{e_0}}}{2^{\jmath_{e_0}}(1 - e^{-t\ell_{e_0}})^2} \,\ell_{e_0}\dd\ell_{e_0}\prod_{e \neq e_{0}} \frac{\ell_e\,\dd\ell_e}{e^{t\ell_e} - 1}.
\end{split}
\end{equation}

\medskip

Let $\mathbf{G}_{g,n}'$ be the set of ordered pairs $(\Gamma,e_0)$ where $\Gamma \in \mathbf{G}_{g,n}$ and $e_0 \in E_{\Gamma}$. We introduce the map
$$
\mathbf{glu}\colon \bigg(\mathbf{G}_{g - 1,n + 2} \sqcup \bigsqcup_{\substack{\{(g_1,J_1),(g_2,J_2)\} \\ g_1 + g_2 = g \\ J_1 \sqcup J_2 = \{1,\ldots,n\}}} \mathbf{G}_{g_1,1 + n_1} \times \mathbf{G}_{g_2,1 + n_2}\bigg) \longrightarrow \mathbf{G}_{g,n}',
$$
which consists in adding an edge between the two special leaves -- the two first leaves in the connected situation and the first leaf of each graph in the disconnected situation. This map is surjective, but not necessarily injective. More precisely, if $(\Gamma,e_0) \in \mathbf{G}_{g,n}'$, let us cut $e_0$ to create the stable graph $\Gamma'$ with $n$ labelled leaves and $2$ unlabelled leaves. Let $a_{\Gamma'}$ be equal to $2$ if $\Gamma'$ is invariant under the permutation of the two unlabelled leaves, and $a_{\Gamma'} = 1$ otherwise. If $\Gamma'$ is disconnected, we must have $a_{\Gamma'} = 1$ because the two connected components can be distinguished by the subsets $J_1$ and $J_2$ of leaves of the initial graph that they contain. Furthermore, the number of automorphisms of $\Gamma$ is the product of the number of automorphism of its connected components. If $\Gamma'$ is connected, it must have genus $g - 1$. If $a_{\Gamma'} = 1$, there are two distinct graphs in $\mathbf{G}_{g -1,n + 2}$, that differ by the labels of the two first leaves, which lead to $(\Gamma,e_0)$ after application of $\mathbf{glu}$. If $a_{\Gamma'} = 2$, these two graphs are actually isomorphic. So, when $\Gamma'$ is connected, we always have
$$
\big|\mathbf{glu}^{-1}(\Gamma,e_0)\big| = \frac{2}{a_{\Gamma'}}.
$$
Finally, we notice that for $(\Gamma,e_0) \in \mathbf{G}_{g,n}'$ and $\tilde{\Gamma} \in \mathbf{glu}^{-1}(\Gamma,e_0)$, we always have
$$
|{\rm Aut}\,\Gamma| = a_{\Gamma'}\,|{\rm Aut}\,\tilde{\Gamma}|.
$$
Partitioning the sum \eqref{sumglu} according to the fibers of $\mathbf{glu}$ we obtain
\begin{align*}
	& \partial_{z = 0}\big(VN_{g,n}^{t,z}(\psi;L_1,\ldots,L_n)\big) \\
	& = \int_{\mathbb{R}_{+}} \frac{\psi(\ell_0)\,e^{-t \ell_0}}{(1 - e^{-t \ell_0})^2}\,\ell_0\,\dd \ell_0 \\
	& \times \bigg\{\frac{1}{2} \sum_{\Gamma \in \mathbf{G}_{g - 1,2 + n}} \frac{1}{|{\rm Aut}\,\Gamma|} \int_{\mathbb{R}_{+}^{E_{\Gamma}}} \prod_{v \in V_{\Gamma}} V\Omega^{{\rm M}}_{h(v),k(v)}\big((\ell_e)_{e \in E(v)},(L_{\lambda})_{\lambda \in \Lambda(v)}\big)\prod_{e \in E_{\Gamma}} \frac{\ell_e\,\dd\ell_e}{e^{t\ell_e} - 1}  \\
	& + \frac{1}{2} \sum_{\substack{\{(g_1,J_1),(g_2,J_2)\} \\ g_1 + g_2 = g \\ J_1 \sqcup J_2 = \{1,\ldots,n\}}} \prod_{i = 1}^2  \bigg(\!\! \sum_{\substack{\Gamma_i \\ \in \mathbf{G}_{g_i,1 + |J_i|}}}\!\!\!\!\!\!\! \frac{1}{|{\rm Aut}\,\Gamma_i|}\int_{\mathbb{R}_{+}^{E_{\Gamma_i}}}  \prod_{v \in V_{\Gamma_i}} V\Omega^{{\rm M}}_{h(v),k(v)}\big((\ell_e)_{e \in E_i(v)},(L_{\lambda})_{\lambda \in \Lambda_i(v)}\big)\!\!\!\prod_{e \in E_{\Gamma_i}} \frac{\ell_e\,\dd \ell_e}{e^{t\ell_e} - 1}\bigg)\bigg\},
\end{align*}
where $E_i(v)$ and $\Lambda_i(v)$ are the set of edges and leaves of $\Gamma_i$, and if $\lambda$ is a special leaf we set $L_{\lambda} = \ell_0$. We stress that $\tfrac{1}{2}$ in the last line comes from $\jmath_{e_0} = 1$. We recognise the sums over stable graphs
$$
VN_{h,k}^{t}(\widetilde{L}_{1},\ldots,\widetilde{L}_{k}) = \sum_{\Gamma \in \mathbf{G}_{\tilde{g},\tilde{n}}} \frac{1}{|{\rm Aut}\,\Gamma|} \int_{\mathbb{R}_{+}^{E_{\Gamma}}} \prod_{v \in V_{\Gamma}} V\Omega^{{\rm M}}_{h(v),k(v)}\big((\ell_e)_{e \in E(v)},(\widetilde{L}_{\lambda})_{\lambda \in \Lambda(v)}\big) \prod_{e \in E_{\Gamma}} \frac{\ell_e\,\dd \ell_e}{e^{t\ell_e} - 1},
$$
which already appeared in the proof of Theorem~\ref{stablg}. We can replace the last sum over pairs with a sum over ordered pairs up to multiplication by an extra factor of $\tfrac{1}{2}$. All in all,
\begin{equation}\label{310310}
\begin{split}
	\partial_{z = 0}\big(VN_{g,n}^{t,z}(\psi;L_1,\ldots,L_n)\big) &= \frac{1}{2} \int_{\mathbb{R}_{+}} \frac{\psi(\ell)\,e^{-t\ell}}{(1 - e^{-t\ell})^2}\bigg( VN_{g - 1,n + 2}^{t}(\ell,\ell,L_1,\ldots,L_n) \\
	&\quad + \frac{1}{2}\!\!\!\!\!\! \sum_{\substack{g_1 + g_2 = g \\ J_1 \sqcup J_2 = \{L_1,\ldots,L_n\}}} \!\!\!\!\!\!\! VN_{g_1,1 + |J_1|}^{t}(\ell,J_1) VN_{g_2,1 + |J_2|}^{t}(\ell,J_2)\bigg) \ell\dd \ell. \\
\end{split}
\end{equation}

\medskip

We multiply the boundary lengths by $\beta$ and divide $t$ by $\beta$ in order to  insert this formula in \eqref{3939}. Notice that the quantity in parenthesis in \eqref{310310} now contributes to an even polynomial in $\ell$, such that the monomial $\ell^{2m}$ is a polynomial in $(\beta/t)$, of top degree $6g - 6 + 2n - (2m + 2)$. We recall that the $\beta \rightarrow \infty$ leading behaviour of $VN_{h,k}^{t/\beta}$ from \eqref{Ncollect} is expressed via the Masur--Veech polynomials $V\Omega^{{\rm MV}}_{h,k}$. Since
\begin{equation}
\label{t2co} \lim_{\beta \rightarrow \infty} \beta^2 \int_{\mathbb{R}_{+}} \frac{\psi(\ell)\,e^{-t\ell/\beta}}{(1 - e^{-t\ell/\beta})^2}\ell^{2m+1}\,\dd \ell = \frac{1}{t^2} \int_{\mathbb{R}_{+}}\psi(\ell) \ell^{2m - 1}\dd \ell
\end{equation}
is finite for any $m \geq 0$, only the $m = 0$ terms will contribute in the leading $\beta \rightarrow \infty$ behaviour of \eqref{3939}, in which case \eqref{t2co} is equal to $t^{-2}\,c_0[\psi]$ which exist since $\ell \mapsto \ell^{-1}\psi(\ell)$ is integrable. We arrive at the formula
\begin{align*}
	& \lim_{\beta \rightarrow \infty} \beta^{-(6g - 6 + 2n)} VN_{g,n}^{+}(\rho_{\beta}^*\phi;\psi;\beta L_1,\ldots,\beta L_n)  \\
	& = \frac{1}{2}\,c_0[\psi]\,\hat{c}[\phi]\bigg[t^{-(6g - 6 + 2n)}  \\
	& \quad \cdot \bigg(V\Omega^{{\rm MV}}_{g-1,n + 2}(0,0,tL_1,\ldots,tL_n) + \frac{1}{2} \sum_{\substack{g_1 + g_2 = g \\ J_1 \sqcup J_2 = \{L_1,\ldots,L_n\}}} V\Omega^{{\rm MV}}_{g_1,1 + |J_1|}(0,J_1) V\Omega^{{\rm MV}}_{g_2,1 + |J_2|}(0,J_2)\bigg)\bigg],
\end{align*}
which is the first desired formula. To obtain the second formula, we remark that all $\beta \rightarrow \infty$ limits used in the previous arguments are uniform for $L_1,\ldots,L_n$ in any compact of $\mathbb{R}_{\geq 0}$. Hence
\begin{align*}
	& \lim_{\beta \rightarrow \infty} \beta^{-(6g - 6 + 2n)} VN_{g,n}^{+}(\rho_{\beta}^*\phi;\psi;L_1,\ldots,L_n)  \\
	& = \frac{1}{2}\,c_0[\psi]\,\hat{c}[\phi]\bigg[t^{-(6g - 6 + 2n)}  \\
	& \quad \cdot \bigg(V\Omega^{{\rm MV}}_{g-1,n + 2}(0) + \frac{1}{2} \sum_{\substack{g_1 + g_2 = g \\ J_1 \sqcup J_2 = \{L_1,\ldots,L_n\}}} V\Omega^{{\rm MV}}_{g_1,1 + |J_1|}(0) V\Omega^{{\rm MV}}_{g_2,1 + |J_2|}(0)\bigg)\bigg].
\end{align*}
The effect of $\hat{c}[\phi]$ factors out to give $c_{6g - 6 + 2n}[\phi]$ and the sum over the partition $J_1 \sqcup J_2 = \{L_1,\ldots,L_n\}$ yields binomial coefficients, hence the formula we sought for.

\medskip

The statistics $\widetilde{N}$ are perhaps more natural. Their expression slightly differs from \eqref{1111A} by one factor $e^{-t\ell/\beta}$ less in front of $\psi$ -- this factor was previously due to the contribution of $\gamma_0$ to the total length that was included before evaluating $\phi$. Namely, we have
$$ 
\widetilde{N}_{\Sigma}(\rho_{\beta}^*\phi;\psi;\sigma) = \int_{\mathbb{R}_{+}} \Phi(t)\,\partial_{z = 0}\big(\widetilde{N}_{\Sigma}^{t/\beta,z}(\psi;\sigma)\big)\,\dd t,
$$
with 
$$
\widetilde{N}_{\Sigma}^{t,z}(\psi;\sigma) = \sum_{c \in M_{\Sigma}'} \prod_{\gamma \in \pi_0(c)} \frac{1}{e^{t\ell_{\sigma}(\gamma)} - 1}\bigg(1 + \frac{z\,\psi(\ell_{\sigma}(\gamma))\,e^{t\ell_{\sigma}(\gamma)}}{2^{\jmath(\gamma)}(1 - e^{-t\ell_{\sigma}(\gamma)})}\bigg).
$$
All the previous argument can be carried over, except that we use instead of \eqref{t2co} the limit
$$
\lim_{\beta \rightarrow \infty} \beta^2 \int_{\mathbb{R}_{+}} \frac{\psi(\ell)}{(1 - e^{-t\ell/\beta})^2} \ell^{2m+1}\,\dd \ell = \frac{1}{t^2} \int_{\mathbb{R}_{+}} \psi(\ell)\ell^{2m - 1}\dd \ell,
$$
which yields the same result.
\end{proof}

\medskip

%*********************************************%
\section{Computing Masur--Veech polynomials}
%***************************************
\label{S5}

The Masur--Veech polynomial $V\Omega_{g,n}^{{\rm MV}}$ has degree $6g - 6 + 2n$. As explained in Section~\ref{Spoly}, we decompose it as follows
\begin{equation}
\label{Vpoly} V\Omega_{g,n}^{{\rm MV}}(L_1,\ldots,L_n) = \sum_{\substack{d_1,\ldots,d_n \geq 0 \\ d_1 + \cdots + d_n \leq 3g - 3 + n}} F_{g,n}[d_1,\ldots,d_n]\,\prod_{j = 1}^n \frac{L_j^{2d_j}}{(2d_j + 1)!},
\end{equation}
In this section we drop the superscript ${}^{{\rm MV}}$ on the $F_{g,n}$'s as it will always refer to the coefficients of \eqref{Vpoly}. By symmetry, $F_{g,n}$ can be considered as a function on the set of partitions of size less or equal to $3g - 3 + n$. It is convenient to give a name to the value of $F_{g,n}$ on partitions with a single row
$$
H_{g,n}[d] = F_{g,n}[d,0,\ldots,0]\qquad {\rm and}\qquad H_{n}[d] = H_{0,n}[d].
$$
By convention, if some $d_i$ is negative or if $2g - 2 + n \leq 0$, we declare $F_{g,n}[d_1,\ldots,d_n] = 0$. We are particularly interested in the Masur--Veech volumes which -- up to normalisation -- are the values of this function on the empty partition:
\begin{equation}  
\label{MVnorm} MV_{g,n} = \frac{2^{4g - 2 + n}(4g - 4 + n)!}{(6g - 7 + 2n)!}\,H_{g,n}[0].
\end{equation}

In this section, we are going to illustrate some computations of $H_{g,n}$ that can be done with stable graphs (Section~\ref{StablegsecH}), give explicitly the Virasoro constraints for the $F_{g,n}$ (Theorem~\ref{thm:intro:2}, detailed in Section~\ref{Virrec}) and the recursion it implies for $H_{g=0,n}$ (Section~\ref{Genus01row}). These computations lead us to conjecture some structural formulas for Masur--Veech volumes for fixed $g$ but any $n$ (Section~\ref{conjMVsecmnc}). We study their consequence for area Siegel--Veech constants in light of Goujard's recursion (Section~\ref{ConjSVsec}) and their behavior when $n \rightarrow \infty$ (Section~\ref{Sasymss}).

%***************************************
\subsection{Leading coefficients via stable graphs}\label{secfrontcoef}
%***************************************
\label{StablegsecH}
We denote $H_{g,n}^*[d] = H_{g,n}[3g - 3 + n - d]$ and consider low values of $d$. In other words, these are the coefficients appearing in front of the terms of high(est) degrees in the Masur--Veech polynomial $V\Omega^{{\rm MV}}_{g,n}(L,0,\ldots,0)$. They can be computed efficiently with the stable graph formula, or equivalently with Equation~\eqref{eqn:interMV1}. We give a few examples of such computations, starting from the expression 
\[
\begin{split}
  H_{g,n}^*[d]
  & =
  \frac{(6g - 5 + 2n - 2d)!}{(3g - 3 + n - d)!}\,2^{-(3g - 3 + n) + d} \times \\
  & \qquad\qquad 
 \sum_{k \geq 0} \sum_{\Gamma \in \mathbf{G}_{g,n}^{k}} \frac{1}{|{\rm Aut}\,\Gamma|} \int_{\overline{\mathfrak{M}}_{\Gamma}} 
     \Biggl[ \prod_{\substack{e \in E_{\Gamma} \\ e = (h,h')}}\bigg(
        \sum_{D \ge 0} (2\pi^{2})^{D+1} \frac{B_{2D + 2}}{2D + 2}\, (- \psi_{h} - \psi_{h'})^{D}\bigg)
\Biggr]_{2(d - k)}\,\psi_1^{3g - 3 + n - d} \\
  & =
  \frac{(6g - 5 + 2n - 2d)!}{(3g - 3 + n - d)!}\,2^{-(3g - 3 + n) + 2d} \, \pi^{2d}\times \\
  & \qquad\qquad \sum_{k \geq 0} \sum_{\Gamma \in \mathbf{G}_{g,n}^{k}} \frac{1}{|{\rm Aut}\,\Gamma|} 
  \int_{\overline{\mathfrak{M}}_{\Gamma}} \Biggl[
      \prod_{\substack{e \in E_{\Gamma} \\ e = (h,h')}}\bigg(
        \sum_{D \ge 0} \frac{|B_{2D + 2}|}{2D + 2}\, (\psi_{h} + \psi_{h'})^{D}\bigg)
  \Biggr]_{2(d - k)}\,\psi_1^{3g - 3 + n - d}
\end{split}
\] 
where $[\,\cdot\,]_{2k}$ extracts the component of cohomological degree $2k$ and we recall that $\mathbf{G}_{g,n}^k$ is the set of stable graphs with $k$ edges.

%***************************************
\subsubsection{Genus zero.}
%***************************************

To compute the vertex weights, we will use the formula \cite{Witten}
\[
  \int_{\overline{\mathfrak{M}}_{0,n}} \prod_{i = 1}^n \psi_i^{m_i} = \delta_{\sum_{i} m_i,n - 3} \frac{(n - 3)!}{d_1!\cdots d_n!},
\]
which is a consequence of the string equation for $\psi$ classes.

\medskip

$\bullet$ $\mathbf{d = 0}$. The computation of the integral is trivial and we have
\[
  H_{n}^*[0] = \frac{(2 n - 5)!}{2^{n-3} (n-3)!} \int_{\overline{\mathfrak{M}}_{0,n}}\psi_1^{n-3} = \frac{(2 n - 5)!}{2^{n-3} (n-3)!}.
\]

\medskip

$\bullet$ $\mathbf{d = 1}$. The only contribution comes from the stable graph with one edge joining two genus zero vertices (Figure~\ref{frontg0d1}). As the $\psi_1^{n - 4}$ carried by the first leaf saturates the dimension of the moduli space at its incident vertex $v_1$, this edge must have degree $0$ and receives a weight $\tfrac{B_2}{2}$. Then, the contribution from each vertex after integration is equal to $1$. It remains to distribute the leaves labelled $2,\ldots,n$ between a first group of $n - 3$ which will be incident to $v_1$, and a second group of $2$ which will be incident to the second vertex. Hence
\[
  H_{n}^*[1] = \frac{(2 n - 7)!}{2^{n - 2} (n-4)!} \frac{(n-1)(n - 2)}{3} \pi^2.
\]

\renewcommand{\figurename}{Figure}
\begin{figure}[h!]
\begin{center}
\begin{tikzpicture}
  \draw (-1,0) -- (1.5,0);
  \filldraw[black] (0,0) circle (3pt);
  \filldraw[black] (1.5,0) circle (3pt);
  \draw (0,0) -- (0,.8);
  \draw (.08,0) -- (.08,.8);
  \draw (-.08,0) -- (-.08,.8);
  \node at (0,1.1) {$\overbrace{~}^{n-3}$};
  \draw (1.5,0) -- ($(1.5,0) + (35:.4)$);
  \draw (1.5,0) -- ($(1.5,0) + (-35:.4)$);
  \node at (-1,0) [above] {\footnotesize$\psi_1^{n-4}$};

  \node at (4,.6) {\small$\displaystyle \frac{B_2}{2}\binom{n-1}{2}$};
\end{tikzpicture}
\caption{\label{frontg0d1} Stable graph contributing to $H_{n}^*[1]$.}
\end{center}
\end{figure}
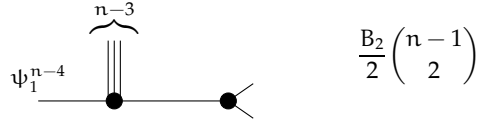

\medskip

$\bullet$ $\mathbf{d = 2}$. We have to consider stable graphs with vertices of genus zero with $1$ or $2$ edges (for cohomological degree reasons the graph with no edges does not contribute). There are four cases (Figure~\ref{frontg0d2}).
\begin{enumerate}
\item[$1a$--] Two vertices are connected by an edge, the extra $\psi$ class lies on the same vertex as $\psi_1^{n-5}$. There are $\binom{n-1}{2}$ ways to pick two leaves carried by the second vertex. The contribution of the first vertex is $\int_{\overline{ \mathfrak{M}}_{0, n-1}} \psi_1^{n-5}\psi = n-4$, and the contribution of the second vertex is $\int_{\overline{ \mathfrak{M}}_{0, 3}}1 =1$. The edge contribution is $\frac{|B_4|}{4}$.
\item[$1b$--] Two vertices are connected by an edge, the first vertex carries $\psi_1^{n - 5}$ and the extra $\psi$ class lies on the second vertex. There are $\binom{n-1}{3}$ ways to pick three leaves to the second vertex. Both vertex contributions are equal to $1$, and the edge contribution is $\tfrac{|B_4|}{4}$.
\item[$2a$--] A central vertex carrying $\psi_1^{n - 5}$ is connected to two other vertices carrying no $\psi$ class. There are $\binom{n-1}{2,2,n-5}$ ways to pick two leaves for each of the two non-central vertices. The contribution of each vertex is $1$, each edge contributes to a factor $\tfrac{B_2}{2}$ and we get an extra factor of a $\tfrac{1}{2}$ from the automorphism of the graph (exchange of the two non-central vertices).
\item[$2b$--] There are three vertices connected by two edges and $\psi_1^{n-5}$ is carried by an extremal vertex. There are $n - 1$ choices for the leaf on the central vertex, and $\binom{n-2}{2}$ ways to pick the two leaves for the second extremal vertex. The contribution of each vertex is $1$, each edge contributes to a factor of $\tfrac{B_2}{2}$, and there are no automorphisms.
\end{enumerate}
Summing up all contributions we obtain:
\begin{align*}
H_{n}^*[2] &= \frac{(2 n - 9)!}{2^{n-7} (n-5)!} \, \pi^4 \bigg\{\frac{|B_4|}{4} \bigg( (n - 4) \binom{n - 1}{2} + \binom{n - 1}{3} \bigg) \\
& \qquad\qquad\qquad + \bigg(\frac{B_2}{2}\bigg)^2 \bigg( \frac{1}{2} \binom{n-1}{2, 2, n - 5} + (n - 1)\binom{n - 2}{2}\bigg)\bigg\} \\
&= \frac{(2 n - 9)!}{2^{n-7} (n-5)!} \frac{(n-1)(n-2)(5n^2 + 17n - 120)}{5760}\pi^4. \nonumber
\end{align*}

\begin{figure}[h!]
\begin{center}
\begin{tikzpicture}
  \draw (-1,0) -- (1.5,0);
  \filldraw[black] (0,0) circle (3pt);
  \filldraw[black] (1.5,0) circle (3pt);
  \draw (0,0) -- (0,.8);
  \draw (.08,0) -- (.08,.8);
  \draw (-.08,0) -- (-.08,.8);
  \node at (0,1.1) {$\overbrace{~}^{n-3}$};
  \draw (1.5,0) -- ($(1.5,0) + (35:.4)$);
  \draw (1.5,0) -- ($(1.5,0) + (-35:.4)$);
  \node at (-1,0) [above] {\footnotesize$\psi_1^{n-5}$};
  \node at (0,0) [below right] {\footnotesize$\psi$};

  \node [right] at (2,.6) {$\frac{|B_4|}{4}\binom{n-1}{2}(n-4)$};

  \begin{scope}[xshift=7.5cm,yshift=0cm]
    \draw (-1,0) -- (0,0);
    \draw (0,0) -- (30:1.2);
    \draw (0,0) -- (-30:1.2);
    \filldraw[black] (0,0) circle (3pt);
    \filldraw[black] (30:1.2) circle (3pt);
    \filldraw[black] (-30:1.2) circle (3pt);
    \draw (0,0) -- (0,.8);
    \draw (.08,0) -- (.08,.8);
    \draw (-.08,0) -- (-.08,.8);
    \node at (0,1.1) {$\overbrace{~}^{n-5}$};
    \draw (30:1.2) -- ($(30:1.2) + (35:.4)$);
    \draw (30:1.2) -- ($(30:1.2) + (-35:.4)$);
    \draw (-30:1.2) -- ($(-30:1.2) + (35:.4)$);
    \draw (-30:1.2) -- ($(-30:1.2) + (-35:.4)$);
    \node at (-1,0) [above] {\footnotesize$\psi_1^{n-5}$};

    \node [right] at (2.3,.6) {$\half \cdot \bigl(\frac{B_2}{2}\bigr)^2 \binom{n-1}{2,2,n-5}$};
  \end{scope}

  \begin{scope}[xshift=0,yshift=-3cm]
    \draw (-1,0) -- (1.5,0);
    \filldraw[black] (0,0) circle (3pt);
    \filldraw[black] (1.5,0) circle (3pt);
    \draw (0,0) -- (0,.8);
    \draw (.08,0) -- (.08,.8);
    \draw (-.08,0) -- (-.08,.8);
    \node at (0,1.1) {$\overbrace{~}^{n-4}$};
    \draw (1.5,0) -- ($(1.5,0) + (35:.4)$);
    \draw (1.5,0) -- ($(1.5,0) + (0:.4)$);
    \draw (1.5,0) -- ($(1.5,0) + (-35:.4)$);
    \node at (-1,0) [above] {\footnotesize$\psi_1^{n-5}$};
    \node at (1.5,0) [above left] {\footnotesize$\psi$};

    \node [right] at (2,.6) {$\frac{|B_4|}{4}\binom{n-1}{3}$};
  \end{scope}

  \begin{scope}[xshift=7.5cm,yshift=-3cm]
    \draw (-1,0) -- (1.8,0);
    \filldraw[black] (0,0) circle (3pt);
    \filldraw[black] (.9,0) circle (3pt);
    \filldraw[black] (1.8,0) circle (3pt);
    \draw (0,0) -- (0,.8);
    \draw (.9,0) -- (.9,.5);
    \draw (.08,0) -- (.08,.8);
    \draw (-.08,0) -- (-.08,.8);
    \node at (0,1.1) {$\overbrace{~}^{n-4}$};
    \draw (1.8,0) -- ($(1.8,0) + (35:.4)$);
    \draw (1.8,0) -- ($(1.8,0) + (-35:.4)$);
    \node at (-1,0) [above] {\footnotesize$\psi_1^{n-5}$};

    \node [right] at (2.3,.6) {$\bigl(\frac{B_2}{2}\bigr)^2 (n-1) \binom{n-2}{2}$};
  \end{scope}
\end{tikzpicture}
\caption{\label{frontg0d2} Stable graphs contributing to $H_{n}^*[2]$.}
\end{center}
\end{figure}
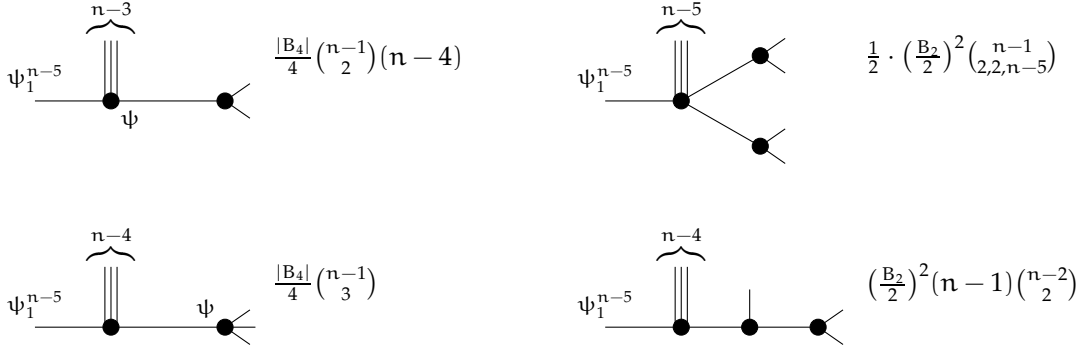

%***************************************
\subsubsection{Genus one}
\label{Sgenus1}
%***************************************

Here we will need the classical formula
\begin{equation}\label{1n1n}
  \int_{\overline{\mathfrak{M}}_{1,n}} \psi_1^{n} = \frac{1}{24},
\end{equation}
which is sufficient to compute $H_{1,n}^*[d]$ for $d = 0,1$. More general $\psi$ classes intersections in genus one would be necessary to push the stable graphs computations further. For instance, we will compute below $H_{1,n}^*[2]$ using
\begin{equation}\label{otherpsigenus1}
  \int_{\overline{\mathfrak{M}}_{1,n}} \psi_1^{n - 1}\psi_2 = \frac{n - 1}{24}.
\end{equation}
We present in Appendix~\ref{AppA} a closed formula for arbitrary genus one $\psi$ classes intersections (including \eqref{1n1n} and \eqref{otherpsigenus1}), which we prove in an elementary way using well-known facts, but for which we could not find a reference.

\medskip

$\bullet$ $\mathbf{d = 0}$. The only stable graph contributing has a single vertex, and with \eqref{1n1n} we obtain
$$
  H_{1,n}^*[0] = \frac{1}{24}\,\frac{(2n + 1)!}{2^{n} n!}.
$$

\medskip

$\bullet$ $\mathbf{d =1}$. We have to consider the stable graphs with a single edge, which is either separating or non-separating (Figure~\ref{frontg1d1}). This edge cannot carry $\psi$ classes and its contribution is $\tfrac{B_2}{2}$. In the separating case, there is a vertex of genus one which carries the $\psi_1^{n - 1}$, which is connected to a second vertex of genus zero. There are  $\binom{n-1}{2}$ ways to distribute the two leaves on the genus zero vertex. The contribution of the genus zero vertex is $1$ and the contribution of the genus one vertex is $\int_{\overline{\mathfrak{M}}_{1,n - 1}} \psi_1^{n - 1} = \tfrac{1}{24}$. In the non-separating case, there is a single vertex, which has genus zero; its contribution is $1$, and we have an automorphism factor of $\tfrac{1}{2}$ (exchange of the two ends of the edge). Hence
\begin{align*}
H_{1,n}^*[1] &= \frac{(2 n - 2)!}{2^{n-2} (n-1)!} \, \pi^2 \,  \frac{B_2}{2}\bigg(\frac{1}{2} + \frac{1}{24}\binom{n - 1}{2}\bigg) \\
&= \frac{(2 n - 2)!}{2^{n-2} (n-1)!} \frac{(n^2 - 3n + 26)}{576} \, \pi^2.
\end{align*}

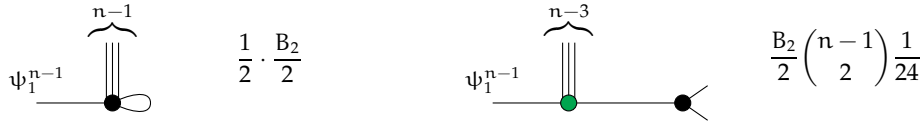
\begin{figure}[h!]
\begin{center}
\begin{tikzpicture}
  \draw (-1,0) -- (0,0);
  \filldraw[black] (0,0) circle (3pt);
  \draw (0,0) -- (0,.8);
  \draw (.08,0) -- (.08,.8);
  \draw (-.08,0) -- (-.08,.8);
  \node at (0,1.1) {$\overbrace{~}^{n-1}$};
  \draw (0,0) to[out=30,in=90] (.5,0);
  \draw (0,0) to[out=-30,in=-90] (.5,0);
  \node at (-1,0) [above] {\footnotesize$\psi_1^{n-1}$};

  \node [right] at (1.5,.6) {\small$\displaystyle \half \cdot \frac{B_2}{2}$};

  \begin{scope}[xshift=6cm,yshift=0cm]
    \draw (-1,0) -- (1.5,0);
    \filldraw[black] (1.5,0) circle (3pt);
    \draw (0,0) -- (0,.8);
    \draw (.08,0) -- (.08,.8);
    \draw (-.08,0) -- (-.08,.8);
    \node at (0,1.1) {$\overbrace{~}^{n-3}$};
    \draw (1.5,0) -- ($(1.5,0) + (35:.4)$);
    \draw (1.5,0) -- ($(1.5,0) + (-35:.4)$);
    \node at (-1,0) [above] {\footnotesize$\psi_1^{n-1}$};

    \node [right] at (2.5,.6) {\small$\displaystyle \frac{B_2}{2} \binom{n-1}{2} \frac{1}{24}$};

    \draw[fill=Green] (0,0) circle (3pt);
  \end{scope}
\end{tikzpicture}
\caption{\label{frontg1d1} Stable graphs contributing to $H_{1,n}^*[1]$. The black vertices have genus zero and the green ones have genus one.}
\end{center}
\end{figure}

$\bullet$ $\mathbf{d = 2}.$ We have to consider stable graphs with one or two edges (Figure~\ref{frontg1d2}). When there is a single edge, its contribution is $\tfrac{|B_4|}{4}$ as we have an extra $\psi$ class to distribute at one of its ends. Four cases appear.
\begin{enumerate} 
\item[$1a$--] There is one non-separating edge on a single vertex of genus zero. The extra $\psi$ class is carried by one extremity of the edge, forbidding non-trivial automorphisms. The contribution of the vertex is $\int_{\overline{\mathfrak{M}}_{0,n + 2}} \psi_1^{n - 2}\psi = (n - 1)$, and the contribution of the edge is $\tfrac{|B_4|}{4}$.
\item[$1b$--] The vertex carrying $\psi_1^{n - 2}$ has genus one and also carries the extra $\psi$ class. It is connected to a genus zero vertex, which has two leaves and there are $\binom{n - 1}{2}$ ways to choose them. The contribution of the genus one vertex is $\int_{\overline{\mathfrak{M}}_{1,n - 1}} \psi_1^{n - 2}\psi = \frac{n - 2}{24}$ using \eqref{otherpsigenus1}, the contribution of the genus zero vertex is $1$.
\item[$1c$--] The vertex carrying $\psi_1^{n - 2}$ has genus one and is connected to a genus zero vertex carrying the extra $\psi$ class. We can pick the $3$ leaves incident to the genus zero vertex in $\binom{n - 1}{3}$ ways. The contribution of the genus one vertex is $\int_{\overline{\mathfrak{M}}_{1,n - 2}} \psi_1^{n - 2} = \tfrac{1}{24}$ while the contribution of the genus zero vertex is $1$. 
\item[$1d$ --] The vertex carrying $\psi_1^{n - 2}$ has genus zero and is connected to a genus one vertex carrying the extra $\psi$ class. The contribution of the genus zero vertex is $1$ and the contribution of the genus one vertex is $\int_{\overline{\mathfrak{M}}_{1,1}} \psi = \frac{1}{24}$.
\end{enumerate}
When there are two edges, each of them contributes by a factor of $\tfrac{B_2}{2}$ and there is no extra $\psi$ class.
\begin{enumerate} 
\item[$2a$--] The vertex carrying $\psi_1^{n - 2}$ has genus zero, is incident to a non-separating edge forming a loop, and the second edge connects it to another vertex of genus zero. There are $\binom{n - 1}{2}$ ways to choose the two leaves on the second vertex. The loop is responsible for a symmetry factor of a $\tfrac{1}{2}$, and the contribution of both vertices is $1$.
\item[$2b$--] The vertex carrying $\psi_1^{n - 2}$ has genus zero and is connected to another vertex of genus zero which carries a loop. The latter yields a symmetry factor of a $\tfrac{1}{2}$ and the contribution of both vertices is $1$.
\item[$2c$--] The vertex carrying $\psi_1^{n - 2}$ has genus zero and is connected to another vertex of genus zero by two edges. To the second vertex should be assigned a leaf and this can be done in $(n - 1)$ ways. There is a symmetry factor of a $\tfrac{1}{2}$ for the exchange of the two edges, and the contribution of both vertices is $1$.
\item[$2d$--] The vertex carrying $\psi_1^{n - 2}$ has genus one, it is connected to a vertex of genus zero with one leaf, which itself is connected to another vertex of genus zero with $2$ leaves. There are $(n - 1)\binom{n - 2}{2}$ ways to assign the leaves. The contribution from the genus one vertex is $\int_{\overline{\mathfrak{M}}_{1,n - 2}} \psi_1^{n - 2} = \tfrac{1}{24}$ and the contribution of the genus zero vertices is $1$. 
\item[$2e$--] There are three vertices, the central one has genus one and carries $\psi_1^{n - 2}$, the extremal ones have genus zero and carry two leaves each. There are $\binom{n - 1}{2,2,n- 5}$ ways to assign the leaves but there is a symmetry factor of a $\tfrac{1}{2}$ for the exchange of the two extremal vertices. The contribution of the genus one vertex is $\int_{\overline{\mathfrak{M}}_{1,n - 2}} \psi_1^{n - 2} = \tfrac{1}{24}$ and the contribution of the genus zero vertices is $1$. 
\end{enumerate}
Summing all contributions, we obtain
\begin{align*}
	H_{1,n}^*[2] & = \frac{(2n - 3)!}{2^{n - 4}(n - 2)!} \, \pi^4
    \bigg\{ \left(\frac{B_2}{2}\right)^2\bigg( \frac{1}{2}\,\binom{n - 1}{2} + \frac{1}{2} + \frac{n - 1}{2} +  \frac{n - 1}{24}\,\binom{n - 2}{2} + \frac{1}{24}\,\frac{1}{2}\,\binom{n - 1}{2,2,n - 5}\bigg) \\ 
	& \qquad\qquad\qquad\qquad\qquad
    + \frac{|B_4|}{4}\bigg(n - 1 + \frac{n - 2}{24}\binom{n - 1}{2} + \frac{1}{24}\binom{n - 1}{3} + \frac{1}{24}\bigg)\bigg\} \\ 
	& = \frac{5n^4 + 2n^3 + 127n^2 + 1162n - 768}{138240} \, \pi^4 .
\end{align*}
 
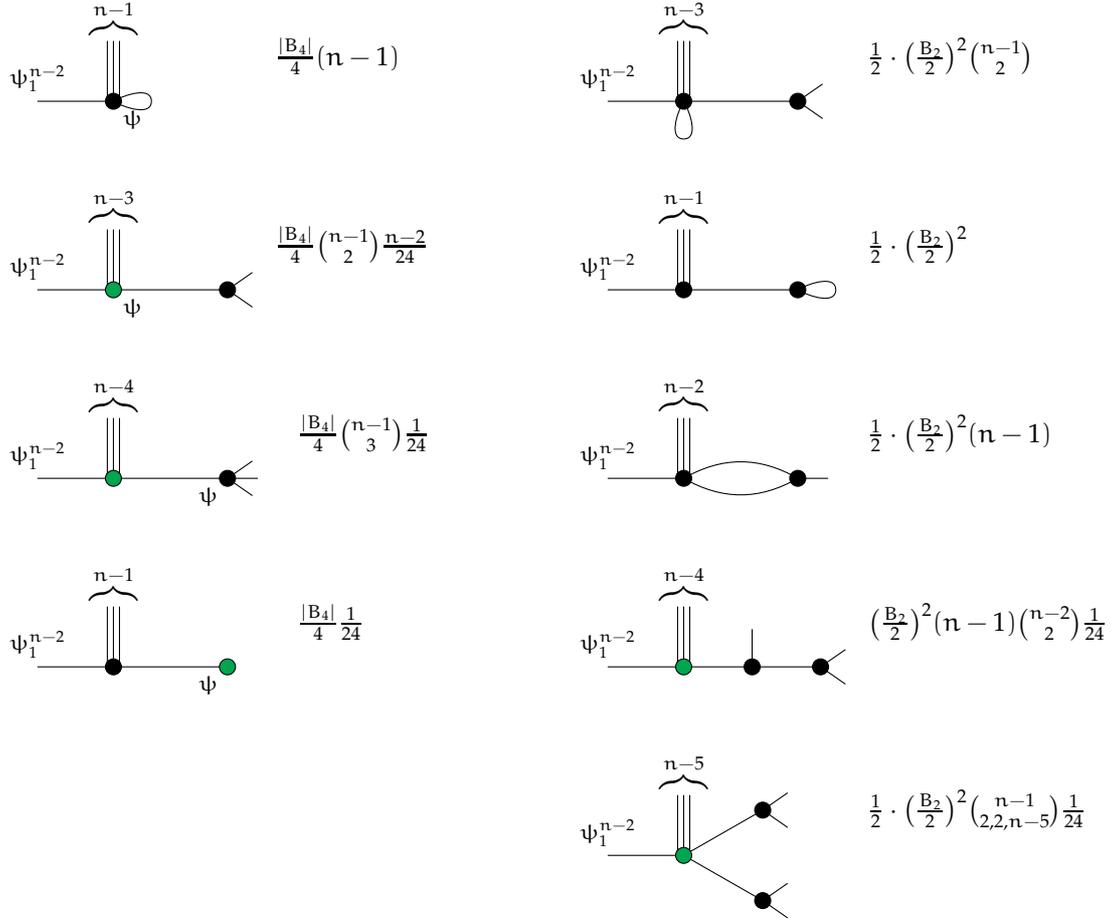
\begin{figure}[h!]
\begin{center}
\begin{tikzpicture}
  % 1

  \draw (-1,0) -- (0,0);
  \draw (0,0) to[out=30,in=90] (.5,0);
  \draw (0,0) to[out=-30,in=-90] (.5,0);
  \filldraw[black] (0,0) circle (3pt);
  \draw (0,0) -- (0,.8);
  \draw (.08,0) -- (.08,.8);
  \draw (-.08,0) -- (-.08,.8);
  \node at (0,1.1) {$\overbrace{~}^{n-1}$};
  \node at (-1,0) [above] {\footnotesize$\psi_1^{n-2}$};
  \node at (0,0) [below right] {\footnotesize$\psi$};

  \node [right] at (2,.6) {$\frac{|B_4|}{4}(n-1)$};

  \begin{scope}[xshift=7.5cm,yshift=0cm]
  % 2

    \draw (-1,0) -- (1.5,0);
    \filldraw[black] (0,0) circle (3pt);
    \filldraw[black] (1.5,0) circle (3pt);
    \draw (0,0) to[out=-60,in=0] (0,-.5);
    \draw (0,0) to[out=-120,in=180] (0,-.5);
    \draw (0,0) -- (0,.8);
    \draw (.08,0) -- (.08,.8);
    \draw (-.08,0) -- (-.08,.8);
    \node at (0,1.1) {$\overbrace{~}^{n-3}$};
    \draw (1.5,0) -- ($(1.5,0) + (35:.4)$);
    \draw (1.5,0) -- ($(1.5,0) + (-35:.4)$);
    \node at (-1,0) [above] {\footnotesize$\psi_1^{n-2}$};

    \node [right] at (2.3,.6) {$\half \cdot \bigl(\frac{B_2}{2}\bigr)^2 \binom{n-1}{2}$};
  \end{scope}
  \begin{scope}[xshift=0,yshift=-2.5cm]
  % 3

    \draw (-1,0) -- (1.5,0);
    \filldraw[black] (1.5,0) circle (3pt);
    \draw (0,0) -- (0,.8);
    \draw (.08,0) -- (.08,.8);
    \draw (-.08,0) -- (-.08,.8);
    \node at (0,1.1) {$\overbrace{~}^{n-3}$};
    \draw (1.5,0) -- ($(1.5,0) + (35:.4)$);
    \draw (1.5,0) -- ($(1.5,0) + (-35:.4)$);
    \draw[fill=Green] (0,0) circle (3pt);
    \node at (-1,0) [above] {\footnotesize$\psi_1^{n-2}$};
    \node at (0,0) [below right] {\footnotesize$\psi$};

    \node [right] at (2,.6) {$\frac{|B_4|}{4}\binom{n-1}{2}\frac{n-2}{24}$};
  \end{scope}
  \begin{scope}[xshift=7.5cm,yshift=-2.5cm]
  % 4

    \draw (-1,0) -- (1.5,0);
    \filldraw[black] (0,0) circle (3pt);
    \filldraw[black] (1.5,0) circle (3pt);
    \draw (0,0) -- (0,.8);
    \draw (.08,0) -- (.08,.8);
    \draw (-.08,0) -- (-.08,.8);
    \node at (0,1.1) {$\overbrace{~}^{n-1}$};
    \node at (-1,0) [above] {\footnotesize$\psi_1^{n-2}$};
    \draw (1.5,0) to[out=30,in=90] (2,0);
    \draw (1.5,0) to[out=-30,in=-90] (2,0);

    \node [right] at (2.3,.6) {$\half \cdot \bigl(\frac{B_2}{2}\bigr)^2$};
  \end{scope}
  \begin{scope}[xshift=0,yshift=-5cm]
  % 5

    \draw (-1,0) -- (1.5,0);
    \filldraw[black] (1.5,0) circle (3pt);
    \draw (0,0) -- (0,.8);
    \draw (.08,0) -- (.08,.8);
    \draw (-.08,0) -- (-.08,.8);
    \node at (0,1.1) {$\overbrace{~}^{n-4}$};
    \draw (1.5,0) -- ($(1.5,0) + (35:.4)$);
    \draw (1.5,0) -- ($(1.5,0) + (0:.4)$);
    \draw (1.5,0) -- ($(1.5,0) + (-35:.4)$);
    \draw[fill=Green] (0,0) circle (3pt);
    \node at (-1,0) [above] {\footnotesize$\psi_1^{n-2}$};
    \node at (1.5,0) [below left] {\footnotesize$\psi$};

    \node [right] at (2.3,.6) {$\frac{|B_4|}{4}\binom{n-1}{3}\frac{1}{24}$};
  \end{scope}
  \begin{scope}[xshift=7.5cm,yshift=-5cm]
  % 6

    \draw (-1,0) -- (0,0);
    \draw (0,0) to[out=30,in=150] (1.5,0);
    \draw (0,0) to[out=-30,in=-150] (1.5,0);
    \filldraw[black] (0,0) circle (3pt);
    \filldraw[black] (1.5,0) circle (3pt);
    \draw (0,0) -- (0,.8);
    \draw (.08,0) -- (.08,.8);
    \draw (-.08,0) -- (-.08,.8);
    \node at (0,1.1) {$\overbrace{~}^{n-2}$};
    \draw (1.5,0) -- ($(1.5,0) + (0:.4)$);
    \node at (-1,0) [above] {\footnotesize$\psi_1^{n-2}$};

    \node [right] at (2.3,.6) {$\half \cdot \bigl(\frac{B_2}{2}\bigr)^2 (n-1)$};
  \end{scope}
  \begin{scope}[xshift=0cm,yshift=-7.5cm]
  % 7

    \draw (-1,0) -- (1.5,0);
    \filldraw[black] (0,0) circle (3pt);
    \draw (0,0) -- (0,.8);
    \draw (.08,0) -- (.08,.8);
    \draw (-.08,0) -- (-.08,.8);
    \node at (0,1.1) {$\overbrace{~}^{n-1}$};
    \draw[fill=Green] (1.5,0) circle (3pt);
    \node at (-1,0) [above] {\footnotesize$\psi_1^{n-2}$};
    \node at (1.5,0) [below left] {\footnotesize$\psi$};

    \node [right] at (2.3,.6) {$\frac{|B_4|}{4}\frac{1}{24}$};
  \end{scope}
  \begin{scope}[xshift=7.5cm,yshift=-7.5cm]
  % 8

    \draw (-1,0) -- (1.8,0);
    \filldraw[black] (.9,0) circle (3pt);
    \filldraw[black] (1.8,0) circle (3pt);
    \draw (0,0) -- (0,.8);
    \draw (.9,0) -- (.9,.5);
    \draw (.08,0) -- (.08,.8);
    \draw (-.08,0) -- (-.08,.8);
    \node at (0,1.1) {$\overbrace{~}^{n-4}$};
    \draw (1.8,0) -- ($(1.8,0) + (35:.4)$);
    \draw (1.8,0) -- ($(1.8,0) + (-35:.4)$);
    \draw[fill=Green] (0,0) circle (3pt);
    \node at (-1,0) [above] {\footnotesize$\psi_1^{n-2}$};

    \node [right] at (2.3,.6) {$\bigl(\frac{B_2}{2}\bigr)^2 (n-1) \binom{n-2}{2} \frac{1}{24}$};
  \end{scope}
  \begin{scope}[xshift=7.5cm,yshift=-10cm]
  % 9

    \draw (-1,0) -- (0,0);
    \draw (0,0) -- (30:1.2);
    \draw (0,0) -- (-30:1.2);
    \filldraw[black] (30:1.2) circle (3pt);
    \filldraw[black] (-30:1.2) circle (3pt);
    \draw (0,0) -- (0,.8);
    \draw (.08,0) -- (.08,.8);
    \draw (-.08,0) -- (-.08,.8);
    \node at (0,1.1) {$\overbrace{~}^{n-5}$};
    \draw (30:1.2) -- ($(30:1.2) + (35:.4)$);
    \draw (30:1.2) -- ($(30:1.2) + (-35:.4)$);
    \draw (-30:1.2) -- ($(-30:1.2) + (35:.4)$);
    \draw (-30:1.2) -- ($(-30:1.2) + (-35:.4)$);
    \draw[fill=Green] (0,0) circle (3pt);
    \node at (-1,0) [above] {\footnotesize$\psi_1^{n-2}$};

    \node [right] at (2.3,.6) {$\half \cdot \bigl(\frac{B_2}{2}\bigr)^2 \binom{n-1}{2,2,n-5}\frac{1}{24}$};
  \end{scope}
\end{tikzpicture}
\caption{\label{frontg1d2} Stable graphs contributing to $H_{1,n}^*[2]$. The black vertices have genus zero and the green ones have genus one.}
\end{center}
\end{figure}

%***************************************
\subsection{Virasoro constraints}
\label{Virrec}
%***************************************

In this paragraph, we write down explicitly the recursion of Theorem~\ref{thm:intro:2} for the coefficients of the Masur--Veech polynomials. It is obtained by inserting the Kontsevich initial data \eqref{K1}-\eqref{K2} and the twist $u_{a,b}$ given by (\ref{udd}) into the general formula \eqref{polytwist} and recursion \eqref{recfgb}. It is equivalent to Virasoro constraints, obtained (see e.g. \cite{CoursToulouse}) by conjugation of the Virasoro constraints for $\psi$ classes intersections, with the operator
$$
\mathcal{U} = \exp\bigg(\sum_{a,b \geq 0} \frac{\hslash}{2}\,u_{a,b}\,\partial_{x_{a}}\partial_{x_{b}}\bigg),\qquad u_{a,b} = \frac{(2a + 2b + 1)!\zeta(2a + 2b + 2)}{(2a + 1)!(2b + 1)!}.
$$

\medskip

\noindent \textsc{Base cases --} When  $2g - 2 + n = 1$ we have
$$
F_{0,3}[d_1,d_2,d_3] = \delta_{d_1,d_2,d_3,0},\qquad F_{1,1}[d] = \delta_{d,0}\,\frac{\zeta(2)}{2} + \delta_{d,1}\,\frac{1}{8}.
$$
We assume $2g + n -2\geq 2$ in what follows.

\medskip

\noindent \textsc{String equation --}
\begin{align*}
F_{g,n}[0,d_2,\ldots,d_n] & =\sum_{i = 2}^n \bigg(F_{g,n - 1}[d_2,\ldots,d_i - 1,\ldots,d_n] + \delta_{d_i,0}\sum_{a \geq 0} \zeta(2a + 2)F_{g,n - 1}(a,d_2,\ldots,\widehat{d_i},\ldots,d_n)\bigg)  \\
& + \frac{1}{2} \sum_{a,b \geq 0} \bigg(\frac{(2a + 2b + 4)!}{(2a + 2)!(2b + 2)!}\,\zeta(2a + 2b + 4) + \zeta(2a + 2)\zeta(2b + 2)\bigg)  \\
&\qquad \quad \,\,\times \bigg(F_{g - 1,n + 1}[a,b,d_2,\ldots,d_n] + \sum_{\substack{h + h' = g \\ J \sqcup J' = \{d_2,\ldots,d_n\}}} F_{h,1+|J|}[a,J]F_{h',1+|J'|}[b,J']\bigg) 
\end{align*}

\noindent \textsc{Dilaton equation --}
\begin{align*}
F_{g,n}[1,d_2,\ldots,d_n] & =\bigg(\sum_{i = 2}^n (2d_i + 1)\bigg) F_{g,n - 1}[d_2,\ldots,d_n] + \frac{1}{2}\sum_{a,b \geq 0} \frac{(2a + 2b + 2)!\zeta(2a + 2b + 2)}{(2a + 1)!(2b + 1)!}  \\
& \qquad \quad\,\,\times \bigg(F_{g - 1,n + 1}[a,b,d_2,\ldots,d_n] + \sum_{\substack{h + h' = g \\ J \sqcup J' = \{d_2,\ldots,d_n\}}} F_{h,1 + |J|}[a,J]F_{h',1+|J'|}[b,J']\bigg) 
\end{align*}

\medskip

\noindent \textsc{For $d_1 \geq 2$}
\begin{align*}
 F_{g,n}[d_1,\ldots,d_n] & =\sum_{i = 2}^n (2d_i + 1)F_{g,n - 1}[d_1 + d_i - 1,d_2,\ldots,\widehat{d_i},\ldots,d_n]  \\
& + \sum_{a,b \geq 0} \bigg(\frac{1}{2}\delta_{a + b,d_1 - 2} + \delta_{a \geq d_1 - 1}\,\frac{(2a + 2b + 3 - 2d_1)!\zeta(2a + 2b + 4 - 2d_1)}{(2b + 1)!(2a + 2 - 2d_1)!}\bigg) \\
& \qquad \quad\,\,\times \bigg(F_{g - 1,n + 1}[a,b,d_2,\ldots,d_n] + \sum_{\substack{h + h ' = g \\ J \sqcup J' = \{d_2,\ldots,d_n\}}} F_{h,1+|J|}[a,J]F_{h',1+|J'|}[b,J']\bigg)  
\end{align*}

In genus zero, the string equation (\textit{i.e.} the first member of the Virasoro constraints) gives a recursion which uniquely determines all $F_{0,n}[d_1,\ldots,d_n]$. Indeed, this number could be non-zero only when $d_1 + \cdots + d_n \leq n - 3$, which implies that at least $3$ of the $d_i$'s are zero. By symmetry we can take one of these zeroes to be $d_1$, and apply the string equation.

%***************************************
\subsection{Recursion for genus zero, one row}
\label{Genus01row}
%***************************************

\renewcommand{\figurename}{Table}

If we specialise the Virasoro constraints to $g = 0$ and $d_2 = \cdots = d_n = 0$, we obtain a recursion for the $H_n[d] = F_{0,n}[d,0,\ldots,0]$.
\begin{cor}
\label{Gn0rec} We have that
\begin{align*}
	H_{n}[0] & = \delta_{n,3} + (n - 1)\sum_{a \geq 0} \zeta(2a + 2)H_{n - 1}[a]  \\
	& + \frac{1}{2} \sum_{\substack{2 \leq j \leq n - 3 \\ a,b \geq 0}} \frac{(n - 1)!}{j!(n - 1 - j)!}\bigg(\frac{(2a + 2b + 4)!\zeta(2a + 2b + 4)}{(2a + 2)!(2b + 2)!} + \zeta(2a + 2)\zeta(2b + 2)\bigg)H_{1 + j}[a]H_{n - j}[b],  \\
	H_{n}[1] & = (n - 1)H_{n - 1}[0] + \frac{1}{2} \sum_{\substack{2 \leq j \leq n - 3 \\ a,b \geq 0}} \frac{(n - 1)!}{j!(n - 1 - j)!} \frac{(2a + 2b + 2)!\zeta(2a + 2b + 2)}{(2a + 1)!(2b + 1)!}\,H_{1 + j}[a] H_{n - j}[b] 
\end{align*}
and for $d \geq 2$
\begin{align*}
	H_{n}[d] & = (n - 1)H_{n - 1}[d - 1] + \sum_{\substack{2 \leq j \leq n - 3 \\ a,b \geq 0}} \frac{(n - 1)!}{j!(n - 1 - j)!}  \\
	& \times \bigg(\frac{1}{2}\delta_{a + b,d - 2} + \delta_{a\geq d - 1} \frac{(2a + 2b + 3 - 2d)!\zeta(2a + 2b + 4 - 2d)}{(2b + 1)!(2a + 2 - 2d)!}\bigg)H_{1 + j}[a]H_{n - j}[b]. 
\end{align*}
\hfill $\blacksquare$
\end{cor}
The last equation could also be written so as to give symmetric roles to $a$ and $b$ in the last term, and it is then easy to see that it is also valid for $d = 1$. This recursion determines uniquely the $H_n[d]$, and a fortiori the genus zero Masur--Veech volumes
$$
MV_{0,n} = \frac{2^{n - 2}(n - 4)!}{(2n - 7)!}\,H_{n}[0].
$$
We have not been able -- even using generating series -- to solve this recursion. It can however be used to generate efficiently the numbers $H_n[d]$.

\medskip

From intersection theory on the moduli space of quadratic differentials, a closed formula is known for area Siegel--Veech constants in genus zero \cite{EKZ} and then the Masur--Veech volumes in genus zero \cite{AEZ}.
\begin{thm} 
\label{MV0n} We have that
$$
MV_{0,n} = \pi^{2(n - 3)}\,2^{5 - n}, \qquad SV_{0,n} = \frac{n + 5}{6\pi^2}.
$$
\hfill $\blacksquare$
\end{thm}
In fact, using Goujard's formula\footnote{The reader looking at Theorem~\ref{thGouj1} may think that to derive a formula for $MV_{0,n}$ from the knowledge of $SV_{0,n}$ one also needs the data of $MV_{0,3}$. Actually, in the literature $MV_{0,3}$ is ill-defined, while for us, in the context of statistics of length of multicurves, it makes perfect sense and is equal to $4$. In the formulation of her result \cite{Goujard}, Goujard wrote separately the terms that we included as contributions of $(0,3)$ in Theorem~\ref{thGouj1}. Therefore the extra value of $MV_{0,3} = 4$ can be seen as a convention and the two formulas for $SV_{0,n}$ and $MV_{0,n}$ are indeed equivalent.} (Theorem~\ref{thGouj1}), it is easy to see that the formula for $SV_{0,n}$ and the formula for $MV_{0,n}$ are equivalent. If one could guess a closed formula for the $H_n[d]$, it should be possible to check that it satisfies the recursion of Corollary~\ref{Gn0rec}, and by uniqueness deduce a new proof of Theorem~\ref{MV0n}.

\medskip

Based on numerical data, we can guess the shape of a formula for fixed $d$ but all $n$.

\begin{conj}\label{Gndconj}
   For each $d \geq 0$, there exists a polynomial $P_d$ of degree $d$ with rational coefficients such that
   \begin{equation}\label{conjGnd}
      H_n[d] = \frac{(2d + 1)}{(2d - 1)!!}\,\frac{P_d(n)}{2^{(n - 3 - d)}}\,\frac{(2(n - 3 - d))!}{(n - 3 - d)!}\,\pi^{2(n - 3 - d)}.
   \end{equation}
   The formula for $d = 0$ uses the convention $(-1)!! = 1$. Equivalently, there exists $\tilde{P}_{d} \in \mathbb{Q}[x]$ such that
   $$
      \mathscr{H}(x;d) = \sum_{n \geq d + 3} \frac{H_n[d]}{\pi^{2(n - 3 - d)}}\,\frac{x^{n - 1}}{(n - 1)!} = \bigg[\tilde{P}_{d}(x)(1 - x)^{3/2}\bigg]_{\geq d + 2}
   $$
   where $[\,\cdot\,]_{\geq d + 2}$ means that we only keep monomials of degree greater than $d + 2$.
\end{conj}

The formula is true for $d = 0$ with $P_0(n) = 1$ according to Theorem~\ref{MV0n}. For low values of $d$, we can find polynomials $P_d(n)$ interpolating the values $H_{d + 3}[d],\ldots,H_{2d + 3}[d]$ (see Table~\ref{Tabl0}). Formula~\eqref{conjGnd} then gives the correct values $H_{n}[d]$ for the $n >  2d + 3$ that appear in Table~\ref{Tablall}.

\begin{figure}[h!]
\begin{center}
\begin{tabular}{|c|l|}
\hline
$d$ {\rule{0pt}{3.2ex}}{\rule[-1.8ex]{0pt}{0pt}}& $P_d(n)$  \\
\hline\hline
$0$ {\rule{0pt}{2.5ex}}{\rule[-1.8ex]{0pt}{0pt}}& $1$ \\
\hline
$1$ {\rule{0pt}{2.5ex}}{\rule[-1.8ex]{0pt}{0pt}}& $n - 3$\\
\hline
$2$ {\rule{0pt}{2.5ex}}{\rule[-1.8ex]{0pt}{0pt}}& $5n^2 - 34n + 52$ \\
\hline
$3$ {\rule{0pt}{2.5ex}}{\rule[-1.8ex]{0pt}{0pt}}& $\frac{3}{2}(32n^3 - 367n^2 + 1307n - 1392)$ \\
\hline
$4$ {\rule{0pt}{2.5ex}}{\rule[-1.8ex]{0pt}{0pt}}& $\frac{1}{6}(4138n^4 - 70496n^3 + 419969n^2 - 1002721n + 751506)$ \\
\hline
$5$ {\rule{0pt}{2.5ex}}{\rule[-1.8ex]{0pt}{0pt}}& $\frac{15}{2}(1766n^5 - 41536n^4 + 365383n^3 - 1459754n^2 + 2493951n - 1221210)$ \\
\hline
$6$ {\rule{0pt}{2.5ex}}{\rule[-1.8ex]{0pt}{0pt}}& $\frac{1}{20}(6377776n^6 - 197270496n^5 + 2385358645n^4 - 14079371820n^3 + 40768140229n^2$ \\ {\rule{0pt}{2.5ex}}{\rule[-1.8ex]{0pt}{0pt}}& \quad $- 48501218874n + 9190581840)$ \\
\hline
$7$ {\rule{0pt}{2.5ex}}{\rule[-1.8ex]{0pt}{0pt}}& $\frac{7}{40}(52783968n^7-2073237920n^6+32861488488n^5-266767548125n^4+1152274787382n^3$ \\ {\rule{0pt}{2.5ex}}{\rule[-1.8ex]{0pt}{0pt}}& \quad $-2422330473875n^2+1627352271762n+713960984880)$ \\
\hline
$8$ {\rule{0pt}{2.5ex}}{\rule[-1.8ex]{0pt}{0pt}}& $\frac{5}{56}(3504015400n^8-170178415232n^7+3416784683368n^6-36378043869776n^5+217683482202865n^4$ \\ {\rule{0pt}{2.5ex}}{\rule[-1.8ex]{0pt}{0pt}}& \quad $-701967732545618n^3+976060154881647n^2+86564417888466n-937368548035920)$ \\ 
\hline
\end{tabular}
\caption{\label{Tabl0} Polynomials appearing in Conjecture~\ref{Gndconj} for $H_n[d]$.}
\end{center}
\end{figure}

%***************************************
\subsection{Conjectures for Masur--Veech volumes with fixed \texorpdfstring{$g$}{g}}
%***************************************
\label{conjMVsecmnc}
\renewcommand{\figurename}{Table}
%\todo{The discussion below about the guessing is very long, and probably not that interesting for the reader. At least, less than the guesses themselves.}

For fixed $g$, the number of a priori non-zero coefficients $F_{g,n}[d_1,\ldots,d_n]$ grows faster than any polynomial in $n$, and the Virasoro constraint determines them by induction on $2g - 2 + n$. If one is interested primarily in obtaining $F_{g,n}[0,\ldots,0]$, there is a more efficient way to use the Virasoro constraints.

\medskip

In genus zero, we already saw that it implies recursion for the values of $F_{0,n}$ on partitions with one row. More generally, if $k > 0$ and we specialise $d_{k + 1} = \cdots = d_{n} = 0$, we also get a recursion expressing the values of $F_{0,n}$ on partitions with at most $k$ rows, in terms of the values of $F_{0,n'}$ for $n' < n$ on partitions with at most $k$ rows. The same specialisation in genus $g > 0$ expresses the values of $F_{g,n}$ on partitions with at most $k$ rows in terms of the values of $F_{g',n'}$ on partitions with at most $k$ rows for $2g' - 2 + n' < 2g - 2 + n$, and the values of $F_{g - 1,n + 1}$ on partitions with at most $k + 1$ rows. In this way, reaching $F_{g,n}[0,\ldots,0]$ only requires the computation of a certain number of values of the $F_{g,n}$'s which grows polynomially with $n$.

\medskip

Based on numerical data, we could guess general formulas for $MV_{g,n}$ for low values of $g$ but all $n$. We start by defining the generating series
\begin{equation}
\label{Gseri}\mathscr{H}_{g}(x) = \sum_{n \geq 1} \frac{H_{g,n}[0]}{\pi^{6g - 6 + 2n}}\,\frac{x^{n}}{n!} + \delta_{g,0}\,\mathscr{A}(x)
\end{equation}
where we allow for a conventional choice of a quadratic polynomial $\mathscr{A}(x)$. Theorem~\ref{MV0n} implies that we can take
$$
\mathscr{H}_0(x) = -\frac{8}{15}(1 - x)^{5/2},\qquad \mathscr{A}(x) = \frac{8}{15} - \frac{4}{3}x + x^2 
$$
where the role of $\mathscr{A}(x)$ is to cancel the coefficients of $x^0,x^1,x^2$ in the expansion of $\mathscr{H}_{0}(x)$, since they do not correspond to Masur--Veech volumes. The \textsc{maple} command \textsf{guessgf} recognises that the values of $H_{1,n}[0]$ that we have computed for $n = 1,\ldots,20$ match with the expansion of 
$$
\mathscr{H}_{1}(x) = -\frac{\ln\sqrt{1 - x}}{12} - \frac{\sqrt{1 - x}}{12} + \frac{1}{12}.
$$
It suggests that, for $g \geq 2$, $\mathscr{H}_{g}(x)$ could be a polynomial of degree $5(g - 1)$ in the variable $(1 - x)^{-1/2}$ with rational coefficients. Although the command \textsf{guessgf} fails for $g \geq 2$, we are on good tracks. If we attempt to match this ansatz for $g = 2$ and $3$ with the data of Table~\ref{TablGnd}, we discover that this polynomial has valuation $4(g - 1)$. This leads us to guess that the generating series we look for may have the form
\begin{equation}
\label{GGx}\mathscr{H}_{g}(x) = -\frac{\ln y}{12}\,\delta_{g,1} + y^{5(1 - g)}\,Q_{g}(y)\qquad {\rm with}\qquad y = \sqrt{1 - x}
\end{equation}
where $Q_g$ is a polynomial of degree $g$ with rational coefficients. We then determine the polynomials $Q_g$ such that \eqref{GGx} reproduces correctly the values $H_{g,n}[0]$ for $n \leq g + 1$, and checked that they predict the correct values $H_{g,n}[0]$ for higher $n$ that we computed in Table~\ref{Tablall} with the recursion of Section~\ref{Virrec}.
\begin{figure}
\begin{center}
\begin{tabular}{|c|l|}
\hline {\rule{0pt}{3.2ex}}{\rule[-1.8ex]{0pt}{0pt}}
$g$ {\rule{0pt}{2.5ex}}{\rule[-1.8ex]{0pt}{0pt}}& $Q_g(y)$ \\
\hline\hline
$0$ {\rule{0pt}{2.5ex}}{\rule[-1.8ex]{0pt}{0pt}}& $-\frac{8}{15}$ \\
\hline
$1$ {\rule{0pt}{2.5ex}}{\rule[-1.8ex]{0pt}{0pt}}& $-\frac{1}{12}y + \frac{1}{12} $ \\
\hline
$2$ {\rule{0pt}{2.5ex}}{\rule[-1.8ex]{0pt}{0pt}}& $\frac{7}{11520}y^2 + \frac{5}{2304}y + \frac{7}{2880}$ \\
\hline
$3$ {\rule{0pt}{2.5ex}}{\rule[-1.8ex]{0pt}{0pt}}&  $\frac{31}{516096}y^3 + \frac{17}{36864}y^2 + \frac{223}{165888}y + \frac{245}{165888}$ \\
\hline
$4$ {\rule{0pt}{2.5ex}}{\rule[-1.8ex]{0pt}{0pt}}&  $\frac{127}{5242880}y^4 + \frac{2521}{8847360}y^3 + \frac{24551}{17694720}y^2 + \frac{8785}{2654208}y + \frac{259553}{79626240}$ \\
\hline
$5$ {\rule{0pt}{2.5ex}}{\rule[-1.8ex]{0pt}{0pt}}& $\frac{6643}{301989888}y^5 + \frac{10949}{31457280}y^4 + \frac{1352317}{566231040}y^3 + \frac{1132327}{127401984}y^2 + \frac{9147257}{509607936}y + \frac{1337455}{84934656}$ \\
\hline
$6$ {\rule{0pt}{2.5ex}}{\rule[-1.8ex]{0pt}{0pt}}& $\frac{24046109}{676457349120}y^6+ \frac{1332533}{1887436800}y^5 +\frac{9522007931}{1522029035520}y^4   + \frac{26920481}{849346560}y^3$ \\
{\rule{0pt}{2.5ex}}{\rule[-1.8ex]{0pt}{0pt}}& $+ \frac{15810556787}{163074539520}y^2 + \frac{2079231455}{12230590464}y + \frac{245229441961}{1834588569600}$ \\
\hline
\end{tabular}
\end{center}
\caption{\label{Tabl2} Conjectural generating series for Masur--Veech polynomials.}
\end{figure}

\medskip

Empirically, we recognise the top coefficients of these polynomials
\begin{equation}
\label{topcoeffQg}{\rm coeff}\,\,{\rm of}\,\,y^{g}\,\,{\rm in}\,\,Q_g(y) = 2^{3 - 2g}\,(4g - 7)!!\,b_g
\end{equation}
where $b_g = (2^{1 - 2g} - 1)(-1)^{g}\frac{B_{2g}}{2g!}$ are the coefficients of the expansion 
$$
\frac{z/2}{\sin(z/2)} = \sum_{g \geq 0} b_g\,z^{2g}.
$$
Equation~\eqref{topcoeffQg} is also valid for $g = 0$ and $g = 1$, if we use the values $(-7)!! = -\tfrac{7}{15}$ and $(-3)!! = -1$ given by the analytic continuation of the double factorial via the Gamma function, and if $g \geq 2$ we discard the coefficients of $x^0$, $x^1$ and $x^2$.

\medskip

Returning to the coefficients of the generating series and then to the Masur--Veech volumes \eqref{MVnorm}, Equation~\eqref{GGx} is equivalent to the following structure for the Masur--Veech volumes. Let us first define
$$
\gamma_{k} = \frac{1}{4^{k}} \binom{2k}{k}.
$$

\begin{conj}
\label{conjMV} For any $g \geq 0$, there exist polynomials $p_g,q_g \in \mathbb{Q}[n]$ of degrees
$$ 
\deg p_g = \left\{\begin{array}{lll} \lfloor (g - 1)/2 \rfloor & & {\rm if}\,\,g > 0 \\ -\infty & & {\rm if}\,\,g = 0 \end{array}\right. \qquad {\rm and}\qquad \deg q_g = \lfloor g/2 \rfloor
$$  
such that, for any $n \geq 0$,
\begin{equation}
\label{MVgnconj} \frac{MV_{g,n}}{\pi^{6g - 6 + 2n}} = 2^{n}\,\frac{(2g - 3 + n)!(4g - 4 + n)!}{(6g - 7 + 2n)!}\big(p_{g}(n) + \gamma_{2g - 3 + n}\,q_g(n)\big).
\end{equation}
\end{conj}

For $g = 0$, formula \eqref{MVgnconj} agrees with Theorem~\ref{MV0n} if we choose $p_0(n) = 0$ and $q_0(n) = \tfrac{1}{4}$. Up to genus $5$, the conjecture is numerically true in the range of Table~\ref{Tablall} for the following choice of polynomials (which can be deduced from Table~\ref{Tabl2}).
\begin{figure}[h!]
\begin{center}
\begin{tabular}{|c|l|l|}
\hline
$g$ {\rule{0pt}{3.2ex}}{\rule[-1.8ex]{0pt}{0pt}}& $p_g(n)$ & $q_g(n)$ \\
\hline\hline
$0$ {\rule{0pt}{2.5ex}}{\rule[-1.8ex]{0pt}{0pt}}& $0$ & $\tfrac{1}{4}$ \\
\hline
$1$ {\rule{0pt}{2.5ex}}{\rule[-1.8ex]{0pt}{0pt}}& $\frac{1}{6}$ & $\frac{1}{6}$ \\
\hline
$2$ {\rule{0pt}{2.5ex}}{\rule[-1.8ex]{0pt}{0pt}}& $\frac{5}{36}$ & $\frac{28}{135}n + \frac{7}{18}$  \\
\hline
$3$ {\rule{0pt}{2.5ex}}{\rule[-1.8ex]{0pt}{0pt}}& $\frac{245}{3888}n + \frac{643}{1944}$ & $\frac{1784}{8505}n + \frac{6523}{8505}$ \\
\hline
$4$ {\rule{0pt}{2.5ex}}{\rule[-1.8ex]{0pt}{0pt}}& $\frac{1757}{23328}n + \frac{95413}{194400}$ & $\frac{1186528}{23455575}n^2 + \frac{40882696}{54729675}n + \frac{5951381}{2296350}$  \\
\hline
$5$ {\rule{0pt}{2.5ex}}{\rule[-1.8ex]{0pt}{0pt}}& $\frac{38213}{3359232}n^2 + \frac{4218671}{16796160}n + \frac{63657059}{48988800}$ & $\frac{83632064}{1196234325}n^2 + \frac{50144427856}{41868201375}n + \frac{63849553}{12629925}$ \\
\hline
$6$ {\rule{0pt}{2.5ex}}{\rule[-1.8ex]{0pt}{0pt}}& $\frac{59406613}{3325639680}n^2 + \frac{11411443987}{27713664000}n + \frac{61888029881}{26453952000}$ & $\frac{2562397434368}{352859220016875}n^3 + \frac{185272285982144}{640374140030625}n^2 + \frac{9008283258227896}{2470014540118125}n + \frac{1636294928657}{110827591875}$ \\
\hline
\end{tabular}
\caption{\label{Tabl3} Polynomials conjecturally appearing in the Masur--Veech volumes.}
\end{center}
\end{figure}

%***************************************
\subsection{Conjectures for area Siegel--Veech with fixed \texorpdfstring{$g$}{g}}
%***************************************
\label{ConjSVsec}
Area Siegel--Veech constants $SV_{g,n}$ can be computed from Masur--Veech volumes thanks to Goujard's formula, see Section~\ref{SGouj}. The correspondence between the notations of Section~\ref{SGouj} and the present one is $\mathscr{F}_{g}(x) = \mathscr{H}_{g}(x) - \delta_{g,0}\mathscr{A}(x)$. If we insert the conjectural formulas for the Masur--Veech volumes, we can obtain conjectural formulas for the area Siegel--Veech constants.

\begin{cor}
\label{conjSV} Assuming Conjecture~\ref{conjMV}, for any $g \geq 0$, there exist polynomials $p_g^{*},q_g^{*} \in \mathbb{Q}[n]$ with degrees
$$
\deg p_g^{*} = \left\{\begin{array}{lll} \lfloor (g + 3)/2 \rfloor & & {\rm if}\,\,g > 0 \\ -\infty & & {\rm if}\,\,g = 0\end{array}\right. \qquad {\rm and}\qquad \deg q_g^{*} = 1 + \lfloor g/2 \rfloor
$$
such that, for any $n \geq 0$ with $2g - 2 + n \geq 2$
$$
\frac{SV_{g,n}\cdot MV_{g,n}}{\pi^{6g - 8 + 2n}} = 2^{n}\,\frac{(2g - 3 + n)!(4g - 4 + n)!}{(6g - 7 + 2n)!}\bigg(\frac{p_g^{*}(n)}{2g - 3 + n} + \gamma_{2g - 3 + n}\,q_g^{*}(n)\bigg),
$$
or equivalently
$$
SV_{g,n} = \frac{1}{\pi^2}\,\frac{\frac{p_g^{*}(n)}{2g - 3 + n} + \gamma_{2g - 3 + n}\,q_g^{*}(n)}{p_g(n) + \gamma_{2g - 3 + n}\,q_g(n)}.
$$
\end{cor}

The expression of the polynomials is displayed in Table~\ref{Tabl4}: it is deduced, after computation of the sums \eqref{eqGouj1}, from Table~\ref{Tabl3}.  For $g = 0$ the conjecture matches with Theorem~\ref{MV0n} with $p_0^{*} = 0$ and $q_1^* = \frac{n + 5}{24}$.

\begin{figure}[h!]
\begin{center}
\begin{tabular}{|c|l|l|}
\hline
$g$ {\rule{0pt}{3.2ex}}{\rule[-1.8ex]{0pt}{0pt}}& $p_g^{*}(n)$ & $q_g^{*}(n)$ \\
\hline
$0$ {\rule{0pt}{2.5ex}}{\rule[-1.8ex]{0pt}{0pt}}& $0$ & $\frac{n + 5}{24}$ \\
\hline
$1$ {\rule{0pt}{2.5ex}}{\rule[-1.8ex]{0pt}{0pt}}& $\frac{1}{36}n^2 - \frac{1}{36}n$ & $\frac{1}{36}n + 1$ \\
\hline
$2$ {\rule{0pt}{2.5ex}}{\rule[-1.8ex]{0pt}{0pt}}& $\frac{5}{216}n^2 + \frac{20}{27}n + \frac{811}{1080}$ & $\frac{14}{405}n^2 - \frac{35}{324}n + \frac{329}{540}$ \\
\hline
$3$ {\rule{0pt}{2.5ex}}{\rule[-1.8ex]{0pt}{0pt}}& $\frac{245}{23328}n^3 - \frac{143}{7776}n^2 + \frac{355}{1458}n + \frac{11861}{9720}$ & $\frac{892}{25515}n^2 + \frac{52907}{51030}n + \frac{69617}{18255}$ \\
\hline
$4$ {\rule{0pt}{2.5ex}}{\rule[-1.8ex]{0pt}{0pt}}& $\frac{1757}{139968}n^3 + \frac{1428289}{3499200}n^2 + \frac{514241}{129600}n + \frac{4368611}{388800}$ & $\frac{593264}{70366725}n^3 - \frac{322892}{164189025}n^2 + \frac{480686827}{1970268300}n + \frac{14820167}{4592700}$ \\
\hline
\multirow{2}{*}{$5$} {\rule{0pt}{2.5ex}}{\rule[-1.8ex]{0pt}{0pt}}& $\frac{38213}{20155392}n^4 + \frac{867413}{50388480}n^3 + \frac{353997223}{3527193600}n^2$ & \multirow{2}{*}{$\frac{41816032}{3588702975}n^3 + \frac{48489191848}{125604604125}n^2 + \frac{1269997838947}{251209208250}n + \frac{957632944}{44778825}$}  \\
{\rule{0pt}{2.5ex}}{\rule[-1.8ex]{0pt}{0pt}}& $+ \frac{124054303}{55112400}n + \frac{128194553}{10497600}$ &  \\
\hline
\multirow{2}{*}{$6$} {\rule{0pt}{2.5ex}}{\rule[-1.8ex]{0pt}{0pt}}& $\frac{59406613}{19953838080}n^4+ \frac{63937638461}{498845952000}n^3+ \frac{8797861897271}{3491921664000}n^2$ & $\frac{1281198717184}{1058577660050625}n^4+ \frac{931707432208544}{51870305342480625}n^3+ \frac{6702081021375716}{51870305342480625}n^2$ \\
{\rule{0pt}{2.5ex}}{\rule[-1.8ex]{0pt}{0pt}}& $\frac{7511464839971}{317447424000}n+\frac{10221213098113}{123451776000}$ & $+ \frac{302389725584289713}{103740610684961250}n+\frac{1719710639461433}{79130900598750}$ \\
\hline
\end{tabular}
\caption{\label{Tabl4} Polynomials conjecturally appearing in the numerator of $SV_{g,n}$.}
\end{center}
\end{figure}

\begin{proof} We already mentioned that an equivalent form of Conjecture~\ref{conjMV} is
$$
\mathscr{H}_{g}(x) = \sum_{n \geq 0} \frac{x^n}{n!}\,\frac{H_{g,n}[0]}{\pi^{6g - 6 + 2n}} = -\frac{\ln y}{12}\,\delta_{g,1} + y^{5(1 - g)}\,Q_g(y),\qquad y = \sqrt{1 - x}.
$$
We recall from Table~\ref{Tabl2} that $Q_0(y) = -\tfrac{8}{15}$ and $Q_1(y) = \tfrac{1 - y}{12}$. Therefore
\begin{equation}
\label{Sgenu} \partial_{x}\mathscr{H}_{g}(x) = y^{3 - 5g}\,Q_{g;1}(y),\qquad \partial_{x}^2 \mathscr{H}_{g}(x) = y^{1 - 5g}\,Q_{g;2}(y),\qquad \partial_{x}\mathscr{A}(x) \partial_{x}\mathscr{H}_{g}(x) = y^{3 - 5g}\,Q_{g + 2;3}(y)
\end{equation}
where $Q_{g;i}$ are polynomials of degree $g$ with rational coefficients:
\begin{align*}
	Q_{g;1}(y) & = \left\{\begin{array}{lll} \frac{16}{15} & & {\rm if}\,\,g = 0 \\[3pt]  \frac{1 + y}{24} & & {\rm if}\,\,g = 1 \\[3pt] \frac{5}{2}(g - 1)Q_g(y) -\frac{1}{2}Q_g'(y) & & {\rm if}\,\,g \geq 2 \end{array}\right. \\[2ex]
	Q_{g;2}(y) & =  \left\{\begin{array}{lll} -\frac{16}{10} & & {\rm if}\,\,g = 0 \\[3pt] \frac{1 + 2y}{48} & & {\rm if}\,\,g = 1 \\[3pt]  \frac{5}{4}(g - 1)(5g - 3)Q_g(y) + \frac{1}{4}(9 - 10g)Q_g'(y) + \frac{1}{4}y^2 Q_g''(y) & & {\rm if}\,\,g \geq 2 \end{array}\right. \\[2ex]
	Q_{g + 2;3}(y) & = 2\big(\tfrac{1}{3} - y^2\big)\,Q_{g;1}(y).
\end{align*}
With $\mathscr{F}_{g}(x) = \mathscr{H}_{g}(x) - \delta_{g,0}\mathscr{A}(x)$, we recall from the proof of Corollary~\ref{cor2Sv} that
\begin{align*}
	\mathscr{S}_{g}(x) & = \sum_{\substack{n \geq 1 \\ 2g + n > 0}} \frac{SV_{g,n}\cdot H_{g,n}[0]}{\pi^{6g - 4 + 2n}}\,\frac{x^{n}}{n!} \\
	& = \frac{1}{4}\bigg( \partial^2_{x} \mathscr{F}_{g - 1}(x) + \frac{1}{2} \sum_{g_1 + g_2 = g} \partial_{x}\mathscr{F}_{g_1}(x)\cdot \partial_{x}\mathscr{F}_{g_2}(x)\bigg).
\end{align*}
From \eqref{Sgenu} we deduce for any $g \geq 0$ the existence of $R_{g + 3} \in \mathbb{Q}_{g}[y]$ such that $\mathscr{S}_{g}(x) = y^{3 - 5g}\,R_{g + 3}(y)$, that is
\begin{equation}
\label{Sgggxxxx}S_{g}(x) = \sum_{k = 0}^{g + 3} \frac{r_{g,k}}{(1 - x)^{2g - 3 + k/2}}
\end{equation}
for some rational numbers $r_{g,k} \in \mathbb{Q}$. For $g \geq 2$, \eqref{Sgggxxxx} contains only negative powers of $y = \sqrt{1 - x}$. From the expansions
\begin{align*}
\frac{1}{(1 - x)^{b + 1}} & = \sum_{n \geq 0} \frac{(b + n)!}{b!}\,\frac{x^n}{n!}, \\
\frac{1}{(1 - x)^{b + 1/2}} & = \sum_{n \geq 0} \frac{b!}{2b!}\,\frac{(2b + 2n)!}{(b + n)!}\,\frac{(x/4)^{n}}{n!}, 
\end{align*}
it easily follows that
$$
SV_{g,n}\,H_{g,n}[0] = (2g - 4 + n)!\,\tilde{p}_{g}^*(n) + \frac{(4g - 6 + 2n)!}{4^{2g - 3 + n}(2g - 3 + n)!}\tilde{q}_{g}^*(n)
$$
for some polynomials $\tilde{p}_{g}^*$ and $\tilde{q}_{g}^*$ with rational coefficients and degrees as announced. Multiplying by the prefactor of Equation~\eqref{MVnorm} yields the claim, with polynomials $p_g^*$ and $q_g^*$ differing from $\tilde{p}_{g}^*$ and $\tilde{q}_{g}^*$ by prefactors that only depend on $g$. The cases $g = 0$ and $g = 1$ can be treated separately, with the same conclusion.
\end{proof}

%***************************************
\subsection{Conjectural asymptotics for fixed \texorpdfstring{$g$}{g} and large \texorpdfstring{$n$}{n}}
%***************************************
\label{Sasymss}
Let us examine the asymptotics when $n \rightarrow \infty$ assuming the conjectural formulas for Masur--Veech volumes and area Siegel--Veech constants. Since $\gamma_{k} \sim (\pi k)^{-1/2}$ when $k \rightarrow \infty$, we obtain when $n \rightarrow \infty$
\begin{equation}
\label{MVasym} MV_{g,n} \sim 2^{-n}\,\pi^{6g - 6 + 2n + \epsilon(g)/2}\,n^{g/2}\,m_g,\qquad \epsilon(g) = \left\{\begin{array}{lll} 0 & & {\rm if}\,\,g\,\,{\rm is}\,\,{\rm even} \\ 1 & & {\rm if}\,\,g\,\,{\rm is}\,\,{\rm odd} \end{array}\right.\,,
\end{equation}
where $2^{6g - 7}m_g \in \mathbb{Q}$ is the top coefficient of $q_g$ if $g$ is even and the top coefficient of $p_g$ if $g$ is odd, see Table~\ref{Tablasym}. We observe that the coefficients which we could recognise in \eqref{topcoeffQg} are not relevant in this leading asymptotics, as they rather appear proportional to the constant term in $p_g$ or $q_g$. For the area Siegel--Veech constants, we find when $n \rightarrow \infty$
\begin{equation}
\label{SVasym} SV_{g,n} = \frac{n + 5 - 5g}{6\pi^2} + \frac{s_{g}}{\pi^{3/2 + \epsilon(g)}n^{1/2}} + O(n^{-1}),
\end{equation}
where $s_{g} \in \mathbb{Q}$ are given in Table~\ref{Tablasym} for $g \leq 5$.

\medskip 

By~\cite[Theorem 2]{EKZ} we have that
\begin{equation}
\label{EKZform}
\frac{\pi^2}{3} SV_{g,n} = \frac{n + 5 - 5g}{18} + \Lambda^+_{g,n},
\end{equation}
where $\Lambda^+_{g,n}$ are the sum of the $g$ Lyapunov exponents of the
Hodge bundle along the Teichm\"uller flow on the moduli space of
area one quadratic differentials $Q^1\mathfrak{M}_{g,n}$. In
particular $\Lambda^+_{g,n} \in [0,g]$ and we can observe the coincidence
of the main term in~\eqref{SVasym} and~\eqref{EKZform}.
Based on extensive numerical
experiments, Fougeron~\cite{Fougeron}
conjectured that for each $g$ we have $\Lambda^+_{g,n} = O(n^{-1/2})$ as $n \to \infty$.
The conjectural asymptotics~\eqref{SVasym} provides a refined version of
Fougeron's conjecture.

\begin{figure}
\begin{center}
\begin{tabular}{|c|cc|} \hline $g$  {\rule{0pt}{3.2ex}}{\rule[-1.8ex]{0pt}{0pt}}& $m_g$ & $s_g$ \\
\hline\hline
$0$  {\rule{0pt}{2.5ex}}{\rule[-1.8ex]{0pt}{0pt}}& $32$ & $0$ \\
$1$  {\rule{0pt}{2.5ex}}{\rule[-1.8ex]{0pt}{0pt}}& $\frac{1}{3}$ &$6$ \\
$2$  {\rule{0pt}{2.5ex}}{\rule[-1.8ex]{0pt}{0pt}}& $\frac{7}{1080}$ & $\frac{225}{56}$ \\
$3$  {\rule{0pt}{2.5ex}}{\rule[-1.8ex]{0pt}{0pt}}& $\frac{245}{7962624}$ & $\frac{171264}{8575}$ \\
$4$  {\rule{0pt}{2.5ex}}{\rule[-1.8ex]{0pt}{0pt}}& $\frac{37079}{96074035200}$ & $\frac{24227775}{2712064}$ \\
$5$  {\rule{0pt}{2.5ex}}{\rule[-1.8ex]{0pt}{0pt}}& $\frac{38213}{28179280429056}$ & $\frac{85639233536}{2322395075}$ \\
$6$  {\rule{0pt}{2.5ex}}{\rule[-1.8ex]{0pt}{0pt}}& $\frac{5004682489}{369999709488414720000}$ & $\frac{19363429564990875}{1311947486396416}$ \\
\hline
\end{tabular}
\caption{\label{Tablasym} Constants in the conjectural asymptotics of $MV_{g,n}$ and $SV_{g,n}$.}
\end{center}
\end{figure}

\medskip

We notice that the power of $\pi$ appearing in the asymptotics depends on the parity of $g$.
Both for $MV_{g,n}$ and $SV_{g,n}$, we have an all-order asymptotic expansion in powers of $n^{-1/2}$ beyond the leading terms \eqref{MVasym}-\eqref{SVasym}.

%***************************************
\subsection{Conjectures \texorpdfstring{for $H_{1,n}[d]$}{in genus one}}
%***************************************

We can generate the numbers $H_{1,n}[d]$ in the following way (see Tables~\ref{TablG1nd}-\ref{TablG1nd2}).
\begin{itemize}
\item[(i)] We record the $H_{n}[d] = F_{0,n}[d,0,\ldots,0]$ computed in Section~\ref{Genus01row}.
\item[(ii)] The specialisation of the Virasoro constraints to genus zero and $d_3 = \cdots = d_n = 0$ gives a recursion (on the variable $n$) for $F_{0,n}[d_1,d_2,0,\ldots,0]$ using (i) as input.
\item[(iii)] The specialisation of the Virasoro constraints to genus one and $d_2 = \cdots = d_n = 0$ gives a recursion (on the variable $n$) for $H_{1,n}[d] = F_{1,n}[d,0,\ldots,0]$ using (i) and (ii) as input.
\end{itemize}
Notice that obtaining $H_{1,n}[d]$ requires from (ii) the knowledge of $F_{0,m}[d_1,d_2,0,\ldots,0]$ for arbitrary $d_1,d_2$ (they can be non-zero only for $d_1 + d_2 \leq m - 3$) and $m \leq n + 1$, and from (i) the knowledge of $H_{n'}[d]$ for arbitrary $d \leq m - 3$ and $n' \leq n$.

\medskip

The data we have generated leads us to propose the ansatz, for $n \geq d + 1$
\begin{equation}\label{G1ndansatz}
  H_{1,n}[d] = 2^{d - 2} \cdot (n - 1 - d)!\big(\rho_{d}(n - 1)\cdots(n - d) + \gamma_{n - 1 - d}\,r_{d}(n)\big)
\end{equation}
for some rational constant $\rho_{d}$ and some polynomial $r_{d}(n)$ of degree $d$ with rational coefficients. This formula makes sense even though the arguments of the factorials can be negative. Indeed, the first term yields $(n - 1)!\rho_{d}$, while for the second term, if $k$ is a negative integer, we use
$$
\lim_{m \rightarrow k} m!\,\gamma_{m} = \lim_{m \rightarrow k} \frac{\Gamma(2m + 1)}{\Gamma(m + 1)} = \frac{1}{2}\,\frac{(k + 1)!}{(2k + 1)!}.
$$
We wrote \eqref{G1ndansatz} in this form to stress the analogy with \eqref{MVgnconj}. Equation \eqref{G1ndansatz} for $d = 0$ indeed matches \eqref{MVgnconj} with the values $r_{0} = q_{1} = \tfrac{1}{6}$ and $\rho_0 = p_1 = \tfrac{1}{6}$ already found in Table~\ref{Tabl3}.

\medskip

For a fixed value of $d \in \{0,1,2,3,4,5\}$, we have determined $R_{d}$ and $\rho_{d}$ (Table~\ref{Tablunu}) by matching the values of $H_{1,d + 1}[d]$, $H_{1,d + 2}[d]$, \ldots, $H_{2d + 2}[d]$ and we checked that \eqref{G1ndansatz} predicts the correct values for $2d + 2 \leq n \leq 14$. We observed that formula \eqref{G1ndansatz} does not give the correct value for $n = d$. However, for this particular case, we prove in Section~\ref{Sgenus1} that $H_{1,n}[n] = \tfrac{1}{24}\,\frac{(2n + 1)!}{2^{n}n!}$ using stable graphs.

\begin{figure}
\begin{center}
\begin{tabular}{|c|c|l|}
\hline
$d$ {\rule{0pt}{3.2ex}}{\rule[-1.8ex]{0pt}{0pt}}& $\rho_{d}$ & $R_{d}(n)$ \\
\hline\hline
$0$ {\rule{0pt}{2.5ex}}{\rule[-1.8ex]{0pt}{0pt}}& $\frac{1}{6}$ & $\frac{1}{6}$ \\
\hline
$1$ {\rule{0pt}{2.5ex}}{\rule[-1.8ex]{0pt}{0pt}}& $\frac{1}{4}$ & $ \frac{1}{4}(n - 1)$ \\
\hline
$2$ {\rule{0pt}{2.5ex}}{\rule[-1.8ex]{0pt}{0pt}}& $\frac{25}{72}$ & $\frac{25}{72}n^2 - \frac{65}{72}n + \frac{35}{144}$ \\
\hline
$3$ {\rule{0pt}{2.5ex}}{\rule[-1.8ex]{0pt}{0pt}}& $\frac{7}{15}$ & $\frac{7}{15}n^3 - \frac{217}{96}n^2 + \frac{973}{480}n + \frac{413}{480}$ \\
\hline
$4$ {\rule{0pt}{2.5ex}}{\rule[-1.8ex]{0pt}{0pt}}& $\frac{2069}{3360}$ & $\frac{2069}{3360}n^4 - \frac{4009}{840}n^3 + \frac{60479}{6720}n^2 + \frac{1189}{840}n - \frac{12549}{2240}$ \\
\hline
$5$ {\rule{0pt}{2.5ex}}{\rule[-1.8ex]{0pt}{0pt}}& $\frac{9713}{12096}$ & $\frac{9713}{12096}n^5 - \frac{55033}{6048}n^4 + \frac{710501}{24192}n^3 - \frac{90299}{6048}n^2 - \frac{61105}{4032}n - \frac{259919}{4032}$ \\
\hline 
\end{tabular}
\caption{\label{Tablunu} Parameters of the conjectural formula \eqref{G1ndansatz} for $H_{1,n}[d]$ with $d \leq 5$.}
\end{center}
\end{figure}

%***************************************
\section{Topological recursion for counting of square-tiled surfaces}
\label{Sectinin}
%***************************************

In~\cite{Delecroix}, Formula~\eqref{MVgraphsss} in Theorem~\ref{thm:intro:1} for Masur--Veech volumes was derived using the asymptotics of the count of square-tiled surfaces. In this section, we introduce square-tiled surfaces with boundary and their associated generating series. We first show that Masur--Veech polynomials can be seen as certain asymptotics of square-tiled surface counting. Next, we show that the generating series of square-tiled surfaces satisfy a topological recursion, and that the topological recursion for the Masur--Veech polynomials derives from it by taking limits.

%***************************************
\subsection{Reminder on the number of ribbon graphs} \label{SecNorbury}
%***************************************

For $g,n \geq 0$ such that $2g - 2 + n > 0$ and $L_1,\ldots,L_n \in \mathbb{Z}_{+}$, let $P_{g,n}(L_1,\ldots,L_n)$ be the number of integral points in the combinatorial moduli space $\mathcal{M}_{g,n}^{{\rm comb}}(L_1,\ldots,L_n)$. The numbers $P_{g,n}(L_1,\ldots,L_n)$ have been extensively studied~\cite{CM92multi,CM92hint,Che93matrix,Norburylattice,ACNP1,ACNP2} and admits several equivalent definitions: it is the number of ribbon graphs of genus $g$ with $n$ labeled faces of perimeter $L_1,\ldots,L_n$; it is the number of maps without internal faces and $n$ labeled (unrooted) boundaries; it is the coefficient of $N^{2 - 2g - n}$ in the cumulant $\langle {\rm Tr}\,M^{L_1}\cdots {\rm Tr}\,M^{L_n} \rangle_{c}$ where $M$ is drawn from the Gaussian Unitary Ensemble of Hermitian matrices of size $N$. The function $P_{g,n}(L_1, \ldots, L_n)$ is a quasi-polynomial of $L_1,\ldots,L_n$. More precisely, it vanishes if the sum of the $L_i$ is odd and for each fixed even integer $k$ the function $P_{g,n}(L_1, \ldots, L_n)$ restricted to the set of integral $(L_1, \ldots, L_n)$ with exactly $k$ odd terms coincide with a polynomial. It can be obtained by topological recursion in the following way, either directly with respect to the lengths in the style of \eqref{TReqn}, or via generating series in the style of Eynard--Orantin.

\medskip

For a function $f \colon \mathbb{R}_{+}^{k} \rightarrow \mathbb{R}$, we denote
\begin{equation}\label{fzZ}
  f_{\mathbb{Z}}(x_1,\ldots,x_k)
  =
  \begin{cases}
    f(x) & \text{if }x_1,\ldots,x_k \in \mathbb{Z}_{+} \text{ and } x_1 + \cdots + x_k \in 2\mathbb{Z}, \\
    0 & \text{otherwise.}
  \end{cases}
\end{equation}

\begin{thm} \cite{Norburylattice0}
\label{discreteTRNor}
  For $X \in \{A,B,C\}$ we set $X^{{\rm P}} = X^{{\rm K}}_{\mathbb{Z}}$ in terms of the Kontsevich initial data of \eqref{WKini}, and
  \[
    VD^{{\rm P}}(L_1) = \frac{L_1^2 - 4}{48}.
  \]
  The topological recursion formula \eqref{TReqn}, with initial data $(A^{{\rm P}},B^{{\rm P}},C^{{\rm P}},VD^{{\rm P}})$ and integrals replaced by sums over positive integers, computes $V\Omega^{{\rm P}}(L_1,\ldots,L_n) = P_{g,n}(L_1,\ldots,L_n)$. \hfill $\blacksquare$
\end{thm}

In other words, we can use the functions $(A^{{\rm K}},B^{{\rm K}},C^{{\rm K}})$ in the recursion, but replace the integrals over $\mathbb{R}_{+}$ by summations over integers satisfying the parity condition coming from \eqref{fzZ}.

\begin{thm} \cite{Norburylattice}
  \label{Norburythm} Let $\omega_{g,n}^{{\rm P}}$ be the output of Eynard--Orantin topological recursion for the spectral curve
  \begin{equation}
  \label{spectri} \mathcal{C} = \mathbb{P}^1,\qquad x(z) = z + \frac{1}{z},\qquad y(z) = -z,\qquad \omega_{0,2}^{{\rm P}}(z_1,z_2) = \frac{\dd z_1 \otimes \dd z_2}{(z_1 - z_2)^2}.
  \end{equation}
  For $2g - 2 + n > 0$, we have $\omega_{g,n} \in \mathscr{V}^{\otimes n}$ and for any $L_1,\ldots,L_n > 0$
  \begin{equation}
  \label{expNobur} P_{g,n}(L_1,\ldots,L_n) = (-1)^n \Res_{z_1 = \infty} \cdots \Res_{z_n = \infty} \omega_{g,n}^{{\rm P}}(z_1,\ldots,z_n) \prod_{i = 1}^n \frac{z_i^{L_i}}{L_i}.
  \end{equation}
  \hfill $\blacksquare$
\end{thm}

If $P_{g,n}(L_1,\ldots,L_n)$ are interpreted in terms of coefficient of expansion of cumulants in the GUE, this result dates back to \cite{E1MM}. 

\medskip

The asymptotics of $P_{g,n}(L)$ for large boundary lengths can be identified with the Kontsevich polynomial appearing in Theorem~\ref{konth}, up to a normalisation constant. To be precise, if $k \in \{0,\ldots,n\}$ is even, we let $P_{g,n}^{(k)}(L_1,\ldots,L_n)$ be the polynomial function coinciding with $P_{g,n}(L_1,\ldots,L_n)$ when $L_1,\ldots,L_k$ are odd and $L_{k + 1},\ldots,L_n$ are even. 

\begin{thm} \cite{Norburylattice0}
\label{Norasym} For $2g - 2 + n > 0$ and $k \in \{0,\ldots,n\}$ an even integer, we have for $T$ a positive even integer
\[
\frac{P_{g,n}^{(k)}(TL_1,\ldots,TL_n)}{T^{6g - 6 + 2n}} = 2^{-(2g - 3 + n)} V\Omega^{{\rm K}}_{g,n}(L_1,\ldots,L_n) + O(1/T),
\]
where the $O(1/T)$ is a polynomial in the $L_i$s. \hfill $\blacksquare$
\end{thm}

%***************************************
\subsection{Square-tiled surfaces with boundaries and Masur--Veech polynomials}
%***************************************

Let us fix $g$ and $n$ so that $2g-2+n > 0$ and a tuple $(L_1, \ldots, L_n)$ of positive real numbers. We consider the moduli space $\mathcal{Q}_{g,n}(L_1, \ldots, L_n)$ of tuples $(X, p_1, \ldots, p_n, q)$ where
\begin{itemize}
\item $X$ is a compact Riemann surface of genus $g$,
\item $p_1$, \ldots, $p_n$ are distinct points on $X$,
\item $q$ is a meromorphic quadratic differential on $X$, holomorphic on $X \setminus \{p_1, \ldots, p_n\}$ and with double poles at $p_j$ with residues
\[
\frac{1}{2 {\rm i}\pi} \int_{\gamma_j} \sqrt{q} = \pm \frac{L_j}{2{\rm i}\pi},
\]
where $\gamma_j$ is a loop around $p_j$. The residue is only defined up to sign corresponding to the choice of a square root of $q$.
\end{itemize}

The quadratic differential $q$ induces a flat metric on $X \setminus \{p_1,\ldots, p_n\}$ with conical singularities at the zeros of $q$. The geometry of the flat metric in a neighbourhood of each pole is a semi-infinite cylinder with periodic horizontal trajectories (the point $p_i$ itself is at infinite distance from the rest of the surface). For each pole, there is a maximal such semi-infinite cylinder that avoids
the zeros of $q$. We call the \emph{convex-core} of $(X, p_1, \ldots, p_n, q)$ the surface obtained by removing the union of the maximal open semi-infinite cylinders around each pole. The convex-core with the metric induced from $q$ is still a flat metric and has $n$ horizontal boundaries of lengths respectively $L_1$, \ldots, $L_n$ which is the union of saddle connections (\emph{i.e.} straight line segments joining zeros
of $q$) bounding the maximal half-infinite cylinder around respectively $p_1$,  \ldots, $p_n$.  The \emph{core area} of $(X, p_1, \ldots, p_n, q)$ denoted $\operatorname{CoreArea}(X, p_1, \ldots, p_n, q)$ is the
area of the convex core of $(X, p_1, \ldots, p_n, q)$. It is a non-negative real number, which is in particular finite contrarily to the area of $X \setminus \{p_1, \ldots, p_n\}$.

\medskip

The space $\mathcal{Q}_{g,n}(L_1, \ldots, L_n)$ admits a stratification with respect to the degree of the zeros. On each stratum, the relative periods of $q$ with respect to its zeros provide coordinates. When all the $L_i$ are integral and all periods are Gauss integers, \emph{i.e.} in $\mathbb{Z} \oplus {\rm i} \mathbb{Z}$, we say that the surface is \emph{square-tiled}. Indeed, such surface can be obtained by gluing side by side as many squares as the core area (which is integral) and leaving open some of the horizontal sides forming $n$ circles of lengths $L_1$, \ldots, $L_n$. For integral $L_1$, \ldots, $L_n$ we define the generating function of square-tiled surfaces with boundary lengths $L_1$, \ldots, $L_n$ as follows
\begin{equation} \label{stscountdef}
P_{g,n}^{\Box,q}(L_1,\ldots,L_n) := \sum_{S} \frac{1}{|{\rm Aut}\,S|}\ \mathsf{q}^{\operatorname{CoreArea}(S)},
\end{equation}
where the sum is taken over square-tiled surfaces $S$ in $\mathcal{Q}_{g,n}(L_1, \ldots, L_n)$.

\begin{prop} \label{stscountassumstgraph}
Let $g$, $n$ be non-negative integers such that $2g-2+n > 0$ and let $(L_1, \ldots, L_n)$
be a tuple of positive integers. We have
\begin{equation}
\label{eq:stsgenfct}
P_{g,n}^{\Box,\mathsf{q}}(L_1, L_2, \ldots, L_n) =
\sum_{\Gamma \in \mathbf{G}_{g,n}}
\frac{1}{|{\rm Aut}\,\Gamma|}
\sum_{\ell \colon E_\Gamma \rightarrow \mathbb{Z}_+}
\prod_{v \in V_\Gamma}
P_{h(v),k(v)} \big( (\ell_e)_{e \in E(v)}, (L_{\lambda})_{\lambda \in \Lambda(v)}\big)
\prod_{e \in E_\Gamma}
\frac{\ell_e \mathsf{q}^{\ell_e}}{1 - \mathsf{q}^{\ell_e}},
\end{equation}
where $P_{h,k}(\ell_1, \ldots, \ell_k)$ is the Norbury quasi-polynomial
as in Section~\ref{SecNorbury}.
\end{prop}

\begin{rem}
The terms in the right hand side of~\eqref{eq:stsgenfct} are similar to \cite[Equation 1.12]{Delecroix} describing polynomials associated to a stable graph $\Gamma$.
\end{rem}

\begin{rem}
Surfaces with vanishing core area are exactly the Strebel differentials \cite{Strebel}, \textit{i.e.} differentials all of whose relative periods are purely real. Hence, one can already identify the constant coefficient of $P^{\Box,\mathsf{q}}_{g,n}(L_1, \ldots, L_n)$ (seen as a $\mathsf{q}$-series) as the Norbury quasi-polynomials of Section~\ref{SecNorbury}. This constant coefficient is also equal to the term associated to the stable graph with a single vertex of genus $g$ and no edge in Formula~\ref{eq:stsgenfct}.
\end{rem}

\begin{proof}
Each square-tiled surface admits a decomposition into horizontal cylinders and saddle connections between the zeros of the differential $q$. The union of all saddle connections forms a union of ribbon graphs that we call the \emph{singular layer} of the square-tiled surface. To such decomposition, we associate a stable graph $\Gamma$ by the following rule.
\begin{itemize}
\item A vertex in $\Gamma$ corresponds to a connected component of the singular layer, where the
genus and number of half edges are respectively the genus and the number of faces of the associated ribbon graph.
\item An edge of $\Gamma$ between two vertices correspond to a cylinder, whose extremities belong to the components of the singular layer corresponding to the two vertices. Note that each of these extremities is a face of the corresponding ribbon graph.
\end{itemize}
Let us now fix a stable graph $\Gamma$ in $\mathbf{G}_{g,n}$. We claim that the term
\[
 \frac{1}{|{\rm Aut}\,\Gamma|}
\sum_{\ell \colon E_\Gamma \rightarrow \mathbb{Z}_+}
\prod_{v \in V_\Gamma}
P_{h(v),k(v)} \big( (\ell_e)_{e \in E(v)}, (L_{\lambda})_{\lambda \in \Lambda(v)}\big)
\prod_{e \in E_\Gamma}
\frac{\ell_e \mathsf{q}^{\ell_e}}{1 - \mathsf{q}^{\ell_e}}.
\]
appearing in the right-hand side of~\eqref{eq:stsgenfct} is the generating series of square-tiled surfaces, whose associated stable graph is $\Gamma$. Indeed, to reconstruct the singular layer one needs to choose a ribbon graph for each vertex $v$ of $V_\Gamma$ and fix the lengths of each edge. This count corresponds to the term $P_{h(v),k(v)} \big( (\ell_e)_{e \in E(v)}, (L_{\lambda})_{\lambda \in \Lambda(v)}\big)$. Next, one needs to reconstruct the cylinders. The cylinders are glued on faces of ribbon graphs and have a height parameter $H$ (which is a positive integer) and a twist parameter $t$ (an non-negative integer strictly smaller than $\ell_i$). The generating series for this cylinder is just
\[
\sum_{H \geq 1}
\sum_{0 \leq t < \ell_e}
\mathsf{q}^{\ell_e H}
=
\frac{\ell_e \mathsf{q}^{\ell_e}}{1 - \mathsf{q}^{\ell_e}}.
\]
This concludes the proof.
\end{proof}

We now show how to retrieve the Masur--Veech polynomials $V \Omega^{{\rm MV}}_{g,n}(L_1,\ldots,L_n)$ by considering certain limits of square-tiled surface counting that are encoded in the generating series $P_{g,n}^{\Box,\mathsf{q}}(L_1,\ldots,L_n)$.

\begin{prop} \label{MVpolyasstscount}
Let $L_1,\ldots,L_n$ be positive integers with even sum. We have
\begin{equation}
\label{Pgnungunaaaa}\lim_{\substack{T \rightarrow \infty \\ T \in 2\mathbb{Z}_{+}}} \frac{P_{g,n}^{\Box,\mathsf{q} = e^{-1/T}}(TL_1,\ldots,TL_n)}{T^{6g - 6 + 2n}} = 2^{-(2g - 3 + n)}\,V\Omega^{{\rm MV}}_{g,n}(L_1,\ldots,L_n).
\end{equation}
\end{prop}

\begin{rem}
In absence of boundaries, the asymptotics of the number of square-tiled surfaces is related to the Masur--Veech volume in \cite[Theorem 1.6]{Delecroix}. The proof there is slightly different as it considers the $T \rightarrow \infty$ asymptotics of the number of square-tiled surface of core area $\leq T$, when $T \rightarrow \infty$, while here we analyse directly the $\mathsf{q}$-series when $\mathsf{q} \rightarrow 1$. It is possible to adapt the proof of \cite[Theorem 1.6]{Delecroix} in the presence of boundaries, \textit{i.e.} study the asymptotics of the number of square-tiled surfaces of core area $\leq T$ with boundaries of length $TL_1,\ldots,TL_n$. The result is then similar to the right-hand side of \eqref{Pgnungunaaaa} except that $V\Omega^{{\rm MV}}_{g,n}$ is replaced with $I_{g,n}[V\Omega^{{\rm MV}}]$, where $I_{g,n}$ is the linear map multiplying the monomial $L_1^{2d_1}\cdots L_n^{2d_n}$ by the factor $1/(6g - 6 + 2n - \sum_i 2d_i)!$. The latter factor is essentially the volume of a simplex and comes from the core area truncation.
\end{rem}

\begin{rem}
The scaling $T L_i$ of the boundary term $L_i$ is of strange nature. As $\mathsf{q}^{T}$ is of order $1$, suggesting that the typical contribution in $P_{g,n}^{\Box,\mathsf{q}}$ comes from surfaces with core area $O(T)$, but scaling the area with $T$ usually rescaled the boundary by $\sqrt{T}$. So the limit in Proposition~\ref{MVpolyasstscount} somehow reflects a blowup of the contribution coming from the boundaries of the square-tiled surface, that is necessary in order to obtain the Masur--Veech polynomials.
\end{rem}

\begin{proof}
Let $T$ be an even integer and set $\mathsf{q} = e^{-1/T}$. Fix a stable graph $\Gamma$ of type $(g,n)$. We want to compute the large $T$ behavior of
\begin{equation}
\label{tehtirngiun}\begin{split}
& \quad \sum_{\ell \colon E_\Gamma \rightarrow \mathbb{Z}_+} \prod_{v \in V_\Gamma}
P_{h(v),k(v)} \big( (\ell_e)_{e \in E(v)}, (T L_{\lambda})_{\lambda \in \Lambda(v)}\big)
\prod_{e \in E_\Gamma} \frac{\ell_e\,\mathsf{q}^{\ell_e}}{1 - \mathsf{q}^{\ell_e}} \\
& = \sum_{\hat{\ell} \colon E_{\Gamma} \rightarrow T^{-1}\mathbb{Z}_{+}} \prod_{v \in V_{\Gamma}} P_{h(v),k(v)}\big((T\hat{\ell}_e)_{e \in E(v)},(T L_{\lambda})_{\lambda \in \Lambda(v)}\big) \prod_{e \in E_{\Gamma}} \frac{T\, \hat{\ell}_{e}}{e^{\hat{\ell}_e} - 1}.
\end{split}
\end{equation}

Recall that $P_{h,k}(x_1,\ldots,x_k)$ vanishes when $x_1 + \cdots + x_k $ is odd.
Let $\mathbb{L} \subseteq \Z^{E_{\Gamma}}$ be the sublattice defined by the congruences
\[
\forall v \in V_{\Gamma},\qquad \bigg(\sum_{e \in E(v)} \ell_{e}\bigg) \in 2\mathbb{Z}.
\] 
By \cite[Corollary 2.2]{Delecroix}, $\mathbb{L}$ has index $2^{|V_{\Gamma}| - 1}$ in $\mathbb{Z}^{E_{\Gamma}}$. In the first line of \eqref{tehtirngiun} we are summing over $\ell \in \mathbb{L}_{+} = \mathbb{Z}_{+}^{E_{\Gamma}} \cap \mathbb{L}$. For $\hat{\ell} \in T^{-1}\mathbb{L}_{+}$, we use Theorem~\ref{Norasym} to make the substitution at each vertex $v \in V_{\Gamma}$
\[
P_{h(v),k(v)}\big((T \hat{\ell}_{e})_{e \in E(v)},(TL_{\lambda})_{\lambda \in \Lambda(v)}\big)
\longrightarrow
T^{6h(v) - 6 + 2k(v)}\,2^{-(2h(v) - 3 + k(v))} \, V\Omega^{{\rm K}}_{h(v),k(v)}\big((\ell_{e})_{e \in E(v)},(L_{\lambda})_{\lambda \in \Lambda(v)}\big),
\]
up to an error that will produce subleading terms when $T \rightarrow \infty$. We are left with analysing for large $T$
\begin{equation}
\label{beforeasym}\sum_{\hat{\ell} \in T^{-1}\mathbb{L}_{+}} \prod_{v \in V_{\Gamma}} T^{6h(v) - 6 + 2k(v)}\,2^{-(2h(v) - 3 + k(v))}\,V\Omega^{{\rm K}}_{h(v),k(v)}\big((\hat{\ell}_{e})_{e \in E(v)},(L_{\lambda})_{\lambda \in \Lambda(v)}\big) \prod_{e \in E_{\Gamma}} \frac{T\,\hat{\ell}_{e}}{e^{\hat{\ell}_{e}} - 1}.
\end{equation}
Now, since $V\Omega^{{\rm K}}_{h,k}(x_1,\ldots,x_k)$ are polynomial in $x_1,\ldots,x_k$, the function of $\hat{\ell}$ appearing in the summands is a continuous function of $\hat{\ell} \in \mathbb{R}_{+}^{E_{\Gamma}}$, which is Riemann-integrable due to the exponential decay in the edge weights. Taking into account the fact that $\mathbb{L}$ has index $2^{|V_{\Gamma}| - 1}$, \eqref{beforeasym} is therefore asymptotically equivalent to
\[
2^{1 - |V_{\Gamma}|}\,\int_{\mathbb{R}_{+}^{E_{\Gamma}}} \prod_{e \in E_{\Gamma}} \prod_{v \in V_{\Gamma}} 2^{-(2h(v) - 3 + k(v))}\,T^{(6h(v) - 6 + 2k(v))}\,V\Omega_{h(v),k(v)}^{{\rm K}}\big((\hat{\ell}_{e \in E(v)},(L_{\lambda})_{\lambda \in \Lambda(v)}\big)\,\frac{T^{2}\,\hat{\ell}_{e}\,\dd \hat{\ell}_{e}}{e^{\hat{\ell}_{e}} - 1}
\]
when $T$ is large. The overall powers of $2$ and $T$ can be computed with Lemma~\ref{combstgraph}, and we respectively find
\begin{equation}
\begin{split}
1 - |V_{\Gamma}| - \sum_{v \in V_{\Gamma}} (2h(v) - 3 + k(v)) & = -(2g - 3 + n), \\ 2|E_{\Gamma}| + \sum_{v \in V_{\Gamma}} (6h(v) - 6 + 2k(v)) & = 6g - 6 + 2n,
\end{split}
\end{equation}
which are independent of $\Gamma$. Performing the (finite) sum over all stable graphs of type $(g,n)$ weighted by automorphisms, and dividing the result by $T^{-(6g - 6 + 2n)}$, one finds exactly the sum over stable graphs defining the Masur--Veech polynomials in \eqref{defn:MVpoly}.
\end{proof}  

%***************************************
\subsection{Topological recursion for \texorpdfstring{$P_{g,n}^{\Box,\mathsf{q}}(L_1,\ldots,L_n)$}{square-tiled surfaces polynomials}}
%***************************************

The expression for the $q$-enumeration of square-tiled surfaces in Proposition~\ref{stscountassumstgraph} is another example of the twisting procedure presented in Section~\ref{Stwists}, with the function
\begin{equation}
\label{fqell} f_{\mathsf{q}}(\ell) = \frac{\mathsf{q}^{\ell}}{1 - \mathsf{q}^{\ell}},
\end{equation}
except that we allow only integer lengths. In fact, the discrete analog of Theorem~\ref{stablth} continues to hold, that is, knowing that the weight of the vertices satisfies the topological recursion (Theorem~\ref{discreteTRNor}) automatically implies that the sum over stable graphs in \eqref{eq:stsgenfct} is also computed by the topological recursion with twisted initial data -- see formula \eqref{initwist} for the twisting of $A,B,C$ and formula~\eqref{VDtwist} for the twisting of $VD$. This is summarised in the following corollary.

\begin{cor}
\label{cotrd} The topological recursion formula \eqref{TReqn}, for the initial data of Theorem~\ref{discreteTRNor} twisted by $f_{\mathsf{q}}$ from \eqref{fqell},  computes
\[
V\Omega^{{\rm P}}_{g,n}[f_{\mathsf{q}}](L_1,\ldots,L_n) = P_{g,n}^{\Box,\mathsf{q}}(L_1,\ldots,L_n).
\]
\hfill $\blacksquare$
\end{cor}

This result can also be brought in the form of Eynard--Orantin topological recursion. Twisting is implemented by a shift of $\omega_{0,2}$, but as the lengths are not continuous variables we cannot use~\eqref{02twist}. Instead we will resort to Theorem~\ref{variationomega02}.
 
\begin{prop}
\label{propsq} Let $\omega_{g,n}^{\rm P,\mathsf{q}}$ be the output of Eynard--Orantin topological recursion for the spectral curve differing from \eqref{spectri} only by the choice of 
$$
\omega_{0,2}^{{\rm P},\mathsf{q}}(z_1,z_2) = \frac{1}{2}\,\frac{\dd z_1 \otimes \dd z_2}{(z_1 - z_2)^2} + \frac{1}{2}\bigg(\wp(u_1 - u_2;\mathsf{q}) + \frac{\pi^2 E_2(\mathsf{q})}{3}\bigg)\dd u_1 \otimes \dd u_2,
$$
where $z_j = \exp(2{\rm i}\pi\,u_j)$, $\wp(u;\mathsf{q})$ is the Weierstra{\ss} function for the elliptic curve $\mathbb{C}/(\mathbb{Z} \oplus \tau\mathbb{Z})$ where $\mathsf{q} = e^{2{\rm i}\pi \tau}$, and $E_2(\mathsf{q})$ is the second Eisenstein series
$$
E_2(\mathsf{q}) = 1 -  24 \sum_{\ell > 0} \frac{\ell\,\mathsf{q}^{\ell}}{1 - \mathsf{q}^{\ell}}.
$$ 
Then, we have for $L_1,\ldots,L_n > 0$
$$
P_{g,n}^{\Box,\mathsf{q}}(L_1,\ldots,L_n) = (-1)^n \Res_{z_1 = \infty}  \cdots \Res_{z_n = \infty} \omega_{g,n}^{{\rm P},\mathsf{q}}(z_1,\ldots,z_n) \prod_{i = 1}^n \bigg(1 + \frac{\mathsf{q}^{L_i}}{2(1 - \mathsf{q}^{L_i})}\bigg)^{-1} \frac{z_i^{L_i}}{L_i}.
$$
\end{prop}
\begin{proof} We first make some preliminary computations. Let us introduce the vector space $\mathscr{V}$ of meromorphic $1$-forms $\phi$ on $\mathbb{P}^1$ whose poles are located at $\pm 1$ and such that $\phi (z) + \phi (1/z) = 0$.  Let us consider the linear map $\mathscr{V}^{\otimes 2} \rightarrow \mathbb{C}\bbraket{\mathsf{q}}$ defined by
\begin{equation}
\label{operatorPO} \mathscr{O}[\varpi] = \sum_{\ell > 0} \frac{\ell\,\mathsf{q}^{\ell}}{1 - \mathsf{q}^{\ell}} \Res_{z_1 = \infty} \Res_{z_2 = \infty} \frac{z_1^{\ell}z_2^{\ell}}{\ell^2}\,\varpi(z_1,z_2).
\end{equation}
Since elements of $\mathscr{V}$ are odd under the involution $z \mapsto 1/z$, we can write
$$
\mathscr{O}[\varpi] = -\frac{1}{2} \bigg(
	\Res_{z_1 = \infty} \Res_{z_2 = 0}
		O_{\mathsf{q}}\Big(\frac{z_1}{z_2}\Big) \varpi(z_1,z_2)
		+
	\Res_{z_1 = 0} \Res_{z_2 = \infty}
		O_{\mathsf{q}}\Big(\frac{z_2}{z_1}\Big) \varpi(z_1,z_2)
	\bigg),
$$
where $O_{\mathsf{q}}(z) = \sum_{\ell > 0} \frac{\mathsf{q}^{\ell}\,z^{\ell}}{\ell \, (1 - q^{\ell})} \in \mathbb{C}[z]\bbraket{\mathsf{q}}$. Recall the expansion of the Weierstra{\ss} function when $u \rightarrow 0$
\begin{equation}
\label{Wpexp}\wp(u;\mathsf{q}) = \frac{1}{u^2} + \sum_{k > 0} 2(2k + 1)G_{2k + 2}(\mathsf{q})u^{2k},
\end{equation}
where for $m > 0$, $G_{2m}(\mathsf{q})$ is the $(2m)$-th Eisenstein series
$$
G_{2m}(\mathsf{q}) \coloneqq \zeta(2m) + \frac{(2{\rm i}\pi)^{2m}}{(2m - 1)!} \sum_{\ell > 0} \frac{\ell^{2m - 1}\mathsf{q}^{\ell}}{1 - \mathsf{q}^{\ell}}.
$$
From the identity
$$
\sum_{k \geq 0} \zeta(2k)u^{2k} = -\frac{\pi u}{2}\,{\rm cotan}(\pi u),
$$
we deduce that
$$
\sum_{k \geq 0} 2(2k + 1)\zeta(2k + 2)\,u^{2k} = \frac{\pi^2}{\sin^2\pi u} - \frac{1}{u^2}.
$$
Adding/subtracting the $k = 0$ term in \eqref{Wpexp}, and computing separately the contribution of the Riemann zeta values, yields
\begin{equation}
\begin{split}
\wp(u;\mathsf{q}) & = \frac{\pi^2}{\sin^2\pi u} - 2G_{2}(\mathsf{q}) + \sum_{\ell > 0}  \frac{\mathsf{q}^{\ell}}{1 - \mathsf{q}^{\ell}} \sum_{k \geq 0} \frac{2(2{\rm i}\pi)^{2k + 2}\ell^{2k + 1}}{(2k)!}\,u^{2k} \\
& = \frac{\pi^2}{\sin^2\pi u} - 2G_2(\mathsf{q}) + (2{\rm i}\pi)^2 \sum_{\ell > 0} \frac{\ell\, \mathsf{q}^{\ell}(z^{\ell} + z^{-\ell})}{1 - \mathsf{q}^{\ell}} 
\end{split}
\end{equation} 
where we have set $z = e^{2{\rm i}\pi u}$. Since $E_2(\mathsf{q}) = \frac{6}{\pi^2}G_2(\mathsf{q})$, setting $z_j = e^{2{\rm i}\pi u_j}$ yields
\[
\begin{split}
& \quad \bigg(\wp(u_1-u_2;\mathsf{q}) + \frac{\pi^2}{3}E_2(\mathsf{q})\bigg)\dd u_1 \otimes \dd u_2 \\
& = \bigg(\frac{1}{(z_1 - z_2)^2} + \frac{1}{z_1z_2} \sum_{\ell > 0} \frac{\ell\,\mathsf{q}^{\ell}\big((z_1/z_2)^{\ell} + (z_2/z_1)^{\ell}\big)}{1 - \mathsf{q}^{\ell}}\bigg) \dd z_1 \otimes \dd z_2 \\
& = \bigg(\frac{1}{(z_1 - z_2)^2} - \Res_{z_1' = \infty} \Res_{z_2' = 0} \frac{O_{\mathsf{q}}(z_1'/z_2')\,\dd z_1'\otimes \dd z_2'}{(z_1 - z_1')^2(z_2 - z_2')^2} - \Res_{z_1 = 0} \Res_{z_2' = \infty} \frac{O_{\mathsf{q}}(z_2'/z_1')\,\dd z_1' \otimes \dd z_2'}{(z_1 - z_1')^2(z_2 - z_2')^2}\,\bigg) \dd z_1 \otimes \dd z_2.
\end{split}
\]
Hence
\begin{equation*}
\begin{split}
	& \quad
	\omega_{0,2}^{{\rm P}}(z_1,z_2)
	- \frac{1}{2} \bigg(
		\Res_{z_1 = \infty} \Res_{z_2 = 0}
			O_{\mathsf{q}}\Big(\frac{z_1'}{z_2'}\Big)
			\omega_{0,2}^{{\rm P}}(z_1,z_1') \otimes \omega_{0,2}^{{\rm P}}(z_2,z_2') \\
	&\qquad\qquad\qquad\qquad\quad
		+
		\Res_{z_1 = 0} \Res_{z_2 = \infty}
			O_{\mathsf{q}}\Big(\frac{z_2'}{z_1'}\Big)
			\omega_{0,2}^{{\rm P}}(z_1,z_1') \otimes \omega_{0,2}^{{\rm P}}(z_2,z_2')
		\bigg) \\
& = \frac{1}{2}\,\frac{\dd z_1 \otimes \dd z_2}{(z_1 - z_2)^2} + \frac{1}{2}\bigg(\wp(u_1 -u_2 ;\mathsf{q}) + \frac{\pi^2E_2(\mathsf{q})}{3}\bigg)\dd u_1 \otimes \dd u_2 
\end{split}
\end{equation*}
which we took as definition for $\omega_{0,2}^{{\rm P},\mathsf{q}}$. We then apply Theorem~\ref{variationomega02}, which expresses $\omega_{g,n}^{{\rm P},\mathsf{q}}$ as a sum over stable graphs, with vertex weights given by $\omega_{h,k}^{{\rm P}}$, the operator $\mathscr{O}$ acting on each edge, and the operator
$$
\widetilde{\mathscr{P}}[\phi](z_0) = \sum_{\alpha \in \{-1,1\}} \Res_{z = \alpha} \bigg(\int^{z} \omega_{0,2}(\cdot,z_0)\bigg)\phi(z), \qquad \phi \in \mathscr{V}
$$
acting on each leaf. This expression should be considered an equality of $\mathsf{q}$-series. By construction of the operator $\mathscr{O}$ in \eqref{operatorPO}, its action on the (products of) $\omega^{{\rm P}}$s realises the summation over integral lengths in the (products of) $P$'s in \eqref{eq:stsgenfct}. It remains to compute the expansion of $\omega_{g,n}^{{\rm P},\mathsf{q}}$ near $z_i \rightarrow \infty$ -- more precisely, one should expand as a $\mathsf{q}$-series, and then expand each term when $z_i \rightarrow \infty$. For $\phi \in \mathscr{V}$, we find
\begin{equation*}
\begin{split}
- \Res_{z_0 = \infty} z_0^{L}\,\widetilde{\mathscr{P}}[\phi](z_0) & = -\sum_{\alpha \in \{-1,1\}} \Res_{z = \alpha} \Res_{z_0 = \infty} z_0^{L} \bigg(\int^{z} \omega_{0,2}^{{\rm P},\mathsf{q}}(\cdot,z_0)\bigg)\,\phi(z)  \\
& = -\sum_{\alpha \in \{-1,1\}} \Res_{z = \alpha} \phi(z) \Res_{z_0 = \infty} z_0^{L}\bigg(\frac{1}{z_0 - z} + \frac{1}{2} \sum_{\ell > 0} \frac{\mathsf{q}^{\ell}}{1 - \mathsf{q}^{\ell}}\big(-z_0^{\ell - 1}z^{-\ell} + z^{\ell}z_0^{-(\ell + 1)}\big)\bigg) \\
& = \sum_{\alpha \in \{-1,1\}} \Res_{z = \alpha} \phi(z)z^{L}\bigg(1 + \frac{\mathsf{q}^{L}}{2(1 - \mathsf{q}^{L})}\bigg) \\
& = - \Res_{z = \infty} \phi(z) z^{L} \bigg(1 + \frac{\mathsf{q}^{L}}{2(1 - \mathsf{q}^{L})}\bigg) .
\end{split}  
\end{equation*}
In the second line we chose a certain primitive of $\omega_{0,2}^{{\rm P},\mathsf{q}}(\cdot,z_0)$. The final result does not depend on this choice, since it is not changing the residue. Recalling \eqref{expNobur}, we deduce that
$$
(-1)^n \Res_{z_1 = \infty} \cdots \Res_{z_n = \infty} \omega_{g,n}^{{\rm P},\mathsf{q}}(z_1,\ldots,z_n) \prod_{i = 1}^n  \bigg(1 + \frac{\mathsf{q}^{L_i}}{2(1 - \mathsf{q}^{L_i})}\bigg)^{-1} \frac{z^{L_i}}{L_i}
$$
coincides with the right-hand side of \eqref{eq:stsgenfct} and this concludes the proof.
\end{proof}

\subsection{Second proof of topological recursion for Masur--Veech polynomials}

We observe that
\[
\mathscr{B}(u_1,u_2) = \bigg(\wp(u_1 - u_2;\mathsf{q}) + \frac{\pi^2 E_2(\mathsf{q})}{3}\bigg)\dd u_1\otimes \dd u_2
\]
is the unique fundamental bidifferential of the second kind on the elliptic curve $\mathbb{C}/(\mathbb{Z} \oplus \tau \mathbb{Z})$ with biresidue $1$ on the diagonal and such that it has zero period on the cycle $u \in [0,1]$. In a sense, $\omega_{0,2}^{{\rm P},\mathsf{q}}$ is the elliptic analog of $\omega_{0,2}^{{\rm MV}}$ from Proposition~\ref{EOMV}.

\medskip

Yet, it is not so easy to derive the Eynard--Orantin form of the topological recursion for the Masur--Veech polynomials (Proposition~\ref{EOMV}) by taking the $\mathsf{q} \rightarrow 1$ limit in Proposition~\ref{propsq}, due to the complicated ($\mathsf{q}$-dependent) way the $P_{g,n}^{\Box,\mathsf{q}}$'s are defined. Instead, we can take limit $\mathsf{q} \rightarrow 1$ in Corollary~\ref{cotrd} to give a second proof of the topological recursion with respect to lengths for Masur--Veech polynomials (the last statement in Proposition~\ref{MVtr1}) namely the topological recursion for Masur--Veech polynomials with respect to lengths.

\begin{prop}
	Corollary~\ref{cotrd} and Proposition~\ref{MVpolyasstscount} imply topological recursion for the Masur–Veech polynomials (see Proposition~\ref{MVtr1}).
\end{prop}

\begin{proof}
The only stable graph of type $(g,n) = (0,3)$ has one vertex with three leaves and no edge
\[
P_{0,3}^{\Box,\mathsf{q}}(L_1,L_2,L_3) = P_{0,3}(L_1,L_2,L_3) = A^{{\rm P}}(L_1,L_2,L_3) = 1.
\]
Rescaling $L_i$ by a positive even integer $T$ and sending $T$ to infinity is transparent here, and $2^{-(2g - 3 + n)} = 1$. Comparing with Proposition~\ref{MVpolyasstscount}, we get the claim for $(0,3)$.

\medskip

In type $(1,1)$, \eqref{eq:stsgenfct} yields for any positive even integer $L_1$
\begin{equation}\label{P11qq} 
P_{1,1}^{\Box,\mathsf{q}}(L_1) = P_{1,1}(L_1) + \frac{1}{2} \sum_{\ell > 0} \frac{\ell\,\mathsf{q}^{\ell}}{1 - \mathsf{q}^{\ell}} = \frac{L_1^2 - 4}{48} + \frac{1}{2} \sum_{\ell > 0} \frac{\ell\,\mathsf{q}^{\ell}}{1 - \mathsf{q}^{\ell}},
\end{equation}
where the $1/2$ comes from the automorphisms of the stable graph with a single vertex and one loop. Setting $\mathsf{q} = e^{-1/T}$, replacing $L_1$ with $TL_1$ and dividing by $T^2$, we get
\[
T^{-2}\,P_{1,1}^{\Box,\mathsf{q}}(L_1)
=
\frac{L_1^2}{48} + O(1/T^2) + \frac{1}{2T} \sum_{\hat{\ell} \in T^{-1}\mathbb{Z}_+} \frac{\hat{\ell}}{e^{\hat{\ell}} - 1} = \frac{L_1^2}{48} + \frac{1}{2} \int_{\mathbb{R}_{+}} \frac{\dd \hat{\ell}\,\hat{\ell}}{e^{\hat{\ell}} - 1} + o(1)
\]
by construction of the Riemann integral. The first term is indeed $V\Omega^{{\rm K}}_{1,1}(L_1)$, and the second term exactly matches the contribution of the stable graph with one vertex and one loop in \eqref{defn:MVpoly} for $V\Omega^{{\rm MV}}_{1,1}(L_1)$. This in fact is a check of Proposition~\ref{MVpolyasstscount} for $(g,n) = (1,1)$ as again $2^{-(2g - 3 + n)} = 1$, but also coincides with $VD^{{\rm MV}}(L_1)$ obtained by integrating $D^{{\rm MV}}$ from \eqref{eq314} over $\mathcal{M}_{1,1}(L_1)$, so proves the claim.

\medskip

For $2g - 2 + n > 0$, we reach $P_{g,n}^{\Box,\mathsf{q}}$ by applying $2g - 3 + n$ times the recursion formula from Corollary~\ref{cotrd}. One step of this formula is, for $L_1,\ldots,L_n$ positive integers such that $L_1 + \cdots + L_n$ is even,
\begin{equation*}
\begin{split}
& \quad P_{g,n}^{\Box,\mathsf{q}}(L_1,L_2,\ldots,L_n) \\
& = \sum_{m = 2}^{n} \sum_{\substack{\ell > 0 \\ L_1 + L_m + \ell = 0\,\,{\rm mod}\,\,2}} \ell\,B^{{\rm K}}[f_{\mathsf{q}}](L_1,L_m,\ell)\,P_{g,n - 1}^{\Box,\mathsf{q}}(\ell,L_2,\ldots,\widehat{L_m},\ldots,L_n) \\
& \quad + \frac{1}{2} \sum_{\substack{\ell,\ell' > 0 \\ L_1 + \ell + \ell' = 0\,\,{\rm mod}\,\,2}} \ell\,\ell'\,C^{{\rm K}}[f_{\mathsf{q}}](L_1,\ell,\ell')\bigg( P_{g-1,n+1}^{\Box,\mathsf{q}}(\ell,\ell',L_2,\ldots,L_n) \\
& \phantom{\qquad \quad + \frac{1}{2} \sum_{\substack{\ell,\ell' > 0 \\ L_1 + \ell + \ell' = 0\,\,{\rm mod}\,\,2}} \ell\,\ell'\,C^{{\rm K}}[f_{\mathsf{q}}](L_1,\ell,\ell') } + \sum_{\substack{h + h' = g \\ J \sqcup J' = \{L_2,\ldots,L_n\}}} P_{h,1 + |J|}^{\Box,\mathsf{q}}(\ell,J)\cdot P_{h',1 + |J'|}^{\Box,\mathsf{q}}(\ell',J')\bigg).
\end{split}
\end{equation*}
If $T$ is a positive even integer, this implies
\begin{equation}
\label{PGngufgn} \begin{split}
& \quad \frac{P_{g,n}^{\Box,\mathsf{q}}(TL_1,TL_2,\ldots,TL_n)}{T^{6g - 6 + 2n}} \\
& = \sum_{m = 2}^{n} \frac{1}{T} \sum_{\hat{\ell} \in 2T^{-1}\mathbb{Z}_{+}} \hat{\ell}\,B^{{\rm K}}[f_{\mathsf{q}}](TL_1,TL_m,T\hat{\ell})\,\frac{P_{g,n - 1}^{\Box,\mathsf{q}}(T\hat{\ell},TL_2,\ldots,\widehat{TL_m},\ldots,TL_n)}{T^{6g - 6 + 2(n - 1)}} \\
& \quad + \frac{1}{2T^2} \sum_{\substack{\hat{\ell},\hat{\ell}' \in T^{-1}\mathbb{Z}_{+} \\   \hat{\ell} + \hat{\ell}' \in 2T^{-1}\mathbb{Z}_{+}}} \hat{\ell}\,\hat{\ell}'\,C^{{\rm K}}[f_{\mathsf{q}}](TL_1,T\hat{\ell},T\hat{\ell}')\bigg( \frac{P_{g-1,n+1}^{\Box,\mathsf{q}}(T\hat{\ell},T\hat{\ell}',TL_2,\ldots,TL_n)}{T^{6(g - 1) - 6 + 2(n + 1)}} \\
& \phantom{\qquad \quad + \frac{1}{2} \sum_{\substack{\ell,\ell' > 0 \\ L_1 + \ell + \ell' = 0\,\,{\rm mod}\,\,2}} \ell\,\ell'\,C^{{\rm K}}[f_{\mathsf{q}}](TL_1,T\hat{\ell})} + \sum_{\substack{h + h' = g \\ J \sqcup J' = \{L_2,\ldots,L_n\}}} \frac{P_{h,1 + |J|}^{\Box,\mathsf{q}}(T\hat{\ell},TJ)}{T^{6h - 6 + 2(1 + |J|)}}\cdot \frac{P_{h',1 + |J'|}^{\Box,\mathsf{q}}(T\hat{\ell}',TJ')}{T^{6h' - 6 + 2(1 + |J'|)}} \bigg).
\end{split}
\end{equation}
We observe that for $X \in \{B,C\}$ and since $\mathsf{q} = e^{-1/T}$, we have the exact relation
\begin{equation}
\label{Xqa}X^{{\rm K}}[f_{\mathsf{q}}](TL_1,TL_2,TL_3) = X^{{\rm K}}[f_{{\rm MV}}](L_1,L_2,L_3),\qquad {\rm with}\,\,f_{{\rm MV}}(\ell) = \frac{1}{e^{\ell} - 1}.
\end{equation}
We observe that $\hat{\ell}$ in the first line (resp. $(\hat{\ell},\hat{\ell}')$ in the second line) is restricted to a sublattice of $\mathbb{Z}$ (resp. $\mathbb{Z}^2$) of index $2$. If we could replace the $P_{h,k}^{\Box,\mathsf{q}}$ directly by their limit provided by Proposition~\ref{MVpolyasstscount} for $2h - 2 + k < 2g - 2 + n$, the approximation of Riemann integrals by Riemann sums would imply Proposition~\ref{MVtr1} for $(g,n)$, and we could conclude by induction.

\medskip

Instead of justifying that such a replacement is allowed, we can proceed in a simpler way. Let us unfold all the $i = 1,\ldots,2g - 3 + n$ steps of the recursion. The result is a formula expressing the left-hand side of \eqref{PGngufgn} as a finite sum of terms of different topological origin and countable sums over discretised variables $\hat{\ell}_i$ (or $\hat{\ell}_i,\hat{\ell}_i'$) each belonging to a sublattice of index $2$. Taking into account \eqref{Xqa}, the summands are finite products of $B^{{\rm K}}[f_{\mathrm{MV}}]$, $C^{{\rm K}}[f_{{\rm MV}}]$, $P_{0,3}^{\Box,\mathsf{q}} \equiv 1$ and $P_{1,1}^{\Box,\mathsf{q}}$. The latter can be replaced by its expression \eqref{P11qq}. The outcome of this unfolding is a big sum over a finite set of discretised variables of specialisations of continuous functions and Riemann integrable functions on some $\mathbb{R}_{+}^{k}$ at those discretised variables. To this, we can apply the principle of approximation of Riemann integrals by Riemann sums, and there will be exactly $2g - 3 + n$ factors of $(1/2)$ coming from the index $2$ sublattices over which the discretised sums range. The factors of $T$, as already exhibited for the one step recursion, disappear in the $T \rightarrow \infty$ limit. We therefore obtain that
\begin{equation}
\label{222nd}2^{2g - 3 + n}\,\lim_{T \rightarrow \infty} \frac{P_{g,n}^{\Box,\mathsf{q} = e^{-1/T}}(TL_1,\ldots,TL_n)}{T^{6g - 6 + 2n}}
\end{equation}
exists and is computed by the unfolded topological recursion, with initial data already identified with
\begin{equation}
\label{ABCDMV}
(A^{{\rm MV}},B^{{\rm MV}},C^{{\rm MV}},VD^{{\rm MV}})
\end{equation}
thanks to \eqref{Xqa} and the cases of $(0,3)$ and $(1,1)$ treated at the beginning of the proof. By recombination, this automatically implies the one-step recursion formula \eqref{TReqn} with this initial data for \eqref{ABCDMV}. As Proposition~\ref{MVpolyasstscount} equates \eqref{222nd} with $V\Omega^{{\rm MV}}_{g,n}(L_1,\ldots,L_n)$, this implies Proposition~\ref{MVtr1}.
\end{proof} 

\newpage

%***************************************
\appendix
%***************************************

%***************************************
\section{Closed formulae for the intersection of \texorpdfstring{$\psi$}{cotangent} classes in genus one}
\label{AppA}
%***************************************

\begin{lem} For a fixed integer $n \geq 1$, we have
\begin{equation}\label{eq:g1aaa}
\int_{\overline{\mathfrak{M}}_{1,n}} \psi_1^{a_1} \cdots \psi_n^{a_n} 
=
 \frac{1}{24} \biggl( \binom{n}{a_1, \ldots, a_n} - \sum_{\substack{b_1, \ldots, b_n \\ b_i \in \{0,1\}}} \binom{n - (b_1 + \cdots + b_n)}{a_1 - b_1, \ldots, a_n - b_n} (b_1 + \cdots +b_n - 2)! 
 \biggr), 
\end{equation}
and the sum of all such integrals is
\begin{equation}\label{eq:g1all}
\int_{\overline{\mathfrak{M}}_{1,n}} \frac{1}{\prod_{i=1}^n(1 - \psi_i)} = \frac{1}{24}\bigg(n^n - \sum_{k=1}^{n-1} \frac{n^{n - k}}{k(k + 1)}\,\frac{(n - 1)!}{(n - k - 1)!}\bigg).
\end{equation}
We use the convention that summands involving negative factorials are excluded from the summation. In particular we retrieve \eqref{1n1n} and \eqref{otherpsigenus1} used in the text.

\begin{proof}
Let us recall the following result.
\begin{thm}[Conjecture of Goulden--Jackson--Vainshtein \cite{GJV}, theorem of Vakil \cite{Vakil}] 
Let $\mu$ be a partition of $d$ of length $n = \ell(\mu)$. The simple connected Hurwitz numbers $h_{g=1, \mu}$ of genus one and ramification profile $\mu$ over zero are given by
\begin{equation}
\frac{h_{g=1, \mu}}{(n + d)!} = \frac{1}{24} \prod_{i=1}^{n} \frac{\mu_i^{\mu_i}}{\mu_i!} 
\Big(
d^n - d^{n-1} - \sum_{k=2}^{n} (k-2)! d^{n - k} s_k(\mu_1, \ldots, \mu_{n})
\Big),
\end{equation}
\end{thm}
\noindent
where $s_k$ is the $k$-th elementary symmetric polynomial. \hfill $\blacksquare$

\medskip

On the other hand the ELSV formula \cite{ELSV} in genus one gives
\begin{equation}
\frac{h_{g=1, \mu}}{(n + d)!} = \prod_{i=1}^n \frac{\mu_i^{\mu_i}}{\mu_i!} \int_{\overline{\mathfrak{M}}_{1,n}} \frac{1 - \lambda_1}{\prod_{i=1}^{n} (1 - \mu_i \psi_i)}.
\end{equation}
Combining the two we obtain
\begin{equation}\label{eq:comb2}
\int_{\overline{\mathfrak{M}}_{1,n}} \frac{1 - \lambda_1}{\prod_{i=1}^n (1 - \mu_i \psi_i)}
=
\frac{1}{24} 
\bigg(
d^n - d^{n-1} - \sum_{k=2}^n (k-2)! d^{n - k} s_k(\mu_1, \ldots, \mu_{n})
\bigg).
\end{equation}
The contribution corresponding to the intersections of $\lambda_1$ is easily removed by erasing the second summand $d^{n - 1}$: indeed, considering $d$ as $\mu_1 + \ldots + \mu_n$, the right hand side is a polynomial in the $\mu_i$ involving only two degrees, which are $n$ and $n - 1$. On the other hand, the summands of degree $n - 1$ must correspond to all and only the monomials in $\psi$ classes intersecting $\lambda_1$. Now, to prove \eqref{eq:g1all} substitute $\mu_i = 1$ for each $i$ and observe that $s_k(1, \ldots, 1) = \binom{n}{k}$. To prove \eqref{eq:g1aaa}, collect the coefficient of $\mu_1^{a_1} \cdots \mu_n^{a_n}$ in the right hand side after the substitution $d = \mu_1 + \ldots + \mu_n$. This concludes the proof of the lemma.
\end{proof}
\end{lem}

A longer but more detailed way to remove the contribution of $\lambda_1$ to prove \eqref{eq:g1all} can be obtained from the $\lambda_g$-theorem. 

\begin{thm}[$\lambda_g$-theorem \cite{FP}]
\begin{equation}
\int_{\overline{\mathfrak{M}}_{g,n}}\psi_1^{a_1} \cdots \psi_n^{a_n}\lambda_g = \binom{2g - 3 + n}{a_1, \ldots, a_n} \int_{\overline{\mathfrak{M}}_{g,1}} \psi_1^{2g - 2} \lambda_g.
\end{equation}
\hfill $\blacksquare$ 
\end{thm}
\noindent
Its specialisation in genus one reads
\begin{equation}
\int_{\overline{\mathfrak{M}}_{1,n}}\psi_1^{a_1} \cdots \psi_d^{a_n}\lambda_1 = \binom{n-1}{a_1, \ldots, a_n} \int_{\overline{\mathfrak{M}}_{1,1}}\lambda_1 = \binom{n-1}{a_1, \ldots, a_n} \frac{1}{24}.
\end{equation}
This can also be seen for instance using that $\lambda_1$ is represented by the Poincar\'e dual of the divisor of curves with at least one non-separating node times $\tfrac{1}{24}$, then pulling-back the class via the attaching map and integrating over $\overline{\mathfrak{M}}_{0, n+2}$ gives the same result. In any case, summing over all $n$-tuples of non-negative integers $a_i$ such that $a_1 + \cdots + a_n = d-1$ gives, using the multinomial theorem,
\begin{align}
\int_{\overline{\mathfrak{M}}_{1,n}}\frac{\lambda_1}{\prod_{i=1}^n (1 - \psi_i)} &= \sum_{\substack{a_1, \ldots, a_n \geq 0 \\ a_1 + \cdots + a_n = n-1}}\int_{\overline{\mathfrak{M}}_{1,d}}\psi_1^{a_1} \cdots \psi_n^{a_n}\lambda_1 \\
&=
\frac{1}{24} \sum_{\substack{a_1, \ldots, a_n \geq 0 \\ a_1 + \cdots + a_n = n - 1}} \binom{n - 1}{a_1, \ldots, a_n} = \frac{1}{24} n^{n-1},
\end{align}
which equals the second summand in equation \eqref{eq:comb2} after the substitution $\mu_i = 1$ for all $i$, and therefore $n = d$. Removing it from \eqref{eq:comb2} and simplifying the expression proves again \eqref{eq:g1all}.

%***************************************
\section{Computing Masur--Veech polynomials and square-tiled surfaces with Eynard--Orantin topological recursion}
%***************************************
\label{AppB}
For readers who are unfamiliar with the topological recursion \`a la Eynard--Orantin, we compute a few Masur--Veech polynomials and square-tiled surfaces generating series via $\omega_{g,n}^{{\rm MV}}$ and $\omega_{g,n}^{{\rm P},\mathsf{q}}$ (resp. Proposition~\ref{EOMV} and Proposition~\ref{propsq}).

%***************************************
\subsubsection*{Masur--Veech polynomials via Eynard--Orantin topological recursion}
%***************************************

Let us apply the residue formula \eqref{TREOeqn} to the spectral curve given by $\mathcal{C} = \mathbb{C}$ and 
\begin{equation}
x(z) = \frac{z^2}{2}, \qquad  y(z) = -z, \qquad \omega_{0,2}^{{\rm MV}}(z_1, z_2) = \frac{\dd z_1 \otimes \dd z_2}{(z_1 - z_2)^2} + \frac{1}{2}\sum_{m \in \Z^*} \frac{\dd z_1 \otimes \dd z_2}{(z_1 - z_2 + m)^2}.
\end{equation}
In this section we drop the superscript ${}^{{\rm MV}}$ on the $\omega_{g,n}$'s as it will always refer to the Masur--Veech topological recursion amplitudes. Let us first compute the recursion kernel
\begin{align*}
  K(z_1,z)
  &=
  \frac{1}{2}\frac{\int_{-z}^{z} \omega_{0,2}(\cdot,z_1)}{(y(z) - y(-z))\,\dd x(z)} \\
  &=
  - \frac{\dd z_1}{4 z^2 \dd z} \int_{-z}^{z} \bigg(\frac{\dd z'}{(z_1 - z')^2} + \frac{1}{2}\sum_{m \in \Z^{*}} \frac{\dd z'}{(z_1 - z' + m)^2}\bigg) \\
  & =
  - \frac{\dd z_1}{2 z \dd z} \bigg(\frac{1}{z_1^2 - z^2}+ \frac{1}{2} \sum_{m \in \Z^{*}} \frac{1}{(z_1 + m)^2 - z^2} \bigg).
\end{align*}
It is handy, in order to compute residues at $z = 0$, to write down the expansion in power series near $z=0$ of the recursion kernel:
\begin{align*}
	\frac{1}{z_1^2 - z^2}  +  \frac{1}{2}  \sum_{m \in \Z^*} \frac{1}{(z_1 + m)^2 - z^2} 
	&= \sum_{d \geq 0} \bigg(\frac{1}{z_1^{2d + 2}} +
	\frac{1}{2} \sum_{m \in \Z^{*} }  \frac{1}{(z_1 + m)^{2d + 2}}\bigg)\,z^{2d}  \\
	&= 
	\sum_{d \geq 0} \zeta_{{\rm H}}(2d + 2; z_1) z^{2d}.
\end{align*}
In the same way we have that
\begin{align*}
	\zeta_{{\rm H}}(2; z - z_i) &= \sum_{d \geq 0} (2d+1)\,\zeta(2d + 2;z_i)\,z^{2d} + \text{odd part in }z,  \\
	\zeta_{{\rm H}}(2k;z) &= \frac{1}{z^{2k}} + \sum_{d \geq 0} \binom{2k - 1 + 2d}{2d}\,\zeta(2k + 2d) \,z^{2d}.  
\end{align*}
The topological recursion formula, specialised to our case and expressed in terms of
$$
W_{g,n}(z_1, \ldots, z_n) = \frac{\omega_{g,n}(z_1, \ldots, z_n)}{\dd z_1 \otimes \cdots \otimes \dd z_n},
$$
reads
\begin{align*} 
	W_{g,n}(z_1,z_2, \ldots, z_n) &= \frac{1}{2} [z^0]  \sum_{d \geq 0} \zeta_{{\rm H}}(2d + 2 ; z_1)\,z^{2d} \bigg\{ W_{g-1,n+1}(z,-z, z_2,\ldots,z_n) \\
	&\qquad\qquad\qquad + \sum_{\substack{h+h'=g \\ J \sqcup J' = \{z_2,\ldots,z_n\}}}^{\text{no $(0,1)$}} W_{h,1 + |J|}(z,J)\,W_{h', 1 + |J'|}(-z, J')
	\bigg\}. 
\end{align*}

$\bullet$ $\mathbf{(g,n) = (0,3)}$
\begin{align*}
	W_{0,3}(z_1,z_2,z_3)  &=  \frac{1}{2} [z^0] \sum_{d \geq 0} \zeta_{{\rm H}}(2d + 2 ; z_1) z^{2d}  \cdot \Big(  W_{0,2}(z,z_2)W_{0,2}(-z,z_3) + W_{0,2}(z,z_3)W_{0,2}(-z,z_2) \Big) \\
 	&=  [z^0]  \sum_{d \geq 0} \zeta_{{\rm H}}(2d + 2 ; z_1)\,z^{2d} \cdot W_{0,2}(z,z_2)W_{0,2}(z,z_3)   \\
 	&=  \zeta_{{\rm H}}(2;z_1)\zeta_{{\rm H}}(2;z_2)\zeta_{{\rm H}}(2;z_3).
\end{align*}
The inverse Laplace transform of the principal part near $z_1 = z_2 = z_3 = 0$ then reads
$$
V\Omega^{{\rm MV}}_{0,3}(L_1,L_2,L_3) = 1.
$$
Multiplying by the combinatorial factor $2^{4g - 2 + n}\,\frac{(4g - 4 + n)!}{(6g - 7 + 2n)!}$ whose value for $g = 0$ and $n \rightarrow 3$ is $4$, we get $MV_{0,3} = 4$.
 
\medskip
 
$\bullet$ $\mathbf{(g,n)=(0,4)}$

\begin{align*}
	W_{0,4}(z_1,z_2,z_3,z_4)
	& = [z^0] \sum_{d \geq 0} \zeta_{{\rm H}}(2d + 2 ; z_1)\,z^{2d} \cdot\Big(W_{0,2}(z,z_2)W_{0,3}(z,z_3,z_4) \\
	& \quad + W_{0,2}(z,z_3)W_{0,3}(z,z_2,z_4) + W_{0,2}(z,z_4) W_{0,3}(z,z_2,z_3)\Big) \\
	& =  [z^0] \sum_{d \geq 0} \zeta_{{\rm H}}(2d + 2 ; z_1)\,z^{2d} \cdot\Big(\zeta_{{\rm H}}(2;z-z_2)\zeta_{{\rm H}}(2 ; z) \zeta_{{\rm H}}(2 ; z_3)\zeta_{{\rm H}}(2 ; z_4)   \\
	& \quad + \zeta_{{\rm H}}(2;z-z_3)\zeta_{{\rm H}}(2 ; z) \zeta_{{\rm H}}(2 ; z_2)\zeta_{{\rm H}}(2 ; z_4) +\zeta_{{\rm H}}(2;z-z_4)\zeta_{{\rm H}}(2 ; z) \zeta_{{\rm H}}(2 ; z_2) \zeta_{{\rm H}}(2 ; z_3)  \Big)  \\ 
	& = 3 \sum_{i=0}^3 \zeta_{{\rm H}}(4;z_i) \prod_{j \in \{1,2,3,4\}\setminus \{ i \}} \zeta_{{\rm H}}(2 ; z_j) + 3 \zeta(2) \zeta_{{\rm H}}(2;z_1)\zeta_{{\rm H}}(2;z_2)\zeta_{{\rm H}}(2;z_3)\zeta_{{\rm H}}(2;z_4). 
\end{align*} 
The inverse Laplace transform of the principal part near $z_1 = z_2 = z_3 = z_4 = 0$ then reads
$$
V\Omega^{{\rm MV}}_{0,4}(L_1,L_2,L_3,L_4) = \frac{1}{2}\bigg(\pi^2 + \sum_{i = 1}^4 L_i^2\bigg)
$$
from which we deduce $MV_{0,4} = 2\pi^2$.

\medskip

$\bullet$ $\mathbf{(g,n)=(1,1)}$

\begin{align*}
	W_{1,1}(z_1) & = \frac{1}{2} [z^0] \sum_{d \geq 0} \zeta_{{\rm H}}(2d + 2 ; z_1)\,z^{2d}
	\,W_{0,2}(z,-z)  \\
	& = \frac{1}{2}\sum_{d \geq 0} \zeta_{{\rm H}}(2d + 2 ; z_1)\,z^{2d}\,\bigg(\frac{1}{4z^2} + \sum_{m \geq 1} \sum_{d' \geq 0} (2d'+1)\frac{(2z)^{2d'}}{m^{2d'+2}}\bigg)   \\
	& = \frac{1}{8}\zeta_{{\rm H}}(4;z_1) + \frac{\pi^2}{12}\zeta_{{\rm H}}(2;z_1).  
\end{align*}

The inverse Laplace transform of the principal part near $z_1 = 0$ then reads
$$
V\Omega^{{\rm MV}}_{1,1}(L_1) = \frac{\pi^2}{12} + \frac{L_1^2}{48}.
$$
Multiplying the constant term by the combinatorial factor $\tfrac{2^{4g - 2 + n}(4g - 4 + n)!}{(6g - 7 + 2n)!} = 8$ we deduce $MV_{1,1} = \tfrac{2\pi^2}{3}$.

\medskip

$\bullet$ $\mathbf{(g,n)=(1,2)}$

\begin{align*}
	& W_{1,2}(z_1, z_2) \\
	& =  [z^0] \sum_{d \geq 0}^{\infty} \zeta_{{\rm H}}(2d + 2 ; z_1)\,z^{2d} \Big(\frac{1}{2}\,W_{0,3}(z,z,z_2) + W_{0,2}(z,z_2)W_{1,1}(z) \Big)  \\
	& =   \frac{1}{2}[z^0] \sum_{d \geq 0} \zeta_{{\rm H}}(2d + 2; z_1)\,z^{2d}  \cdot \zeta_{{\rm H}}(2;z)^2 \zeta_{{\rm H}}(2;z_2) + \frac{1}{8}[z^0]  \sum_{d \geq 0} \zeta_{{\rm H}}(2d + 2; z_1)\,z^{2d}\cdot \zeta_{{\rm H}}(2;z-z_1)\zeta_{{\rm H}}(4;z) \\
	&\qquad\qquad + \frac{1}{2}[z^0] \sum_{d \geq 0} \zeta_{{\rm H}}(2d + 2 ; z_1)\,z^{2d} \cdot \zeta_{{\rm H}}(2;z-z_2)\zeta_{{\rm H}}(2;z)\zeta(2)  \\
	& = \frac{\zeta_{{\rm H}}(2;z_2)}{2}\biggl( \zeta_{{\rm H}}(6;z_1) +  2\zeta_{\rm H}(4;z_1)\zeta(2) + 6 \zeta_{\rm H}(2; z_1)\zeta(4) +  \zeta_{\rm H}(2;z_0)\zeta(2)^2 \biggr) \\
	&\qquad\qquad + \frac{1}{8} \biggl( \zeta_{\rm H}(2; z_1)\zeta_{{\rm H}}(2;z_2)\zeta(4) + \zeta_{{\rm H}}(6;z_1)\zeta_{{\rm H}}(2;z_2) + 3\zeta_{\rm H}(4;z_1)\zeta_{\rm H}(4;z_2) + 5\zeta_{\rm H}(2;z_1)\zeta_{\rm H}(6;z_2)\biggr) \\
	&\qquad\qquad + \frac{1}{2}\biggl( \zeta_{\rm H}(2; z_1)\zeta_{\rm H}(2;z_2)\,\zeta(2)^2 + \zeta_{\rm H}(4;z_1)\zeta_{\rm H}(2;z_2)\,\zeta(2) + 3\zeta_{\rm H}(2;z_1)\zeta_{\rm H}(4;z_2)\,\zeta(2)\biggr).
\end{align*}
Re-arranging the terms, we obtain 
\begin{align*}
	W_{1,2}(z_1, z_2) &= \frac{5}{8}\Big(\zeta_{\rm H}(6;z_1)\zeta_{{\rm H}}(2;z_2) + \zeta_{|\rm H}(2;z_1)\zeta_{\rm H}(6;z_2)\big) + \frac{3}{8} \zeta_{\rm H}(4;z_1)\zeta_{\rm H}(4;z_2) \\
	& + \frac{\pi^2}{4}\Big(\zeta_{\rm H}(4;z_1)\zeta_{\rm H}(2;z_2) + \zeta_{\rm H}(2;z_1)\zeta_{\rm H}(4;z_2)\Big) + \frac{\pi^4}{16} \zeta_{\rm H}(2;z_1)\zeta_{\rm H}(2;z_2).
\end{align*}

The inverse Laplace transform of the principal part near $z_1 = z_2 = 0$ then reads
$$
V\Omega^{{\rm MV}}_{1,2}(L_1,L_2) = \frac{1}{192}(L_1^4 + L_2^4) + \frac{1}{96}L_1^2L_2^2 + 
 \frac{\pi^2}{24}(L_1^2 + L_2^2)  + \frac{\pi^4}{16}.
$$
Multiplying the constant term by the combinatorial factor $\tfrac{2^{4g - 2 + n}(4g - 4 + n)!}{(6g - 7 + 2n)!} = \tfrac{16}{3}$ we obtain $MV_{1,2} = \tfrac{\pi^4}{3}$.

%***************************************
\subsubsection*{Square-tiled surfaces via Eynard--Orantin topological recursion}
%***************************************

Let us apply the residue formula \eqref{TREOeqn} to the spectral curve given by $\mathcal{C} = \mathbb{C}$ and 
\begin{equation}
  x(z) = z + \frac{1}{z},
  \qquad
  y(z) = -z,
  \qquad
  \omega_{0,2}^{{\rm P},\mathsf{q}}(z_1,z_2) = \frac{1}{2}\,\frac{\dd z_1 \otimes \dd z_2}{(z_1 - z_2)^2} + \frac{1}{2}\bigg(\wp(u_1 - u_2;\mathsf{q}) + \frac{\pi^2 E_2(\mathsf{q})}{3}\bigg)\dd u_1 \otimes \dd u_2,
  \end{equation}
where $z_j = \exp(2{\rm i}\pi\,u_j)$. In this section, we drop the superscript ${}^{{\rm P},\mathsf{q}}$ on the $\omega_{g,n}$'s as it will always refer to the square-tiled surfaces topological recursion amplitudes. We can rewrite $\omega_{0,2}$ as
\[
  \omega_{0,2}(z_1,z_2)
  =
  \frac{\dd z_1 \otimes \dd z_2}{(z_1 - z_2)^2}
  +
  \frac{\dd z_1 \otimes \dd z_2}{2 z_1 z_2} \sum_{\ell \ge 0} \frac{\ell \mathsf{q}^{\ell}}{1-\mathsf{q}^{\ell}} \left(
        \Bigl( \frac{z_1}{z_2} \Bigr)^{\ell} + \Bigl( \frac{z_1}{z_2} \Bigr)^{-\ell}
      \right).
\]
Let us first compute the recursion kernel
\begin{align*}
  K(z_1,z)
  &=
  \frac{1}{2}\frac{\int_{1/z}^{z} \omega_{0,2}(\cdot,z_1)}{(y(z) - y(1/z)) \, \dd x(z)} \\
  &=
  - \frac{z \, \dd z_1}{2 (z - z^{-1})^2 \, \dd z}
    \int_{1/z}^{z} \biggl(
      \frac{\dd z'}{(z_1 - z')^2}
      +
      \frac{\dd z'}{2 z_1 z'} \sum_{\ell \ge 0} \frac{\ell \mathsf{q}^{\ell}}{1-\mathsf{q}^{\ell}} \left(
        \Bigl( \frac{z'}{z_1} \Bigr)^{\ell} + \Bigl( \frac{z'}{z_1} \Bigr)^{-\ell}
      \right)
    \biggr) \\
  & =
  - \frac{z^2 \, \dd z_1}{2 (z^2 - 1) \dd z} \bigg(
    \frac{1}{(z_1-z)(z_1-z^{-1})}
    +
    \frac{1}{2z_1} \sum_{\ell >0} \frac{\mathsf{q}^{\ell}}{1-\mathsf{q}^{\ell}}
      (z_1^\ell + z_1^{-\ell}) \sum_{i=0}^{\ell-1} z^{2i-\ell+1}
  \bigg),
\end{align*}
and define $\hat{K}(z_1,z)$ via the formula $K(z_1,z) = - \frac{z^2 \, \dd z_1}{2 (z^2 - 1) \dd z} \hat{K}(z_1,z)$. Moreover, let us set
$$
W_{g,n}(z_1, \ldots, z_n) = \frac{\omega_{g,n}(z_1, \ldots, z_n)}{\dd z_1 \otimes \cdots \otimes \dd z_n}.
$$

$\bullet$ $\mathbf{(g,n) = (0,3)}$
\[
\begin{split}
  W_{0,3}(z_1,z_2,z_3)
  & =
  \sum_{\alpha \in \{-1,1\}} \Res_{z = \alpha} \frac{\hat{K}(z_1,z)}{2(z^2-1)}
    W_{0,2}(z,z_2) W_{0,2}(z^{-1},z_3)\dd z
    +
    (2 \leftrightarrow 3)
    \\
  & = \sum_{\alpha \in \{-1,1\}} \frac{\hat{K}(z_1,z)}{2(z + \alpha)}
    W_{0,2}(z,z_2) W_{0,2}(z^{-1},z_3)\Big|_{z = \alpha}
    +
    (2 \leftrightarrow 3) \\
  & = \sum_{\alpha \in \{-1,1\}} \frac{1}{2(z + \alpha)}
  \Bigg[
  \bigg(
    \frac{1}{(z_1-z)(z_1-z^{-1})}
    +
    \frac{1}{2z_1} \sum_{\ell >0} \frac{\mathsf{q}^{\ell}}{1-\mathsf{q}^{\ell}}
      (z_1^\ell + z_1^{-\ell}) \sum_{i=0}^{\ell-1} z^{2i-\ell+1}
  \bigg) \\
  &\hphantom{= \sum_{\alpha \in \{-1,1\}} \frac{1}{2(z + \alpha)}}
  \bigg(
    \frac{1}{(z - z_2)^2}
    +
    \frac{1}{2 z z_2} \sum_{\ell \ge 0} \frac{\ell \mathsf{q}^{\ell}}{1-\mathsf{q}^{\ell}} \left(
          \Bigl( \frac{z}{z_2} \Bigr)^{\ell} + \Bigl( \frac{z}{z_2} \Bigr)^{-\ell}
        \right) \\
  &\hphantom{= \sum_{\alpha \in \{-1,1\}} \frac{1}{2(z + \alpha)}}
  \bigg(
    \frac{1}{(z^{-1} - z_3)^2}
    +
    \frac{z}{2 z_3} \sum_{\ell \ge 0} \frac{\ell \mathsf{q}^{\ell}}{1-\mathsf{q}^{\ell}} \left(
          \Bigl( \frac{z^{-1}}{z_3} \Bigr)^{\ell} + \Bigl( \frac{z^{-1}}{z_3} \Bigr)^{-\ell}
        \right)
  \Bigg]\Bigg|_{z = \alpha}
  +
  (2 \leftrightarrow 3) \\
  & =
  \frac{1}{2} \bigl(
    X_{\mathsf{q}}(z_1)X_{\mathsf{q}}(z_2)X_{\mathsf{q}}(z_3)
    -
    X_{\mathsf{q}}(-z_1)X_{\mathsf{q}}(-z_2)X_{\mathsf{q}}(-z_3)
  \bigr),
\end{split}
\]
where we have set
\[
  X_{\mathsf{q}}(z)
  =
  \frac{1}{(1 - z)^2}
  +
  \frac{1}{2 z} \sum_{\ell \ge 0} \frac{\ell \mathsf{q}^{\ell}}{1-\mathsf{q}^{\ell}}
    \left( z^{\ell} + z^{-\ell} \right).
\]
It can be rewritten in terms of the Weierstra\ss{} function, as $X_{\mathsf{q}}(z) \dd z = \bigl( \wp(u;\mathsf{q}) + \frac{\pi^2 E_2(\mathsf{q})}{3}\bigr) \dd u$. We can then compute the generating series of square-tiled surfaces of type $(0,3)$ with boundaries as
\[
\begin{split}
  P_{0,3}^{\Box,\mathsf{q}}(L_1,L_2,L_3)
  & =
  - \Res_{z_1 = \infty}\Res_{z_2 = \infty}\Res_{z_3 = \infty} W_{0,3}^{{\rm P},\mathsf{q}}(z_1,z_2,z_3) \prod_{i = 1}^3 \bigg(1 + \frac{\mathsf{q}^{L_i}}{2(1 - \mathsf{q}^{L_i})}\bigg)^{-1} \frac{z_i^{L_i}}{L_i} \dd z_i \\
  & =
  \frac{1 + (-1)^{L_1+L_2+L_3}}{2}.
\end{split}
\]
This coincides with $P_{0,3}(L_1,L_2,L_3)$.

$\bullet$ $\mathbf{(g,n) = (1,1)}$
\[
\begin{split}
  W_{1,1}(z_1)
  & =
  \sum_{\alpha \in \{-1,1\}} \Res_{z = \alpha} \frac{\hat{K}(z_1,z)}{2(z^2-1)}
    W_{0,2}(z,z^{-1})\dd z \\
  & =
  \sum_{\alpha \in \{-1,1\}}
    \bigg[
      \frac{\hat{K}(z_1,z)}{2(z+\alpha)} \frac{1}{2} \sum_{\ell \ge 0} \frac{\ell \mathsf{q}^{\ell}}{1-\mathsf{q}^{\ell}} \left( z^{2\ell} + z^{-2\ell} \right)
    \bigg]\bigg|_{z = \alpha}
  +
  \sum_{\alpha \in \{-1,1\}} \frac{1}{2} \frac{d^2}{dz^2} \bigg[
    \frac{z^2 \, \hat{K}(z_1,z)}{2(z+\alpha)^3}
  \bigg]\bigg|_{z = \alpha} \\
  & =
  \frac{X_{\mathsf{q}}(z) - X_{\mathsf{q}}(-z)}{2} \frac{1}{2} \sum_{\ell \ge 0} \frac{\ell \mathsf{q}^{\ell}}{1-\mathsf{q}^{\ell}}
  +
  \frac{1}{32}\bigg( Y_{\mathsf{q}}(z) - Y_{\mathsf{q}}(-z)
  - X_{\mathsf{q}}(z) + X_{\mathsf{q}}(-z)\bigg),
\end{split}
\]
where $X_{\mathsf{q}}$ is defined as before, and
\[
  Y_{\mathsf{q}}(z) = \frac{2z}{(1-z)^4} + \frac{1}{2z} \sum_{\ell \ge 0} \frac{\ell^2-1}{3} \frac{\ell \mathsf{q}^{\ell}}{1-\mathsf{q}^{\ell}}
    \left( z^{\ell} + z^{-\ell} \right) = \frac{1}{3}\bigg(z \frac{\dd}{\dd z} z \frac{\dd}{\dd z}  - 1 \bigg)X_{\mathsf{q}}(z)
\]
We can then compute the generating series of square-tiled surfaces of type $(1,1)$ with boundaries as
\[
\begin{split}
  P_{1,1}^{\Box,\mathsf{q}}(L_1)
  & =
  - \Res_{z_1 = \infty} W_{1,1}^{{\rm P},\mathsf{q}}(z_1) \bigg(1 + \frac{\mathsf{q}^{L_1}}{2(1 - \mathsf{q}^{L_1})}\bigg)^{-1} \frac{z_1^{L_1}}{L_1} \dd z_1 \\
  & =
  \frac{1 + (-1)^{L_1}}{2} \bigg( \frac{1}{2} \sum_{\ell \ge 0} \frac{\ell \mathsf{q}^{\ell}}{1-\mathsf{q}^{\ell}} + \frac{L_1^2-4}{48}\bigg).
\end{split}
\]
Since $P_{1,1}(L_1) = \frac{1 + (-1)^{L_1}}{2} \frac{L_1^2-4}{48}$, we can write it as
\[
  P_{1,1}^{\Box,\mathsf{q}}(L_1)
  =
  P_{1,1}(L_1) + \frac{1}{2} \sum_{\ell > 0} P_{0,3}(L_1,\ell,\ell) \frac{\ell \mathsf{q}^{\ell}}{1-\mathsf{q}^{\ell}},
\]
which coincides with the sum over stable graphs of type $(1,1)$ of Proposition~\ref{stscountassumstgraph}.

%*************************************
\newpage
\section{Numerical data}
%*************************************
\label{AppC}
\vfill
\begin{figure}[h!]
\centering
\savebox{\mytable}{
\begin{tabular}{|c|ccccccc|}
\hline
$n$ \textbackslash \,\,$g$ {\rule{0pt}{3.2ex}}{\rule[-1.8ex]{0pt}{0pt}}& $0$ & $1$ & $2$ & $3$ & $4$ & $5$ & $6$ \\
\hline\hline
$0$ {\rule{0pt}{2.5ex}}{\rule[-1.8ex]{0pt}{0pt}}& - & - & $\frac{1}{15}$ & $\frac{115}{33264}$ & $\frac{2106241}{11548293120}$  & $\frac{7607231}{790778419200}$ & $\frac{51582017261473}{101735601235107840000}$ \\
$1$  {\rule{0pt}{2.5ex}}{\rule[-1.8ex]{0pt}{0pt}}& - & $\mathbf{\frac{2}{3}}$ & $\mathbf{\frac{29}{840}}$ & $\mathbf{\frac{4111}{2223936}}$ & $\mathbf{\frac{58091}{592220160}}$ & $\mathbf{\frac{35161328707}{6782087854080000}}$ & $\mathbf{\frac{1725192578138153}{6307607276576686080000}}$ \\
$2$  {\rule{0pt}{2.5ex}}{\rule[-1.8ex]{0pt}{0pt}}& - & $\mathbf{\frac{1}{3}}$ & $\mathbf{\frac{337}{18144}}$ & $\mathbf{\frac{77633}{77837760}}$ & $\mathbf{\frac{160909109}{3038089420800}}$ & $\mathbf{\frac{27431847097}{9796349122560000}}$ & $\mathbf{\frac{236687293214441}{1601932006749634560000}}$ \\
$3$  {\rule{0pt}{2.5ex}}{\rule[-1.8ex]{0pt}{0pt}}& $\mathbf{4}$ & $\frac{11}{60}$ & $\mathbf{\frac{29}{2880}}$ & $\mathbf{\frac{207719}{384943104}}$ & $\mathbf{\frac{14674841399}{512424415641600}}$ & $\mathbf{\frac{5703709895459}{3767985230929920000}}$ & $\mathbf{\frac{37679857842043}{471817281090355200000}}$ \\
$4$  {\rule{0pt}{2.5ex}}{\rule[-1.8ex]{0pt}{0pt}}& $\mathbf{2}$ & $\frac{1}{10}$ & $\frac{919}{168480}$ & $\mathbf{\frac{16011391}{54854392320}}$ & $\mathbf{\frac{9016171639}{582300472320000}}$ & $\mathbf{\frac{143368101519407}{175211313238241280000}}$ & $\mathbf{\frac{13237209152580169}{306665505466027868160000}}$  \\
$5$  {\rule{0pt}{2.5ex}}{\rule[-1.8ex]{0pt}{0pt}}& $\mathbf{1}$ & $\frac{163}{3024}$ & $\frac{653}{221760}$ & $\frac{6208093}{39382640640}$ & $\mathbf{\frac{442442475179}{52900285261824000}}$ & $\mathbf{\frac{259645860580231}{587375069141532672000}}$ & $\mathbf{\frac{6359219722433607397}{272686967460391980367872000}}$ \\
$6$  {\rule{0pt}{2.5ex}}{\rule[-1.8ex]{0pt}{0pt}}& $\mathbf{\frac{1}{2}}$ & $\frac{29}{1008}$ & $\frac{88663}{56010240}$ & $\frac{5757089}{67781007360}$ & $\frac{1537940628689}{340912949465088000}$ & $\mathbf{\frac{229686916047007}{962777317187911680000}}$ & $\mathbf{\frac{43310941179948284069}{3440050974115714213871616000}}$ \\
$7$  {\rule{0pt}{2.5ex}}{\rule[-1.8ex]{0pt}{0pt}}& $\mathbf{\frac{1}{4}}$ & $\frac{1255}{82368}$ & $\frac{295133}{348281856}$ & $\frac{2598992519}{56936046182400}$ & $\frac{643391778377}{264869710110720000}$ & $\frac{11267167909498433}{87618715847436533760000}$ & $\mathbf{\frac{74408487930504838727}{10957199399035237866405888000}}$ \\
$8$  {\rule{0pt}{2.5ex}}{\rule[-1.8ex]{0pt}{0pt}}& $\mathbf{\frac{1}{8}}$ & $\frac{2477}{308880}$ & $\frac{1835863}{4063288320}$ & $\frac{1769539}{720943441920}$ & $\frac{127802659622551}{97895844856922112000}$& $\frac{2762333771707}{39907473380632166400}$ & $\frac{76034947449385560773}{20780895411963382160424960000}$ \\
$9$  {\rule{0pt}{2.5ex}}{\rule[-1.8ex]{0pt}{0pt}}& $\mathbf{\frac{1}{16}}$ & $\frac{39203}{9335040}$ & $\frac{12653167}{52718561280}$ & $\frac{6756335603}{516534771916800}$& $\frac{76170641989903}{108773160952135680000}$ & $\frac{46331482996262911}{1245354014578231266508800}$ & \textcolor{lightgray}{$\frac{7583038108310022233611}{3850996789771271334071894016000}$} \\
$10$  {\rule{0pt}{2.5ex}}{\rule[-1.8ex]{0pt}{0pt}}& $\mathbf{\frac{1}{32}}$ & $\frac{1363}{622336}$ & $\frac{5219989}{41079398400}$ & $\frac{2863703603}{410578921267200}$ & $\frac{364975959330977}{973541287193739264000}$ & \textcolor{lightgray}{$\frac{110488317513510709}{5533939090421837306265600}$} & \textcolor{lightgray}{$\frac{1597788327762805352162251}{1509590741590338362956182454272000}$} \\
$11$  {\rule{0pt}{2.5ex}}{\rule[-1.8ex]{0pt}{0pt}}& $\mathbf{\frac{1}{64}}$ & $\frac{308333}{270885888}$ & $\frac{644710519}{9612579225600}$ & $\frac{28221517763}{7606514751897600}$ & \textcolor{lightgray}{$\frac{26274127922961227}{131162562511011053568000}$} & \textcolor{lightgray}{$\frac{39074093749702556551}{3652399799678412622135296000}$} & \textcolor{lightgray}{$\frac{32893791972666409219189}{57890914971883041214200545280000}$} \\
\hline
\end{tabular}}
\rotatebox{90}{
    \begin{minipage}{\wd\mytable}
      \usebox{\mytable}
      \caption{\label{Tablall} Masur--Veech volumes $\pi^{-(6g - 6 + 2n)}MV_{g,n}$. We display in black the values that were honestly computed from the recursion, in bold the values used to determine the polynomials appearing in Conjecture~\ref{conjMV}, and in grey the values that the conjecture predicts. The first column reproduces Theorem~\ref{MV0n}. The first row is computed by the relation $MV_{g,0} = \frac{2^{4g-2}(4g-4)!}{(6g-g)!} H_{g,1}[1]$ proved in Lemma~\ref{dilatonthm0}.}
    \end{minipage}}
\end{figure}
\vfill

\begin{figure}[h!]
\begin{center}
\begin{tabular}{|c|ccccccc|}
\hline
$n$ \textbackslash \,\,$g$ {\rule{0pt}{3.2ex}}{\rule[-1.8ex]{0pt}{0pt}}& $0$ & $1$ & $2$ & $3$ & $4$ & $5$ & $6$ \\
\hline\hline
$0$ {\rule{0pt}{2.5ex}}{\rule[-1.8ex]{0pt}{0pt}}& - & - & $\frac{19}{6}$ & $\frac{24199}{8625}$ & $\frac{283794163}{105312050}$ & $\frac{180693680}{68465079}$ & $\frac{806379495590975}{309492103568838}$ \\
$1$ {\rule{0pt}{2.5ex}}{\rule[-1.8ex]{0pt}{0pt}}& - & - & $\frac{230}{87}$ & $\frac{529239}{205550}$  & $\frac{14053063}{5518645}$ & $\frac{533759417507}{210967972242}$ & $\frac{4346055982466800}{1725192578138153}$ \\
$2$ {\rule{0pt}{2.5ex}}{\rule[-1.8ex]{0pt}{0pt}}& - & $\frac{7}{3}$ & $\frac{8131}{3370}$ & $\frac{2843354}{1164495}$  & $\frac{11842209371}{4827273270}$ & $\frac{606925117339}{246886623873}$ & $\frac{122318875814791931}{49704331575032610}$ \\
$3$ {\rule{0pt}{2.5ex}}{\rule[-1.8ex]{0pt}{0pt}}& - & $\frac{47}{22}$ & $\frac{11041}{4785}$ & $\frac{73870699}{31157850}$ & $\frac{35221419482}{14674841399}$ & $\frac{82681229028041}{34222259372754}$ & $\frac{5057811587495459887}{2085014933689449405}$  \\
$4$ {\rule{0pt}{2.5ex}}{\rule[-1.8ex]{0pt}{0pt}}& $\frac{3}{2}$ & $\frac{44}{21}$ & $\frac{688823}{303270}$ & $\frac{187549387}{80056955}$ & $\frac{1414826039249}{595067328174}$ & $\frac{1031120131654286}{430104304558221}$ & $\frac{1339844245835171101}{555962784408367098}$ \\
$5$ {\rule{0pt}{2.5ex}}{\rule[-1.8ex]{0pt}{0pt}}& $\frac{5}{3}$ & $\frac{2075}{978}$ & $\frac{96716}{42445}$ & $\frac{87365995}{37248558}$ & $\frac{15788133716389}{6636637127685}$ & $\frac{1245335246460801}{519291721160462}$ & $\frac{321899861240823487478}{133543614171105755337}$  \\
$6$ {\rule{0pt}{2.5ex}}{\rule[-1.8ex]{0pt}{0pt}}& $\frac{11}{6}$ & $\frac{697}{319}$ & $\frac{8622217}{3723846}$ & $\frac{1433623484}{604494345}$ & $\frac{7380284015613}{3075881257378}$ & $\frac{18305424406953487}{7579668229551231}$ & $\frac{3150765025310943712637}{1299328235398448522070}$ \\
$7$ {\rule{0pt}{2.5ex}}{\rule[-1.8ex]{0pt}{0pt}}& $2$ & $\frac{17101}{7530}$ & $\frac{10506949}{4426995}$ & $\frac{12557689333}{5197985038}$ & $\frac{32906433038620}{13511227345917}$ & $\frac{165332043184123111}{67603007456990598}$ & \textcolor{lightgray}{$\frac{1276869600669686371105}{520859415513533871089}$} \\
$8$ {\rule{0pt}{2.5ex}}{\rule[-1.8ex]{0pt}{0pt}}& $\frac{13}{6}$ & $\frac{17630}{7431}$ & $\frac{44927707}{18358630}$ & $\frac{3273823127}{1322965425}$ & $\frac{1905176709014543}{766815957735306}$ & \textcolor{lightgray}{$\frac{931701551880070892}{374503401099176525}$}  & \textcolor{lightgray}{$\frac{32923598627691820002839}{13230080856193087574502}$} \\
$9$ {\rule{0pt}{2.5ex}}{\rule[-1.8ex]{0pt}{0pt}}& $\frac{7}{3}$ & $\frac{194829}{78406}$ & $\frac{480821458}{189797505}$ & $\frac{515867741141}{202690068090}$ & \textcolor{lightgray}{$\frac{3294839869674121}{1294900913828351}$} & \textcolor{lightgray}{$\frac{13416096198217292533}{5281789061573971854}$} & \textcolor{lightgray}{$\frac{403660475951758341605956}{159243800274510466905831}$} \\
$10$ {\rule{0pt}{2.5ex}}{\rule[-1.8ex]{0pt}{0pt}}& $\frac{5}{2}$ & $\frac{202415}{77691}$ & $\frac{905804827}{344519274}$ & \textcolor{lightgray}{$\frac{488680850166}{186140734195}$} & \textcolor{lightgray}{$\frac{658216299971112017}{251833411938374130}$} & \textcolor{lightgray}{$\frac{2586449275763662283}{994394857621596381}$}& \textcolor{lightgray}{$\frac{57921215793035879725637191}{22369036588679274930271514}$} \\
$11$ {\rule{0pt}{2.5ex}}{\rule[-1.8ex]{0pt}{0pt}}& $\frac{8}{3}$ & $\frac{5054467}{1849998}$ & \textcolor{lightgray}{$\frac{1761936475}{644710519}$} & \textcolor{lightgray}{$\frac{2297552653219}{846645532890}$} & \textcolor{lightgray}{$\frac{212103557000574050}{78822383768883681}$} & \textcolor{lightgray}{$\frac{208627514502680586639}{78148187499405113102}$} & \textcolor{lightgray}{$\frac{9156519282251402538004459}{3453848157129972968014845}$} \\
\hline
\end{tabular}
\caption{\label{TablSVA} Area Siegel--Veech constants $\pi^2 SV_{g,n}$. They are computed from Table~\ref{Tablall} thanks to Theorem~\ref{thGouj1}. Theorem~\ref{MV0n} gives the first column.}
\end{center}
\end{figure}

\begin{landscape}
\begin{figure}[h!]
\begin{center}
\begin{tabular}{|c|ccccccc|}
\hline
$n$ $\backslash$\,\,$d$ {\rule{0pt}{3.2ex}}{\rule[-1.8ex]{0pt}{0pt}}& $0$ & $1$ & $2$ & $3$ & $4$ & $5$ & $6$\\
\hline\hline
$3$ {\rule{0pt}{2.5ex}}{\rule[-1.8ex]{0pt}{0pt}}&  $\mathbf{1}$ & & & & & &\\
\hline
$4$ {\rule{0pt}{2.5ex}}{\rule[-1.8ex]{0pt}{0pt}}& $\mathbf{\frac{1}{2}}$ & {\footnotesize $\mathbf{3}$} & & & & &\\
\hline
$5$ {\rule{0pt}{2.5ex}}{\rule[-1.8ex]{0pt}{0pt}}& $\mathbf{\frac{3}{4}}$ & {\footnotesize $\mathbf{3}$} & {\footnotesize $\mathbf{15}$} & & & & \\ \hline
$6$ {\rule{0pt}{2.5ex}}{\rule[-1.8ex]{0pt}{0pt}}& $\mathbf{\frac{15}{8}}$ & $\frac{27}{4}$ & {\footnotesize $\mathbf{25}$} & {\footnotesize $\mathbf{105}$} & & & \\
\hline
$7$ {\rule{0pt}{2.5ex}}{\rule[-1.8ex]{0pt}{0pt}}& $\mathbf{\frac{105}{16}}$ & $\frac{45}{2}$ & $\mathbf{\frac{305}{4}}$ & $\mathbf{\frac{525}{2}}$ & {\footnotesize $\mathbf{945}$} & & \\
\hline
$8$ {\rule{0pt}{2.5ex}}{\rule[-1.8ex]{0pt}{0pt}}& $\mathbf{\frac{945}{32}}$ & $\frac{1575}{16}$ & $\frac{1275}{4}$ & {\footnotesize $\mathbf{1029}$} & $\mathbf{\frac{6615}{2}}$ & {\footnotesize $\mathbf{10395}$} & \\
\hline
$9$ {\rule{0pt}{2.5ex}}{\rule[-1.8ex]{0pt}{0pt}}& $\mathbf{\frac{10395}{64}}$ & $\frac{8505}{16}$ & $\frac{26775}{16}$ & $\mathbf{\frac{20853}{4}}$ & $\mathbf{\frac{32193}{2}}$ & {\footnotesize $\mathbf{48510}$} & {\footnotesize $\mathbf{135135}$} \\
\hline
$10$ {\rule{0pt}{2.5ex}}{\rule[-1.8ex]{0pt}{0pt}}& $\mathbf{\frac{135135}{128}}$ & $\frac{218295}{64}$ & $\frac{168525}{16}$ & $\frac{512883}{16}$ & $\mathbf{\frac{386271}{4}}$ & $\mathbf{\frac{571725}{2}}$ & {\footnotesize $\mathbf{810810}$}\\
\hline
$11$ {\rule{0pt}{2.5ex}}{\rule[-1.8ex]{0pt}{0pt}}& $\mathbf{\frac{2027025}{256}}$ & $\frac{405405}{16}$ & $\frac{4937625}{64}$ & $\frac{7388955}{32}$ & $\mathbf{\frac{10938159}{16}}$ & $\mathbf{\frac{3992175}{2}}$ & {\footnotesize $\mathbf{5675670}$} \\
\hline
$12$ {\rule{0pt}{2.5ex}}{\rule[-1.8ex]{0pt}{0pt}} & $\frac{54729675}{256}$ & $\frac{41216175}{64}$ & $\frac{60904305}{32}$ & $\frac{177968745}{32}$ & $\mathbf{\frac{256944105}{16}}$ & $\mathbf{\frac{182035425}{4}}$ & $\mathbf{\frac{497972475}{4}}$ \\
\hline
$13$ {\rule{0pt}{2.5ex}}{\rule[-1.8ex]{0pt}{0pt}} & $\frac{516891375}{256}$ & $\frac{1543917375}{256}$ & $\frac{564729165}{32}$ & $\frac{816623775}{16}$ & {\footnotesize $\mathbf{146029950}$} & $\mathbf{\frac{6588578139}{16}}$ & $\mathbf{\frac{9070035975}{8}}$ \\
\hline
$14$ {\rule{0pt}{2.5ex}}{\rule[-1.8ex]{0pt}{0pt}} & $\frac{21606059475}{1024}$ & $\frac{16023632625}{256}$ & $\frac{46514953485}{256}$ & $\frac{16675523865}{32}$ & $\frac{23672927025}{16}$ &  $\mathbf{\frac{66411007767}{16}}$ & $\mathbf{\frac{45759368655}{4}}$ \\
\hline
$15$ {\rule{0pt}{2.5ex}}{\rule[-1.8ex]{0pt}{0pt}} & $\frac{123743795175}{512}$ & $\frac{730022918625}{1024}$ & $\frac{1052618271705}{512}$ & $\frac{1499012960445}{256}$ & $\frac{264272735725}{16}$ & {\footnotesize $\mathbf{46109143548}$} & $\mathbf{\frac{2030915622645}{16}}$ \\
\hline
$16$ {\rule{0pt}{2.5ex}}{\rule[-1.8ex]{0pt}{0pt}} & $\frac{12333131585775}{4096}$ & $\frac{9051629461875}{1024}$ & $\frac{811209323925}{32}$ & $\frac{36751659059115}{512}$ & $\frac{51544870752375}{256}$ & $\frac{17905106587095}{32}$ & $\mathbf{\frac{49185509458365}{32}}$ \\
\hline
$17$ {\rule{0pt}{2.5ex}}{\rule[-1.8ex]{0pt}{0pt}} & $\frac{166022925193125}{4096}$ & $\frac{485419409850375}{4096}$ & $\frac{346386381315975}{1024}$ & $\frac{487957139745075}{512}$ & $\frac{340526280419325}{128}$ & $\frac{1884818339393535}{256}$ & $\mathbf{\frac{1291176578473425}{64}}$ \\
\hline
$18$ {\rule{0pt}{2.5ex}}{\rule[-1.8ex]{0pt}{0pt}} & $\frac{9605612100459375}{16384}$ & $\frac{6996680418853125}{4096}$ & $\frac{19889546438136375}{4096}$ & $\frac{13948724683018125}{1024}$ & $\frac{19385928833646375}{512}$ & $\frac{1670327729904015}{16}$ & $\frac{73041994202191875}{256}$ \\
\hline
$19$ {\rule{0pt}{2.5ex}}{\rule[-1.8ex]{0pt}{0pt}} & $\frac{18570850060888125}{2048}$ & $\frac{431541017698415625}{16384}$ & $\frac{611196796805970375}{8192}$ & $\frac{854034802440694875}{4096}$ & $\frac{295622640130442625}{512}$ & $\frac{50770632935200815}{32}$ & $\frac{2214495362565119025}{512}$ \\
\hline
$20$ {\rule{0pt}{2.5ex}}{\rule[-1.8ex]{0pt}{0pt}} & $\frac{9786837982088041875}{65536}$ & $\frac{7087874439905634375}{16384}$ & $\frac{10007297452695918375}{8192}$ & $\frac{27873406624278340125}{8192}$ & $\frac{38464769339580709125}{4096}$ & $\frac{26342984670839932095}{1024}$ & $\frac{71637824483699384925}{1024}$ \\
\hline
\end{tabular}
\caption{\label{TablGnd} Values of $\pi^{-2(n - 3 - d)}H_{n}[d]$ computed from the recursion of Section~\ref{Genus01row}, for $n \leq 20$. We display in bold the values that were used to determine the polynomials $P_d(n)$ in Table~\ref{Tabl0}.}\end{center}
\end{figure}

\begin{figure}[h!]
\begin{center}
\begin{tabular}{|c|ccccccc|}
\hline
$n$ $\backslash$\,\,$d$ {\rule{0pt}{3.2ex}}{\rule[-1.8ex]{0pt}{0pt}}& $7$ & $8$ & $9$ & $10$ & $11$ & $12$ & $13$ \\
\hline\hline
$10$ {\rule{0pt}{2.5ex}}{\rule[-1.8ex]{0pt}{0pt}}
	& {\footnotesize $\mathbf{2027025}$} & & & & & & \\
\hline
$11$ {\rule{0pt}{2.5ex}}{\rule[-1.8ex]{0pt}{0pt}}
	& $\mathbf{\frac{30405375}{2}}$ & {\footnotesize $\mathbf{34459425}$} & & & & & \\
\hline
$12$ {\rule{0pt}{2.5ex}}{\rule[-1.8ex]{0pt}{0pt}}
	& $\mathbf{\frac{34459425}{512}}$ & $\mathbf{\frac{631756125}{2}}$ & {\footnotesize $654729075$} & & & & \\
\hline
$13$ {\rule{0pt}{2.5ex}}{\rule[-1.8ex]{0pt}{0pt}}
	& $\mathbf{\frac{654729075}{1024}}$ & $\mathbf{\frac{11952825885}{4}}$ & {\footnotesize $7202019825$} & {\footnotesize $13749310575$} & & & \\
\hline
$14$ {\rule{0pt}{2.5ex}}{\rule[-1.8ex]{0pt}{0pt}}
	& $\mathbf{\frac{13749310575}{2048}}$ & $\mathbf{\frac{245021061285}{8}}$ & $\frac{311520093885}{4}$ & {\footnotesize $178741037475$} & {\footnotesize $316234143225$} & & \\
\hline
$15$ {\rule{0pt}{2.5ex}}{\rule[-1.8ex]{0pt}{0pt}}
	& $\mathbf{\frac{316234143225}{4096}}$ & $\mathbf{\frac{2737188448995}{8}}$ & {\footnotesize $891435459915$} & $\frac{8758310836275}{4}$ & $\frac{9592435677825}{2}$ & {\footnotesize $7905853580625$}  & \\
\hline
$16$ {\rule{0pt}{2.5ex}}{\rule[-1.8ex]{0pt}{0pt}}
	& $\mathbf{\frac{7905853580625}{8192}}$ & $\mathbf{\frac{66469370924775}{16}}$ & $\frac{87660410042205}{8}$ & $\frac{111209007034125}{4}$ & $\frac{132080460486975}{2}$ & $\frac{276704875321875}{2}$ & {\footnotesize $213458046676875$} \\
\hline
$17$ {\rule{0pt}{2.5ex}}{\rule[-1.8ex]{0pt}{0pt}}
	& $\mathbf{\frac{213458046676875}{16384}}$ & $\mathbf{\frac{1746012465049635}{32}}$ & $\frac{1158795684753525}{8}$ & $\frac{2992184547677325}{8}$ & $\frac{1849593268648125}{2}$ & {\footnotesize $2126674613188125$} & {\footnotesize $4269160933537500$} \\
\hline
$18$ {\rule{0pt}{2.5ex}}{\rule[-1.8ex]{0pt}{0pt}}
	& $\mathbf{\frac{6190283353629375}{32768}}$ & $\mathbf{\frac{49363995431765475}{64}}$ & $\frac{65744175999773079}{32}$ & $\frac{42871352128175025}{8}$ & $\frac{108418387292179965}{8}$ & $\frac{65350162166664375}{2}$ & {\footnotesize $72817654989704625$} \\
\hline
$19$ {\rule{0pt}{2.5ex}}{\rule[-1.8ex]{0pt}{0pt}}
	& $\mathbf{\frac{191898783962510625}{65536}}$ & $\mathbf{\frac{2990161669353791775}{256}}$ & $\frac{498588953042078271}{16}$ & $\frac{2616690164139006165}{32}$ & $\frac{3354058884543031665}{16}$ & $\frac{4155321837251710125}{8}$ & {\footnotesize $1221898913059324500$} \\
\hline
$20$ {\rule{0pt}{2.5ex}}{\rule[-1.8ex]{0pt}{0pt}}
	& $\mathbf{\frac{6332659870762850625}{131072}}$ & $\frac{96601743007909225875}{512}$ & $\frac{128930498895120139989}{256}$ & $\frac{42442944415051043115}{32}$ & $\frac{27436266153939996045}{8}$ & $\frac{138292854856604823375}{16}$ & $\frac{167929854347224344375}{8}$ \\
\hline
\end{tabular}
\end{center}
\end{figure}

\begin{figure}[h!]
\begin{center}
\begin{tabular}{|c|cccc|}
\hline
$n$ $\backslash$\,\,$d$ {\rule{0pt}{3.2ex}}{\rule[-1.8ex]{0pt}{0pt}}& $14$ & $15$ & $16$ & $17$ \\
\hline\hline
$17$ {\rule{0pt}{2.5ex}}{\rule[-1.8ex]{0pt}{0pt}}
	& {\footnotesize $6190283353629375$} & & &  \\
\hline
$18$ {\rule{0pt}{2.5ex}}{\rule[-1.8ex]{0pt}{0pt}}
	& {\footnotesize $140313089348932500$} & {\footnotesize $191898783962510625$} & &  \\
\hline
$19$ {\rule{0pt}{2.5ex}}{\rule[-1.8ex]{0pt}{0pt}}
	& $\frac{5282787813987308625}{2}$ & $\frac{9786837982088041875}{2}$ & {\footnotesize $6332659870762850625$} & \\
\hline
$20$ {\rule{0pt}{2.5ex}}{\rule[-1.8ex]{0pt}{0pt}}
	& $\frac{192834342120022477875}{4}$ & $\frac{404714535376934908125}{4}$ & $\frac{360961612633482485625}{2}$ & {\footnotesize $221643095476699771875$} \\
\hline
\end{tabular}
\end{center}
\caption{\label{TablGnd2} (Continued) Values of $\pi^{-2(n - 3 - d)}H_{n}[d]$ computed from the recursion of Section~\ref{Genus01row}, for $n \leq 20$.}
\end{figure}

\begin{figure}[h!]
\begin{center}
\begin{tabular}{|c|ccccccccc|}
\hline
$n$ $\backslash$\,\,$d$ {\rule{0pt}{3.2ex}}{\rule[-1.8ex]{0pt}{0pt}}& $0$ & $1$ & $2$ & $3$ & $4$ & $5$ & $6$ & $7$ & $8$ \\
\hline\hline
$1$ {\rule{0pt}{2.5ex}}{\rule[-1.8ex]{0pt}{0pt}}& $\mathbf{\frac{1}{12}}$ & $\frac{1}{8}$ & & & & & & & \\
\hline
$2$  {\rule{0pt}{2.5ex}}{\rule[-1.8ex]{0pt}{0pt}}& $\mathbf{\frac{1}{16}}$ & $\mathbf{\frac{1}{4}}$ & $\frac{5}{8}$ & & & & & & \\
\hline 
$3$ {\rule{0pt}{2.5ex}}{\rule[-1.8ex]{0pt}{0pt}}& $\frac{11}{96}$ & $\mathbf{\frac{3}{8}}$ & $\mathbf{\frac{65}{48}}$ & $\frac{35}{8}$ & & & & & \\ 
\hline
$4$ {\rule{0pt}{2.5ex}}{\rule[-1.8ex]{0pt}{0pt}}& $\frac{21}{64}$ & $\mathbf{\frac{33}{32}}$ & $\mathbf{\frac{305}{96}}$ & $\mathbf{\frac{175}{16}}$ & $\frac{315}{8}$ & & & & \\
\hline
$5$ {\rule{0pt}{2.5ex}}{\rule[-1.8ex]{0pt}{0pt}}& $\frac{163}{128}$ & $\frac{63}{16}$ & $\mathbf{\frac{745}{64}}$ & $\mathbf{\frac{1127}{32}}$ & $\frac{945}{8}$ & $\frac{3465}{8}$ & & & \\
\hline
$6$ {\rule{0pt}{2.5ex}}{\rule[-1.8ex]{0pt}{0pt}}& $\frac{1595}{256}$ & $\frac{2445}{128}$ & $\mathbf{\frac{21275}{384}}$ & $\mathbf{\frac{10283}{64}}$ & $\mathbf{\frac{15477}{32}}$ & $\mathbf{\frac{12705}{8}}$ & $\frac{45045}{8}$ & & \\
\hline
$7$ {\rule{0pt}{2.5ex}}{\rule[-1.8ex]{0pt}{0pt}}& $\frac{18825}{512}$ & $\frac{14355}{128}$ & $\frac{82375}{256}$ & $\mathbf{\frac{116907}{128}}$ & $\mathbf{\frac{84279}{32}}$ & $\mathbf{\frac{252945}{32}}$ & $\frac{405405}{16}$ & $\frac{675675}{8}$ & \\
\hline
$8$ {\rule{0pt}{2.5ex}}{\rule[-1.8ex]{0pt}{0pt}}& $\frac{260085}{1024}$ & $\frac{395325}{512}$ & $\frac{1126475}{512}$ & $\mathbf{\frac{1578339}{256}}$ & $\mathbf{\frac{2222919}{128}}$ & $\mathbf{\frac{1600005}{32}}$  & $\frac{2387385}{16}$ & $\frac{7432425}{16}$ & $\frac{11486475}{8}$ \\
\hline
$9$ {\rule{0pt}{2.5ex}}{\rule[-1.8ex]{0pt}{0pt}}& $\frac{4116315}{2048}$ & $\frac{780255}{128}$ & $\frac{17702825}{1024}$ & $\frac{24596187}{512}$ & $\mathbf{\frac{17069781}{128}}$ & $\mathbf{\frac{47966325}{128}}$ & $\frac{34462285}{32}$ & $\frac{25450425}{8}$ & $\frac{19144125}{2}$ \\
\hline
$10$ {\rule{0pt}{2.5ex}}{\rule[-1.8ex]{0pt}{0pt}}& $\frac{73417995}{4096}$ & $\frac{11140505}{2048}$ & $\frac{314129025}{2048}$ & $\frac{433887111}{1024}$ & $\mathbf{\frac{596350053}{512}}$ & $\mathbf{\frac{412393245}{128}}$ & $\frac{1157510783}{128}$ & $\frac{103528425}{4}$ & $\frac{602657055}{8}$ \\
\hline
$11$ {\rule{0pt}{2.5ex}}{\rule[-1.8ex]{0pt}{0pt}}& $\frac{1456873425}{8192}$ & $\frac{1101269925}{2048}$ & $\frac{6209382375}{4096}$ & $\frac{8539707015}{2048}$ & $\frac{5827734675}{512}$ & $\mathbf{\frac{15945353325}{512}}$ & $\frac{22022084487}{256}$ & $\frac{30846831015}{128}$ & $\frac{21914407515}{32}$ \\
\hline
$12$ {\rule{0pt}{2.5ex}}{\rule[-1.8ex]{0pt}{0pt}}& $\frac{31832972325}{16384}$ & $\frac{48076823025}{8192}$ & $\frac{135269701875}{8192}$ & $\frac{185426394615}{4096}$ & $\frac{251769660615}{2048}$ & $\mathbf{\frac{170859190425}{512}}$ & $\frac{7290122177}{8}$ & $\frac{643590359685}{256}$ & $\frac{898172490215}{128}$ \\
\hline
$13$ {\rule{0pt}{2.5ex}}{\rule[-1.8ex]{0pt}{0pt}}& $\frac{759408232275}{32768}$ & $\frac{286496750925}{4096}$ & $\frac{3219307441625}{16384}$ & $\frac{4401709330095}{8192}$ & $\frac{2976421242465}{2048}$ & $\frac{8031549267225}{2048}$ & $\frac{339777246393}{32}$ & $\frac{3707081523495}{128}$ & $\frac{10205899285855}{128}$ \\
\hline
$14$ {\rule{0pt}{2.5ex}}{\rule[-1.8ex]{0pt}{0pt}}& $\frac{19639202658075}{65536}$ & $\frac{29616921058725}{32768}$ & $\frac{83094183797625}{32768}$ & $\frac{113381114321475}{16384}$ & $\frac{152856379904085}{8192}$ & $\frac{102641114315025}{2048}$ & $\frac{276095298454833}{2048}$ & $\frac{46658673053055}{128}$ & $\frac{127090011696915}{128}$ \\
\hline
\end{tabular}
\caption{\label{TablG1nd} Values of $\pi^{-2(n - d)}H_{1,n}[d]$ computed from the Virasoro constraints. We display in bold the values that were used to determine $\rho_d$ and $R_d(n)$ in Table~\ref{Tablunu}.}
\end{center}
\end{figure}

\begin{figure}[h!]
\begin{center}
\begin{tabular}{|c|cccccc|}
\hline
$n$ $\backslash$\,\,$d$ {\rule{0pt}{3.2ex}}{\rule[-1.8ex]{0pt}{0pt}}& $9$ & $10$ & $11$ & $12$ & $13$ & $14$ \\
\hline\hline
$9$ {\rule{0pt}{2.5ex}}{\rule[-1.8ex]{0pt}{0pt}}& $\frac{218243025}{8}$ & & & & & \\
\hline
$10$ {\rule{0pt}{2.5ex}}{\rule[-1.8ex]{0pt}{0pt}}& {\footnotesize $218243025$} & $\frac{4583103525}{8}$ & & & & \\
\hline
$11$ {\rule{0pt}{2.5ex}}{\rule[-1.8ex]{0pt}{0pt}}& $\frac{31296049785}{16}$ & $\frac{87078966975}{16}$ & $\frac{105411381075}{8}$ & & & \\
\hline
$12$ {\rule{0pt}{2.5ex}}{\rule[-1.8ex]{0pt}{0pt}}& $\frac{631580764815}{32}$ & $\frac{1764494857125}{32}$ & $\frac{2354187510675}{16}$ & $\frac{2635284526875}{8}$ & & \\
\hline
$13$ {\rule{0pt}{2.5ex}}{\rule[-1.8ex]{0pt}{0pt}}& $\frac{28319986049355}{128}$ & $\frac{39312334336275}{64}$ & $\frac{53584118713125}{32}$ & $\frac{34258698849375}{8}$ & $\frac{71152682225625}{8}$ & \\
\hline
$14$ {\rule{0pt}{2.5ex}}{\rule[-1.8ex]{0pt}{0pt}}& $\frac{348671089506285}{128}$ & $\frac{479892186999225}{64}$ & $\frac{1311975146807325}{64}$ & $\frac{1741923072264575}{32}$ & $\frac{1067290233384375}{8}$ & $\frac{2063427784543125}{8}$ \\
\hline
\end{tabular}
\caption{\label{TablG1nd2} Values of $\pi^{-2(n - d)}H_{1,n}[d]$ computed from the Virasoro constraints (continued).}
\end{center}
\end{figure}

\begin{figure}[h!]
\begin{center}
\begin{tabular}{|c|cccccccccccc|}
\hline
$n$ $\backslash$\,\,$d$ {\rule{0pt}{3.2ex}}{\rule[-1.8ex]{0pt}{0pt}}& $0$ & $1$ & $2$ & $3$ & $4$ & $5$ & $6$ & $7$ & $8$ & $9$ & $10$ & $11$ \\
\hline\hline
$1$ {\rule{0pt}{2.5ex}}{\rule[-1.8ex]{0pt}{0pt}}& $\frac{29}{2560}$ & $\frac{1}{32}$ & $\frac{119}{1152}$ & $\frac{35}{96}$ & $\frac{105}{128}$ & & & & & & & \\
$2$ {\rule{0pt}{2.5ex}}{\rule[-1.8ex]{0pt}{0pt}}& $\frac{337}{9216}$ & $\frac{261}{2560}$ & $\frac{75}{256}$ & $\frac{119}{128}$ & $\frac{105}{32}$ & $\frac{1155}{128}$ & & & & & & \\
$3$ {\rule{0pt}{2.5ex}}{\rule[-1.8ex]{0pt}{0pt}}& $\frac{319}{2048}$ & $\frac{337}{768}$ & $\frac{1399}{1152}$ & $\frac{2695}{768}$ & $\frac{2779}{256}$ & $\frac{9625}{256}$ & $\frac{15015}{128}$ & & & & & \\
$4$ {\rule{0pt}{2.5ex}}{\rule[-1.8ex]{0pt}{0pt}}& $\frac{10109}{12288}$ & $\frac{4785}{2048}$ & $\frac{19583}{3072}$ & $\frac{567}{32}$ & $\frac{26161}{512}$ & $\frac{79695}{512}$ & $\frac{135135}{256}$ & $\frac{225225}{128}$ & & & & \\
$5$ {\rule{0pt}{2.5ex}}{\rule[-1.8ex]{0pt}{0pt}}& $\frac{42445}{8192}$ & $\frac{30327}{2048}$ & $\frac{185063}{4608}$ & $\frac{105007}{960}$ & $\frac{1559847}{5120}$ & $\frac{897039}{1024}$ & $\frac{1354353}{512}$ & $\frac{1126125}{128}$ & $\frac{3828825}{128}$ & & & \\
$6$ {\rule{0pt}{2.5ex}}{\rule[-1.8ex]{0pt}{0pt}}& $\frac{620641}{16384}$ & $\frac{891345}{8192}$ & $\frac{1205735}{4096}$ & $\frac{24343627}{30720}$ & $\frac{4426961}{2048}$ & $\frac{12338073}{2048}$ & $\frac{53113775}{3072}$ & $\frac{26591565}{512}$ & $\frac{21696675}{128}$ & $\frac{72747675}{128}$ & & \\
$7$ {\rule{0pt}{2.5ex}}{\rule[-1.8ex]{0pt}{0pt}}& $\frac{10329655}{32768}$ & $\frac{1861923}{2048}$ & $\frac{15099635}{6144}$ & $\frac{26907839}{4096}$ & $\frac{362089239}{20480}$ & $\frac{197828785}{4096}$ & $\frac{827349809}{6144}$ & $\frac{49424375}{128}$ & $\frac{591425835}{512}$ & $\frac{945719775}{256}$ & $\frac{1527701175}{128}$ & \\
$8$ {\rule{0pt}{2.5ex}}{\rule[-1.8ex]{0pt}{0pt}}& $\frac{192765615}{65536}$ & $\frac{278900685}{32768}$ & $\frac{1130917375}{49152}$ & $\frac{250801845}{4096}$ & $\frac{1337766877}{8192}$ & $\frac{3604505311}{8192}$ & $\frac{4929397655}{4096}$ & $\frac{6875894025}{2048}$ & $\frac{4926846925}{512}$ & $\frac{916620705}{32}$ & $\frac{22915517625}{256}$ & $\frac{35137127025}{128}$ \\
\hline
\end{tabular}
\caption{\label{TablG2nd} Values of $\pi^{-2(n + 3 - d)}\,H_{2,n}[d]$ computed from the Virasoro constraints.}
\end{center}
\end{figure}

\begin{figure}[h!]
\begin{center}
{\tabcolsep=0.11cm
\begin{tabular}{|c|ccccccccccccc|}
\hline
$n$ $\backslash$\,\,$d$ {\rule{0pt}{3.2ex}}{\rule[-1.8ex]{0pt}{0pt}}& $0$ & $1$ & $2$ & $3$ & $4$ & $5$ & $6$ & $7$ & $8$ & $9$ & $10$ & $11$ & $12$ \\
\hline\hline
$1$ {\rule{0pt}{2.5ex}}{\rule[-1.8ex]{0pt}{0pt}}& $\frac{20555}{1327104}$ & $\frac{575}{14336}$ & $\frac{8099}{73728}$ & $\frac{56749}{184320}$ & $\frac{8203}{9218}$ & $\frac{17479}{6144}$ & $\frac{5005}{512}$ & $\frac{25025}{1024}$ & & & & & \\
$2$ {\rule{0pt}{2.5ex}}{\rule[-1.8ex]{0pt}{0pt}}& $\frac{77633}{884736}$ & $\frac{102775}{442368}$ & $\frac{1920563}{3096576}$ & $\frac{624463}{368640}$ & $\frac{48189}{10240}$ & $\frac{500489}{36864}$ & $\frac{87087}{2048}$ & $\frac{75075}{512}$ & $\frac{425425}{1024}$ & & & & \\
$3$ {\rule{0pt}{2.5ex}}{\rule[-1.8ex]{0pt}{0pt}}& $\frac{1038595}{1769472}$ & $\frac{77633}{49152}$ & $\frac{11069909}{2654208}$ & $\frac{8245679}{737280}$ & $\frac{3737107}{122880}$ & $\frac{6218927}{73728}$ & $\frac{26863837}{110592}$ & $\frac{575575}{768}$ & $\frac{15740725}{6144}$ & $\frac{8083075}{1024}$ & & & \\
$4$ {\rule{0pt}{2.5ex}}{\rule[-1.8ex]{0pt}{0pt}}& $\frac{16011391}{3538944}$ & $\frac{7270165}{589824}$ & $\frac{172014797}{5308416}$ & $\frac{379718161}{4423680}$ & $\frac{2354953}{10240}$ & $\frac{92087039}{147456}$ & $\frac{383277895}{221184}$ & $\frac{183688505}{36864}$ & $\frac{186931745}{12288}$ & $\frac{105079975}{2048}$ & $\frac{169744575}{1024}$ & & \\
$5$ {\rule{0pt}{2.5ex}}{\rule[-1.8ex]{0pt}{0pt}}& $\frac{31040465}{786432}$ & $\frac{16011391}{147456}$ & $\frac{1008891097}{3538944}$ & $\frac{6627865109}{8847360}$ & $\frac{2277007409}{1146880}$ & $\frac{3658297225}{688128}$ & $\frac{2129633311}{147456}$ & $\frac{739407955}{18432}$ & $\frac{8496138365}{73728}$ & $\frac{1430704275}{4096}$ & $\frac{1188212025}{1024}$ & $\frac{3904125225}{1024}$ & \\
$6$ {\rule{0pt}{2.5ex}}{\rule[-1.8ex]{0pt}{0pt}}& $\frac{201498115}{524288}$ & $\frac{279364185}{262144}$ & $\frac{2200898005}{786432}$ & $\frac{8629787107}{1179648}$ & $\frac{9925860397}{516096}$ & $\frac{210466671811}{4128768}$ & $\frac{156612960491}{1146880}$ & $\frac{18256755985}{49152}$ & $\frac{456931637591}{442368}$ & $\frac{72883701605}{24576}$ & $\frac{36630879285}{4096}$ & $\frac{29931626725}{1024}$ & $\frac{97603130625}{1024}$ \\
\hline
\end{tabular}
}
\caption{\label{TablG3nd} Values of $\pi^{-2(n + 6 - d)}\,H_{3,n}[d]$ computed from the Virasoro constraints.}
\end{center}
\end{figure}

\begin{figure}[h!]
\begin{center}
\begin{tabular}{|c|cccccccc|}
\hline
$n$ $\backslash$\,\,$d$ {\rule{0pt}{3.2ex}}{\rule[-1.8ex]{0pt}{0pt}}& $0$ & $1$ & $2$ & $3$ & $4$ & $5$ & $6$ & $7$ \\
\hline\hline
$1$ {\rule{0pt}{2.5ex}}{\rule[-1.8ex]{0pt}{0pt}}& $\frac{1103729}{18874368}$ & $\frac{2106241}{14155776}$ & $\frac{249375997}{637009920}$ & $\frac{55140311}{53084160}$ & $\frac{7636607}{2752512}$ & $\frac{26723675}{3538944}$ & $\frac{111772375}{5308416}$ & $\frac{3379805}{55296}$ \\
$2$ {\rule{0pt}{2.5ex}}{\rule[-1.8ex]{0pt}{0pt}}& $\frac{160909109}{339738624}$ & $\frac{7726103}{6291456}$ & $\frac{209435}{65536}$ & $\frac{8920693349}{1061683200}$ & $\frac{1834236191}{82575360}$ & $\frac{4903282417}{82575360}$ & $\frac{571236653}{3538944}$ & $\frac{793954733}{1769472}$ \\
$3$ {\rule{0pt}{2.5ex}}{\rule[-1.8ex]{0pt}{0pt}}& $\frac{14674841399}{3397386240}$ & $\frac{160909109}{14155776}$ & $\frac{14996838119}{509607936}$ & $\frac{2173908251}{28311552}$ & $\frac{62374893397}{309657600}$ & $\frac{17585088257}{33030144}$ & $\frac{1058028040511}{743178240}$ & $\frac{1140294935}{294912}$ \\
$4$ {\rule{0pt}{2.5ex}}{\rule[-1.8ex]{0pt}{0pt}}& $\frac{99177888029}{2264924160}$ & $\frac{14674841399}{125829120}$ & $\frac{399887731}{1327104}$ & $\frac{110651102527}{141557760}$ & $\frac{2524091828537}{1238630400}$ & $\frac{756362323453}{141557760}$ & $\frac{466656238555}{33030144}$ & $\frac{29716909709}{786432}$ \\
$5$ {\rule{0pt}{2.5ex}}{\rule[-1.8ex]{0pt}{0pt}}& $\frac{442442475179}{905969664}$ & $\frac{99177888029}{75497472}$ & $\frac{2305616474459}{679477248}$ & $\frac{1240791326569}{141557760}$ & $\frac{3575909389403}{157286400}$ & $\frac{117418444025993}{1981808640}$ & $\frac{153980456785597}{990904320}$ & $\frac{969929385535}{2359296}$ \\
\hline
\end{tabular}
\caption{\label{TablG4ndcon} Values of $\pi^{-2n + 9 - d)}\,H_{4,n}[d]$ computed from the Virasoro constraints.}
\end{center}
\end{figure}

\begin{figure}[h!]
\begin{center}
\begin{tabular}{|c|ccccccc|}
\hline
$n$ $\backslash$\,\,$d$ {\rule{0pt}{3.2ex}}{\rule[-1.8ex]{0pt}{0pt}}& $8$ & $9$ & $10$ & $11$ & $12$ & $13$ & $14$ \\
\hline\hline
$1$ {\rule{0pt}{2.5ex}}{\rule[-1.8ex]{0pt}{0pt}}& $\frac{28503475}{147456}$ & $\frac{8083075}{12288}$ & $\frac{56581525}{32768}$ & & & & \\
$2$ {\rule{0pt}{2.5ex}}{\rule[-1.8ex]{0pt}{0pt}}& $\frac{1147322605}{884736}$ & $\frac{66281215}{16384}$ & $\frac{56581525}{4096}$ & $\frac{1301375075}{32768}$ & & & \\
$3$ {\rule{0pt}{2.5ex}}{\rule[-1.8ex]{0pt}{0pt}}& $\frac{56993061697}{5308416}$ & $\frac{4563242255}{147456}$ & $\frac{6257916665}{65536}$ & $\frac{63767378675}{196608}$ & $\frac{32534376875}{32768}$ & & \\
$4$ {\rule{0pt}{2.5ex}}{\rule[-1.8ex]{0pt}{0pt}}& $\frac{272361575941}{2654208}$ & $\frac{42015685283}{147456}$ & $\frac{322445178055}{393216}$ & $\frac{986702581865}{393216}$ & $\frac{553084406875}{65536}$ & $\frac{878428175625}{32768}$ & \\
$5$ {\rule{0pt}{2.5ex}}{\rule[-1.8ex]{0pt}{0pt}}& $\frac{7785790643045}{7077888}$ & $\frac{3525813704023}{1179648}$ & $\frac{2177058238355}{262144}$ & $\frac{18781333535965}{786432}$ & $\frac{9519558673625}{131072}$ & $\frac{7905853580625}{32768}$ & $\frac{25474417093125}{32768}$ \\
\hline
\end{tabular}
\caption{\label{TablG4nd} Values of $\pi^{-2(n + 9 - d)}\,H_{4,n}[d]$ computed from the Virasoro constraints (continued).\vspace{-.3cm}}
\end{center}
\end{figure}

\begin{figure}[h!]
\begin{center}
\begin{tabular}{|c|ccccccccc|}
\hline
$n$ $\backslash$\,\,$d$ {\rule{0pt}{3.2ex}}{\rule[-1.8ex]{0pt}{0pt}}& $0$ & $1$ & $2$ & $3$ & $4$ & $5$ & $6$ & $7$ & $8$ \\
\hline\hline
$1$ {\rule{0pt}{2.5ex}}{\rule[-1.8ex]{0pt}{0pt}}& $\frac{35161328707}{81537269760}$ & $\frac{22821693}{20971520}$ & $\frac{7644175835}{2717908992}$ & $\frac{109134615047}{14948499456}$ & $\frac{35307225403}{1857945600}$ & $\frac{296352439757}{5945425920}$ & $\frac{87114140087}{660602880}$ & $\frac{4992748537}{14155776}$ & $\frac{244417472273}{254803968}$ \\
$2$ {\rule{0pt}{2.5ex}}{\rule[-1.8ex]{0pt}{0pt}}& $\frac{27431847097}{6039797760}$ & $\frac{35161328707}{3019898880}$ & $\frac{1463192937811}{48922361856}$ & $\frac{210033343633}{2717908992}$ & $\frac{174554274549551}{871995801600}$ & $\frac{12396787304077}{23781703680}$ & $\frac{2438015036977}{1783627776}$ & $\frac{68208841817}{18874368}$ & $\frac{547141537559}{56623104}$ \\
$3$ {\rule{0pt}{2.5ex}}{\rule[-1.8ex]{0pt}{0pt}}& $\frac{5703709895459}{108716359680}$ & $\frac{27431847097}{201326592}$ & $\frac{3793692043765}{10871635968}$ & $\frac{14642582778043}{16307453952}$ & $\frac{68867984889821}{29727129600}$ & $\frac{8918977503973}{1486356480}$ & $\frac{2787426440647199}{178362777600}$ & $\frac{23198438807111}{566231040}$ & $\frac{551941078807889}{5096079360}$ \\
$4$ {\rule{0pt}{2.5ex}}{\rule[-1.8ex]{0pt}{0pt}}& $\frac{143368101519407}{217432719360}$ & $\frac{62740808850049}{36238786560}$ & $\frac{288900738269347}{65229815808}$ & $\frac{617405851484437}{54358179840}$ & $\frac{3086856384144917}{105696460800}$ & $\frac{7165942204049803}{95126814720}$ & $\frac{19890284967197827}{101921587200}$ & $\frac{1499002879859}{2949120}$ & $\frac{4527331166325601}{3397386240}$ \\
\hline
\end{tabular}
\caption{\label{TablG5nd} Values of $\pi^{-2n + 12 - d)}\,H_{5,n}[d]$ computed from the Virasoro constraints.
\vspace{-.3cm}}
\end{center}
\end{figure}

\begin{figure}[h!]
\begin{center}
\begin{tabular}{|c|cccccccc|}
\hline
$n$ $\backslash$\,\,$d$ {\rule{0pt}{3.2ex}}{\rule[-1.8ex]{0pt}{0pt}}& $9$ & $10$ & $11$ & $12$ & $13$ & $14$ & $15$ & $16$ \\
\hline\hline
$1$ {\rule{0pt}{2.5ex}}{\rule[-1.8ex]{0pt}{0pt}}& $\frac{4730354057}{1769472}$ & $\frac{18324331025}{2359296}$ & $\frac{115301831645}{4718592}$ & $\frac{32534376875}{393216}$ & $\frac{58561878375}{262144}$ & & & \\
$2$ {\rule{0pt}{2.5ex}}{\rule[-1.8ex]{0pt}{0pt}}& $\frac{247779861745}{9437184}$ & $\frac{689116092665}{9437184}$ & $\frac{1993595068465}{9437184}$ & $\frac{344864394875}{524288}$ & $\frac{292809391875}{131072}$ & $\frac{1698294472875}{262144}$ & & \\
$3$ {\rule{0pt}{2.5ex}}{\rule[-1.8ex]{0pt}{0pt}}& $\frac{81986849479019}{283115520}$ & $\frac{44535267213073}{56623104}$ & $\frac{123749786026867}{56623104}$ & $\frac{59509093518875}{9437184}$ & $\frac{5102691669075}{262144}$ & $\frac{34531987615125}{524288}$ & $\frac{52647128659125}{262144}$ & \\
$4$ {\rule{0pt}{2.5ex}}{\rule[-1.8ex]{0pt}{0pt}}& $\frac{221681265259901}{62914560}$ & $\frac{1067126014204999}{113246208}$ & $\frac{966187786983419}{37748736}$ & $\frac{447428553196625}{6291456}$ & $\frac{214834529684775}{1048576}$ & $\frac{658598596580925}{1048576}$ & $\frac{1105589701841625}{524288}$ & $\frac{1737355245751125}{262144}$ \\
\hline
\end{tabular}
\caption{\label{TablG5ndcon} Values of $\pi^{-2(n + 12 - d)}\,H_{5,n}[d]$ computed from the Virasoro constraints (continued).\vspace{-.3cm}}
\end{center}
\end{figure}

\begin{figure}
\begin{center}
\begin{tabular}{|c|ccccccc|}
\hline
$n$ $\backslash$\,\,$d$ {\rule{0pt}{3.2ex}}{\rule[-1.8ex]{0pt}{0pt}}& $0$ & $1$ & $2$ & $3$ & $4$ & $5$ & $6$\\
\hline\hline
$1$ {\rule{0pt}{2.5ex}}{\rule[-1.8ex]{0pt}{0pt}}& $\frac{1725192578138153}{328758271672320}$ & $\frac{51582017261473}{3913788948480}$ & $\frac{118778241943205}{3522410053632}$ & $\frac{28213838434057}{326149079040}$ & $\frac{769565168808731549}{3462616055808000}$ & $\frac{7853909947536887}{13698261319680}$ & $\frac{1678887318135645397}{1130106558873600}$ \\
$2$ {\rule{0pt}{2.5ex}}{\rule[-1.8ex]{0pt}{0pt}}& $\frac{236687293214441}{3478923509760}$ & $\frac{18977118359519683}{109586090557440}$ & $\frac{2072797311487703}{4696546738176}$ & $\frac{13256652596923381}{11741366845440}$ & $\frac{18867920053695301}{6522981580800}$ & $\frac{10295189499821403871}{1385046422323200}$ & $\frac{14445165541157676091}{753404372582400}$ \\
$3$ {\rule{0pt}{2.5ex}}{\rule[-1.8ex]{0pt}{0pt}}& $\frac{19857285082756661}{20873541058560}$ & $\frac{236687293214441}{96636764160}$ & $\frac{341353246695400481}{54793045278720}$ & $\frac{4608517054909483}{289910292480}$ & $\frac{11130008225501778383}{273965226393600}$ & $\frac{271549970945534501}{2609192632320}$ & $\frac{940355203755005824369}{3515887072051200}$ \\
\hline
\end{tabular}
\end{center}
\end{figure}
\begin{figure}[h!]
\begin{center}
\begin{tabular}{|c|cccccc|}
\hline
$n$ $\backslash$\,\,$d$ {\rule{0pt}{3.2ex}}{\rule[-1.8ex]{0pt}{0pt}}& $7$ & $8$ & $9$ & $10$ & $11$ & $12$\\
\hline\hline
$1$ {\rule{0pt}{2.5ex}}{\rule[-1.8ex]{0pt}{0pt}} & $\frac{2208394994566309}{570760888320}$ & $\frac{18384224416373}{1811939328}$ & $\frac{15195072225949}{566231040}$ & $\frac{585329319725327}{8153726976}$ & $\frac{530536293549353}{2717908992}$ & $\frac{246412953611825}{452984832}$ \\
$2$ {\rule{0pt}{2.5ex}}{\rule[-1.8ex]{0pt}{0pt}}& $\frac{534554804193487589}{10762919608320}$ & $\frac{221467745692382411}{171228266496}$ & $\frac{9216122383558207}{27179089920}$ & $\frac{1624463784480019}{1811939328}$ & $\frac{39095700824091245}{16307453952}$ & $\frac{1967183255694475}{301989888}$ \\
$3$ {\rule{0pt}{2.5ex}}{\rule[-1.8ex]{0pt}{0pt}}& $\frac{112491864394018121}{163074539520}$ & $\frac{57694593556278486017}{32288758824960}$ & $\frac{196729681979823383}{42278584320}$ & $\frac{132613775029696513}{10871635968}$ & $\frac{1051799750173130861}{32614907904}$ & $\frac{4218609560944475125}{48922361856}$ \\
\hline
\end{tabular}
\end{center}
\end{figure}
\begin{figure}[h!]
\begin{center}
\begin{tabular}{|c|cccccc|}
\hline
$n$ $\backslash$\,\,$d$ {\rule{0pt}{3.2ex}}{\rule[-1.8ex]{0pt}{0pt}}& $13$ & $14$ & $15$ & $16$ & $17$ & $18$\\
\hline\hline
$1$ {\rule{0pt}{2.5ex}}{\rule[-1.8ex]{0pt}{0pt}}& $\frac{4972554161575}{3145728}$ & $\frac{20794672323425}{4194304}$ & $\frac{17549042886375}{1048576}$ & $\frac{193039471750125}{4194304}$ & &\\
$2$ {\rule{0pt}{2.5ex}}{\rule[-1.8ex]{0pt}{0pt}}& $\frac{304054573898075}{16777216}$ & $\frac{1320243340390975}{25165824}$ & $\frac{684412672568625}{4194304}$ & $\frac{579118415250375}{1048576}$ & $\frac{6756381511254375}{4194304}$ & \\
$3$ {\rule{0pt}{2.5ex}}{\rule[-1.8ex]{0pt}{0pt}}& $\frac{17683899568619375}{75497472}$ & $\frac{24580850021252525}{37748736}$ & $\frac{141960893257558375}{75497472}$ & $\frac{48710293371614875}{8388608}$ & $\frac{164405283440523125}{8388608}$ & $\frac{249986115916411875}{4194304}$ \\
\hline
\end{tabular}
\end{center}
\caption{\label{TablG6nd} Values of $\pi^{-2(n + 15 - d)}\,H_{6,n}[d]$ computed from the Virasoro constraints.}
\end{figure}

\end{landscape}

\begin{figure}[h!]
\begin{center}
\caption{\label{TablFmulti1} Values of $\pi^{-2(3g - 3 + n) + 2d_1 + \cdots + 2d_k}\,F_{g,n}[d_1,\ldots,d_k,0,\ldots,0]$ for $k \geq 2$, computed from the Virasoro constraints.} 
\begin{tabular}{ccc}
\begin{tabular}[t]{|c|c|c|}
\hline 
$(g,n)$ {\rule{0pt}{3.2ex}}{\rule[-1.8ex]{0pt}{0pt}}& $(d_1,\ldots,d_k)$ & $*$ \\
\hline\hline
$(0,5)$ {\rule{0pt}{2.5ex}}{\rule[-1.8ex]{0pt}{0pt}}& $(1,1)$ & $18$ \\
\hline
\hline
\multirow{3}{*}{$(0,6)$} {\rule{0pt}{2.5ex}}{\rule[-1.8ex]{0pt}{0pt}}& $(1,1)$ & $27$ \\
{\rule{0pt}{2.5ex}}{\rule[-1.8ex]{0pt}{0pt}}& $(2,1)$ & $135$ \\
{\rule{0pt}{2.5ex}}{\rule[-1.8ex]{0pt}{0pt}}& $(1,1,1)$ & $162$ \\
\hline\hline
\multirow{7}{*}{$(0,7)$} {\rule{0pt}{2.5ex}}{\rule[-1.8ex]{0pt}{0pt}}& $(1,1)$ & $81$ \\
{\rule{0pt}{2.5ex}}{\rule[-1.8ex]{0pt}{0pt}}& $(2,1)$ &  $300$ \\
{\rule{0pt}{2.5ex}}{\rule[-1.8ex]{0pt}{0pt}}& $(3,1)$ & $1260$ \\
{\rule{0pt}{2.5ex}}{\rule[-1.8ex]{0pt}{0pt}}& $(2,2)$ & $1350$ \\
{\rule{0pt}{2.5ex}}{\rule[-1.8ex]{0pt}{0pt}}& $(1,1,1)$ & $324$ \\
{\rule{0pt}{2.5ex}}{\rule[-1.8ex]{0pt}{0pt}}& $(2,1,1)$ & $1620$ \\
{\rule{0pt}{2.5ex}}{\rule[-1.8ex]{0pt}{0pt}}& $(1,1,1,1)$ & $1944$ \\
\hline\hline
\multirow{13}{*}{$(0,8)$} {\rule{0pt}{2.5ex}}{\rule[-1.8ex]{0pt}{0pt}}& $(1,1)$ & $\frac{675}{2}$ \\
{\rule{0pt}{2.5ex}}{\rule[-1.8ex]{0pt}{0pt}}& $(2,1)$ & $\frac{4575}{4}$ \\
{\rule{0pt}{2.5ex}}{\rule[-1.8ex]{0pt}{0pt}}& $(3,1)$ & $\frac{7875}{2}$ \\
{\rule{0pt}{2.5ex}}{\rule[-1.8ex]{0pt}{0pt}}& $(4,1)$ & $14175$ \\ 
{\rule{0pt}{2.5ex}}{\rule[-1.8ex]{0pt}{0pt}}& $(2,2)$ & $4125$ \\
{\rule{0pt}{2.5ex}}{\rule[-1.8ex]{0pt}{0pt}}& $(3,2)$ & $15750$ \\
{\rule{0pt}{2.5ex}}{\rule[-1.8ex]{0pt}{0pt}}& $(1,1,1)$ & $1215$ \\
{\rule{0pt}{2.5ex}}{\rule[-1.8ex]{0pt}{0pt}}& $(2,1,1)$ & $4500$ \\
{\rule{0pt}{2.5ex}}{\rule[-1.8ex]{0pt}{0pt}}& $(3,1,1)$ & $18900$ \\
{\rule{0pt}{2.5ex}}{\rule[-1.8ex]{0pt}{0pt}}& $(2,2,1)$ & $20250$ \\
{\rule{0pt}{2.5ex}}{\rule[-1.8ex]{0pt}{0pt}}& $(1,1,1,1)$ & $4860$ \\
{\rule{0pt}{2.5ex}}{\rule[-1.8ex]{0pt}{0pt}}& $(2,1,1,1)$ & $24300$ \\
{\rule{0pt}{2.5ex}}{\rule[-1.8ex]{0pt}{0pt}}& $(1,1,1,1,1)$ & $29160$ \\
\hline
\end{tabular} 
&
\begin{tabular}[t]{|c|c|c|}
\hline 
$(g,n)$ {\rule{0pt}{3.2ex}}{\rule[-1.8ex]{0pt}{0pt}}& $(d_1,\ldots,d_k)$ & $*$ \\
\hline\hline
$(1,2)$ {\rule{0pt}{2.5ex}}{\rule[-1.8ex]{0pt}{0pt}}& $(1,1)$ & $\frac{3}{8}$ \\
\hline
\hline
\multirow{3}{*}{$(1,3)$} {\rule{0pt}{2.5ex}}{\rule[-1.8ex]{0pt}{0pt}}& $(1,1)$ & $\frac{3}{2}$ \\
{\rule{0pt}{2.5ex}}{\rule[-1.8ex]{0pt}{0pt}}& $(2,1)$ & $\frac{15}{4}$ \\
{\rule{0pt}{2.5ex}}{\rule[-1.8ex]{0pt}{0pt}}& $(1,1,1)$ & $\frac{9}{4}$ \\
\hline\hline
\multirow{7}{*}{$(1,4)$} {\rule{0pt}{2.5ex}}{\rule[-1.8ex]{0pt}{0pt}}& $(1,1)$ & $\frac{27}{8}$ \\
{\rule{0pt}{2.5ex}}{\rule[-1.8ex]{0pt}{0pt}}& $(2,1)$ &  $\frac{195}{16}$ \\
{\rule{0pt}{2.5ex}}{\rule[-1.8ex]{0pt}{0pt}}& $(3,1)$ & $\frac{315}{8}$ \\
{\rule{0pt}{2.5ex}}{\rule[-1.8ex]{0pt}{0pt}}& $(2,2)$ & $\frac{75}{2}$ \\
{\rule{0pt}{2.5ex}}{\rule[-1.8ex]{0pt}{0pt}}& $(1,1,1)$ & $\frac{27}{2}$ \\
{\rule{0pt}{2.5ex}}{\rule[-1.8ex]{0pt}{0pt}}& $(2,1,1)$ & $\frac{135}{4}$ \\
{\rule{0pt}{2.5ex}}{\rule[-1.8ex]{0pt}{0pt}}& $(1,1,1,1)$ & $\frac{81}{4}$ \\
\hline\hline
\multirow{13}{*}{$(1,5)$} {\rule{0pt}{2.5ex}}{\rule[-1.8ex]{0pt}{0pt}}& $(1,1)$ & $\frac{99}{8}$ \\
{\rule{0pt}{2.5ex}}{\rule[-1.8ex]{0pt}{0pt}}& $(2,1)$ & $\frac{305}{8}$ \\
{\rule{0pt}{2.5ex}}{\rule[-1.8ex]{0pt}{0pt}}& $(3,1)$ & $\frac{525}{4}$ \\
{\rule{0pt}{2.5ex}}{\rule[-1.8ex]{0pt}{0pt}}& $(4,1)$ & $\frac{945}{2}$ \\ 
{\rule{0pt}{2.5ex}}{\rule[-1.8ex]{0pt}{0pt}}& $(2,2)$ & $\frac{1075}{8}$ \\
{\rule{0pt}{2.5ex}}{\rule[-1.8ex]{0pt}{0pt}}& $(3,2)$ & $\frac{3675}{8}$ \\
{\rule{0pt}{2.5ex}}{\rule[-1.8ex]{0pt}{0pt}}& $(1,1,1)$ & $\frac{81}{2}$ \\
{\rule{0pt}{2.5ex}}{\rule[-1.8ex]{0pt}{0pt}}& $(2,1,1)$ & $\frac{585}{4}$ \\
{\rule{0pt}{2.5ex}}{\rule[-1.8ex]{0pt}{0pt}}& $(3,1,1)$ & $\frac{945}{2}$ \\
{\rule{0pt}{2.5ex}}{\rule[-1.8ex]{0pt}{0pt}}& $(2,2,1)$ & $450$ \\
{\rule{0pt}{2.5ex}}{\rule[-1.8ex]{0pt}{0pt}}& $(1,1,1,1)$ & $162$ \\
{\rule{0pt}{2.5ex}}{\rule[-1.8ex]{0pt}{0pt}}& $(2,1,1,1)$ & $405$ \\
{\rule{0pt}{2.5ex}}{\rule[-1.8ex]{0pt}{0pt}}& $(1,1,1,1,1)$ & $243$ \\
\hline
\end{tabular}
& 
\begin{tabular}[t]{|c|c|c|}
\hline 
$(g,n)$ {\rule{0pt}{3.2ex}}{\rule[-1.8ex]{0pt}{0pt}}& $(d_1,\ldots,d_k)$ & $*$ \\
\hline\hline
\multirow{23}{*}{$(1,6)$} {\rule{0pt}{2.5ex}}{\rule[-1.8ex]{0pt}{0pt}}& $(1,1)$ & $\frac{945}{16}$ \\
{\rule{0pt}{2.5ex}}{\rule[-1.8ex]{0pt}{0pt}}& $(2,1)$ & $\frac{11175}{64}$ \\
{\rule{0pt}{2.5ex}}{\rule[-1.8ex]{0pt}{0pt}}& $(3,1)$ & $\frac{16905}{32}$ \\
{\rule{0pt}{2.5ex}}{\rule[-1.8ex]{0pt}{0pt}}& $(4,1)$ & $\frac{14175}{8}$ \\
{\rule{0pt}{2.5ex}}{\rule[-1.8ex]{0pt}{0pt}}& $(5,1)$ & $\frac{51975}{8}$ \\
{\rule{0pt}{2.5ex}}{\rule[-1.8ex]{0pt}{0pt}}& $(2,2)$ & $\frac{6475}{12}$ \\
{\rule{0pt}{2.5ex}}{\rule[-1.8ex]{0pt}{0pt}}& $(3,2)$ & $\frac{29225}{16}$ \\
{\rule{0pt}{2.5ex}}{\rule[-1.8ex]{0pt}{0pt}}& $(4,2)$ & $\frac{51975}{8}$ \\
{\rule{0pt}{2.5ex}}{\rule[-1.8ex]{0pt}{0pt}}& $(3,3)$ & $\frac{25725}{4}$ \\
{\rule{0pt}{2.5ex}}{\rule[-1.8ex]{0pt}{0pt}}& $(1,1,1)$ & $\frac{1485}{8}$ \\
{\rule{0pt}{2.5ex}}{\rule[-1.8ex]{0pt}{0pt}}& $(2,1,1)$ & $\frac{4575}{8}$ \\
{\rule{0pt}{2.5ex}}{\rule[-1.8ex]{0pt}{0pt}}& $(3,1,1)$ & $\frac{7875}{4}$ \\
{\rule{0pt}{2.5ex}}{\rule[-1.8ex]{0pt}{0pt}}& $(4,1,1)$ & $\frac{14175}{2}$ \\
{\rule{0pt}{2.5ex}}{\rule[-1.8ex]{0pt}{0pt}}& $(2,2,1)$ & $\frac{16125}{8}$ \\
{\rule{0pt}{2.5ex}}{\rule[-1.8ex]{0pt}{0pt}}& $(3,2,1)$ & $\frac{55125}{8}$ \\
{\rule{0pt}{2.5ex}}{\rule[-1.8ex]{0pt}{0pt}}& $(2,2,2)$ & $6750$ \\
{\rule{0pt}{2.5ex}}{\rule[-1.8ex]{0pt}{0pt}}& $(1,1,1,1)$ & $\frac{1215}{2}$ \\
{\rule{0pt}{2.5ex}}{\rule[-1.8ex]{0pt}{0pt}}& $(2,1,1,1)$ & $\frac{8775}{4}$ \\
{\rule{0pt}{2.5ex}}{\rule[-1.8ex]{0pt}{0pt}}& $(3,1,1,1)$ & $\frac{14175}{2}$ \\
{\rule{0pt}{2.5ex}}{\rule[-1.8ex]{0pt}{0pt}}& $(2,2,1,1)$ & $6750$ \\
{\rule{0pt}{2.5ex}}{\rule[-1.8ex]{0pt}{0pt}}& $(1,1,1,1,1)$ & $2430$ \\
{\rule{0pt}{2.5ex}}{\rule[-1.8ex]{0pt}{0pt}}& $(2,1,1,1,1)$ & $6075$ \\
{\rule{0pt}{2.5ex}}{\rule[-1.8ex]{0pt}{0pt}}& $(1,1,1,1,1,1)$ & $3645$  \\
\hline
\end{tabular}
\end{tabular}
\end{center}
\end{figure}

\begin{figure}[h!]
\begin{center}
\caption{\label{TablFmulti2} Values of $\pi^{-2(3g - 3 + n) + 2d_1 + \cdots + 2d_k}\,F_{g,n}[d_1,\ldots,d_k,0,\ldots,0]$ for $k \geq 2$ (continued).}

\begin{tabular}{ccc}
\begin{tabular}[t]{|c|c|c|}
\hline
$(g,n)$ {\rule{0pt}{3.2ex}}{\rule[-1.8ex]{0pt}{0pt}}& $(d_1,\ldots,d_k)$ & $*$ \\
\hline\hline
\multirow{6}{*}{$(2,2)$} {\rule{0pt}{2.5ex}}{\rule[-1.8ex]{0pt}{0pt}}& $(1,1)$ & $\frac{9}{32}$ \\
{\rule{0pt}{2.5ex}}{\rule[-1.8ex]{0pt}{0pt}}& $(2,1)$ & $\frac{119}{128}$ \\
{\rule{0pt}{2.5ex}}{\rule[-1.8ex]{0pt}{0pt}}& $(3,1)$ & $\frac{105}{32}$ \\
{\rule{0pt}{2.5ex}}{\rule[-1.8ex]{0pt}{0pt}}& $(4,1)$ & $\frac{945}{128}$ \\
{\rule{0pt}{2.5ex}}{\rule[-1.8ex]{0pt}{0pt}}& $(2,2)$ & $\frac{1225}{384}$ \\
{\rule{0pt}{2.5ex}}{\rule[-1.8ex]{0pt}{0pt}}& $(3,2)$ & $\frac{1015}{128}$ \\
\hline\hline
\multirow{16}{*}{$(2,3)$} {\rule{0pt}{2.5ex}}{\rule[-1.8ex]{0pt}{0pt}}& $(1,1)$ & $\frac{783}{640}$ \\
{\rule{0pt}{2.5ex}}{\rule[-1.8ex]{0pt}{0pt}}& $(2,1)$ & $\frac{225}{64}$ \\
{\rule{0pt}{2.5ex}}{\rule[-1.8ex]{0pt}{0pt}}& $(3,1)$ & $\frac{357}{32}$ \\
{\rule{0pt}{2.5ex}}{\rule[-1.8ex]{0pt}{0pt}}& $(4,1)$ & $\frac{315}{8}$ \\
{\rule{0pt}{2.5ex}}{\rule[-1.8ex]{0pt}{0pt}}& $(5,1)$ & $\frac{3465}{32}$ \\
{\rule{0pt}{2.5ex}}{\rule[-1.8ex]{0pt}{0pt}}& $(2,2)$ & $\frac{25585}{2304}$ \\
{\rule{0pt}{2.5ex}}{\rule[-1.8ex]{0pt}{0pt}}& $(3,2)$ & $\frac{14875}{384}$ \\
{\rule{0pt}{2.5ex}}{\rule[-1.8ex]{0pt}{0pt}}& $(4,2)$ & $\frac{3465}{32}$ \\
{\rule{0pt}{2.5ex}}{\rule[-1.8ex]{0pt}{0pt}}& $(3,3)$ & $\frac{7105}{64}$ \\
{\rule{0pt}{2.5ex}}{\rule[-1.8ex]{0pt}{0pt}}& $(1,1,1)$ & $\frac{27}{8}$ \\
{\rule{0pt}{2.5ex}}{\rule[-1.8ex]{0pt}{0pt}}& $(2,1,1)$ & $\frac{357}{32}$ \\
{\rule{0pt}{2.5ex}}{\rule[-1.8ex]{0pt}{0pt}}& $(3,1,1)$ & $\frac{315}{8}$ \\
{\rule{0pt}{2.5ex}}{\rule[-1.8ex]{0pt}{0pt}}& $(4,1,1)$ & $\frac{2835}{32}$ \\
{\rule{0pt}{2.5ex}}{\rule[-1.8ex]{0pt}{0pt}}& $(2,2,1)$ & $\frac{1225}{32}$ \\
{\rule{0pt}{2.5ex}}{\rule[-1.8ex]{0pt}{0pt}}& $(3,2,1)$ & $\frac{3045}{32}$ \\
{\rule{0pt}{2.5ex}}{\rule[-1.8ex]{0pt}{0pt}}& $(2,2,2)$ & $\frac{1575}{16}$ \\
\hline
\end{tabular}
&
\begin{tabular}[t]{|c|c|c|}
\hline
$(g,n)$ {\rule{0pt}{3.2ex}}{\rule[-1.8ex]{0pt}{0pt}}& $(d_1,\ldots,d_k)$ & $*$ \\
\hline\hline
\multirow{29}{*}{$(2,4)$} {\rule{0pt}{2.5ex}}{\rule[-1.8ex]{0pt}{0pt}}& $(1,1)$ & $\frac{1685}{256}$ \\
{\rule{0pt}{2.5ex}}{\rule[-1.8ex]{0pt}{0pt}}& $(2,1)$ & $\frac{6995}{384}$ \\
{\rule{0pt}{2.5ex}}{\rule[-1.8ex]{0pt}{0pt}}& $(3,1)$ & $\frac{13475}{256}$ \\
{\rule{0pt}{2.5ex}}{\rule[-1.8ex]{0pt}{0pt}}& $(4,1)$ & $\frac{41685}{256}$ \\
{\rule{0pt}{2.5ex}}{\rule[-1.8ex]{0pt}{0pt}}& $(5,1)$ & $\frac{144375}{256}$ \\
{\rule{0pt}{2.5ex}}{\rule[-1.8ex]{0pt}{0pt}}& $(6,1)$ & $\frac{225225}{128}$ \\
{\rule{0pt}{2.5ex}}{\rule[-1.8ex]{0pt}{0pt}}& $(2,2)$ & $\frac{241825}{4608}$ \\
{\rule{0pt}{2.5ex}}{\rule[-1.8ex]{0pt}{0pt}}& $(3,2)$ & $\frac{125167}{768}$ \\
{\rule{0pt}{2.5ex}}{\rule[-1.8ex]{0pt}{0pt}}& $(4,2)$ & $\frac{72135}{128}$ \\
{\rule{0pt}{2.5ex}}{\rule[-1.8ex]{0pt}{0pt}}& $(5,2)$ & $\frac{3465}{2}$ \\
{\rule{0pt}{2.5ex}}{\rule[-1.8ex]{0pt}{0pt}}& $(3,3)$ & $\frac{71785}{128}$ \\
{\rule{0pt}{2.5ex}}{\rule[-1.8ex]{0pt}{0pt}}& $(4,3)$ & $\frac{112455}{64}$ \\
{\rule{0pt}{2.5ex}}{\rule[-1.8ex]{0pt}{0pt}}& $(1,1,1)$ & $\frac{2349}{128}$ \\
{\rule{0pt}{2.5ex}}{\rule[-1.8ex]{0pt}{0pt}}& $(2,1,1)$ & $\frac{3375}{64}$ \\
{\rule{0pt}{2.5ex}}{\rule[-1.8ex]{0pt}{0pt}}& $(3,1,1)$ & $\frac{5355}{32}$ \\
{\rule{0pt}{2.5ex}}{\rule[-1.8ex]{0pt}{0pt}}& $(4,1,1)$ & $\frac{4725}{8}$ \\
{\rule{0pt}{2.5ex}}{\rule[-1.8ex]{0pt}{0pt}}& $(5,1,1)$ & $\frac{51975}{32}$ \\
{\rule{0pt}{2.5ex}}{\rule[-1.8ex]{0pt}{0pt}}& $(2,2,1)$ & $\frac{127925}{768}$ \\
{\rule{0pt}{2.5ex}}{\rule[-1.8ex]{0pt}{0pt}}& $(3,2,1)$ & $\frac{74375}{128}$ \\
{\rule{0pt}{2.5ex}}{\rule[-1.8ex]{0pt}{0pt}}& $(3,3,1)$ & $\frac{106575}{64}$ \\
{\rule{0pt}{2.5ex}}{\rule[-1.8ex]{0pt}{0pt}}& $(2,2,2)$ & $\frac{18375}{32}$ \\
{\rule{0pt}{2.5ex}}{\rule[-1.8ex]{0pt}{0pt}}& $(3,2,2)$ & $\frac{13125}{8}$ \\
{\rule{0pt}{2.5ex}}{\rule[-1.8ex]{0pt}{0pt}}& $(1,1,1,1)$ & $\frac{405}{8}$ \\
{\rule{0pt}{2.5ex}}{\rule[-1.8ex]{0pt}{0pt}}& $(2,1,1,1)$ & $\frac{5355}{32}$ \\
{\rule{0pt}{2.5ex}}{\rule[-1.8ex]{0pt}{0pt}}& $(3,1,1,1)$ & $\frac{4725}{8}$ \\
{\rule{0pt}{2.5ex}}{\rule[-1.8ex]{0pt}{0pt}}& $(4,1,1,1)$ & $\frac{42525}{32}$ \\
{\rule{0pt}{2.5ex}}{\rule[-1.8ex]{0pt}{0pt}}& $(2,2,1,1)$ & $\frac{18375}{32}$ \\
{\rule{0pt}{2.5ex}}{\rule[-1.8ex]{0pt}{0pt}}& $(3,2,1,1)$ & $\frac{45675}{32}$ \\
{\rule{0pt}{2.5ex}}{\rule[-1.8ex]{0pt}{0pt}}& $(2,2,2,1)$ & $\frac{23625}{16}$    \\
\hline 
\end{tabular}
&
\begin{tabular}[t]{|c|c|c|}
\hline
$(g,n)$ {\rule{0pt}{3.2ex}}{\rule[-1.8ex]{0pt}{0pt}}& $(d_1,\ldots,d_k)$ & $*$ \\
\hline\hline
\multirow{16}{*}{$(3,2)$} {\rule{0pt}{2.5ex}}{\rule[-1.8ex]{0pt}{0pt}}& $(1,1)$ & $\frac{8625}{14336}$ \\
{\rule{0pt}{2.5ex}}{\rule[-1.8ex]{0pt}{0pt}}& $(2,1)$ & $\frac{40495}{24576}$ \\
{\rule{0pt}{2.5ex}}{\rule[-1.8ex]{0pt}{0pt}}& $(3,1)$ & $\frac{56749}{12288}$ \\
{\rule{0pt}{2.5ex}}{\rule[-1.8ex]{0pt}{0pt}}& $(4,1)$ & $\frac{41015}{3072}$ \\
{\rule{0pt}{2.5ex}}{\rule[-1.8ex]{0pt}{0pt}}& $(5,1)$ & $\frac{87395}{2048}$ \\
{\rule{0pt}{2.5ex}}{\rule[-1.8ex]{0pt}{0pt}}& $(6,1)$ & $\frac{75075}{512}$ \\
{\rule{0pt}{2.5ex}}{\rule[-1.8ex]{0pt}{0pt}}& $(7,1)$ & $\frac{375375}{1024}$ \\
{\rule{0pt}{2.5ex}}{\rule[-1.8ex]{0pt}{0pt}}& $(2,2)$ & $\frac{56765}{12288}$ \\
{\rule{0pt}{2.5ex}}{\rule[-1.8ex]{0pt}{0pt}}& $(3,2)$ & $\frac{743917}{55296}$ \\
{\rule{0pt}{2.5ex}}{\rule[-1.8ex]{0pt}{0pt}}& $(4,2)$ & $\frac{87353}{2048}$ \\
{\rule{0pt}{2.5ex}}{\rule[-1.8ex]{0pt}{0pt}}& $(5,2)$ & $\frac{297605}{2048}$ \\
{\rule{0pt}{2.5ex}}{\rule[-1.8ex]{0pt}{0pt}}& $(6,2)$ & $\frac{385385}{1024}$ \\
{\rule{0pt}{2.5ex}}{\rule[-1.8ex]{0pt}{0pt}}& $(3,3)$ & $\frac{131467}{3072}$ \\
{\rule{0pt}{2.5ex}}{\rule[-1.8ex]{0pt}{0pt}}& $(4,3)$ & $\frac{37485}{256}$ \\
{\rule{0pt}{2.5ex}}{\rule[-1.8ex]{0pt}{0pt}}& $(5,3)$ & $\frac{193655}{512}$ \\
{\rule{0pt}{2.5ex}}{\rule[-1.8ex]{0pt}{0pt}}& $(4,4)$ & $\frac{191205}{512}$ \\
\hline
\end{tabular}
\end{tabular}
\end{center}
\end{figure}

\clearpage 
\bibliographystyle{plain}
\bibliography{BibliMV}

\end{document}